\documentclass{amsart}
\newcommand{\RR}[1]{\textcolor{red}{#1}}
\RequirePackage{ntrees}
\usepackage{amssymb}
\usepackage{amsmath}
\usepackage{amsfonts}
\usepackage{geometry}
\usepackage{bbm}
\usepackage{hyperref}
\usepackage{tikz}
\usepackage{enumitem}
\usepackage[utf8]{inputenc}
\usepackage{xcolor}
\usepackage{mathrsfs}
\usepackage{multirow}
\usetikzlibrary{matrix,arrows,decorations.pathmorphing}
\usepackage{verbatim }
\usepackage{mathtools}
\usepackage{mathrsfs}
\usetikzlibrary{arrows.meta, positioning}
\usepackage[author={Sebastian}]{pdfcomment}
\newcommand{\triplenorm}[1]{\lvert\!\lVert #1 \rVert\!\rvert}

\newcounter{cprop}[section]

\newtheorem{theorem}[cprop]{Theorem}
\usepackage{tikz}
\usetikzlibrary{arrows.meta,calc,decorations.markings}

\definecolor{baseblue}{RGB}{0,28,210}
\definecolor{stablegreen}{RGB}{0,107,14}
\definecolor{unstablered}{RGB}{255,18,13}
\definecolor{fiberpurple}{RGB}{92,0,157}

\usepackage{tikz}
\usetikzlibrary{decorations.markings}

\theoremstyle{plain}

\newtheorem{corollary}[cprop]{Corollary}

\newtheorem{lemma}[cprop]{Lemma}
\newtheorem{proposition}[cprop]{Proposition}

\newtheorem{assumption}[cprop]{Assumption}

\numberwithin{equation}{section}

\newtheorem{definition}[cprop]{Definition}

\theoremstyle{remark}
\newtheorem{remark}[cprop]{Remark}

\newcommand{\E}{\mathbb{E}}
\renewcommand{\P}{\mathbb{P}}

\newcommand{\R}{\mathbb{R}}

\newcommand{\Z}{\mathbb{Z}}
\newcommand{\T}{\mathbb{T}}
\renewcommand{\d}{\mathrm{d}}

\newcommand{\vertiii}[1]{{\left\vert\kern-0.25ex\left\vert\kern-0.25ex\left\vert #1 
		\right\vert\kern-0.25ex\right\vert\kern-0.25ex\right\vert}}

\begin{document}
	\title[Invariant manifolds for Singular SPDEs]{Invariant Manifolds and Stability for Singular SPDEs
	}
	
	\author{M. Ghani Varzaneh}
	\address{Mazyar Ghani Varzaneh\\
		Fachbereich Mathematik und Statistik, Universität Konstanz, Konstanz, Germany}
	\email{mazyar.ghani-varzaneh@uni-konstanz.de \ \ and \ \ mazyarghani69@gmail.com}

\subjclass[2020]{60H17, 60L30, 37H10, 37H15, 37H30, 37L55.}

\keywords{singular stochastic partial differential equations, invariant manifolds, regularity structures, random dynamical systems, Lyapunov exponents, stability}

	\begin{abstract}
We study the long-time dynamics of a class of singular stochastic partial differential equations with multiplicative Gaussian noise. Using the theory of regularity structures, we construct the associated random dynamical system and establish an ergodic framework for the underlying random models. Combining pathwise estimates for the singular equation with a multiplicative ergodic theorem in Banach spaces, we analyse the Lyapunov spectrum of the linearized dynamics and establish the existence of local stable, unstable, and center invariant manifolds around stationary points. As a consequence, when the linearized equation is exponentially stable, we obtain a local exponential stability result for the nonlinear equation.
	\end{abstract}

	\maketitle
	\tableofcontents
	\section*{Introduction}
	The main goal of this work is to establish the existence of invariant manifolds for a class of SPDEs with multiplicative noise. More precisely, we consider semilinear stochastic partial differential equations of the form
	\begin{align}\label{MAIN0}
		\begin{cases}
			(\partial_t-\Delta)u=\mu u+\Sigma(u)\xi, \quad t>0,\\
			u(0,\cdot)=u_0\in C(\mathbb{T}^d).
		\end{cases}
	\end{align}
	Here, $\mu\leq 0$, $\Sigma$ denotes the multiplicative coefficient, and $\xi$ is a spatially and temporally irregular noise, interpreted within the framework of regularity structures. We assume that the temporal component of the noise is generated by a process with stationary increments. In the setting of stochastic differential equations, typical examples are Brownian motion and fractional Brownian motion.
	
	\noindent\textbf{Motivation and main idea.}
	The motivation for studying invariant manifolds in this setting comes from the theory of random dynamical systems (RDS), which provides a natural framework for understanding the long-time behaviour of stochastic equations. In particular, one is interested in how randomness influences the stability and asymptotic behaviour of solutions, including the existence of invariant objects, synchronization phenomena, and Lyapunov exponents.
	Random dynamical systems extend classical concepts from dynamical systems to systems evolving under random forcing. A central object in this framework is a family of random maps describing the evolution of the state from a given initial condition under a realization of the noise, together with the temporal evolution of the noise itself. The theory has been developed extensively for stochastic differential equations and stochastic evolution equations; see, for instance, the monograph \cite{Arn98} for a comprehensive treatment and an account of its applications to stochastic dynamics.
	
	The construction of a random dynamical system for an SPDE requires more than the existence of solutions for individual initial conditions. One also needs suitable continuity of the resulting solution flow with respect to the initial condition. For stochastic differential equations in finite dimensions, a standard approach is to regard the solution as a random function of the initial condition and apply the Kolmogorov--Chentsov continuity theorem. Suitable moment estimates for the increments of the flow then yield the required continuity. This argument, however, relies essentially on the fact that the initial condition ranges over a finite-dimensional parameter space. For an SPDE, the initial condition belongs to an infinite-dimensional state space, and this standard application of the Kolmogorov--Chentsov theorem is therefore no longer available. Consequently, establishing the continuity with respect to the initial condition required for the construction of a random dynamical system becomes a fundamental difficulty in the SPDE setting.
	
	A common way to overcome this difficulty is to transform the SPDE into a random PDE, so that, for each realization of the noise, the resulting equation can be treated pathwise without stochastic integration. This approach is particularly effective for equations with additive noise. Its applicability to multiplicative noise, however, is considerably more restricted: such a transformation is generally available only under particular structural assumptions on the diffusion coefficient, for example when it depends linearly on the solution, as in \cite{DLS04}. For a discussion of stochastic flows for SPDEs, see also \cite[Section~9.1.2]{DZ14}. It is worth mentioning that the existence of a stochastic flow can already be nontrivial in the finite-dimensional setting; see \cite{LS11}.
	
	These limitations make the theory of random dynamical systems considerably less accessible for SPDEs with genuinely multiplicative noise. When the noise is additive, the construction of the associated random dynamical system is, in many important settings, relatively well understood; see, for example, \cite{MZ08}. This has enabled the study of important dynamical questions, most notably the existence and properties of random attractors; see \cite{CF94,BLW09}. The restriction to additive noise, however, leaves outside this framework many problems in which the interaction between the noise and the state of the system plays an essential role. This motivates the development of methods capable of treating genuinely multiplicative noise.
	
	In recent years, several approaches have been developed to overcome related difficulties and extend the theory of random dynamical systems to more general classes of stochastic equations. Among these, rough path theory has played an important role; see \cite{FV10,GH19,FH20}. A central idea of this theory is that, for sufficiently irregular driving signals, the first-order increments of the path do not contain enough information to define integration against the driving signal. One therefore enriches the driving signal by higher-order iterated integrals, which encode the interactions between its increments. For Gaussian signals, these iterated integrals can in many cases be constructed canonically.
	Rough path theory has subsequently been incorporated into the random dynamical systems framework, leading to a number of results concerning the long-time behaviour of systems driven by irregular noise; see \cite{BRS17}. In this direction, several works have established the existence of random attractors for rough PDEs; see \cite{YLZ23,BNS24}. Related results on invariant manifolds can be found in \cite{GVR25}, where more refined properties of the rough-path structure are exploited to obtain the estimates required for their construction. Similar ideas have also proved useful for stochastic delay equations, where the standard stochastic formulation does not directly provide the dynamical structure required for an RDS; see \cite{GVRS22,GVR25B}.
	
	This rough-path perspective has proved particularly effective for constructing and analysing stochastic flows driven by irregular signals depending only on time. For many SPDEs, however, the situation is fundamentally different, since the noise depends simultaneously on time and space and may be irregular in both variables. A natural framework for treating such equations is provided by the theory of regularity structures, introduced by Hairer; see \cite{Hai14}. This theory is closely related in spirit to rough path theory, but is designed to handle the substantially richer local structure arising from singular space--time distributions.
	Regularity structures provide an abstract framework for describing the local behaviour of highly irregular distributions and for defining the nonlinear operations required to formulate the corresponding equations. The theory has proved particularly successful for singular SPDEs, where classical analytical operations, such as multiplication of distributions, cannot in general be performed directly. These features make regularity structures a natural candidate for combining the analysis of singular SPDEs with the theory of random dynamical systems.
	In particular, the continuity properties of the solution map can be studied directly within the pathwise solution theory, rather than relying on an application of the Kolmogorov--Chentsov theorem with the initial condition treated as a parameter. This provides a natural route towards constructing the dynamical-system structure required in the infinite-dimensional setting and, in turn, towards applying tools from random dynamical systems to singular SPDEs.
	
	Following the foundational work of Hairer, the theory of regularity structures has developed rapidly, with substantial advances in the well-posedness, renormalization, and analytical properties of broad classes of singular SPDEs; see \cite{Ha13,HP15,BHZ19,MW20,CAW23}. Nevertheless, much of this literature has focused on the construction and properties of solutions rather than on their long-time dynamical behaviour. To the best of our knowledge, apart from a few works such as \cite{R22,GT20}, which consider singular SPDEs with additive noise, the interaction between regularity structures and random dynamical systems has not yet been systematically developed. One of the aims of the present work is to take a step in this direction.
	Establishing such a connection is useful from both perspectives. On the one hand, random dynamical systems provide a rich collection of tools for studying the long-time behaviour of stochastic systems, including invariant manifolds, random attractors, synchronization, Lyapunov exponents, and stability. On the other hand, the analytical machinery of regularity structures makes it possible to treat classes of singular SPDEs that are not accessible through classical approaches to random dynamical systems. Bringing these two theories together therefore provides a framework in which genuinely dynamical questions for singular SPDEs can be formulated and investigated.
	
	The main contribution of the present work, however, goes beyond the construction of this connection. Our principal objective is the study of invariant manifolds and stability near stationary points, or random equilibria, for singular SPDEs. Once such a stationary point has been identified, a fundamental question in dynamical systems is to understand the behaviour of solutions in its neighbourhood. Invariant-manifold theory provides one of the principal tools for this purpose, since it gives a nonlinear geometric description of the local dynamics.
	The basic idea is to analyse the linearization of the dynamics around the stationary point and to use the spectral properties of the resulting linearized dynamics to identify directions along which the system exhibits stable, unstable, or neutral behaviour. Under suitable conditions, these linear structures persist at the nonlinear level in the form of invariant manifolds. In this way, invariant manifolds provide a geometric description of the nonlinear dynamics that reflects the spectral structure of the linearized equation.
	
	For deterministic dynamical systems, this theory has been extensively developed, leading to three fundamental types of invariant manifolds: stable, unstable, and center manifolds; see \cite{Wig03,GH83}. Stable and unstable manifolds describe, respectively, the directions along which trajectories approach or depart from an equilibrium, and their existence provides precise information about the corresponding local dynamics. In particular, stable manifolds are closely connected to questions of local stability, as they characterize initial conditions whose trajectories remain near and converge towards the equilibrium.
	Center manifolds arise in the presence of neutral directions, typically associated with eigenvalues having vanishing real part. They reduce the analysis to an invariant set associated with these directions and can capture nonlinear phenomena that are not visible at the level of the linearization. Center manifolds are particularly useful for analysing bifurcations and more intricate forms of local dynamical behaviour; see \cite{Car81,VF13,GW13}.
	
	These invariant manifolds also have natural counterparts in the theory of random dynamical systems. The main difference is that, in the random setting, the invariant manifolds themselves depend on the realization of the driving noise and therefore form random families of sets. Their evolution is coupled to the temporal shift of the noise: as the noise is shifted in time, the corresponding random invariant manifold changes accordingly. Thus, the deterministic notion of an invariant manifold is replaced by a random invariant family whose evolution is governed jointly by the dynamics of the system and the underlying noise. This provides the natural framework for extending stable, unstable, and center manifold theory to stochastic dynamical systems.
	
	Despite this close analogy, the construction of random invariant manifolds presents substantial additional difficulties. In particular, the construction of stable manifolds is technically demanding in infinite-dimensional systems, since it requires control of the dynamics over arbitrarily long time intervals. In the random setting, the relevant estimates depend on the realization of the noise and must remain sufficiently controlled along the entire forward trajectory.
	The unstable-manifold problem introduces a different difficulty when the underlying solution flow is not invertible, since one cannot rely on a globally defined backward evolution. This point is particularly relevant in the present setting, where the solution map is naturally defined only forward in time. Related results on invariant manifolds for stochastic partial differential equations can be found, for example, in
	\cite{DLS004,LS207,MZ08,MZ10,GLS010,TDLS10,MS11,JKB20,WQT23}.
	
	The application of invariant-manifold theory to stochastic equations remains an active area of research, and further developments are needed in settings involving singular equations and infinite-dimensional state spaces. As mentioned above, the construction is particularly delicate in the Banach-space setting; for the Hilbert-space case, see \cite{Rue82}. In the present work, we approach this problem through the multiplicative ergodic theorem (MET) in Banach spaces.
	The multiplicative ergodic theorem provides the Lyapunov spectral structure of the linearized dynamics and, under appropriate assumptions, the corresponding invariant splittings; see \cite{LL10,GTQ15,Blu16}. Once suitable bounds on the nonlinear dynamics have been established, we analyse the linearized equation and identify the relevant Lyapunov subspaces. We then verify the \emph{first-order approximation condition} introduced in Definition~\ref{first-order-condition}, which allows us to deduce the existence of the corresponding invariant manifolds.
	
	The Banach-space formulation of the MET is essential for our analysis. Indeed, the dynamics are formulated on the space $C(\mathbb{T}^d)$, which is a Banach space but not a Hilbert space. Another important feature of our setting is that the solution map is not invertible. Consequently, the dynamics are naturally defined only in the forward time direction, and in general no global backward solution flow is available.
	This approach has several advantages. From a technical point of view, the construction can be reduced to the verification of suitable estimates for the linearized dynamics and the nonlinear remainder, making the argument more modular than a direct construction based on the Lyapunov--Perron method. Moreover, the same framework yields stable, unstable, and center manifolds within a unified argument, rather than requiring separate constructions for each type of manifold. We apply this approach to the SPDE~\eqref{MAIN0}.
	
	\noindent\textbf{Main contributions.}
Let us now summarize the main contributions of this paper, which, to the best of the author's knowledge, are new in the literature.
	\begin{enumerate}
		
		\item
		We establish the existence of stable, unstable, and center invariant manifolds for the SPDE~\eqref{MAIN0} within the framework of regularity structures; see Theorems~\ref{stable_manifold}, \ref{unstable}, and \ref{center}. The construction combines the pathwise solution theory provided by regularity structures with the Lyapunov spectral information furnished by the multiplicative ergodic theorem.
		
		At an informal level, the main invariant-manifold result can be stated as follows.
		\begin{theorem}
			Let $Y$ be a stationary point for the random dynamical system generated by \eqref{MAIN0}. Then the dynamics admit local stable manifolds associated with the negative Lyapunov spectrum. If the Lyapunov spectrum contains positive, respectively zero, exponents, then local unstable, respectively center, manifolds also exist.
		\end{theorem}
		
		\item
The existence of the stable manifold yields a nonlinear stability result around the stationary point. More precisely, whenever the linearized equation around the stationary point is exponentially stable, solutions starting sufficiently close to it converge exponentially towards the corresponding stationary trajectory. In Theorem~\ref{stabi}, we provide a concrete criterion for this situation when the stationary point is zero and the noise intensity is sufficiently small. Thus, exponential stability of the linearized random dynamics can be transferred to the nonlinear SPDE.
		
		\item
		The unstable- and center-manifold results provide complementary information about the local dynamics. The center manifold reduces the dynamics to the directions associated with the central Lyapunov spectrum and therefore provides a natural framework for analysing local behaviour near bifurcation regimes. The unstable-manifold result is particularly relevant in our setting because we do not assume that the solution flow is invertible. Indeed, the dynamics are naturally defined only in the forward time direction and do not, in general, admit a global backward evolution. Our construction therefore produces an unstable manifold without requiring invertibility of the underlying flow.
		
		\item
		Although regularity structures are naturally suited to a pathwise formulation and therefore have a natural connection with random dynamical systems, implementing this connection rigorously requires the verification of several structural properties. In particular, one must establish suitable measurability, stationarity, and ergodicity properties of the objects entering the solution theory. Our analysis addresses these issues and places the solution map in the framework required for the application of ergodic-theoretic results, including the multiplicative ergodic theorem; see Section~\ref{PRO}.
		
		The specific choice of noise considered here is motivated mainly by the desire to remain close to the classical setting of SPDEs. Nevertheless, a substantial part of the analytical and dynamical framework is formulated independently of this particular choice and may be applicable to other temporal noise structures whenever the corresponding pathwise and global solution theory is available. Moreover, the framework allows stationary points that are genuinely random. This is particularly relevant in synchronization problems, where random equilibria naturally arise as states towards which solutions may synchronize and where the stable-manifold result provides a tool for analysing their local stability.
		
		\item
		We also obtain several analytical estimates that are of interest independently of the particular invariant-manifold problem considered in this paper. In particular, the version of Gronwall's lemma developed in Proposition~\ref{onwall}, together with the estimates established throughout Section~\ref{GPAM}, is formulated in a relatively general setting. These estimates control the solution map, its linearization, and the nonlinear remainder, and may therefore be useful in related problems involving long-time estimates for singular stochastic equations.
		
	\end{enumerate}
	
	In this work, we consider noise of regularity $\mathcal{C}^{-1-\kappa}$ with $\kappa\in(0,1/3)$ and assume that the corresponding global solution theory and the a priori bounds required in our analysis are available. Such bounds, together with global well-posedness, are established in \cite{CFW26,SZZ26} for sufficiently small $\kappa$. Very recently, a priori bounds for the generalized parabolic Anderson model have also been obtained in \cite{ES26} throughout the full subcritical regime. Thus, the estimates required in the present work are available from these results in their respective regimes.
	The restriction to $\kappa<1/3$ is mainly technical. Treating $\kappa\in[1/3,1)$ directly would require a substantially more involved regularity-structure construction, without introducing essentially new ideas for the invariant-manifold and stability analysis. We therefore restrict the detailed construction to $\kappa<1/3$. Nevertheless, the framework and technical arguments developed here for invariant manifolds and stability extend to the full subcritical regime. In particular, once the corresponding global solution theory and the required integrability estimates are available, the main invariant-manifold and stability results remain valid in this more general setting.
	
	\medskip
	\noindent\textbf{Structure of the paper.}
	In Section~\ref{pre}, we provide a concise and accessible presentation of the basic elements of the theory of regularity structures needed in the sequel, following \cite{Hai14}. In Section~\ref{GPAM}, we establish several analytical estimates under fairly general assumptions. The main results of this section are Lemma~\ref{global}, Proposition~\ref{onwall}, and Proposition~\ref{UJUJUJU}, which provide the key estimates used in the subsequent dynamical analysis.
	In Section~\ref{PRO}, we develop the framework needed to incorporate the SPDE~\eqref{MAIN0} into the theory of random dynamical systems. In particular, we establish the structural properties required for the later application of ergodic-theoretic arguments. The main results of this section are Proposition~\ref{Ujasn785sw} and Theorem~\ref{Ujasm}. Finally, in Section~\ref{TNT}, we combine the analytical and dynamical ingredients developed in Sections~\ref{GPAM} and~\ref{PRO} to establish the main results of the paper. These include Theorems~\ref{stable_manifold}, \ref{unstable}, \ref{center}, and \ref{stabi}.

	\section{Preliminaries}\label{pre}
	Let us now briefly state the main technical tools, which are mainly taken from \cite{Hai14}. Note that, in this reference, the results are formulated in a more general and abstract framework. For the purposes of this manuscript, we simplify these tools and adapt them to our specific setting.
	\subsection{Notations}
	In this subsection we collect some notations and conventions which will be used
	throughout the manuscript.
	\begin{itemize}
	\item Throughout this manuscript, when we say that an object is spatially
	periodic, we mean that it is $2\pi\mathbb Z^d$-periodic in the spatial
	variables, or equivalently that it is defined on the torus
	\[
	\mathbb T^d=(\mathbb R/2\pi\mathbb Z)^d.
	\]
	This convention applies to functions, distributions, and modelled
	distributions.
	
	\item For $z=(t,x)$ and $\tilde z=(\tilde t,\tilde x)$ in
	$\mathbb R^{d+1}$, define
	\[
	d(z,\tilde z)
	:=
	\max\left\{
	\sqrt{|t-\tilde t|},\,\|x-\tilde x\|_{\R^d}
	\right\},
	\]
	where $\|\cdot\|$ denotes the Euclidean norm on $\mathbb R^d$.
	For points $z=(t,x)$ and $\tilde z=(\tilde t,\tilde x)$ in
	$\mathbb R\times\mathbb T^d$, we use the same notation $d$ for
	the quotient metric
	\[
	d(z,\tilde z)
	:=
	\max\left\{
	\sqrt{|t-\tilde t|},\,
	\|x-\tilde x\|_{\mathbb T^d}
	\right\},
	\]
	where, using any representatives of $x$ and $\tilde x$,
	\[
	\|x-\tilde x\|_{\mathbb T^d}
	:=
	\inf_{k\in\mathbb Z^d}\|x-\tilde x+2\pi k\|.
	\]
	For $L>0$ and $z=(t,x)\in\mathbb R\times\mathbb T^d$, set
	\[
	B(z,L)
	:=
	\left\{
	\tilde z=(\tilde t,\tilde x)\in\mathbb R\times\mathbb T^d:
	d(z,\tilde z)<L,\ \tilde t<t
	\right\},
	\]
	and denote its closure by $\overline{B(z,L)}$.
	For each $i=1,\ldots,d$, we define
	\[
	X^i:\mathbb R^{d+1}\to\mathbb R,
	\qquad
	X^i(z):=x_i,
	\]
	where $z=(t,x_1,\ldots,x_d)\in\mathbb R^{d+1}$.
	Finally, if $(Y,d_Y)$ is a metric space, we denote by $B_Y(y,M)$
	the open ball of radius $M>0$ centered at $y\in Y$, and by
	$\overline{B_Y(y,M)}$ its closure.
		\item We may regard functions defined on $\mathbb{R}\times\mathbb{T}^d$ as functions defined on $\mathbb{R}^{d+1}$ in the canonical way, without explicitly mentioning it.
		
		\item For a topological space $\mathcal{X}$, $\mathcal{B}(\mathcal{X})$ denotes the Borel $\sigma$-algebra on $\mathcal{X}$.
		
		\item For $k\in\mathbb{N}_0$ and Banach spaces $V$ and $W$, we denote by
		$C^{k}(V,W)$ the space of mappings from $V$ to $W$ that are $k$ times
		continuously Fréchet differentiable. In particular, for $k\in\mathbb{N}_0$,
		$n,m\in\mathbb{N}$, and an open set $D\subseteq\mathbb{R}^{n}$,
		$C^{k}(D,\mathbb{R}^{m})$ denotes the space of functions whose derivatives
		up to order $k$ exist and are continuous. The subset of
		$C^{k}(D,\mathbb{R}^{m})$ consisting of functions whose derivatives up to
		order $k$ are bounded is denoted by $C^{k}_{b}(D,\mathbb{R}^{m})$.
		Moreover, $C^{k}_{c}(D,\mathbb{R}^{m})$ denotes the space of compactly
		supported functions in $C^{k}(D,\mathbb{R}^{m})$.
		
		\item For $\nu\in(0,1]$, we use the notation
		\[
	C_b^{3,\nu}(\mathbb{R},\mathbb{R})
		:=
		\left\{
		f\in C^3(\mathbb{R},\mathbb{R}):
		\|f\|_{C_b^3}+[f''']_{\nu}<\infty
		\right\},
		\]
		where
		\[
		\|f\|_{C_b^3}:=\max_{0\leq j\leq3}\|f^{(j)}\|_\infty,
		\qquad
		[f''']_{\nu}:=
		\sup_{x\neq y}
		\frac{|f'''(x)-f'''(y)|}{|x-y|^\nu}.
		\]
	Analogously, we define $C_b^{2,\nu}(\mathbb{R},\mathbb{R})$.
		\item We denote by $\mathcal{D}'(\mathbb{R}^{d+1})$ the space of distributions on
		$\mathbb{R}^{d+1}$, and by $\mathcal{S}'(\mathbb{R}^{d+1})$ the space of tempered
		distributions on $\mathbb{R}^{d+1}$.
		\item For $s,T\in\mathbb{R}$, let
		\[
		\mathbf{R}^{+}_{s}(t,x):=\mathbf{1}_{\{t>s\}},
		\qquad
		O_{T}:=(-\infty,T]\times\mathbb{R}^{d}.
		\]
		Also, for $a,b\in\mathbb{R}$ with $a<b$, we set
		\begin{align}
			\begin{split}
				&O_b^a:=(a,b]\times\mathbb{R}^d,
				\\
				&\overline{O_b^a}:=[a,b]\times\mathbb{R}^d.
			\end{split}
		\end{align}
		
		\item Let $z\in\mathbb{R}^{d+1}$, $l>0$, and
		$Q:\mathbb{R}^{d+1}\to\mathbb{R}$. We set
		\begin{align}\label{AAASSsssww}
			\begin{split}
				&Q^{l}_{z}:\mathbb{R}^{d+1}\rightarrow\mathbb{R},
				\\
				&Q^{l}_{z}(\tilde{z})
				:=l^{-2-d}Q\left(
				\frac{\tilde{t}-t}{l^2},
				\frac{\tilde{x}-x}{l}
				\right).
			\end{split}
		\end{align}
		Also, $Q^{l}:=Q^{l}_{0}$.
		
		\item For $r\in\mathbb{N}$, we set
		\[
\mathcal{B}^{r}_{0}
:=
\left\{
Q\in C_c^\infty(\mathbb{R}^{d+1},\mathbb{R}):
\|Q\|_{C_b^r}\leq1,\;
\operatorname{supp}(Q)\subseteq
B_{\mathbb{R}^{d+1}}(0,1)
\right\}.
\]
		Let $\alpha<0$ and choose $r\in\mathbb{N}$ such that $r>-\alpha$.
		We denote by $\mathcal{C}^{\alpha}_{(2,1)}$ the space of
		$\xi\in\mathcal{D}'(\mathbb{R}^{d+1})$ such that, for every compact set
		$\mathcal{O}\subseteq\mathbb{R}^{d+1}$,
		\[
		\|\xi\|_{\alpha,\mathcal{O}}
		:=
		\sup_{\substack{z\in\mathcal{O}\\0<l\leq1}}
		\sup_{Q\in\mathcal{B}^{r}_{0}}
		l^{-\alpha}
		\left|\langle\xi,Q_{z}^{l}\rangle\right|
		<\infty.
		\]
		If $\alpha\in(0,1)$, then $\mathcal{C}^{\alpha}_{(2,1)}$ denotes the space
		of locally $\alpha$-H\"older functions with respect to the metric $d$.
		\item Let
		\[
		\mathcal{T}
		=
		\bigoplus_{a\in\mathcal{A}}\mathcal{T}_a,
		\qquad
		\mathcal{A}
		=
		\{\rb,\gIb,\gXXXX\rb,\gY,\gI,\gXXXX:i=1,\ldots,d\},
		\qquad
		\mathcal{T}_a:=\mathbb{R}a,
		\]
		be the real vector space generated by the abstract symbols in $\mathcal{A}$.
		Let $0<\kappa<\frac{1}{3}$. The homogeneity of the elements of $\mathcal{T}$ is
		defined by
		\[
		|\rb|_{h}=-1-\kappa,
		\quad
		|\gIb|_{h}=-2\kappa,
		\quad
		|\gXXXX\rb|_{h}=-\kappa,
		\quad
		|\gY|_{h}=0,
		\quad
		|\gI|_{h}=1-\kappa,
		\quad
		|\gXXXX|_{h}=1,
		\]
		for $i=1,\ldots,d$.
		For $\tau\in\mathcal{T}$ and $a\in\mathcal{A}$, we denote by
		$\|\tau\|_{a}$ the absolute value of the coefficient of the component of
		$\tau$ in $\mathcal{T}_{a}$, and by $\mathrm{Pr}_{a}(\tau)$ the corresponding
		projection of $\tau$ onto $\mathcal{T}_{a}$.	
		
		\item Let $\mathcal{R}=(\mathcal{A},\mathcal{T},\mathcal{G})$ denote the corresponding regularity structure for \eqref{MAIN}, where $\mathcal{G}$ is the structure group acting on $\mathcal{T}$. More precisely, $\mathcal{G}$ is a group of linear transformations on $\mathcal{T}$ satisfying
		\[
		g(a)-a\in\mathcal{T}_{<|a|_h},
		\qquad a\in\mathcal{A},
		\]
		and is given explicitly by the following transformations:
		\begin{align*}
			&g(\rb)=\rb,\qquad
			g(\gI)=\gI-p_1\gY,\qquad
			g(\gY)=\gY,\\
			&g(\gXXXX)=\gXXXX-q_i\gY,\qquad
			g(\gIb)=\gIb-p_2\rb,\\
			&g(\gXXXX\rb)=\gXXXX\rb-k_i\rb ,
		\end{align*}
		for $p_1,p_2,q_i,k_i\in\mathbb{R}$.
		We define the model $\mathcal{Z}=(\Pi,\Gamma)$ on $\mathcal{R}$ by the
		maps
		\[
		\Gamma:\mathbb{R}^{d+1}\times\mathbb{R}^{d+1}\to\mathcal{G},
		\]
		where the linear operator $\Gamma_{z,w}:\mathcal{T}\to\mathcal{T}$ is
		given by
		\begin{align}\label{B1478}
			\begin{split}
				&\Gamma_{z,w}(\rb)=\rb,\qquad
				\Gamma_{z,w}(\gI)
				=
				\gI-\Pi_z(\gI)(w)\gY,\qquad
				\Gamma_{z,w}(\gY)=\gY,\\
				&\Gamma_{z,w}(\gXXXX)
				=
				\gXXXX-\Pi_z(\gXXXX)(w)\gY, \quad\Gamma_{z,w}(\gIb)
				=
				\gIb-\Pi_z(\gI)(w)\rb,\\
				&\Gamma_{z,w}(\gXXXX\rb)
				=
				\gXXXX\rb-\Pi_z(\gXXXX)(w)\rb .
			\end{split}
		\end{align}
		Here,
		\[
		\Pi_z:\mathcal{T}\to\mathcal{D}'(\mathbb{R}^{d+1})
		\]
		denotes the realization map of the model at the base point $z$. The
		maps $\Pi$ and $\Gamma$ satisfy the consistency relations
		\begin{align*}
			\Gamma_{z,w}\Gamma_{w,v}&=\Gamma_{z,v},\\
			\Pi_z\Gamma_{z,w}&=\Pi_w .
		\end{align*}
		Moreover, for every homogeneous element $a\in\mathcal{A}$,
		\begin{align}\label{BB1}
			\Gamma_{z,w}a-a\in\mathcal{T}_{<|a|_h}.
		\end{align}
		Throughout, we assume that the model is spatially periodic. More
		precisely, for every $m\in\mathbb{Z}^d$,
		\begin{align*}
			\Gamma_{z+2\pi(0,m),\,w+2\pi(0,m)}
			&=
			\Gamma_{z,w},
			\\
			\left\langle
			\Pi_{z+2\pi(0,m)}(a),
			\phi(\cdot-2\pi(0,m))
			\right\rangle
			&=
			\left\langle
			\Pi_z(a),\phi
			\right\rangle ,
		\end{align*}
	for all $z,w\in\mathbb{R}\times\mathbb{R}^d$,
	$a\in\mathcal{T}$, and
	$\phi\in C_c^\infty(\mathbb{R}\times\mathbb{R}^d)$.
		\item 
		Let $\mathcal{Z}=(\Pi,\Gamma)$ and
		$\widetilde{\mathcal{Z}}=(\widetilde{\Pi},\widetilde{\Gamma})$ be
		(spatially periodic) models, and let
		$\mathcal{O}\subseteq\mathbb{R}^{d+1}$ be a compact set.
		For $a\in\mathcal{A}$, we set
		\begin{align}\label{NBAVsss}
			\begin{split}
				\|\Pi\|_{a,\mathcal{O}}
				&:=
				\sup_{Q\in\mathcal{B}^{2}_{0}}
				\sup_{0<l\leq 1}
				\sup_{z\in\mathcal{O}}
				l^{-|a|_h}
				\left|
				\left\langle
				\Pi_z(a),Q^l_z
				\right\rangle
				\right|,
				\\
				\|\Pi;\widetilde{\Pi}\|_{a,\mathcal{O}}
				&:=
				\sup_{Q\in\mathcal{B}^{2}_{0}}
				\sup_{0<l\leq 1}
				\sup_{z\in\mathcal{O}}
				l^{-|a|_h}
				\left|
				\left\langle
				(\Pi_z-\widetilde{\Pi}_z)(a),Q^l_z
				\right\rangle
				\right|.
			\end{split}
		\end{align}
		
		We then define
		\[
		\|\Pi\|_{\mathcal{O}}
		:=
		\max_{a\in\mathcal{A}}\|\Pi\|_{a,\mathcal{O}},
		\qquad
		\|\Pi;\widetilde{\Pi}\|_{\mathcal{O}}
		:=
		\max_{a\in\mathcal{A}}
		\|\Pi;\widetilde{\Pi}\|_{a,\mathcal{O}}.
		\]
		
		Moreover, we set
		\[
		\|\Gamma\|_{\mathcal{O}}
		:=
		1+\max_{\substack{a_1,a_2\in\mathcal{A}\\
				|a_1|_h<|a_2|_h}}
		\sup_{\substack{z,w\in\mathcal{O}\\z\neq w}}
		\frac{
			\|\Gamma_{z,w}a_2\|_{a_1}
		}{
			d(z,w)^{|a_2|_h-|a_1|_h}
		}.
		\]
		In our setting, one can show that
		\[
		\|\Gamma\|_{\mathcal{O}}
		\lesssim
		1+[ \ \gI\ ]_{\mathcal{O}},
		\]
		where
		\[
		[\ \gI\ ]_{\mathcal{O}}
		:=
		\sup_{\substack{z,w\in\mathcal{O}\\z\neq w}}
		\frac{
			\|\gI(w)-\gI(z)\|
		}{
			d(z,w)^{1-\kappa}
		}.
		\]
		Furthermore, for every
		$\tau\in\mathcal{T}_{a_2}$,
		$z,w\in\mathcal{O}$ with $z\neq w$, and
		$a_1,a_2\in\mathcal{A}$ with
		$|a_1|_h\leq |a_2|_h$,
		we have
		\begin{align}\label{BB2}
			\frac{\|\Gamma_{z,w}\tau\|_{a_1}}
			{d(z,w)^{|a_2|_h-|a_1|_h}}
			\leq
			\|\Gamma\|_{\mathcal{O}}
			\|\tau\|_{a_2}.
		\end{align}
		We set
		\[
		\|\mathcal{Z}\|_{\mathcal{O}}
		:=
		\|\Gamma\|_{\mathcal{O}}
		+
		\|\Pi\|_{\mathcal{O}} .
		\]
		Moreover, we define
		\begin{align}\label{NBA1}
			\|\Gamma;\widetilde{\Gamma}\|_{\mathcal{O}}
			:=
			\max_{\substack{a_1,a_2\in\mathcal{A}\\
					|a_1|_h<|a_2|_h}}
			\sup_{\substack{z,w\in\mathcal{O}\\z\neq w}}
			\frac{
				\|\Gamma_{z,w}a_2-\widetilde{\Gamma}_{z,w}a_2\|_{a_1}
			}{
				d(z,w)^{|a_2|_h-|a_1|_h}
			},
		\end{align}
		and
		\begin{align}\label{NBA2}
			\|\mathcal{Z};\widetilde{\mathcal{Z}}\|_{\mathcal{O}}
			:=
			\|\Pi;\widetilde{\Pi}\|_{\mathcal{O}}
			+
			\|\Gamma;\widetilde{\Gamma}\|_{\mathcal{O}} .
		\end{align}
		These quantities define the corresponding local model seminorms on
		$\mathcal{M}=\mathcal{M}(\mathcal{R})$. Using these seminorms, one can equip
		$\mathcal{M}$ with a complete metrizable topology. Let
		\[
		\mathcal{O}_n:=[-n,n]\times[-\pi,\pi]^d.
		\]
		Then this topology is generated by the metric
		\begin{align}\label{NHMAsi78}
			d_{\mathcal{M}}(\mathcal{Z},\widetilde{\mathcal{Z}})
			=
			\sum_{n\geq1}
			\frac{
				\|\mathcal{Z};\widetilde{\mathcal{Z}}\|_{\mathcal{O}_n}
			}{
				2^n
				\left(
				1+\|\mathcal{Z};\widetilde{\mathcal{Z}}\|_{\mathcal{O}_n}
				\right)
			}.
		\end{align}
		\item A sector $\mathcal{V}=\bigoplus_{a\in\mathcal{A}}\mathcal{V}_a$ of
		$\mathcal{T}$ is a graded subspace such that, for each $a\in\mathcal{A}$,
		$\mathcal{V}_a\subseteq\mathcal{T}_a$, and $\mathcal{V}$ is invariant under
		the action of every element of $\mathcal{G}$.
		For $\alpha\in\mathbb{R}$, we set
		\[
		\mathcal{T}_{<\alpha}
		:=
		\bigoplus_{|a|_h<\alpha}\mathcal{T}_a .
		\]
		Similarly, we define $\mathcal{T}_{\leq\alpha}$,
		$\mathcal{T}_{\geq\alpha}$, and $\mathcal{T}_{>\alpha}$.
		It is easy to verify that $\mathcal{T}_{<0}$ and $\mathcal{T}_{\geq0}$
		are sectors of $\mathcal{T}$.
		For a sector $\mathcal{V}$ and a set
		$D\subseteq\mathbb{R}^{d+1}$, we denote by
		$\mathcal{F}(D,\mathcal{V})$ the set of functions from $D$ to $\mathcal{V}$.
		For $\RR{F}\in\mathcal{F}(D,\mathcal{T})$ and $a\in\mathcal{A}$, we define
		$F^{a}(z)\in\mathbb{R}$ by
		\[
		\mathrm{Pr}_{a}\big(\RR{F}(z)\big)=F^{a}(z)a.
		\]
		In particular, for
		$\RR{F}\in\mathcal{F}(D,\mathcal{T}_{\geq0})$, we write
		\[
		\textcolor{red}{F}(z)
		:=
		{F}^{\gY}(z)\gY
		+
		{F}^{\small{\gI}}(z)\gI
		+
		\sum_{i=1}^{d}{F}^{\gXXXX}(z)\gXXXX .
		\]
		For
		$\textcolor{red}{F},\textcolor{red}{K}
		\in\mathcal{F}(\mathbb{R}^{d+1},\mathcal{T}_{\geq0})$,
		we define
		$\textcolor{red}{F}\star\textcolor{red}{K}
		\in\mathcal{F}(\mathbb{R}^{d+1},\mathcal{T}_{\geq0})$ by
		\begin{align*}
			(\textcolor{red}{F}\star\textcolor{red}{K})(z)
			&:=
			{F}^{\gY}(z){K}^{\gY}(z)\gY+
			\Big[
			{F}^{\gY}(z){K}^{\small{\gI}}(z)
			+
			{K}^{\gY}(z){F}^{\small{\gI}}(z)
			\Big]\gI
			\\
			&\quad+
			\sum_{i=1}^{d}
			\Big[
			{F}^{\gY}(z){K}^{\gXXXX}(z)
			+
			{K}^{\gY}(z){F}^{\gXXXX}(z)
			\Big]\gXXXX .
		\end{align*}
		It is straightforward to verify that $\star$ is associative.
		For
		$\textcolor{red}{F}\in\mathcal{F}(D,\mathcal{T}_{\geq0})$, we define
		\[
		\textcolor{red}{F}\tilde{\star}\rb
		\in
		\mathcal{F}(D,\mathcal{T}_{<0})
		\]
		by
		\[
		(\textcolor{red}{F}\tilde{\star}\rb)(z)
		:=
		{F}^{\gY}(z)\rb
		+
		{F}^{\small{\gI}}(z)\gIb
		+
		\sum_{i=1}^{d}{F}^{\gXXXX}(z)\gXXXX\rb .
		\]
		\item 
		Let $\Sigma\in C^{2}(\mathbb{R},\mathbb{R})$ and
		$\RR{F}\in\mathcal{F}(D,\mathcal{T}_{\geq0})$. Then
		$\Sigma(\RR{F})\in\mathcal{F}(D,\mathcal{T}_{\geq0})$ is defined by
		\begin{align}\label{Hnasssss}
			\Sigma(\RR{F})(z)
			&=
			\Sigma(F^{\gY}(z))\gY
			+
			\Sigma^{\prime}(F^{\gY}(z))
			\left[
			{F}^{\small{\gI}}(z)\gI
			+
			\sum_{i=1}^{d}{F}^{\gXXXX}(z)\gXXXX
			\right].
		\end{align}
		For $\RR{F},\RR{\tilde{F}}\in\mathcal{F}(D,\mathcal{T}_{\geq0})$ and
		$\Sigma\in C^{3}(\mathbb{R},\mathbb{R})$, it follows that
		\begin{align}\label{TYUa}
			\Sigma(\RR{F})-\Sigma(\RR{\tilde{F}})
			=
			\int_0^1
			\Sigma'\big(\RR{\tilde{F}}+\alpha(\RR{F}-\RR{\tilde{F}})\big)
			\star(\RR{F}-\RR{\tilde{F}})
			\,\mathrm d\alpha .
		\end{align}
		More generally, if $\Sigma\in C^{j+3}(\mathbb R,\mathbb R)$, then
		\begin{align}\label{TYUa1}
			\Sigma(\RR{F}) - \Sigma(\RR{\tilde{F}})
			= \sum_{i=1}^{j} \frac{1}{i!} \Sigma^{(i)}(\RR{\tilde{F}}) \star \left[\star^i (\RR{F} - \RR{\tilde{F}})\right]
			+ \frac{1}{j!} \int_0^1 (1-\alpha)^j \Sigma^{(j+1)}\big(\RR{\tilde{F}} + \alpha (\RR{F} - \RR{\tilde{F}})\big) \star \left[\star^{j+1} (\RR{F} - \RR{\tilde{F}})\right] \, \mathrm{d}\alpha,
		\end{align}
		where $\Sigma^{(i)}$ denotes the $i$-th derivative of $\Sigma$ and
		\[
		\star^i(\RR{F}-\RR{\tilde{F}})
		:=
		\underbrace{
			(\RR{F}-\RR{\tilde{F}})\star\cdots\star(\RR{F}-\RR{\tilde{F}})
		}_{i\text{ times}} .
		\]
		\item 
		For $s\in \R$, let \( P_s = \{s\} \times \mathbb{R}^d \), and
		\begin{align*}
			\Vert z \Vert_{P_{s}} &:= \min\{1, d(z,P_{s})\}, \\
			\Vert z,w \Vert_{P_{s}} &:= \min\{\Vert z\Vert_{P_s},\Vert w\Vert_{P_s}\}.
		\end{align*}
		Given a model $\mathcal{Z}=(\Pi,\Gamma)$, the space
		$\mathcal{D}^{\gamma,\eta}_{P_s}$ consists of spatially periodic
		functions
		\[
		\RR{F}:\mathbb{R}^{d+1}\setminus P_s\to\mathcal{T}_{<\gamma}
		\]
		such that, for every compact set
		$\mathcal{O}\subset\mathbb{R}^{d+1}$,
		\begin{align}\label{Q00}
			\begin{split}
				\triplenorm{\RR{F}}_{\gamma,\eta,\mathcal{O}}^{(s)}
				:={}&
				\sup_{\substack{a\in\mathcal{A}\\|a|_h<\gamma}}
				\sup_{\substack{z,w\in\mathcal{O}\setminus P_s,\; z\neq w\\
						d(z,w)\leq\Vert z,w\Vert_{P_s}}}
				\frac{
					\Vert z,w\Vert_{P_s}^{\,\gamma-\eta}
					\,
					\Vert
					\RR{F}(z)-\Gamma_{z,w}\RR{F}(w)
					\Vert_a
				}{
					d(z,w)^{\gamma-|a|_h}
				}
				\\
				&\quad+
				\sup_{\substack{a\in\mathcal{A}\\|a|_h<\gamma}}
				\sup_{z\in\mathcal{O}\setminus P_s}
				\Vert z\Vert_{P_s}^{\,\max\{0,|a|_h-\eta\}}
				\Vert\RR{F}(z)\Vert_a
				<\infty .
			\end{split}
		\end{align}
		\begin{comment}
			content...
			
			Whenever $\eta=0$, we omit the parameter $\eta$ from the notation and write
			\[
			\triplenorm{\RR{F}}_{\gamma,\mathcal{O}}^{(s)}
			:=
			\triplenorm{\RR{F}}_{\gamma,0,\mathcal{O}}^{(s)}.
			\]
		\end{comment}
		For every
		$D\subseteq\mathbb{R}^{d+1}\setminus P_s$,
		we define
		$\triplenorm{\RR{F}}_{\gamma,\eta,D}^{(s)}$
		analogously whenever it is well defined.
		For $\alpha<\gamma$, we say that
		$\RR{F}\in\mathcal{D}^{\gamma,\eta}_{P_s}$
		is of regularity $\alpha$ if
		\[
		\RR{F}(z)\in\mathcal{T}_{\geq\alpha},
		\qquad
		z\in\mathbb{R}^{d+1}\setminus P_s.
		\]
		The corresponding subspace is denoted by
		$\mathcal{D}^{\gamma,\eta}_{P_s,\alpha}$.
		Finally, for a sector $\mathcal{V}$ of $\mathcal{T}$, we define
		\begin{align*}
			\mathcal{D}^{\gamma,\eta}_{P_s}(\mathcal{V})
			&:=
			\mathcal{D}^{\gamma,\eta}_{P_s}
			\cap
			\mathcal{F}(\mathbb{R}^{d+1}\setminus P_s,\mathcal{V}),
			\\
			\mathcal{D}^{\gamma,\eta}_{P_s,\alpha}(\mathcal{V})
			&:=
			\mathcal{D}^{\gamma,\eta}_{P_s,\alpha}
			\cap
			\mathcal{F}(\mathbb{R}^{d+1}\setminus P_s,\mathcal{V}).
		\end{align*}
		The space
		$\mathcal{D}^{\gamma,\eta}_{P_s}(\mathcal{V})$
		is complete with respect to the topology generated by the local norms
		$\triplenorm{\cdot}_{\gamma,\eta,\mathcal{O}}^{(s)}$.
		For $\Sigma\in C^{2}(\mathbb{R},\mathbb{R})$,
		$\RR{F}\in\mathcal{D}^{\gamma,\eta}_{P_s}(\mathcal{T}_{\geq0})$,
		$\eta\geq0$, and $0<\gamma<2(1-\kappa)$, 
		\begin{align}\label{Asasqqq}
			\Sigma(\RR{F})\in \mathcal{D}^{\gamma,\eta}_{P_s}.
		\end{align}
	
		\item  Let $\gamma>0$ and let $\mathcal{Z}=(\Pi,\Gamma)$ be a model. We denote by
		$\mathcal{D}^{\gamma}$
		the space of spatially periodic functions
		\[
		\RR{F}:\mathbb{R}^{d+1}\rightarrow\mathcal{T}_{<\gamma}
		\]
		such that, for every compact set
		$\mathcal{O}\subset\mathbb{R}^{d+1}$,
		\[
		\triplenorm{\RR{F}}_{\gamma,\mathcal{O}}
		:=
		\max_{\substack{a\in\mathcal{A}\\|a|_h<\gamma}}
		\sup_{\substack{z,w\in\mathcal{O}\\z\neq w}}
		\frac{
			\Vert
			\RR{F}(z)-\Gamma_{z,w}\RR{F}(w)
			\Vert_a
		}{
			d(z,w)^{\gamma-|a|_h}
		}
		+
		\max_{\substack{a\in\mathcal{A}\\|a|_h<\gamma}}
		\sup_{z\in\mathcal{O}}
		\Vert\RR{F}(z)\Vert_a
		<\infty.
		\]
		Observe that, if $\RR{F}\in\mathcal{D}^{\gamma}$, then
		\[
		\RR{F}|_{\mathbb{R}^{d+1}\setminus P_0}
		\in
		\mathcal{D}^{\gamma,\gamma}_{P_0}.
		\]
		For $\alpha<\gamma$, we say that
		$\RR{F}\in\mathcal{D}^{\gamma}$
		is of regularity $\alpha$ if
		\[
		\RR{F}(z)\in\mathcal{T}_{\geq\alpha},
		\qquad
		z\in\mathbb{R}^{d+1}.
		\]
		The corresponding subspace is denoted by
		$\mathcal{D}^{\gamma}_{\alpha}$.
		For a sector $\mathcal{V}$ of $\mathcal{T}$, we further define
		\[
		\mathcal{D}^{\gamma}_{\alpha}(\mathcal{V})
		:=
		\mathcal{D}^{\gamma}_{\alpha}
		\cap
		\mathcal{F}(\mathbb{R}^{d+1},\mathcal{V}).
		\]
	\end{itemize}
	\subsection{Ingredients}
	We now briefly summarize the main tools used in this manuscript.
	\begin{itemize}
		\item 
		Let
		\[
		G:(0,\infty)\times\mathbb{R}^{d}\to\mathbb{R}
		\]
		denote the heat kernel associated with the operator
		$(\partial_t-\Delta)$ on $\mathbb{R}^{d}$, extended by zero for negative
		times. For $s\in\mathbb{R}$, we set
		\[
		G^{s}(t,x):=G(t-s,x),\qquad t>s,
		\]
		and extend it by zero for earlier times, i.e.,
		\begin{align}\label{NNAsss}
			G^{s}(t,x)=0,\qquad t<s,\ x\in\mathbb{R}^{d}.
		\end{align}
		
		Let
		\[
		g\in\mathcal{D}'(\mathbb{R}\times\mathbb{R}^{d}),
		\qquad
		f\in\mathcal{D}'(\mathbb{R}^{d}),
		\]
		be spatially periodic distributions. Whenever the pairing below is
		well-defined, we denote by $\mathbf{R}^{+}_{s}g$ the restriction of $g$
		to $(s,\infty)\times\mathbb{R}^{d}$, characterized by
		\[
		\left\langle \mathbf{R}^{+}_{s}g,\varphi\right\rangle
		:=
		\left\langle
		g,\mathbf{1}_{(s,\infty)}(t)\varphi(t,x)
		\right\rangle,
		\qquad
		\varphi\in C_c^\infty(\mathbb{R}\times\mathbb{R}^{d}).
		\]
		
		We define the convolutions in the distributional sense. The space-time
		convolution is given by
		\[
		(G*\mathbf{R}^{+}_{s}g)(t,x)
		:=
		\left\langle
		\mathbf{R}^{+}_{s}g,
		G(t-\cdot,x-\cdot)
		\right\rangle,
		\]
		whenever the corresponding pairing is well-defined. Similarly, the spatial
		convolution with $f$ is defined by
		\[
		(G^{s}f)(t,x)
		:=
		\left\langle
		f,G^{s}(t,x-\cdot)
		\right\rangle.
		\]
		
		Assume that all the quantities below are well-defined and admit
		representatives which are functions of $(t,x)$. Then, for $s<r<t$,
		the semigroup property of the heat kernel yields
		\begin{align}\label{SEMIJ}
			\begin{split}
				&(G*\mathbf{R}^{+}_{s}g)(t,x)+(G^{s}f)(t,x)
				\\
				&\qquad=
				(G*\mathbf{R}^{+}_{r}g)(t,x)
				+
				G^{r}\Big(
				(G*\mathbf{R}^{+}_{s}g)(r,\cdot)
				+
				(G^{s}f)(r,\cdot)
				\Big)(t,x).
			\end{split}
		\end{align}
		\item From \cite[Proposition 6.12]{Hai14}, it follows that for
		\(\RR{\mathcal{U}_1},\RR{\mathcal{U}_2}\in
		\mathcal{D}^{\gamma_1,0}_{P_s,0}\), we have
		\[
		\RR{\mathcal{U}_1}\star\RR{\mathcal{U}_2}
		\in
		\mathcal{D}^{\gamma_1,0}_{P_s,0}.
		\]
		Moreover, for every compact set $\mathcal{O}$, there exists a constant
		$C_{\mathcal{O}}>0$ such that
		\begin{align}\label{multi}
			\triplenorm{\RR{\mathcal{U}_1}\star\RR{\mathcal{U}_2}}_{\gamma_1,0,\mathcal{O}}^{(s)}
			\leq
			C_{\mathcal{O}}
			\triplenorm{\RR{\mathcal{U}_1}}_{\gamma_1,0,\mathcal{O}}^{(s)}
			\triplenorm{\RR{\mathcal{U}_2}}_{\gamma_1,0,\mathcal{O}}^{(s)}.
		\end{align}
		With a slight abuse of notation, for
		$\mathbf{R}_s^+(t,x):=\chi_{\{t>s\}}$, we set
		\[
		(\mathbf{R}_s^+\gY)(t,x)
		:=
		\mathbf{R}_s^+(t,x)\gY.
		\]
		Then, for
		\(\RR{\mathcal U}\in\mathcal D^{\gamma,0}_{P_s,0}\), we define
		\[
		\mathbf{R}_s^+\RR{\mathcal U}
		:=
		(\mathbf{R}_s^+\gY)\star\RR{\mathcal U}.
		\]
		\item 
		Let $\gamma>0$, $-2<\alpha <0 $, and let $r\in\mathbb{N}$ satisfy
		$r>|\alpha|$. Then, by inspecting the proofs of
		\cite[Proposition~3.25 and Theorem~3.10]{Hai14}, there exists a unique
		continuous linear map
		\[
		\mathcal{R}:\mathcal{D}^{\gamma}_{\alpha}
		\rightarrow
		\mathcal{C}^{\alpha}_{(2,1)}.
		\]
		More precisely, for $z=(t,x)$ and $0<l\leq1$, let
		\[
		\mathcal{O}_{z,l}
		:=
		[t-4l^2,t+4l^2]
		\times
		\overline{B_{\mathbb{R}^d}(x,2l)}.
		\]
		Then, for every $\RR{F}\in\mathcal{D}^{\gamma}_{\alpha}$ and
		$Q\in\mathcal{B}^{r}_{0}$,
		\begin{align}\label{AA}
			\left|
			\left\langle
			\mathcal{R}\RR{F}-\Pi_{z}\RR{F}(z),
			Q^{l}_{z}
			\right\rangle
			\right|
			&\lesssim
			l^{\gamma}
			\|\Pi\|_{\mathcal{O}_{z,l}}
			\sup_{\substack{w,\tilde w\in\mathcal{O}_{z,l}\\w\neq\tilde w}}
			\sup_{\substack{a\in\mathcal A\\|a|_h<\gamma}}
			\frac{
				\|\RR{F}(w)-\Gamma_{w,\tilde w}\RR{F}(\tilde w)\|_a
			}{
				d(w,\tilde w)^{\gamma-|a|_h}
			}.
		\end{align}
		This map is called the \emph{reconstruction operator} and depends on
		the underlying model.
		Moreover, in our setting,
		\begin{align}\label{521}
			\mathcal{R}\RR{F}(z)
			=
			F^{\gY}(z),
			\qquad
			\forall\,\RR{F}\in\mathcal{D}^{\gamma}_{0}.
		\end{align}
		The reconstruction operator is local. Consequently, for every
		$\RR{F}\in\mathcal{D}^{\gamma,\eta}_{P_s,\alpha}$,
		there exists a unique distribution
		\[
		\overline{\mathcal{R}}\RR{F}
		\in
		\bigl(C_c^\infty(\mathbb R^{d+1}\setminus P_s)\bigr)'
		\]
		such that
		\[
		\left|
		\left\langle
		\overline{\mathcal{R}}\RR{F}
		-
		\Pi_z\RR{F}(z),
		Q_z^l
		\right\rangle
		\right|
		\lesssim
		l^\gamma
		\Vert z\Vert_{P_s}^{\,\eta-\gamma},
		\]
		for every compact set $\mathcal O\subseteq\mathbb R^{d+1}$,
		uniformly over $z\in\mathcal O\setminus P_s$,
		$Q\in\mathcal B_0^r$, and
		\[
		0<l\leq\min\left\{1,\frac12 d(z,P_s)\right\}.
		\]
		The implicit constant may depend on $\mathcal O$, $\RR{F}$,
		and the model.
		
		By \cite[Proposition~6.9]{Hai14}, if moreover
		$\eta\leq\gamma$ and
		\[
		\min\{\alpha,\eta\}>-2,
		\]
		then there exists a unique distribution
		\[
		\mathcal{R}\RR{F}
		\in
		\mathcal{C}^{\min\{\alpha,\eta\}}_{(2,1)}
		\]
		such that
		\[
		\left\langle
		\mathcal{R}\RR{F},
		Q
		\right\rangle
		=
		\left\langle
		\overline{\mathcal{R}}\RR{F},
		Q
		\right\rangle,
		\]
		for every
		$Q\in C_c^\infty(\mathbb R^{d+1}\setminus P_s)$.
		
		\item  
		By \cite[Lemmas~5.5 and 5.24]{Hai14}, we can choose a decomposition
		\begin{align*}
			G=K+S,
		\end{align*}
		where
		\(S\in C^{\infty}(\mathbb{R}^{d+1},\mathbb{R})\)
		is symmetric in the spatial variable and vanishes for $t<0$.
		
		The kernel \(K\) can be chosen such that:
		\begin{itemize}
			\item It is compactly supported in
			\[
			\{(t,x)\in[0,\infty)\times\mathbb{R}^{d}:
			\|x\|^{2}+t\leq1\},
			\]
			is symmetric in the spatial variable \(x\), and vanishes for \(t<0\).
			
			\item For $(t,x)\in(0,\infty)\times\mathbb{R}^{d}$ such that
			\[
			\|x\|^{2}+t<\frac12,
			\]
			one has
			\[
			K(t,x)
			=
			\frac{1}{\left(4\pi(t)\right)^{\frac d2}}
			\exp\!\left(-\frac{\|x\|^{2}}{4t}\right),
			\]
			and $K$ is smooth on the set
			\[
			\{(t,x):\|x\|^{2}+t\geq\tfrac14\}.
			\]
			
			\item It annihilates all polynomials of parabolic degree at most \(r\), for some \(r>\gamma_1+2\), and satisfies Assumptions~5.1 and~5.4 of \cite{Hai14}.
		\end{itemize}
		
		\item Let $0<\kappa<\frac13$. We assume that
		\[
		1+\kappa<\gamma_{1}<2-2\kappa,
		\]
		and define
		\[
		\gamma_{0}:=\gamma_{1}-1-\kappa.
		\]
		\item 
		For $K$, one can define an abstract linear integration operator
		\[
		\mathcal{I} : \mathcal{T}_{<0} \rightarrow \mathcal{T}_{\geq 1-\kappa}.
		\]
		For the equation considered here, this operator is easy to define. We set
		\[
		\mathcal{I}(a)=
		\begin{cases}
			\gI, & \text{if } a=\rb,\\
			0, & \text{otherwise}.
		\end{cases}
		\]
		We say that the model $\mathcal{Z}=(\Pi,\Gamma)$ realizes $K$ for
		$\mathcal{I}$ if, for every $z,w\in\mathbb{R}^{d+1}$,
		\begin{align}\label{trtetre}
			\Pi_z(\gI)(w)
			=
			\gII(w)-\gII(z)
			=
			\left\langle
			\Pi_z(\rb),
			K(w-\cdot)-K(z-\cdot)
			\right\rangle .
		\end{align}
		
		\item Let the model $\mathcal{Z}=(\Pi,\Gamma)$ realize $K$ for
		$\mathcal{I}$. For every $\theta\in(0,1-\kappa]$, by
		\cite[Proposition~6.16 and Theorem~7.1]{Hai14}
		and the continuous embedding
		\[
		\mathcal{D}^{\gamma_0+2,1-\kappa-\theta}_{P_s}
		\hookrightarrow
		\mathcal{D}^{\gamma_1,1-\kappa-\theta}_{P_s},
		\]
		there exists a continuous linear operator
		\begin{align}\label{ASXZa}
			\mathcal{K}_{\gamma_0,\theta}^{s}
			:
			\mathcal{D}^{\gamma_0,-1-\kappa}_{P_s}
			\longrightarrow
			\mathcal{D}^{\gamma_1,1-\kappa-\theta}_{P_s},
		\end{align}
		such that, for every
		$\RR{F}\in\mathcal{D}^{\gamma_0,-1-\kappa}_{P_s}$,
		\begin{align}\label{YHNa}
			\mathcal{R}\big(
			\mathcal{K}_{\gamma_0,\theta}^{s}(\RR{F})
			\big)
			=
			K*\mathcal{R}(\RR{F}).
		\end{align}
		It can also be shown that, for each
		$z\in\mathbb{R}^{d+1}\setminus P_s$,
		\begin{align}\label{A1}
			\mathrm{Pr}_{\gI}
			\big(
			\mathcal{K}_{\gamma_0,\theta}^{s}\RR{F}(z)
			\big)
			=
			{F}^{\rb}(z)\gI .
		\end{align}
		
		By inspecting the proofs of
		\cite[Lemma~6.5, Proposition~6.16, Theorem~7.1]{Hai14},
		we obtain that, for $T\in(0,1]$ and every
		$\theta\in(0,1-\kappa]$,
		\begin{align}\label{A2}
			\begin{split}
				&\triplenorm{
					\mathcal{K}_{\gamma_0,\theta}^{s}
					(\mathbf{R}^{+}_{s}\RR{F})
				}_{\gamma_1,\,1-\kappa-\theta,\,O_{T+s}^s}^{(s)}
				\\
				&\qquad\lesssim
				T^{\frac{\theta}{2}}
				\big(1+\|\mathcal{Z}\|_{\overline{O_{s+1}^{s-1}}}^2\big)
				\triplenorm{\RR{F}}_{\gamma_0,-1-\kappa,O_{T+s}^s}^{(s)},
			\end{split}
		\end{align}
		for every $\RR{F}\in\mathcal{D}^{\gamma_0,-1-\kappa}_{P_s}$,
		where the implicit constant is uniform in $s\in\mathbb{R}$.
		
		For the smooth function $S$ and a spatially periodic distribution
		$\xi\in\mathcal C^{-1-\kappa}_{(2,1)}$ supported in
		$[s,\infty)\times\mathbb R^d$, we define
		\[
		\mathcal S_{\gamma_0,\theta}^{s}(\xi)(z)
		:=
		\langle\xi,S(z-\cdot)\rangle\gY
		+
		\sum_{i=1}^{d}
		\left\langle
		\xi,\frac{\partial S(z-\cdot)}{\partial x_i}
		\right\rangle\gXXXX.
		\]
		
		Suppose that
		\[
		\RR{F}\in\mathcal{D}^{\gamma_0,-1-\kappa}_{P_s}.
		\]
		Since $\mathcal{R}(\mathbf{R}^{+}_{s}\RR{F})$ is supported in
		$[s,\infty)\times\mathbb{R}^{d}$, we have
		\[
		\mathcal{S}_{\gamma_0,\theta}^{s}
		\left(
		\mathcal{R}\left(\mathbf{R}^{+}_{s}\RR{F}\right)
		\right)(t,x)
		=0,
		\qquad t\leq s.
		\]
		By the same argument as in \cite[Lemma~7.3]{Hai14}, for every
		$\theta\in(0,1-\kappa]$ and $T\in(0,1]$, we have
		\begin{align}\label{A3}
			\begin{split}
				&\triplenorm{
					\mathcal{S}_{\gamma_0,\theta}^{s}
					\left(\mathcal{R}\left(\mathbf{R}^{+}_{s}\RR{F}\right)\right)
				}_{\gamma_1,\,1-\kappa-\theta,\,O_{T+s}^s}^{(s)}
				\\
				&\qquad\lesssim
				T
				\big(1+\|\mathcal{Z}\|_{\overline{O_{s+1}^{s-1}}}^2\big)
				\triplenorm{\RR{F}}_{\gamma_0,-1-\kappa,O_{T+s}^s}^{(s)},
			\end{split}
		\end{align}
		where the implicit constant is uniform in $s\in\mathbb{R}$.
		Moreover, it is immediate from the definition that
		\begin{align}\label{Z0}
			\mathcal{R}\left(
			\mathcal{S}_{\gamma_0,\theta}^{s}\left(
			\mathcal{R}\left(\mathbf{R}^{+}_{s}\RR{F}\right)
			\right)\right)
			=
			S*\mathcal{R}\left(\mathbf{R}^{+}_{s}\RR{F}\right).
		\end{align}
		
		Let $\theta=1-\kappa$. We view $\mathbf{R}^{+}_{s}$ as the natural map on
		modelled distributions and define
		\begin{align*}
			\mathcal{G}^{s}_{\gamma_{0},1-\kappa}
			:=
			\mathcal{K}_{\gamma_0,1-\kappa}^{s}\circ\mathbf{R}^{+}_{s}
			+
			\mathcal{S}_{\gamma_0,1-\kappa}^{s}
			\circ\mathcal{R}\circ\mathbf{R}^{+}_{s}.
		\end{align*}
		Then, from \eqref{A2} and \eqref{A3}, together with
		\eqref{YHNa} and \eqref{Z0}, there exists a continuous linear map
		\begin{align}\label{ASXZb}
			\mathcal{G}^{s}_{\gamma_{0},1-\kappa}
			:
			\mathcal{D}^{\gamma_0,-1-\kappa}_{P_s}
			\longrightarrow
			\mathcal{D}^{\gamma_1,0}_{P_s,0},
		\end{align}
		such that, for every $T\in(0,1]$ and every
		$\RR{F}\in\mathcal{D}^{\gamma_0,-1-\kappa}_{P_s}$,
		\begin{align}\label{A2222}
			\begin{split}
				&\triplenorm{
					\mathcal{G}^{s}_{\gamma_{0},1-\kappa}
					(\RR{F})
				}_{\gamma_1,0,O_{T+s}^s}^{(s)}
				\lesssim
				T^{\frac{1-\kappa}{2}}
				\big(1+\|\mathcal{Z}\|_{\overline{O_{s+1}^{s-1}}}^2\big)
				\triplenorm{\RR{F}}_{\gamma_0,-1-\kappa,O_{T+s}^s}^{(s)},
				\\
				&\mathcal{R}\big(
				\mathcal{G}^{s}_{\gamma_{0},1-\kappa}(\RR{F})
				\big)
				=
				G*
				\mathcal{R}\left(\mathbf{R}^{+}_{s}\RR{F}\right).
			\end{split}
		\end{align}
		
		\item Let $\tilde{\theta}\in(0,2]$.
		By \cite[Proposition~6.16 and Theorem~7.1]{Hai14} and the continuous embedding
		\[
		\mathcal{D}^{\gamma_1+2,2-\tilde{\theta}}_{P_s,0}
		\hookrightarrow
		\mathcal{D}^{\gamma_1,\,2-\tilde{\theta}}_{P_s,0},
		\]
		there exists a continuous linear operator
		\begin{align}\label{B1}
			\mathcal{K}_{\gamma_1,\tilde{\theta}}^{s}
			:
			\mathcal{D}^{\gamma_1,0}_{P_s,0}
			\longrightarrow
			\mathcal{D}^{\gamma_1,\,2-\tilde{\theta}}_{P_s,0},
		\end{align}
		such that
		\[
		\mathcal{R}\big(
		\mathcal{K}_{\gamma_1,\tilde{\theta}}^{s}(\RR{H})
		\big)
		=
		K*\mathcal{R}(\RR{H}).
		\]
		Moreover, for every
		$\RR{H}\in\mathcal{D}^{\gamma_1,0}_{P_s,0}$ and
		$T\in(0,1]$,
		\begin{align*}
			\triplenorm{
				\mathcal{K}_{\gamma_1,\tilde{\theta}}^{s}(
				\mathbf{R}^{+}_{s}\RR{H})
			}_{\gamma_1,\,2-\tilde{\theta},\,O_{T+s}^s}^{(s)}
			\lesssim
			T^{\frac{\tilde{\theta}}{2}}
			\big(1+\|\mathcal{Z}\|_{\overline{O_{s+1}^{s-1}}}^2\big)
			\triplenorm{\RR{H}}_{\gamma_1,0,O_{T+s}^s}^{(s)} .
		\end{align*}
		For the smooth function $S$, we define
		\begin{align*}
			\mathcal{S}_{\gamma_1,\tilde{\theta}}^{s}
			&:
			\left\{
			u\in L^\infty_{\mathrm{loc}}(\mathbb R\times\mathbb T^d):
			u=0\ \text{a.e. on }(-\infty,s)\times\mathbb T^d
			\right\}\longrightarrow
			\mathcal{D}^{\gamma_1,\,2-\tilde{\theta}}_{P_s,0},
			\\
			\mathcal{S}_{\gamma_1,\tilde{\theta}}^{s}(u)(z)
			&=
			\langle u,S(z-\cdot)\rangle\gY
			+
			\sum_{i=1}^{d}
			\left\langle
			u,
			\frac{\partial S(z-\cdot)}{\partial x_i}
			\right\rangle
			\gXXXX.
		\end{align*}
		For
		$\RR{H}\in\mathcal{D}^{\gamma_1,0}_{P_s,0}$, since
		$\mathcal{R}(\mathbf{R}^{+}_{s}\RR{H})$ is supported in
		$[s,\infty)\times\mathbb{R}^{d}$, we have
		\[
		\mathcal{S}_{\gamma_1,\tilde{\theta}}^{s}
		\left(
		\mathcal{R}\left(\mathbf{R}^{+}_{s}\RR{H}\right)
		\right)(t,x)
		=0,
		\qquad t\leq s .
		\]
		Then, similarly to \eqref{A3}, for every
		$\tilde{\theta}\in(0,2]$ and $T\in(0,1]$, we have
		\begin{align}\label{B2}
			\begin{split}
				&
				\triplenorm{
					\mathcal{S}_{\gamma_1,\tilde{\theta}}^{s}
					\left(
					\mathcal{R}\left(\mathbf{R}^{+}_{s}\RR{H}\right)
					\right)
				}_{\gamma_1,\,2-\tilde{\theta},\,O_{T+s}^s}^{(s)}
				\\
				&\qquad\lesssim
				T
				\big(1+\|\mathcal{Z}\|_{\overline{O_{s+1}^{s-1}}}^2\big)
				\triplenorm{\RR{H}}_{\gamma_1,0,O_{T+s}^s}^{(s)},
			\end{split}
		\end{align}
		where the implicit constant is uniform in $s\in\mathbb{R}$.
		Moreover,
		\[
		\mathcal{R}\left(
		\mathcal{S}_{\gamma_1,\tilde{\theta}}^{s}
		\left(
		\mathcal{R}\left(\mathbf{R}^{+}_{s}\RR{H}\right)
		\right)\right)
		=
		S*\mathcal{R}\left(\mathbf{R}^{+}_{s}\RR{H}\right).
		\]
		
		Analogously to \eqref{ASXZb}, we conclude that for
		$\tilde{\theta}=2$,
		\begin{align}\label{Azyaa}
			\mathcal{G}^{s}_{\gamma_{1},2}
			&:=
			\mathcal{K}_{\gamma_1,2}^{s}\circ\mathbf{R}^{+}_{s}
			+
			\mathcal{S}_{\gamma_1,2}^{s}
			\circ\mathcal{R}\circ\mathbf{R}^{+}_{s}
			:
			\mathcal{D}^{\gamma_1,0}_{P_s,0}
			\longrightarrow
			\mathcal{D}^{\gamma_1,0}_{P_s,0}
		\end{align}
		is a continuous linear map. Moreover, for every $T\in(0,1]$ and every
		$\RR{H}\in\mathcal{D}^{\gamma_1,0}_{P_s,0}$,
		\begin{align}\label{A222233}
			\begin{split}
				&\triplenorm{
					\mathcal{G}^{s}_{\gamma_{1},2}
					(\RR{H})
				}_{\gamma_1,0,O_{T+s}^s}^{(s)}
				\lesssim
				T
				\big(1+\|\mathcal{Z}\|_{\overline{O_{s+1}^{s-1}}}^2\big)
				\triplenorm{\RR{H}}_{\gamma_1,0,O_{T+s}^s}^{(s)},
				\\
				&\mathcal{R}\big(
				\mathcal{G}^{s}_{\gamma_{1},2}(\RR{H})
				\big)
				=
				G*
				\mathcal{R}\left(\mathbf{R}^{+}_{s}\RR{H}\right).
			\end{split}
		\end{align}
		\item For every periodic function 
		$u_0\in L^{\infty}(\mathbb{T}^d)$ and $s\in\mathbb{R}$,
		one can enhance
		\[
		G^{s}(u_0)(z)
		=
		\langle u_0,G^{s}(t,x-\cdot)\rangle
		=
		\langle u_0,G(t-s,x-\cdot)\rangle ,
		\]
		to obtain
		\begin{align}\label{Azu}
			\begin{split}
				\RR{G^{s}}(u_0)
				&\in
				\mathcal{D}^{\gamma_1,0}_{P_s,0},
				\\
				\RR{G^{s}}(u_0)(z)
				&=
				\langle u_0,G^{s}(t,x-\cdot)\rangle \gY
				+
				\sum_{i=1}^{d}
				\Big\langle
				u_0,
				\frac{\partial G^{s}(t,x-\cdot)}{\partial x_i}
				\Big\rangle
				\gXXXX ,
			\end{split}
		\end{align}
		where $z=(t,x)$.
		By the standard properties of the heat kernel, one has
		\begin{align}\label{A-1}
			\sup_{s\in\mathbb{R}}
			\triplenorm{\RR{G^{s}}(u_0)}_{\gamma_1,0,O_{s+1}^s}^{(s)}
			\lesssim
			\|u_0\|_{L^{\infty}(\mathbb{T}^d)}.
		\end{align}
		Moreover, for $t<s$ and every $z=(t,x)$,
		\[
		\RR{G^{s}}(u_0)(z)=0.
		\]
		Finally, by the reconstruction property,
		\begin{align}\label{IKLa}
			\mathcal{R}\left(\RR{G^{s}}(u_0)\right)(z)
			=
			G^{s}(u_0)(z).
		\end{align}
		\item 
		Recall that \(G=K+S\). Let
		\(\Sigma\in C^{2}(\mathbb{R},\mathbb{R})\) and
		\(u_{0}\in L^{\infty}(\mathbb{T}^{d})\).
		Let
		\(\mathcal{D}^{\gamma_{1},0}_{P_{s},0}\big|_{O_{T+s}^s}\)
		denote the restriction of
		\(\mathcal{D}^{\gamma_{1},0}_{P_{s},0}\) to \(O_{T+s}^s\).
		By \eqref{Asasqqq}, \eqref{ASXZb}, \eqref{Azyaa}, and \eqref{Azu},
		we can define the continuous map
		\begin{align}\label{FIXED}
			\begin{split}
				\mathcal{Q}_{T}^{s}
				&:
				L^{\infty}(\mathbb{T}^{d})
				\times
				\mathcal{D}^{\gamma_1,0}_{P_{s},0}\big|_{O_{T+s}^s}
				\longrightarrow
				\mathcal{D}^{\gamma_1,0}_{P_{s},0}\big|_{O_{T+s}^s},
				\\
				\mathcal{Q}_T^{s}(u_0,\RR{\mathcal{U}})
				&=
				\mathcal{G}^{s}_{\gamma_{0},1-\kappa}
				\left(
				\Sigma(\RR{\mathcal{U}})\tilde{\star}\rb
				\right)
				+
				\mu\mathcal{G}^{s}_{\gamma_{1},2}(\RR{\mathcal{U}})
				+
				\RR{G^{s}}(u_0).
			\end{split}
		\end{align}
		From \eqref{A2222}, \eqref{A222233}, and \eqref{A-1},
		there exists a constant \(C\geq1\) such that, for every
		\(T\in(0,1]\),
		\begin{align}\label{fixed-point}
			\begin{split}
				&
				\triplenorm{
					\mathcal{Q}_T^{s}(u_0,\RR{\mathcal{U}})
				}_{\gamma_1,0,O_{T+s}^s}^{(s)}
				\\
				&\quad\leq
				CT^{\frac{1-\kappa}{2}}
				\left[
				1+\|\mathcal{Z}\|_{\overline{O_{s+1}^{s-1}}}^{2}
				\right]
				\Bigg(
				\triplenorm{
					\mathbf{R}^+_s
					\Sigma(\RR{\mathcal{U}})
					\tilde{\star}\rb
				}_{\gamma_0,-1-\kappa,O_{T+s}^s}^{(s)}
				+
				|\mu|
				\triplenorm{
					\RR{\mathcal{U}}
				}_{\gamma_1,0,O_{T+s}^s}^{(s)}
				\Bigg)
				+
				C\|u_0\|_{L^{\infty}(\mathbb{T}^d)},
				\\[0.5em]
				&
				\triplenorm{
					\mathbf{R}^+_s
					\mathcal{Q}_T^{s}(u_0,\RR{\mathcal{U}_2})
					-
					\mathbf{R}^+_s
					\mathcal{Q}_T^{s}(u_0,\RR{\mathcal{U}_1})
				}_{\gamma_1,0,O_{T+s}^s}^{(s)}
				\\
				&\quad\leq
				CT^{\frac{1-\kappa}{2}}
				\left[
				1+\|\mathcal{Z}\|_{\overline{O_{s+1}^{s-1}}}^{2}
				\right]
				\Bigg(
				\triplenorm{
					\mathbf{R}^+_s
					\Sigma(\RR{\mathcal{U}_2})\tilde{\star}\rb
					-
					\mathbf{R}^+_s
					\Sigma(\RR{\mathcal{U}_1})\tilde{\star}\rb
				}_{\gamma_0,-1-\kappa,O_{T+s}^s}^{(s)}+
				|\mu|
				\triplenorm{
					\RR{\mathcal{U}_2}
					-
					\RR{\mathcal{U}_1}
				}_{\gamma_1,0,O_{T+s}^s}^{(s)}
				\Bigg).
			\end{split}
		\end{align}
		Thanks to \eqref{fixed-point} and \eqref{Asasqqq}, one can use a
		fixed point argument to obtain a fixed point for the map
		\(\mathcal{Q}_{T}^{s}(u_0,\cdot)\) around
		\(\RR{G^{s}}(u_0)\), provided that \(T\) is sufficiently small.
		More precisely, for every
		\(u_0\in L^{\infty}(\mathbb{T}^{d})\), there exists \(T>0\) and
		\[
		\RR{\mathcal{U}}_{u_0}
		\in
		\mathcal{D}^{\gamma_1,0}_{P_s,0}\big|_{O_{T+s}^s}
		\]
		such that
		\[
		\mathcal{Q}_{T}^{s}
		\left(u_0,\RR{\mathcal{U}}_{u_0}\right)
		=
		\RR{\mathcal{U}}_{u_0}.
		\]
		Moreover, this fixed point is unique in a sufficiently small neighborhood of
		$\RR{G^s}(u_0)$ in $\mathcal{D}^{\gamma_1,0}_{P_s,0}\big|_{O_{T+s}^s}.$
	\end{itemize}
	\begin{remark}\label{TGsasa}
		In the concrete model considered in this manuscript, the realization map $\Pi$ is given by
		\begin{align*}
			& \Pi_z(\rb) = \bullet, \quad
			\Pi_z(\gXXXX) = X^{i} - X^{i}(z), \quad
			\Pi_z(\gI) = \gII - \gII(z), \\
			&
			\Pi_z(\gY) = 1, \quad
			\Pi_z(\gXXXX\rb) = \Pi_z(\gXXXX)\,\bullet .
		\end{align*}
		Here $ \bullet \in \mathcal{D}'(\mathbb{R}^{d+1}) $ and $ \gII \ \in C^{1-\kappa}_{(2,1)}$ are spatially periodic. The definition of $\Pi_z((\gIb))$ is more subtle and is naively given by
		\[
		(\gII-\gII(z))\bullet-\infty
		\]
		in the renormalisation sense.
	\end{remark}
	\section{Generalised parabolic Anderson model}\label{GPAM}
	In the previous part, we summarized the main tools that we will use for
	Equation \eqref{MAIN}. In this section, we derive several technical results
	that will be needed for our subsequent analysis. The local solution theory
	for this type of equation was developed in \cite{Hai14}. However, the
	solution theory established there is local in time, and obtaining a global
	solution theory for this class of equations requires additional arguments.
	Starting from the local existence theory, we derive several useful properties
	of the solution. In particular, these results will serve as a basis for the
	investigation of the dynamical behaviour of the equation.
	We now fix the assumptions that will be used throughout this section.
	\begin{assumption}\label{UASassasa}
		We assume that
		\begin{enumerate}
			\item  
			We consider the following $2\pi$-periodic problem:
			\begin{align}\label{MAIN}
				\begin{cases}
					(\partial_t-\Delta)u
					=
					\mu u+\Sigma(u)\,\xi,
					& t>s,
					\\
					u(s,\cdot)=u_0,
					& u_0\in L^{\infty}(\mathbb{T}^d).
				\end{cases}
			\end{align}
			Here
			\[
			\mathbb{T}^d=(\mathbb{R}/2\pi\mathbb{Z})^d,
			\]
			$\xi\in\mathcal{C}^{-1-\kappa}$ with
			$0<\kappa<\frac13$, $\mu\in\mathbb{R}$, and
			$\Sigma\in C_b^2(\mathbb{R},\mathbb{R})$.
			
			\item We continue to work with the notation and assumptions introduced in
			the previous part. In particular, for $0<\kappa<\frac13$, we set
			\begin{align}
				\begin{split}
					1+\kappa &< \gamma_1 < 2-2\kappa,\\
					\gamma_0 &:= \gamma_1-1-\kappa.
				\end{split}
			\end{align}
			
			\item We fix a model $\mathcal{Z}=(\Pi,\Gamma)$.
		\end{enumerate}
	\end{assumption}
	Let us now begin with an elementary lemma showing that multiplication by the noise symbol $\rb$ shifts the regularity of a modelled distribution according to the homogeneity of $\rb$. This simple observation will be used repeatedly throughout the sequel.
	\begin{lemma}\label{stat}
		Let $\RR{\mathcal{U}}\in\mathcal{D}^{\gamma_1,0}_{P_{s},0}|_{O_{T+s}^s}$, then
		\begin{align*}
			\triplenorm{\RR{\mathcal{U}}\tilde{\star}\rb}_{\gamma_0,-1-\kappa,O_{T+s}^s}^{(s)}
			=
			\triplenorm{\RR{\mathcal{U}}}_{\gamma_1,0,O_{T+s}^s}^{(s)}.
		\end{align*}
	\end{lemma}
	\begin{proof}
		By definition, for $z,w\in \R^{d+1}\setminus P_{s}$, we have
		\begin{align*}
			\RR{\mathcal{U}}\tilde{\star}\rb(z)
			&:=
			{\mathcal{U}}^{\gY}(z)\rb
			+
			{\mathcal{U}}^{\small{\gI}}(z)\gIb
			+
			\sum_{i=1}^{d}
			{\mathcal{U}}^{\gXXXX}(z)\gXXXX\rb,
		\end{align*}
		and
		\begin{align*}
			\RR{\mathcal{U}}\tilde{\star}\rb(z)
			-
			\Gamma_{z,w}
			\left(\RR{\mathcal{U}}\tilde{\star}\rb(w)\right)
			=
			\big(
			\RR{\mathcal{U}}(z)
			-
			\Gamma_{z,w}\left(\RR{\mathcal{U}}(w)\right)
			\big)\tilde{\star}\rb.
		\end{align*}
		The claim now follows directly from the definition.
	\end{proof}
	\begin{remark}\label{BOUNDED}
		Assume that $F\in C_b^2(\mathbb{R},\mathbb{R})$ and
		$\RR{\mathcal{U}}\in
		\mathcal{D}^{\gamma_1,0}_{P_s,0}\big|_{O_{T+s}}$.
		Then, by standard estimates, for every $T\in(0,1]$,
		\[
		\triplenorm{
			\mathbf{R}^{+}_{s}F(\RR{\mathcal{U}})\tilde{\star}\rb
		}_{\gamma_0,-1-\kappa,O_{s+T}^{\,s}}^{(s)}
		=
		\triplenorm{
			\mathbf{R}^{+}_{s}F(\RR{\mathcal{U}})
		}_{\gamma_1,0,O_{s+T}^{\,s}}^{(s)}
		\lesssim
		1+\big([ \ \gI \ ]_{O_{T+s}^{\,s}}\big)^2
		+
		\bigg(
		\triplenorm{
			\mathbf{R}^{+}_{s}\RR{\mathcal{U}}
		}_{\gamma_1,0,O_{s+T}^{\,s}}^{(s)}
		\bigg)^2.
		\]
	\end{remark}
	We now define the concept of a solution in the sense of regularity structures.
	\begin{definition}\label{dfn}
		Let $s\in\mathbb{R}$, $T>0$, and $u_0\in L^\infty(\mathbb{T}^d)$.
		Suppose that
		\[
		\RR{\mathcal{U}}_{u_0}
		\in
		\mathcal{D}^{\gamma_1,0}_{P_s,0}\big|_{O^s_{s+T}}
		\]
		is the fixed point of the operator $\mathcal{Q}_T^s$
		defined in \eqref{FIXED}. For $z=(t,x)\in O^s_{s+T}$, write
		\[
		\RR{\mathcal{U}}_{u_0}(z)
		=
		\mathcal{U}^{\gY}_{u_0}(z)\,\gY
		+
		\mathcal{U}^{\gI}_{u_0}(z)\,\gI
		+
		\sum_{i=1}^{d}
		\mathcal{U}^{\gXXXX}_{u_0}(z)\,\gXXXX .
		\]
At the initial time, we set
\[
\mathcal{U}^{\gY}_{u_0}(s,\cdot):=u_0
\quad\text{in }L^\infty(\mathbb{T}^d).
\]
		We then say that $\RR{\mathcal{U}}_{u_0}$ solves
		\begin{align}\label{MAIN_1}
			\begin{cases}
				(\partial_t-\Delta)\RR{\mathcal{U}}
				=
				\mu\RR{\mathcal{U}}
				+
				\Sigma(\RR{\mathcal{U}})\tilde{\star}\rb,
				& t>s,\\
				\mathcal{U}^{\gY}(s,\cdot)=u_0,
				& t=s,
			\end{cases}
		\end{align}
		on $[s,s+T]$ with initial condition $u_0$ and with respect to the model
		$\mathcal{Z}$.
		The corresponding actual solution is defined by
		\[
		\mathcal{R}(\RR{\mathcal{U}}_{u_0})(z)
		=
		\mathcal{U}^{\gY}_{u_0}(z),
		\qquad
		z\in(s,s+T]\times\mathbb{R}^d .
		\]
	\end{definition}
	\begin{remark}
		From \eqref{A1}, we can see that for each
		$z\in(s,T+s]\times\mathbb{R}^d$,
		\begin{align*}
			\RR{\mathcal{U}}_{u_0}(z)
			=
			\mathcal{U}_{u_0}^{\gY}(z)\,\gY
			+
			\Sigma\big(\mathcal{U}_{u_0}^{\gY}(z)\big)\,\gI
			+
			\sum_{i=1}^{d}
			\mathcal{U}^{\gXXXX}_{u_0}(z)\,\gXXXX.
		\end{align*}
	\end{remark}
	
	\begin{lemma}
		For each $s\in\mathbb{R}$, there exists $T>0$ such that
		\eqref{MAIN_1} admits a unique local solution on $O^s_{s+T}$.
	\end{lemma}
	
	\subsection{A patching lemma}
	The previous lemma establishes local existence of the solution. However, as
	mentioned above, when the noise is not too singular (namely, when $\kappa$ is
	sufficiently small), global existence can be obtained. Before investigating the
	solution of \eqref{MAIN_1}, we first state the following general lemma for
	modelled distributions. This lemma provides a local-to-global estimate by
	patching together estimates on small time intervals and will be used
	repeatedly throughout the sequel.
	\begin{lemma}\label{global}
		Let $T\in(0,1]$ and $	\RR{\mathcal{U}}\in
		\mathcal{D}^{\gamma_1,0}_{P_s,0}\big|_{O^s_{s+T}}.$
		Suppose that there exist $0<\delta_0\leq T$ and a constant
		$\Lambda>0$ such that
		\begin{align}\label{ASASa}
			\forall\, s\leq b<c\leq s+T,\quad c-b\leq\delta_0:
			\qquad
			\triplenorm{
				\mathbf{R}^{+}_{b}\RR{\mathcal{U}}
			}_{\gamma_1,0,O_c^{b}}^{(b)}
			\leq \Lambda .
		\end{align}
		Then
		\begin{align*}
			\triplenorm{
				\mathbf{R}^{+}_{s}\RR{\mathcal{U}}
			}_{\gamma_1,0,O^s_{s+T}}^{(s)}
			\lesssim
			\left(\frac{1}{\delta_0}\right)^{\frac{\gamma_1+2}{2}}
			\big(\|\Gamma\|_{\overline{O^s_{s+T}}}\big)^2
			\Lambda .
		\end{align*}
	\end{lemma}
	\begin{proof}
		For a better flow of the proof, we divide it into several steps.
		
		\textbf{Step 1.}\label{Step1}
		Let
		\begin{align*}
			\mathcal{A}^{+} := \{ \gY, \gI, \gXXXX : i = 1,\ldots,d \}.
		\end{align*}
		The large bulk of the argument is devoted to estimating
		\begin{align}\label{Kmasas}
			\sup_{\substack{a\in\mathcal{A}^+}} \sup_{\substack{ z,w\in O_{T+s}^{s},\ z\neq w,\\ d(z,w)\leq\Vert z,w\Vert_{P_s} }}
			\frac{
				\Vert z,w \Vert_{P_{s}}^{\gamma_1}
				\, \Vert \RR{\mathcal{U}}(z) - \Gamma_{z,w}\RR{\mathcal{U}}(w) \Vert_{a}
			}{
				d(z,w)^{\gamma_1-|a|_h}
			}.
		\end{align}
		Let $z=(t_0,y_0)$ and $w=(t_1,y_1)$ be elements of
		$O_{s+T}^{s} = (s, s+T] \times \mathbb{R}^d$ such that
		$d(z,w)\leq \Vert z,w \Vert_{P_{s}}$.
		Then
		\begin{align}\label{2369}
			\max \Big\{ \sqrt{|t_1 - t_0|}, \Vert y_1 - y_0\Vert_{\mathbb{T}^d} \Big\}
			\leq \Vert z,w \Vert_{P_{s}}
			= \min\bigl\{ \sqrt{t_0 - s}, \sqrt{t_1 - s} \bigr\}.
		\end{align}
		For simplicity during the proof, we assume that
		\begin{enumerate}[label=\textbf{(*)}]
			\item \label{cond:A}
			\[
			\|z,w\|_{P_s} = \|z\|_{P_s} = \sqrt{t_0 - s}.
			\]
		\end{enumerate}
		The argument for the other case is similar.
		Let $\tilde{z}=(t_1,y_0)$. Then
		\begin{align}\label{BNasr}
			d(z,\tilde{z}) \leq \Vert z,\tilde{z}\Vert_{P_s} = \Vert z\Vert_{P_s}
			\leq \Vert\tilde{z},w\Vert_{P_s}=\Vert\tilde{z}\Vert_{P_s}
			\qquad \text{and} \qquad
			\max\lbrace	d(z,\tilde{z}), d(\tilde{z},w)\rbrace \leq d(z,w).
		\end{align}
		\begin{center}
			\begin{tikzpicture}[scale=.6]
				
				% Colors
				\definecolor{mypink}{RGB}{220,20,120}
				\definecolor{myblue}{RGB}{30,50,220}
				\definecolor{mygreen}{RGB}{50,180,50}
				
				% Rectangle
				\draw[line width=1.1pt] (0,0) rectangle (8,5);
				
				% Horizontal dashed lines
				\draw[dashed] (0,1.6) -- (8,1.6);
				\draw[dashed] (0,4.0) -- (8,4.0);
				
				% Vertical dashed lines
				\draw[dashed] (2,0) -- (2,1.6);
				\draw[dashed] (6,0) -- (6,1.6);
				
				% Blue path
				\draw[line width=1.8pt,myblue] (2,1.6) -- (6,1.6);
				\draw[line width=1.8pt,myblue] (6,1.6) -- (6,4.0);
				
				% Green points
				\fill[mypink!0!mygreen,mygreen] (2,1.6) circle (2pt);
				\fill[mypink!0!mygreen,mygreen] (6,1.6) circle (2pt);
				\fill[mypink!0!mygreen,mygreen] (6,4.0) circle (2pt);
				
				% Labels points
				\node[above left] at (2,1.6) {$z$};
				\node[right] at (6,2) {$\tilde z$};
				\node[above] at (6,4.0) {$w$};
				
				% Axis-like labels
				\node[left] at (-0.2,2.8) {\textcolor{red}{$\mathbb{T}^d$}};
				
				\node[left] at (0,1.6) {$y_0$};
				\node[left] at (0,4.0) {$y_1$};
				
				% Bottom labels
				\node[below] at (2,0) {$t_0$};
				\node[below] at (6,0) {$t_1$};
				
				\node[below left] at (0,0) {\textcolor{mypink}{$s$}};
				\node[below right] at (8,0) {\textcolor{mypink}{$s+T$}};
				
			\end{tikzpicture}
		\end{center}
		Also, we have
		\begin{align}\label{Kaiss}
			\RR{\mathcal{U}}(z) - \Gamma_{z,w}\RR{\mathcal{U}}(w)
			&=
			\RR{\mathcal{U}}(z) - \Gamma_{z,\tilde{z}}\RR{\mathcal{U}}(\tilde{z})
			+ \Gamma_{z,\tilde{z}}
			\bigl(
			\RR{\mathcal{U}}(\tilde{z}) - \Gamma_{\tilde{z},w}\RR{\mathcal{U}}(w)
			\bigr).
		\end{align}
		Using \eqref{BB1} and \eqref{BB2}, we obtain
		\begin{align}\label{Oalss}
			\begin{split}
				&\frac{\Vert {z}\Vert_{P_{s}}^{\gamma_1}\big\Vert \Gamma_{z,\tilde{z}}
					\bigl(
					\RR{\mathcal{U}}(\tilde{z}) - \Gamma_{\tilde{z},w}\RR{\mathcal{U}}(w)
					\bigr)\big\Vert_a}{d(z,w)^{\gamma_1-|a|_h}} \\
				&\leq \|\Gamma\|_{\overline{O_{s+T}^s}}
				\sum_{\substack{a_1\in\mathcal{A}^+,\\ |a|_h\leq|a_1|_h}}
				\frac{d(z,\tilde{z})^{|a_1|_h-|a|_h}d(\tilde{z},w)^{\gamma_1-|a_1|_h}}{d(z,w)^{\gamma_1-|a|_h}}
				\frac{\Vert {z}\Vert_{P_s}^{\gamma_1}}{\Vert \tilde{z}\Vert_{P_s}^{\gamma_1}}
				\frac{
					\Vert \tilde{z}\Vert_{P_{s}}^{\gamma_1}
					\, \Vert \RR{\mathcal{U}}(\tilde{z}) - \Gamma_{\tilde{z},w}\RR{\mathcal{U}}(w) \Vert_{a_1}
				}{
					d(\tilde{z},w)^{\gamma_1-|a_1|_h}
				}.
			\end{split}
		\end{align}
		Therefore, from \eqref{cond:A} and \eqref{2369}--\eqref{Oalss}, for every
		$a \in \mathcal{A}^+$ we have
		\begin{align}\label{BNasa}
			\begin{split}
				&\frac{
					\Vert z,w \Vert_{P_{s}}^{\gamma_1}
					\, \Vert \RR{\mathcal{U}}(z) - \Gamma_{z,w}\RR{\mathcal{U}}(w) \Vert_{a}
				}{
					d(z,w)^{\gamma_1-|a|_h}
				}
				\lesssim
				\frac{
					\Vert z\Vert_{P_{s}}^{\gamma_1}
					\, \Vert \RR{\mathcal{U}}(z) - \Gamma_{z,\tilde{z}}\RR{\mathcal{U}}(\tilde{z}) \Vert_{a}
				}{
					d(z,\tilde{z})^{\gamma_1-|a|_h}
				} \\
				&\quad +
				\|\Gamma\|_{\overline{O_{s+T}^s}}
				\sum_{\substack{a_1\in\mathcal{A}^+,\\ |a|_h\leq|a_1|_h}}
				\frac{
					\Vert \tilde{z}\Vert_{P_{s}}^{\gamma_1}
					\, \Vert \RR{\mathcal{U}}(\tilde{z}) - \Gamma_{\tilde{z},w}\RR{\mathcal{U}}(w) \Vert_{a_1}
				}{
					d(\tilde{z},w)^{\gamma_1-|a_1|_h}
				}.
			\end{split}
		\end{align}
		Consequently, by assuming \eqref{cond:A}, in order to estimate \eqref{Kmasas},
		we need to estimate the following two terms separately:
		\begin{itemize}
			\item
			\begin{align}\label{NMass}
				\sup_{a\in\mathcal{A}^+}
				\frac{
					\Vert z\Vert_{P_{s}}^{\gamma_1}
					\, \Vert \RR{\mathcal{U}}(z) - \Gamma_{z,\tilde{z}}\RR{\mathcal{U}}(\tilde{z}) \Vert_{a}
				}{
					d(z,\tilde{z})^{\gamma_1-|a|_h}
				}
			\end{align}
			\item
			\begin{align}\label{NMas1s}
				\sup_{a\in\mathcal{A}^+}\frac{
					\Vert \tilde{z}\Vert_{P_{s}}^{\gamma_1}
					\, \Vert \RR{\mathcal{U}}(\tilde{z}) - \Gamma_{\tilde{z},w}\RR{\mathcal{U}}(w) \Vert_{a}
				}{
					d(\tilde{z},w)^{\gamma_1-|a|_h}
				}
			\end{align}
		\end{itemize}
		The main task is to estimate \eqref{NMass}, which is the focus of the subsequent step.
		
		\textbf{Step 2.}\label{Step2}
		To estimate \eqref{NMass}, we first define
		\begin{align}\label{UJnasexs}
			(r_{-1}, r_0, r_1)=
			\begin{cases}
				(s, t_0, s+\delta_0) & \text{if } 0 < t_0 - s \leq\frac{\delta_0}{2}, \\
				\left(t_{0}-\frac{\delta_0}{2}, t_0, t_0+\frac{\delta_0}{2}\right) & \text{otherwise.}
			\end{cases}
		\end{align}
		
		With this definition, we set
		\begin{align*}
			k := \min \left\{ j \geq 0 : r_0 < t_{1} - \frac{j \delta_0}{2} \leq r_1 \right\}.
		\end{align*}
		Note that, since $r_{1}-r_{0}\geq \frac{\delta_0}{2}$, such a value exists.
		\begin{enumerate}[label=\textbf{(\Roman*)}]
			\item\label{I1} If $k=0$, we have
			\[
			0<t_1-t_0\leq \frac{\delta_0}{2}.
			\]
			We update $r_1=t_1$. Thus,
			\begin{align*}
				(r_{-1}, r_0, r_1)=
				\begin{cases}
					(s, t_0, t_1) & \text{if } 0 < t_0 - s \leq \frac{\delta_0}{2}, \\
					\left(t_{0} - \frac{\delta_0}{2}, t_0, t_1\right) & \text{otherwise.}
				\end{cases}
			\end{align*}
			We set $w_0 = z$ and $w_1 = \tilde{z}$. Then, together with \eqref{cond:A}
			\begin{align*}
				w_0, w_1 &\in O_{r_1}^{r_{-1}}, \quad r_1 - r_{-1} \leq \delta_0, \\
				d(w_0,w_1)=\sqrt{t_1-t_0}
				&\leq \| w_0,w_1 \|_{P_{r_{-1}}}= 
				\begin{cases}
					\| w_0 \|_{P_s} & \text{if } 0 < t_0 - s \leq \frac{\delta_0}{2}, \\
					\sqrt{\frac{\delta_0}{2}} & \text{otherwise.}
				\end{cases}
			\end{align*}
			%\ref{I1}
			\item \label{I2} If $k=1$, we set $r_2 = t_1$ and
			\begin{align*}
				w_{0} = z = (t_0, y_0), \quad 
				w_{1} = (r_1, y_0), \quad 
				w_{2} = \tilde{z} = (r_2, y_0).
			\end{align*}
			Since $r_{1} < t_1 \leq r_1 + \frac{\delta_0}{2}$, we conclude that
			\begin{align*}
				w_0, w_1 &\in O^{r_{-1}}_{r_1}, \quad
				r_{1} - r_{-1} = \delta_0, \\
				d(w_0,w_1)
				&
				\leq \| w_0 \|_{P_{r_{-1}}}
				=\sqrt{\frac{\delta_0}{2}} 
			\end{align*}
			Note that, from \eqref{2369} and \eqref{cond:A}, one can easily see that for $r_1$ in \eqref{UJnasexs}, we must have
			\begin{align*}
				r_1=t_0+\frac{\delta_0}{2}.
			\end{align*}
			Otherwise, this contradicts $r_1<t_1\leq r_1+\frac{\delta_0}{2}.$
			Thus,
			\begin{align*}
				& w_1, w_2 \in O_{r_2}^{r_1 - \frac{\delta_0}{2}}, \quad
				r_2 - \left(r_1 - \frac{\delta_0}{2}\right) \leq \delta_0, \\
				& d(w_1,w_2)\leq \| w_1,w_2 \|_{P_{r_1 - \frac{\delta_0}{2}}}
				= \sqrt{\frac{\delta_0}{2}}.
			\end{align*}
			Also,
			\begin{align}\label{Bnassa}
				\max\big\{ d(w_0,w_1),\, d(w_1,w_{2}) \big\} \leq d(w_0,w_2).
			\end{align}
			\item \label{I3}  If $k > 1$, we set
			\begin{align*}
				r_j &= r_1 + \frac{(j-1)\delta_0}{2}, \quad 1 < j \leq k, & 
				w_j &= (r_j, y_0), \quad 0 \leq j \leq k, \\
				r_{k+1} &= t_1, &
				w_{k+1} &= (t_1, y_0).
			\end{align*}
			In particular, we have
			\begin{align}\label{Nnasss}
				\max\big\{ d(w_0,w_j),\, d(w_j,w_{j+1}) \big\} \leq d(w_0,w_{k+1}), \quad 0 \leq j \leq k.
			\end{align}
			Moreover, we must have $t_{0}-s>\frac{\delta_0}{2}$  and $	r_1=t_0+\frac{\delta_0}{2}.$ Also
			\begin{align}\label{UJmss}
				w_0,w_1
				&\in O^{r_{-1}}_{r_1}, \quad
				r_1-r_{-1}=\delta_0, \quad
				d(w_0,w_1)
				\leq
				\|w_0,w_1\|_{P_{r_{-1}}}
				=
				\sqrt{\frac{\delta_0}{2}}
				\notag\\
				w_j,w_{j+1}
				&\in O^{r_{j-1}}_{r_{j+1}}, \quad
				r_{j+1}-r_{j-1}=\delta_0, \quad
				d(w_j,w_{j+1})
				\leq
				\|w_j,w_{j+1}\|_{P_{r_{j-1}}}
				=
				\sqrt{\frac{\delta_0}{2}},
				\quad 1\leq j<k,
				\notag\\
				w_k,w_{k+1}
				&\in O^{r_{k-1}}_{r_{k+1}}, \quad
				r_{k+1}-r_{k-1}\leq\delta_0, \quad
				d(w_k,w_{k+1})
				\leq
				\|w_k,w_{k+1}\|_{P_{r_{k-1}}}
				=
				\sqrt{\frac{\delta_0}{2}}.
			\end{align}
			\begin{center}
				\begin{tikzpicture}[scale=.7]
					
					\definecolor{mygreen}{RGB}{50,180,50}
					
					% Main rectangle
					\draw[thick] (-1,0) rectangle (18,5);
					
					% Vertical reference lines
					\draw[dashed] (1,0) -- (1,5);
					\draw[dashed] (2.2,0) -- (2.2,5);
					\draw[dashed] (3.5,0) -- (3.5,5);
					\draw[dashed] (5,0) -- (5,5);
					\draw[dashed] (8,0) -- (8,5);
					\draw[dashed] (9.5,0) -- (9.5,5);
					\draw[dashed] (11,0) -- (11,5);
					\draw[dashed] (15,0) -- (15,5);
					\draw[dashed] (18,0) -- (18,5);
					
					% Points
					\fill[mygreen] (2.2,2.5) circle (2pt);
					\node[right] at (2.2,2.7) {$w_0$};
					
					\fill[mygreen] (3.5,2.5) circle (2pt);
					\node[right] at (3.5,2.7) {$w_1$};
					
					\fill[mygreen] (5,2.5) circle (2pt);
					\node[right] at (5,2.7) {$w_2$};
					
					\node[right] at (6,2) {$\cdots$};
					
					\fill[mygreen] (8,2.5) circle (2pt);
					\node[left] at (8.2,2.7) {$w_{j-1}$};
					
					\fill[mygreen] (9.5,2.5) circle (2pt);
					\node[right] at (9.5,2.7) {$w_j$};
					
					\fill[mygreen] (11,2.5) circle (2pt);
					\node[right] at (11,2.7) {$w_{j+1}$};
					
					\node[right] at (12.5,2) {$\cdots$};
					
					\fill[mygreen] (15,2.5) circle (2pt);
					\node[right] at (15,2.7) {$w_{k+1}$};
					
					% Bottom labels
					\node at (1,-0.4) {$r_{-1}$};
					\node at (2.2,-0.4) {$r_0$};
					\node at (3.5,-0.4) {$r_1$};
					\node at (5,-0.4) {$r_2$};
					\node at (8,-0.4) {$r_{j-1}$};
					\node at (9.5,-0.4) {$r_j$};
					\node at (11,-0.4) {$r_{j+1}$};
					\node at (15,-0.4) {$r_{k+1}$};
					
					% End labels
					\node[red] at (-1,-0.4) {$s$};
					\node[red] at (18,-0.4) {$s+T$};
					
					% Horizontal spatial line
					\draw[black,dash dot] (-1,2.5) -- (18,2.5);
					
					% Spatial label
					\node[left] at (-1.2,2.8) {\textcolor{red}{$\mathbb{T}^d$}};
					
					% Local intervals
					\draw[blue,thick] (1,0) rectangle (3.5,5);
					\draw[red,dash dot,very thick] (2.2,0) rectangle (5,5);
					\draw[brown,thick] (8,0) rectangle (11,5);
					
				\end{tikzpicture}
			\end{center}
		\end{enumerate}
		We now continue the analysis for \eqref{I3}, i.e., the case $k > 1$.
		The remaining two cases, \eqref{I1} and \eqref{I2}, are simpler and can be handled analogously.
		First, note that, by \eqref{ASASa} and \eqref{UJmss},
		\begin{align}\label{NMAsse}
			\max\left\{
			\triplenorm{\mathbf{R}^+_{r_{-1}} \RR{\mathcal{U}}}_{\gamma_1,0,O_{r_1}^{r_{-1}}}^{(r_{-1})},  \,
			\triplenorm{\mathbf{R}^+_{r_{j-1}} \RR{\mathcal{U}}}_{\gamma_1,0,O_{r_{j+1}}^{r_{j-1}}}^{(r_{j-1})}
			: 1 \leq j \leq k
			\right\}
			\leq \Lambda.
		\end{align}
		Also, we know that
		\[
		\Gamma_{w_j,w_{j+1}} \circ \Gamma_{w_{j+1},w_{k+1}} = \Gamma_{w_j,w_{k+1}}, \quad \text{for } 0 \leq j \leq k-1.
		\]
		Thus,
		\begin{align}\label{Bnasz}
			\begin{split}
				&\RR{\mathcal{U}}(z)
				-
				\Gamma_{z,\tilde{z}}\RR{\mathcal{U}}(\tilde{z})
				=
				\RR{\mathcal{U}}(w_0)
				-
				\Gamma_{w_0,w_{k+1}}\RR{\mathcal{U}}(w_{k+1})
				\\
				&=
				\RR{\mathcal{U}}(w_0)
				-
				\Gamma_{w_0,w_1}\RR{\mathcal{U}}(w_1)
				+
				\sum_{1\leq j\leq k}
				\Gamma_{w_0,w_j}
				\Big(
				\RR{\mathcal{U}}(w_j)
				-
				\Gamma_{w_j,w_{j+1}}\RR{\mathcal{U}}(w_{j+1})
				\Big).
			\end{split}
		\end{align}
		Let us define
		\begin{align*}
			D(w_0,w_j,a,a_1)
			&= d(w_{0},w_{j})^{|a_1|_h-|a|_h}\, d(w_{j},w_{j+1})^{\gamma_1-|a_1|_{h}} \\
			&\quad \text{for } a,a_1 \in \mathcal{A}^{+} = \{ \gY, \gI, \gXXXX : i=1,\ldots,d \}, \quad |a|_h \leq |a_1|_h, \qquad 1 \leq j \leq k.
		\end{align*}
		Then, by \eqref{BB1}, \eqref{BB2}, and \eqref{Bnasz}, for each $a \in \mathcal{A}^{+}$, we obtain
		\begin{align}\label{Bnass}
			\begin{split}
				&\Vert \RR{\mathcal{U}} (z) - \Gamma_{z,\tilde{z}} \RR{\mathcal{U}} (\tilde{z}) \Vert_a\leq\Vert \RR{\mathcal{U}}(w_0) - \Gamma_{w_0,w_1}\RR{\mathcal{U}}(w_1)\Vert_{a}
				+ \sum_{1\leq j\leq k} \big\Vert \Gamma_{w_0,w_j} \big(\RR{\mathcal{U}}(w_j) - \Gamma_{w_j,w_{j+1}}\RR{\mathcal{U}}(w_{j+1}) \big) \big\Vert_{a}\\& \lesssim d(w_0,w_1)^{\gamma_1-|a|_h}
				\frac{\Vert \RR{\mathcal{U}}(w_0) - \Gamma_{w_0,w_1}\RR{\mathcal{U}}(w_1)\Vert_{a}}
				{d(w_0,w_1)^{\gamma_1-|a|_h}} \\
				%& + \|\Gamma\|_{\overline{O_{s+T}^s}}
				%	\sum_{\substack{a_1\in\mathcal{A}^+,\\ |a|_h\leq|a_1|_h}}
				%		d(w_0,w_1)^{|a_1|_h-|a|_h}
				%	\Vert \RR{\mathcal{U}}(w_1) - \Gamma_{w_1,w_{2}}\RR{\mathcal{U}}(w_{2}) \Vert_{a_1}\\&
				&	+ \|\Gamma\|_{\overline{O_{s+T}^s}}
				\sum_{1\leq j\leq k}\sum_{\substack{a_1\in\mathcal{A}^+,\\ |a|_h\leq|a_1|_h}}
				d(w_0,w_j)^{|a_1|_h-|a|_h}
				\Vert \RR{\mathcal{U}}(w_j) - \Gamma_{w_j,w_{j+1}}\RR{\mathcal{U}}(w_{j+1}) \Vert_{a_1} \\
				&= \frac{d(w_0,w_1)^{\gamma_1-|a|_h}}{\| w_0,w_1 \|_{P_{r_{-1}}}^{\gamma_1}}
				\frac{	\| w_0,w_1 \|_{P_{r_{-1}}}^{\gamma_1}\Vert \RR{\mathcal{U}}(w_0) - \Gamma_{w_0,w_1}\RR{\mathcal{U}}(w_1)\Vert_{a}}
				{d(w_0,w_1)^{\gamma_1-|a|_h}}\\
				& + \|\Gamma\|_{\overline{O_{s+T}^s}}
				\sum_{1\leq j\leq k}\sum_{\substack{a_1\in\mathcal{A}^+,\\ |a|_h\leq|a_1|_h}}
				\frac{
					D(w_0,w_j,a,a_1)
				}{
					\Vert w_j,w_{j+1} \Vert_{P_{r_{j-1}}}^{\gamma_1}
				}
				\frac{
					\Vert w_j,w_{j+1} \Vert_{P_{r_{j-1}}}^{\gamma_1}
					\Vert \RR{\mathcal{U}}(w_j) - \Gamma_{w_j,w_{j+1}}\RR{\mathcal{U}}(w_{j+1}) \Vert_{a_1}
				}{
					d(w_j,w_{j+1})^{\gamma_1-|a_1|_h}
				}.
			\end{split}
		\end{align}
		Thus, it follows from \eqref{Nnasss}, \eqref{UJmss} and \eqref{Bnass} that
		\begin{align}\label{HNasda}
			\begin{split}
				&\frac{\Vert w_0,w_{k+1}\Vert_{P_s}^{\gamma_1}\Vert\RR{\mathcal{U}} (w_0) - \Gamma_{w_0,w_{k+1}} \RR{\mathcal{U}} (w_{k+1})\Vert_{a}}{d(w_0,w_{k+1})^{\gamma_1-|a|_h}}\leq\frac{\Vert w_0,w_{k+1}\Vert_{P_s}^{\gamma_1}\Vert\RR{\mathcal{U}}(w_0) - \Gamma_{w_0,w_1}\RR{\mathcal{U}}(w_1)\Vert_{a}}{d(w_0,w_{k+1})^{\gamma_1-|a|_h}}\\&+\sum_{1\leq j\leq k} \frac{\Vert w_0,w_{k+1}\Vert_{P_s}^{\gamma_1}\big\Vert\Gamma_{w_0,w_j} \big(\RR{\mathcal{U}}(w_j) - \Gamma_{w_j,w_{j+1}}\RR{\mathcal{U}}(w_{j+1}) \big)\big\Vert_{a} }{d(w_0,w_{k+1})^{\gamma_1-|a|_h}} \\&\lesssim \frac{\Vert w_0,w_{k+1}\Vert_{P_s}^{\gamma_1}}{\| w_0,w_1 \|_{P_{r_{-1}}}^{\gamma_1}}\frac{d(w_0,w_1)^{\gamma_1-|a|_h}}{d(w_0,w_{k+1})^{\gamma_1-|a|_h}}\frac{\| w_0,w_1 \|_{P_{r_{-1}}}^{\gamma_1}\Vert\RR{\mathcal{U}}(w_0) - \Gamma_{w_0,w_1}\RR{\mathcal{U}}(w_1)\Vert_{a}}{d(w_0,w_1)^{\gamma_1-|a|_h}} %\\&+ \|\Gamma\|_{\overline{O_{s+T}^s}}\sum_{\substack{a_1\in\mathcal{A}^+,\\ |a|_h\leq|a_1|_h}}\frac{\Vert w_0,w_{k+1}\Vert_{P_s}^{\gamma_1}}{\Vert w_{1},w_2\Vert_{P_{r_1-\frac{\delta_0}{2}}}^{\gamma_1}}\frac{D(w_0,w_1,a,a_1)}{d(w_0,w_{k+1})^{\gamma_1-|a|_h}}\frac{\Vert w_{1},w_2\Vert_{P_{r_1-\frac{\delta_0}{2}}}^{\gamma_1}\Vert\RR{\mathcal{U}}(w_1) - \Gamma_{w_1,w_{2}}\RR{\mathcal{U}}(w_{2}) \Vert_{a_1}}{d(w_{1},w_{2})^{\gamma_1-|a_1|_{h}}}
				\\&+ \|\Gamma\|_{\overline{O_{s+T}^s}}	\sum_{1\leq j \leq k}\sum_{\substack{a_1\in\mathcal{A}^+,\\ |a|_h\leq|a_1|_h}}\frac{\Vert w_0,w_{k+1}\Vert_{P_s}^{\gamma_1}}{\Vert w_{j},w_{j+1}\Vert_{P_{r_{j-1}}}^{\gamma_1}}\frac{D(w_0,w_j,a,a_1)}{d(w_0,w_{k+1})^{\gamma_1-|a|_h}}\frac{\Vert w_{j},w_{j+1}\Vert_{P_{r_{j-1}}}^{\gamma_1}\Vert\RR{\mathcal{U}}(w_j) - \Gamma_{w_j,w_{j+1}}\RR{\mathcal{U}}(w_{j+1}) \Vert_{a_1}}{d(w_{j},w_{j+1})^{\gamma_1-|a_1|_{h}}}\\&\lesssim \left(\frac{1}{\delta_0}\right)^{\frac{\gamma_1+2}{2}} \|\Gamma\|_{\overline{O_{s+T}^s}}\Lambda.
			\end{split}
		\end{align}
		Note that, to derive the above bound, we also used \eqref{NMAsse} and the fact that $k\lesssim \frac{1}{\delta_0}$. It is easy to see that this bound also holds in the two cases \eqref{I1} and \eqref{I2}.
		To summarize, from \eqref{HNasda}, we conclude that, by assuming \eqref{cond:A},
		\begin{align}\label{NBN}
			\sup_{\substack{a\in\mathcal{A}^+}}
			\frac{\Vert z,\tilde{z} \Vert_{P_{s}}^{\gamma_1} \, \Vert \RR{\mathcal{U}}(z) - \Gamma_{z,\tilde{z}} \RR{\mathcal{U}}(\tilde{z}) \Vert_{a}}{d(z,\tilde{z})^{\gamma_1-|a|_h}}
			\lesssim \left(\frac{1}{\delta_0}\right)^{\frac{\gamma_1+2}{2}} \|\Gamma\|_{\overline{O_{s+T}^s}}\Lambda.
		\end{align}
		
		\textbf{Step 3.}\label{Step3}
		Now we estimate \eqref{NMas1s}. This is quite straightforward.
		First, recall that $\tilde{z}=(t_1,y_0)$ and $w=(t_1,y_1)$.
		There are two possible cases:
		\begin{itemize}
			\item
			If $t_1-s\leq\delta_0$, then $\tilde{z},w\in O_{t_1}^{s}$ and,
			thanks to \eqref{ASASa} and \eqref{2369},
			\begin{align*}
				\sup_{a\in\mathcal{A}^+}
				\frac{
					\Vert \tilde{z}\Vert_{P_{s}}^{\gamma_1}
					\, \Vert \RR{\mathcal{U}}(\tilde{z})
					- \Gamma_{\tilde{z},w}\RR{\mathcal{U}}(w) \Vert_{a}
				}{
					d(\tilde{z},w)^{\gamma_1-|a|_h}
				}
				\leq \Lambda.
			\end{align*}
			
			\item If $t_1-s>\delta_0$, then by considering the straight line
			between $y_0$ and $y_1$, we define a sequence
			$(\bar{y}_{j})_{0\leq j\leq m}$ such that
			\begin{align*}
				&\bar{y}_{0}=y_0,\qquad
				\bar{y}_{m}=y_1,\qquad
				\Vert \bar{y}_{m}-\bar{y}_{m-1}\Vert_{\mathbb{T}^d}
				\leq \delta_0,
				\\
				&\Vert \bar{y}_{j}-\bar{y}_{j-1}\Vert_{\mathbb{T}^d}
				=\delta_0,
				\qquad 1\leq j<m.
			\end{align*}
			In this way, by setting $\bar{w}_{j}=(t_1,\bar{y}_{j})$, we have
			\begin{align*}
				\bar{w}_j,\bar{w}_{j+1}
				&\in O^{t_1-\delta_0}_{t_1},
				\\
				d(\bar{w}_j,\bar{w}_{j+1})
				&\leq
				\|\bar{w}_j,\bar{w}_{j+1}\|_{P_{t_1-\delta_0}}
				=
				\sqrt{\delta_0},
				\\
				&\hspace{3.8cm}
				0\leq j<m,
				\qquad
				m\lesssim\frac{1}{\delta_0}.
			\end{align*}
			
			Similarly to \eqref{Bnasz},
			\begin{align}\label{Mkopass}
				\begin{split}
					&\RR{\mathcal{U}}(\tilde{z})
					-\Gamma_{\tilde{z},w}\RR{\mathcal{U}}(w)
					\\
					&=
					\RR{\mathcal{U}}(\bar{w}_0)
					-\Gamma_{\bar{w}_0,\bar{w}_1}
					\RR{\mathcal{U}}(\bar{w}_1)+
					\sum_{1\leq j<m}
					\Gamma_{\bar{w}_0,\bar{w}_j}
					\Big(
					\RR{\mathcal{U}}(\bar{w}_j)
					-\Gamma_{\bar{w}_j,\bar{w}_{j+1}}
					\RR{\mathcal{U}}(\bar{w}_{j+1})
					\Big).
				\end{split}
			\end{align}
			Using the same argument as in \eqref{Bnass} and \eqref{HNasda},
			we obtain from \eqref{Mkopass} that
			\begin{align}\label{NNMa}
				\sup_{a\in\mathcal{A}^+}
				\frac{
					\Vert\tilde{z},w\Vert_{P_s}^{\gamma_1}
					\,\Vert
					\RR{\mathcal{U}}(\tilde{z})
					-\Gamma_{\tilde{z},w}\RR{\mathcal{U}}(w)
					\Vert_a
				}{
					d(\tilde{z},w)^{\gamma_1-|a|_h}
				}
				\lesssim
				\left(\frac{1}{\delta_0}\right)^{\frac{\gamma_1+2}{2}}
				\|\Gamma\|_{\overline{O_{s+T}^s}}
				\Lambda.
			\end{align}
		\end{itemize}
		\textbf{Step 4.}\label{Step4} Now we finalize the proof. From \eqref{BNasa}, \eqref{NBN}, and \eqref{NNMa}, we conclude that, by assuming \eqref{cond:A},
		\begin{align*}
			\sup_{a\in\mathcal{A}^+}\sup_{\substack{ z,w\in O_{T+s}^s,\ z\neq w\\ d(z,w)\leq\|z,w\|_{P_s}\\ \|z,w\|_{P_s}=\|z\|_{P_s} }}
			\frac{
				\Vert z,w \Vert_{P_{s}}^{\gamma_1}
				\, \Vert \RR{\mathcal{U}}(z) - \Gamma_{z,w}\RR{\mathcal{U}}(w) \Vert_{a}
			}{
				d(z,w)^{\gamma_1-|a|_h}
			}
			\lesssim \left(\frac{1}{\delta_0}\right)^{\frac{\gamma_1+2}{2}}
			\big( \|\Gamma\|_{\overline{O_{s+T}^s}}\big)^2 \Lambda.
		\end{align*}
		As mentioned earlier, the assumption \eqref{cond:A} was made for simplicity, and the above argument from the beginning until here can be adapted to the case $z,w \in O_{T+s}^s$ with $\|z,w\|_{P_s} = \|w\|_{P_s}$. Thus, to summarize, we can state that
		\begin{align}\label{Bczxxxxz}
			\sup_{a\in\mathcal{A}^+}\sup_{\substack{ z,w\in O_{T+s}^s,\ z\neq w\\ d(z,w)\leq\|z,w\|_{P_s} }}
			\frac{
				\Vert z,w \Vert_{P_{s}}^{\gamma_1}
				\, \Vert \RR{\mathcal{U}}(z) - \Gamma_{z,w}\RR{\mathcal{U}}(w) \Vert_{a}
			}{
				d(z,w)^{\gamma_1-|a|_h}
			}
			\lesssim \left(\frac{1}{\delta_0}\right)^{\frac{\gamma_1+2}{2}}
			\big( \|\Gamma\|_{\overline{O_{s+T}^s}}\big)^2 \Lambda.
		\end{align}
		Also note that, from \eqref{ASASa}, for $z = (t_0,x)$ and $a\in \mathcal{A}^+$, we have
		\begin{align}\label{MBM}
			\Vert z\Vert_{P_s}^{|a|_{h}}\Vert\RR{\mathcal{U}} (z)\Vert_{a}
			\leq 
			\begin{cases}
				\Lambda, & \text{if } \sqrt{t_0 - s} = \Vert z\Vert_{P_s} \leq \sqrt{\delta_0}, \\[0.5em]
				\displaystyle
				\frac{\Vert z\Vert_{P_s}^{|a|_{h}}}{\Vert z\Vert_{P_{t_0-\delta_0}}^{|a|_{h}}}
				\Vert z\Vert_{P_{t_0-\delta_0}}^{|a|_{h}}
				\Vert\RR{\mathcal{U}}(z)\Vert_{a}
				\lesssim \left(\frac{1}{\delta_0}\right)^{\frac{\gamma_1}{2}} \Lambda,
				& \text{otherwise}.
			\end{cases}
		\end{align}
		
		The claim then follows from \eqref{Bczxxxxz} and \eqref{MBM}.
	\end{proof}
	As a first application of the previous lemma, we show that once the existence
	of $\RR{\mathcal{U}}_{u_0}$ on $O^s_{\,s+T}$ is established and an a priori
	bound for $\mathcal{U}^{\gY}_{u_0}$ is available, one can naturally derive
	an a priori estimate for $\triplenorm{\mathbf{R}^+_{s}\RR{\mathcal{U}}_{u_0}}_{\gamma_1,0,O_{s+T}}^{(s)}.$
	\begin{lemma}\label{Alex}
		Let $s\in\mathbb{R}$ and $T\in(0,1]$. Suppose that
		$\RR{\mathcal{U}}_{u_0}$ is a solution to \eqref{MAIN_1} on $O^s_{s+T}$. Then
		\begin{align}\label{SSAS78sasassss}
			\begin{split}
				&\triplenorm{\mathbf{R}^+_{s}\RR{\mathcal{U}}_{u_0}}_{\gamma_1,0,O_{s+T}^s}^{(s)}
				\lesssim P_{1}\Big(\|\mathcal{Z}\|_{\overline{O_{s+1+T}^{s-1}}},[\ \gI\ ]_{\overline{{O}_{T+s}^s}}, \|\Gamma\|_{\overline{O_{s+T}^s}},\sup_{w\in O_{s+T}^s}
				\|\mathcal{U}^{\gY}_{u_0}(w)\|,\|u_0\|_{L^{\infty}(\mathbb{T}^d)}\Big)	
			\end{split}
		\end{align}
		where
		\begin{align}\label{ASasas78qw}
			\begin{split}
				P_{1}(b,c,d,f,g):=&
				d^2
				(1+b^2)^{\frac{\gamma_1+2}{1-\kappa}}
				\big(1+|\mu|+c+f^{\gamma_1-1}\big)^{\frac{\gamma_1+2}{1-\kappa}}
				\big(
				f+g+(1+b^2)(1+c^2+f^{\gamma_1})
				\big),
				\\
				&\qquad \text{for } b,c,d,f,g \geq 0.
			\end{split}
		\end{align}
	\end{lemma}
	\begin{proof}
		The proof is a bit lengthy, so we defer it to Appendix~\ref{Alexi}.
	\end{proof}
	\subsection{Linearization}
	The key element of our analysis is the linearization of the equation. 
	Since the existence theory is based on a fixed point argument, the associated solution map is Fréchet differentiable. 
	In this part, we discuss this property in more detail and first establish the Fréchet differentiability of the solution map. 
	This issue has also been addressed in \cite{FK22}; however, we provide a short proof following the approach of \cite{Bai15b, GVR21}. 
	In addition, we derive an explicit formula for the linearized equation that will be used frequently in the sequel. Let us first start with an auxiliary lemma.
	\begin{lemma}\label{ATGFRAs}
		Assume that   $n \ge 3$ and $\Sigma \in C^{n}(\R,\R)$. Then, for every $T>0$, the map
		\begin{align}\label{TGBG}
			\begin{split}
				\overline{\mathcal{Q}}^{s}_T &: 
				L^{\infty}(\mathbb{T}^{d}) \times 
				\mathcal{D}^{\gamma_1,0}_{P_{s},0}\big|_{O_{T+s}} 
				\longrightarrow 
				\mathcal{D}^{\gamma_1,0}_{P_{s},0}\big|_{O_{T+s}}, \\
				\overline{\mathcal{Q}}^{s}_T(\tilde{u},\RR{\mathcal{U}}) &=
				\mathcal{G}^{s}_{\gamma_{0},1-\kappa}
				\Big(
				\mathbf{R}^{+}_{s}\Sigma(\RR{\mathcal{U}} + \RR{G^{s}}\tilde{u}) \,\tilde{\star}\, \rb
				\Big)+\mu\mathcal{G}^{s}_{\gamma_{1},2}\big(\mathbf{R}^{+}_{s}\RR{\mathcal{U}} + \RR{G^{s}}(\tilde{u})\big)
			\end{split}
		\end{align}
		is $C^{\,n-2}$–Fréchet differentiable in both of its arguments.
	\end{lemma}
	\begin{proof}
		For $i=1,2$, set
		\[
		\RR{\mathcal{V}}_i
		:=
		\RR{\mathcal{U}}_i+\RR{G^s}\tilde{u}_i.
		\]
		For \(1 \leq j \leq n-2\), the composition estimates \eqref{TYUa1}, together with \eqref{multi}, yield
		\begin{align}\label{TYLOR}
			\begin{split}
				\Sigma(\RR{\mathcal{V}_2}) - \Sigma(\RR{\mathcal{V}_1})
				&=
				\sum_{i=1}^{j}\frac{1}{i!}
				\Sigma^{(i)}(\RR{\mathcal{V}_1})
				\star^{i}
				\bigl(\RR{\mathcal{V}_2}-\RR{\mathcal{V}_1}\bigr)+o\!\left(
				\left(
				\triplenorm{\mathcal{V}_2-\mathcal{V}_1}_{\gamma_1,\,0,\,O_{T+s}}^{(s)}
				\right)^{j}
				\right).
			\end{split}
		\end{align}
Here the little-$o$ notation has its usual meaning.
		Since $\RR{G^s}$ is linear and continuous,
		\[
		\triplenorm{\mathcal{V}_2-\mathcal{V}_1}_{\gamma_1,\,0,\,O_{T+s}}^{(s)}
		\lesssim
		\triplenorm{\mathcal{U}_2-\mathcal{U}_1}_{\gamma_1,\,0,\,O_{T+s}}^{(s)}
		+
		\|\tilde{u}_2-\tilde{u}_1\|_{L^\infty(\mathbb{T}^d)}.
		\]
		Since $\mathcal{G}^{s}_{\gamma_{0},1-\kappa}$,
		$\mathcal{G}^{s}_{\gamma_{1},2}$,
		and $\RR{G^{s}}$ are linear and continuous, the desired conclusion follows from \eqref{TYLOR}.
	\end{proof}
	We are now ready to prove the Fréchet differentiability of the solution.
	\begin{lemma}\label{YHATSs}
		Let $n \geq 3$ and $\Sigma \in C^{n}(\mathbb{R},\mathbb{R})$. Assume that, for a given initial condition
		$u_{0}\in L^{\infty}(\mathbb{T}^{d})$, there exist constants $T>0$ and $\sigma>0$, and a map
		\[
		\RR{\mathcal{J}_T} :
		B_{L^{\infty}(\mathbb{T}^{d})}(u_{0},\sigma)
		\longrightarrow
		\mathcal{D}^{\gamma_1,0}_{P_s,0}\big|_{O_{T+s}},
		\]
		such that, for every
		$\tilde{u}\in B_{L^{\infty}(\mathbb{T}^{d})}(u_{0},\sigma)$,
		the modelled distribution
		$\RR{\mathcal{J}_T}(\tilde{u})=\RR{\mathcal{U}}_{\tilde{u}}$
		is the solution of \eqref{MAIN_1} on $[s,s+T]$ with initial condition
		$\tilde{u}$. Then
		\[
		\RR{\mathcal{J}_T}
		\in
		C^{\,n-2}\!\left(
		B_{L^{\infty}(\mathbb{T}^{d})}(u_{0},\sigma),
		\,
		\mathcal{D}^{\gamma_1,0}_{P_s,0}\big|_{O_{T+s}}
		\right).
		\]
		Moreover, for each fixed $t\in[0,T]$, define
		\[
		\mathcal{J}_{t} :
		B_{L^{\infty}(\mathbb{T}^{d})}(u_{0},\sigma)
		\longrightarrow
		L^{\infty}(\mathbb{T}^{d}),
		\qquad
		\mathcal{J}_{t}(\tilde{u})
		:=
		\mathcal{R}\!\left(\RR{\mathcal{J}_T}(\tilde{u})\right)(s+t,\cdot).
		\]
		Then
		\[
		\mathcal{J}_{t}
		\in
		C^{\,n-2}\!\left(
		B_{L^{\infty}(\mathbb{T}^{d})}(u_{0},\sigma),
		L^{\infty}(\mathbb{T}^{d})
		\right),
		\]
		and
		\[
		\mathcal{J}_{0}=\mathrm{Id}.
		\]
	\end{lemma}
	\begin{proof}
		First, by Lemma~\ref{ATGFRAs}, the map
		$\overline{\mathcal{Q}}_{T}^{s}$ is $C^{\,n-2}$--Fréchet differentiable.
		Moreover, using \eqref{FIXED}, \eqref{fixed-point}, and Lemma~\ref{stat},
		we can choose $T_0\in(0,T]$ and $\sigma_0\in(0,\sigma]$ sufficiently small
		such that, for some $M>0$,
		\[
		\overline{\mathcal Q}_{T_0}^{s}:
		\overline{B_{L^\infty(\mathbb T^d)}(u_0,\sigma_0)}
		\times
		\overline{
			B_{\mathcal D^{\gamma_1,0}_{P_s,0}|_{O_{T_0+s}}}(0,M)}
		\longrightarrow
		\overline{
			B_{\mathcal D^{\gamma_1,0}_{P_s,0}|_{O_{T_0+s}}}(0,M)}
		\]
		and
		\begin{align*}
			&\triplenorm{
				\overline{\mathcal Q}_{T_0}^{s}(u_2,\RR{\mathcal U}_2)
				-
				\overline{\mathcal Q}_{T_0}^{s}(u_1,\RR{\mathcal U}_1)}
			_{\gamma_1,0,O_{T_0+s}}^{(s)}
			\\
			&\qquad\leq
			\frac12
			\left(
			\triplenorm{
				\mathbf R_s^+\RR{\mathcal U}_2-
				\mathbf R_s^+\RR{\mathcal U}_1}
			_{\gamma_1,0,O_{T_0+s}}^{(s)}
			+
			\|u_2-u_1\|_{L^\infty(\mathbb T^d)}
			\right).
		\end{align*}
		Therefore, for every
		$\tilde u\in B_{L^\infty(\mathbb T^d)}(u_0,\sigma_0)$,
		the map
		\[
		\RR{\mathcal U}\mapsto
		\overline{\mathcal Q}_{T_0}^{s}
		(\tilde u,\RR{\mathcal U})
		\]
		is a contraction. Hence, by the contraction mapping theorem,
		there exists a unique fixed point $\RR g(\tilde u)$ satisfying
		\begin{align}\label{IAS78asaas}
			\overline{\mathcal Q}_{T_0}^{s}
			(\tilde u,\RR g(\tilde u))
			=
			\RR g(\tilde u).
		\end{align}
		The corresponding solution can be written as
		\[
		\RR{\mathcal J}_{T_0}(\tilde u)
		=
		\RR g(\tilde u)+\RR G^s(\tilde u).
		\]
		We now prove the regularity of the fixed point with respect to the
		initial condition. Fix
		$u_1\in B_{L^\infty(\mathbb T^d)}(u_0,\sigma_0)$.
		Let $N_{u_1}$ and $W_{u_1}$ be neighborhoods of $u_1$ and
		$\RR g(u_1)$, respectively, such that
		\[
		\mathbf D:
		N_{u_1}\times W_{u_1}
		\to
		\mathcal D^{\gamma_1,0}_{P_s,0}|_{O_{T_0+s}}
		\]
		given by
		\[
		\mathbf D(\tilde u,\RR{\mathcal U})
		=
		\RR{\mathcal U}
		-
		\overline{\mathcal Q}_{T_0}^{s}
		(\tilde u,\RR{\mathcal U})
		\]
		is well defined. By \eqref{IAS78asaas},
		\[
		\mathbf D(u_1,\RR g(u_1))=0.
		\]
		For fixed $\tilde u$, define
		\[
		\mathbf D_{\tilde u}(\RR{\mathcal U})
		=
		\mathbf D(\tilde u,\RR{\mathcal U}).
		\]
		Its derivative with respect to $\RR{\mathcal U}$ is
		\[
		\frac{\partial\mathbf D_{\tilde u}}
		{\partial\RR{\mathcal U}}
		(\RR{\mathcal U})
		=
		I-
		\frac{\partial\overline{\mathcal Q}_{T_0}^{s}}
		{\partial\RR{\mathcal U}}
		(\tilde u,\RR{\mathcal U}).
		\]
		Set
		\[
		A_{\tilde u,\RR{\mathcal U}}
		=
		\frac{\partial\overline{\mathcal Q}_{T_0}^{s}}
		{\partial\RR{\mathcal U}}
		(\tilde u,\RR{\mathcal U}).
		\]
		The contraction estimate implies
		\[
		\|A_{\tilde u,\RR{\mathcal U}}\|\leq \frac12.
		\]
		Hence $I-A_{\tilde u,\RR{\mathcal U}}$ is invertible and its inverse
		is given by the Neumann series
		\[
		\left(
		\frac{\partial\mathbf D_{\tilde u}}
		{\partial\RR{\mathcal U}}
		(\RR{\mathcal U})
		\right)^{-1}
		=
		\sum_{k=0}^{\infty}
		A_{\tilde u,\RR{\mathcal U}}^{\,k},
		\]
		where $A_{\tilde u,\RR{\mathcal U}}^{\,k}$ denotes the $k$-fold
		composition of the same bounded linear operator.
		Therefore,
		\[
		\frac{\partial\mathbf D_{\tilde u}}
		{\partial\RR{\mathcal U}}
		(\RR{\mathcal U})
		\]
		is an isomorphism. Applying the implicit function theorem in Banach spaces
		\cite[Theorem~2.5.7]{AMR88}, we conclude that
		\[
		\RR g:
		N_{u_1}\longrightarrow
		\mathcal D^{\gamma_1,0}_{P_s,0}|_{O_{T_0+s}}
		\]
		is of class $C^{n-2}$. Since $u_1$ was arbitrary, and since
		$\RR G^s$ is linear and continuous, it follows that
		\[
		\RR{\mathcal J}_{T_0}
		\in
		C^{n-2}\!\left(
		B_{L^\infty(\mathbb T^d)}(u_0,\sigma_0),
		\mathcal D^{\gamma_1,0}_{P_s,0}|_{O_{T_0+s}}
		\right).
		\]
		To extend the result to $[s,s+T]$, we use the restarting property.
		At time $s+T_0$, the initial condition for the restarted equation is
		\[
		\mathcal J_{T_0}(\tilde u).
		\]
		Since $\mathcal J_{T_0}$ is of class $C^{n-2}$ with respect to
		$\tilde u$, the same argument applied to the restarted equation shows
		that the solution depends $C^{n-2}$--smoothly on $\tilde u$ on the next
		time interval. Iterating this argument and using that $[s,s+T]$ can be
		covered by finitely many such intervals, we obtain
		\[
		\RR{\mathcal J}_{T}
		\in
		C^{n-2}\!\left(
		B_{L^\infty(\mathbb T^d)}(u_0,\sigma),
		\mathcal D^{\gamma_1,0}_{P_s,0}|_{O_{T+s}}
		\right).
		\]
		Finally, for every fixed $t\in(0,T]$, the reconstruction and evaluation
		map
		\[
		\RR{\mathcal U}
		\longmapsto
		\mathcal R(\RR{\mathcal U})(s+t,\cdot)
		\]
		is continuous into $L^\infty(\mathbb T^d)$. Hence
		\[
		\mathcal J_t
		\in
		C^{n-2}\!\left(
		B_{L^\infty(\mathbb T^d)}(u_0,\sigma),
		L^\infty(\mathbb T^d)
		\right).
		\]
		For $t=0$, the initial condition gives
		\[
		\mathcal J_0=\mathrm{Id}.
		\]
	\end{proof}
	For our purposes, we should derive the equation satisfied by the linearized solution. This is the content of the next proposition.
	\begin{proposition}\label{kkkl}
		Consider the setting of Lemma \ref{YHATSs} and let 
		$u_1 \in B_{L^{\infty}(\T^d)}(u_0, \sigma)$. Then the Fréchet derivative
		\[
		D_{u_1}\RR{\mathcal{J}_T} : L^{\infty}(\T^d) \longrightarrow \mathcal{D}^{\gamma_1,0}_{P_{s},0}\big|_{O_{T+s}}
		\]
		is well-defined. Moreover, for any $\tilde{u} \in L^{\infty}(\T^d)$, the derivative $D_{u_1}\RR{\mathcal{J}_T}[\tilde{u}]$ satisfies the linearized equation on $[s,T+s]$:
		\begin{align}\label{MAIN2}
			\begin{cases}
				(\partial_t - \Delta) D_{u_1}\RR{\mathcal{J}_T}[\tilde{u}]
				= \mu D_{u_1}\RR{\mathcal{J}_T}[\tilde{u}]+\big(\Sigma'(\RR{\mathcal{J}_T}(u_1)) \star D_{u_1}\RR{\mathcal{J}_T}[\tilde{u}]\big) \tilde{\star} \rb, & t>s, \\
				\mathcal{R}\left(D_{u_1}\RR{\mathcal{J}_T}[\tilde{u}]\right)(s,\cdot) = \tilde{u}, & t=s.
			\end{cases}
		\end{align}
		Equivalently, in mild form on $[s,s+T]$, we have
		\begin{align}\label{UIO}
			\mathbf{R}^+_{s}D_{u_1}\RR{\mathcal{J}_T}[\tilde{u}]
			&=
			\mathcal{G}^{s}_{\gamma_{0},1-\kappa}
			\Big(
			\big(\mathbf{R}^+_{s}\Sigma'(\RR{\mathcal{J}_T}(u_1)) \star D_{u_1}\RR{\mathcal{J}_T}[\tilde{u}]\big) \tilde{\star} \rb
			\Big)+\mu\mathcal{G}^{s}_{\gamma_{1},2}(\mathbf{R}^+_{s} D_{u_1}\RR{\mathcal{J}_T}[\tilde{u}])+ \RR{G^{s}}(\tilde{u}).
		\end{align}
	\end{proposition}
	\begin{proof}
		Thanks to Lemma~\ref{YHATSs}, the derivative operator exists. In particular, for every 
		$u_1 \in B_{L^{\infty}(\T^d)}(u_0, \sigma)$ and $\tilde{u} \in L^{\infty}(\T^d)$,
		\begin{align}\label{epsilon}
			\mathbf{R}^+_{s}D_{u_1}\RR{\mathcal{J}_T}[\tilde{u}]
			=
			\lim_{\epsilon \to 0}
			\frac{\mathbf{R}^+_{s}\RR{\mathcal{J}_T}(u_1+\epsilon\tilde{u})
				-
				\mathbf{R}^+_{s}\RR{\mathcal{J}_T}(u_1)}{\epsilon}.
		\end{align}
		For sufficiently small $\epsilon$, set
		\begin{align*}
			\begin{split}
				&\RR{\mathcal{H}_1}(u_1,\tilde{u},\epsilon)
				:=
				\frac{
					\Sigma\big(\RR{\mathcal{J}_T}(u_1+\epsilon\tilde{u})\big)
					-
					\Sigma\big(\RR{\mathcal{J}_T}(u_1)\big)
				}{\epsilon}
				-
				\Sigma'(\RR{\mathcal{J}_T}(u_1))
				\star
				D_{u_1}\RR{\mathcal{J}_T}[\tilde{u}],
				\\
				&\RR{\mathcal{H}_2}(u_1,\tilde{u},\epsilon)
				:=
				\frac{\RR{\mathcal{J}_T}(u_1+\epsilon\tilde{u})
					-
					\RR{\mathcal{J}_T}(u_1)}{\epsilon}
				-
				D_{u_1}\RR{\mathcal{J}_T}[\tilde{u}].
			\end{split}
		\end{align*}
		Subtracting the mild equations satisfied by
		$\RR{\mathcal{J}_T}(u_1+\epsilon\tilde{u})$
		and $\RR{\mathcal{J}_T}(u_1)$, and dividing by $\epsilon$, it is sufficient to show that
		\begin{align}\label{Q0}
			\begin{split}
				&\lim_{\epsilon \to 0}
				\big\Vert\big\vert
				\mathcal{G}^{s}_{\gamma_{0},1-\kappa}
				\big(
				\mathbf{R}^+_{s}\RR{\mathcal{H}_1}(u_1,\tilde{u},\epsilon)
				\tilde{\star}
				\rb
				\big)
				\big\vert\big\Vert_{\gamma_1,\,0,\,O_{T+s}^s}^{(s)}
				=0,
				\\
				&\lim_{\epsilon \to 0}
				\big\Vert\big\vert
				\mathcal{G}^{s}_{\gamma_{1},2}
				\big(
				\mathbf{R}^+_{s}\RR{\mathcal{H}_2}(u_1,\tilde{u},\epsilon)
				\big)
				\big\vert\big\Vert_{\gamma_1,\,0,\,O_{T+s}^s}^{(s)}
				=0.
			\end{split}
		\end{align}
		Without loss of generality, we may assume that $T\in (0,1]$ and then apply a
		patching argument to treat arbitrary $T$. Thanks to \eqref{A2222}, \eqref{A222233},
		and \eqref{multi}, it suffices to prove that
		\begin{align}\label{UJNMa}
			\begin{split}
				&\lim_{\epsilon \to 0}
				\triplenorm{
					\mathbf{R}^+_{s}\RR{\mathcal{H}_1}(u_1,\tilde{u},\epsilon)
				}_{\gamma_1,0,O_{T+s}^s}^{(s)}
				=0,
				\\
				&\lim_{\epsilon \to 0}
				\triplenorm{
					\mathbf{R}^+_{s}\RR{\mathcal{H}_2}(u_1,\tilde{u},\epsilon)
				}_{\gamma_1,0,O_{T+s}^s}^{(s)}
				=0.
			\end{split}
		\end{align}
		The second one follows from \eqref{epsilon}. Let us address the first one.
		To this end, by the definition of $\RR{\mathcal{H}_1}$ and the Taylor expansion
		\eqref{TYUa}, we write
		\begin{align*}
			&\RR{\mathcal{H}_1}(u_1,\tilde{u},\epsilon)
			\\
			&=
			\underbrace{
				\int_{0}^{1}
				\bigg[
				\Sigma'\Big(
				\RR{\mathcal{J}_T}(u_1)
				+
				\alpha\big(
				\RR{\mathcal{J}_T}(u_1+\epsilon\tilde{u})
				-
				\RR{\mathcal{J}_T}(u_1)
				\big)
				\Big)
				-
				\Sigma'(\RR{\mathcal{J}_T}(u_1))
				\bigg]
				\star
				\bigg[
				\frac{
					\RR{\mathcal{J}_T}(u_1+\epsilon\tilde{u})
					-
					\RR{\mathcal{J}_T}(u_1)
				}{\epsilon}
				\bigg]
				\,\mathrm{d}\alpha
			}_{=: \textcolor{red}{I}(u_1,\tilde{u},\epsilon)}
			\\
			&\quad+
			\underbrace{
				\Sigma'(\RR{\mathcal{J}_T}(u_1))
				\star
				\bigg[
				\frac{
					\RR{\mathcal{J}_T}(u_1+\epsilon\tilde{u})
					-
					\RR{\mathcal{J}_T}(u_1)
				}{\epsilon}
				-
				D_{u_1}\RR{\mathcal{J}_T}[\tilde{u}]
				\bigg]
			}_{=: \textcolor{red}{II}(u_1,\tilde{u},\epsilon)} .
		\end{align*}
		From \eqref{epsilon}, \eqref{TYUa}, and \eqref{multi}, we conclude that
		\begin{align*}
			\lim_{\epsilon \to 0}
			\triplenorm{
				\mathbf{R}^+_{s}\textcolor{red}{I}(u_1,\tilde{u},\epsilon)
			}_{\gamma_1,0,O_{T+s}^s}^{(s)}
			=0,
			\qquad
			\lim_{\epsilon \to 0}
			\triplenorm{
				\mathbf{R}^+_{s}\textcolor{red}{II}(u_1,\tilde{u},\epsilon)
			}_{\gamma_1,0,O_{T+s}^s}^{(s)}
			=0.
		\end{align*}
		This yields \eqref{UJNMa}, and hence, after passing to the limit in the
		difference quotient of the mild equations,
		\begin{align*}
			\mathbf{R}^+_{s}D_{u_1}\RR{\mathcal{J}_T}[\tilde{u}]
			&=
			\mathcal{G}^{s}_{\gamma_{0},1-\kappa}
			\Big(
			\mathbf{R}^+_{s}
			\Big(
			\Sigma'(\RR{\mathcal{J}_T}(u_1))
			\star
			D_{u_1}\RR{\mathcal{J}_T}[\tilde{u}]
			\Big)
			\tilde{\star}\rb
			\Big)+
			\mu\mathcal{G}^{s}_{\gamma_{1},2}
			\Big(
			\mathbf{R}^+_{s}D_{u_1}\RR{\mathcal{J}_T}[\tilde{u}]
			\Big)
			+
			\RR{G^{s}}(\tilde{u}).
		\end{align*}
		This proves \eqref{UIO}. Finally, since $\mathcal{J}_0=\mathrm{Id}$ by
		Lemma~\ref{YHATSs}, differentiating at $u_1$ in the direction $\tilde{u}$ yields
		\[
		D_{u_1}\mathcal{J}_0[\tilde{u}]=\tilde{u},
		\]
		and therefore
		\[
		\mathcal{R}\left(
		D_{u_1}\RR{\mathcal{J}_T}[\tilde{u}]
		\right)(s,\cdot)
		=
		\tilde{u}.
		\]
		Hence \eqref{MAIN2} follows.
	\end{proof}
	\subsection{Gr\"onwall-type inequality}
	Now that we understand the regularity of the solution map and its linearization, our next goal is to derive a priori estimates. To this end, we follow a classical approach based on Grönwall's inequality. We first prove a version of Grönwall's inequality adapted to our setting. This result will be used repeatedly throughout the sequel and will considerably simplify many of the subsequent estimates. A similar result in the context of rough partial differential equations can be found in \cite{BGS2025}.
	
	\begin{proposition}\label{onwall}
		Let 
		$\RR{\mathcal{U}}, \RR{\mathcal{V}} \in \mathcal{D}^{\gamma_1,0}_{P_s,0}\big|_{O_{1+s}}$, 
		$\RR{\mathcal{W}} \in \mathcal{D}^{\gamma_0,-1-\kappa}_{P_s}\big|_{O_{1+s}}$, 
		and $u \in L^{\infty}(\mathbb{T}^d)$. 
		Assume that on $[s, 1+s] \times \mathbb{R}^d$ the following identity holds:
		\begin{align}\label{OL<45}
			\begin{cases}
				\mathbf{R}^+_s \RR{\mathcal{U}}
				=\mathcal{G}^{s}_{\gamma_{0},1-\kappa}
				\Big(
				\big(\mathbf{R}^+_s\RR{\mathcal{V}} \star \RR{\mathcal{U}}\big) \tilde{\star} \rb
				+\mathbf{R}^+_s\RR{\mathcal{W}}
				\Big)+\mu\mathcal{G}^{s}_{\gamma_{1},2}\left(\mathbf{R}^+_s \RR{\mathcal{U}}\right) + \RR{G^{s}} (u), & t > s, \\
				\mathcal{R}\RR{\mathcal{U}}(s,\cdot) = {u}, & t = s.
			\end{cases}
		\end{align}
		Then for the constants $\widetilde{C}_1,\widetilde{C}_2\geq 1$, we have
		\begin{align}\label{Atyasww}
			\begin{split}
				&\triplenorm{\mathbf{R}^+_s \RR{\mathcal{U}}}_{\gamma_1,0,O_{1+s}^s}^{(s)}
				\\
				&\leq
				\widetilde{C}_1
				\big(\|\Gamma\|_{\overline{O_{s+1}^s}}\big)^2
				\big(1+\|\mathcal{Z}\|_{\overline{O_{s+2}^{s-1}}}^2\big)^{
					\frac{\gamma_1+2}{1-\kappa}}
				\big(
				1+
				\triplenorm{\mathbf{R}^+_s\RR{\mathcal{V}}}_
				{\gamma_1,0,O_{1+s}^s}^{(s)}
				\big)^{\frac{\gamma_1+2}{1-\kappa}}	
				\\
				&\times	\left(
				\triplenorm{\mathbf{R}^+_s\RR{\mathcal{W}}}_
				{\gamma_0,-1-\kappa,O_{1+s}^s}^{(s)}
				+
				\Vert u \Vert_{L^{\infty}(\T^d)}
				\right)
				\exp\Big(
				\widetilde{C}_2
				\big(1+\|\mathcal{Z}\|_{\overline{O_{s+2}^{s-1}}}^2\big)^{
					\frac{2}{1-\kappa}}
				\big(
				1+
				\triplenorm{\mathbf{R}^+_s\RR{\mathcal{V}}}_
				{\gamma_1,0,O_{1+s}^s}^{(s)}
				\big)^{\frac{2}{1-\kappa}}
				\Big).
			\end{split}
		\end{align}
	\end{proposition}
	\begin{proof}
		For better tracking of the proof, we divide it into several steps.
		
		\textbf{Step 1.}\label{Step10}  For $r \in [s,1+s]$, set
		\begin{align}\label{P1a}
			\mathbf{R}^+_r\RR{\mathcal{U}^{r}} 
			:= & \mathcal{G}^{r}_{\gamma_{0},1-\kappa}
			\Big(
			(\mathbf{R}^+_r\RR{\mathcal{V}} \star \RR{\mathcal{U}}) \tilde{\star} \rb
			+ \mathbf{R}^+_r\RR{\mathcal{W}}
			\Big)+\mu\mathcal{G}^{r}_{\gamma_{1},2}\left(\mathbf{R}^+_r \RR{\mathcal{U}}\right)
			+ \RR{G^{r}} \big( \mathcal{R}\RR{\mathcal{U}} (r,.)\big).
		\end{align}
		First, we show that
		\begin{align}\label{JMssa}
			\mathbf{R}^+_r \RR{\mathcal{U}} = \mathbf{R}^+_r \RR{\mathcal{U}^{r}} 
			\quad \text{on } (r,1+s] \times \R^d.
		\end{align}
		To this end, since
		\[
		\mathbf{R}^+_r\RR{\mathcal{U}^{r}},\mathbf{R}^+_r\RR{\mathcal{U}}
		\in \mathcal{D}^{\gamma_1,0}_{P_r,0}\big|_{O_{1+s}^r},
		\]
		by \cite[Proposition 3.29]{Hai14} and \eqref{521}, it suffices to show that 
		\begin{align}\label{Jmma}
			\mathcal{R}(\mathbf{R}^+_r \RR{\mathcal{U}}) 
			= \mathcal{R}(\mathbf{R}^+_r \RR{\mathcal{U}^{r}}) 
			\quad \text{on } (r,1+s] \times \R^d.
		\end{align}
		From \eqref{A2222}, \eqref{A222233} and \eqref{IKLa}, we have
		\begin{align}\label{HYAs}
			\mathcal{R} \RR{\mathcal{U}}(r,.)
			=
			\Big(G * \mathcal{R}
			\big(
			(\mathbf{R}^+_{s}\RR{\mathcal{V}} \star \RR{\mathcal{U}}) \tilde{\star} \rb
			+ \mathbf{R}^+_{s}\RR{\mathcal{W}} \big)
			+\mu G * \mathcal{R}\big(\mathbf{R}^+_{s}\RR{\mathcal{U}}\big)
			\Big)(r,.) + G^{s}(u)(r,.).
		\end{align}
		Again, using \eqref{A2222}, \eqref{A222233} and \eqref{IKLa}, together with \eqref{P1a} and \eqref{HYAs}, for $(t,x)\in (r,1+s] \times \R^d$, we have 
		\begin{align*}
			&\mathcal{R}(\mathbf{R}^+_r \RR{\mathcal{U}^{r}})(t,x)
			=
			\Big(
			G * \mathcal{R}\big(
			(\mathbf{R}^+_{r}\RR{\mathcal{V}} \star \RR{\mathcal{U}}) \tilde{\star} \rb
			+ \mathbf{R}^+_{r}\RR{\mathcal{W}}
			\big)
			+\mu G * \mathcal{R}\big(\mathbf{R}^+_{r}\RR{\mathcal{U}}\big)
			\Big)(t,x)
			+ G^{r}\big( \mathcal{R}\RR{\mathcal{U}} (r,.)\big)(t,x)
			\\
			&=
			\Big(
			G * \mathcal{R}\big(
			(\mathbf{R}^+_{r}\RR{\mathcal{V}} \star \RR{\mathcal{U}}) \tilde{\star} \rb
			+ \mathbf{R}^+_{r}\RR{\mathcal{W}}
			\big)
			+\mu G * \mathcal{R}\big(\mathbf{R}^+_{r}\RR{\mathcal{U}}\big)
			\Big)(t,x)
			\\
			&\quad
			+ G^{r}\Bigg(
			\Big(
			G * \mathcal{R}\big(
			(\mathbf{R}^+_{s}\RR{\mathcal{V}} \star \RR{\mathcal{U}}) \tilde{\star} \rb
			+ \mathbf{R}^+_{s}\RR{\mathcal{W}}
			\big)
			+\mu G * \mathcal{R}\big(\mathbf{R}^+_{s}\RR{\mathcal{U}}\big)
			\Big)(r,\cdot)
			+ (G^{s}u)(r,\cdot)
			\Bigg)(t,x).
		\end{align*}
		Consequently, using \eqref{SEMIJ},
		\begin{align*}
			\mathcal{R}(\mathbf{R}^+_r \RR{\mathcal{U}^{r}})(t,x)
			&=
			\Big(
			G * \mathcal{R}\big(
			(\mathbf{R}^+_{s}\RR{\mathcal{V}} \star \RR{\mathcal{U}}) \tilde{\star} \rb
			+ \mathbf{R}^+_{s}\RR{\mathcal{W}}
			\big)
			+\mu G * \mathcal{R}\big(\mathbf{R}^+_{s}\RR{\mathcal{U}}\big)
			\Big)(t,x)
			+(G^{s}u)(t,x)
			\\
			&=
			\mathcal{R}(\mathbf{R}^+_s \RR{\mathcal{U}})(t,x)
			=
			\mathcal{R}(\mathbf{R}^+_r \RR{\mathcal{U}})(t,x).
		\end{align*}
		Therefore, \eqref{JMssa} is proved.
		\\
		\textbf{Step 2.}\label{Step20}
		We now proceed to prove the claimed inequality. In particular, we aim to apply Lemma~\ref{global}. We first verify its conditions.
		From \eqref{multi}, \eqref{A2222}, \eqref{A222233}, \eqref{A-1}, and Lemma~\ref{stat}, there exists a constant $C \geq 1$ such that, for each $\delta \in (0,1]$ and $r\in[s,s+1-\delta]$,
		\begin{align}\label{D0}
			\begin{split}
				\triplenorm{\mathbf{R}^+_r\RR{\mathcal{U}^r}}_{\gamma_1,0,O_{\delta+r}^r}^{(r)}
				&\leq
				C\delta^{\frac{1-\kappa}{2}}
				\left(
				\|\mathcal{Z}\|_{\overline{O_{r+1}^{r-1}}}^2+1
				\right)
				\left(
				1+
				\triplenorm{\mathbf{R}^+_r\RR{\mathcal{V}}}_{\gamma_1,0,O_{\delta+r}^r}^{(r)}
				\right)
				\triplenorm{\mathbf{R}^+_r\RR{\mathcal{U}^r}}_{\gamma_1,0,O_{\delta+r}^r}^{(r)}
				\\
				&\quad+
				C\delta^{\frac{1-\kappa}{2}}
				\left(
				\|\mathcal{Z}\|_{\overline{O_{r+1}^{r-1}}}^2+1
				\right)
				\triplenorm{\mathbf{R}^+_r\RR{\mathcal{W}}}_{\gamma_0,-1-\kappa,O_{\delta+r}^r}^{(r)}+
				C\Vert \mathcal{R}\RR{\mathcal{U}}(r,\cdot)\Vert_{L^{\infty}(\T^{d})}.
			\end{split}
		\end{align}
		For an arbitrary $\beta \in (0,1)$, we define
		\begin{align}\label{kappa}
			\tilde{\delta}_{\beta}
			=
			\Bigg[
			\frac{\beta}{
				C\big(\|\mathcal{Z}\|_{\overline{O_{s+2}^{s-1}}}^2+1\big)
				\left(
				1+\triplenorm{\mathbf{R}^+_s\RR{\mathcal{V}}}_{\gamma_1,0,O_{s+1}^s}^{(s)}
				\right)}
			\Bigg]^{\frac{2}{1-\kappa}}.
		\end{align}
		We note that
		\begin{align*}
			\triplenorm{\mathbf{R}^+_r\RR{\mathcal{V}}}_{\gamma_1,0,O_{\tilde{\delta}_{\beta}+r}^r}^{(r)}
			&\leq
			\triplenorm{\mathbf{R}^+_s\RR{\mathcal{V}}}_{\gamma_1,0,O_{1+s}^s}^{(s)},
			\\
			\triplenorm{\mathbf{R}^+_r\RR{\mathcal{W}}}_{\gamma_0,-1-\kappa,O_{\tilde{\delta}_{\beta}+r}^r}^{(r)}
			&\leq
			\triplenorm{\mathbf{R}^+_s\RR{\mathcal{W}}}_{\gamma_0,-1-\kappa,O_{1+s}^s}^{(s)}
		\end{align*}
		for all
		\[
		r\in[s,s+1-\tilde{\delta}_{\beta}].
		\]
		Thus, from \eqref{kappa},
		\begin{align}\label{OMaslsa}
			\begin{split}
				&\sup_{r\in[s,s+1-\tilde{\delta}_{\beta}]}
				C\tilde{\delta}_{\beta}^{\frac{1-\kappa}{2}}
				\left(
				\|\mathcal{Z}\|_{\overline{O_{r+1}^{r-1}}}^2+1
				\right)
				\left(
				1+
				\triplenorm{\mathbf{R}^+_r\RR{\mathcal{V}}}_{\gamma_1,0,O_{\tilde{\delta}_{\beta}+r}^r}^{(r)}
				\right)
				\leq \beta,
				\\
				&\sup_{r\in[s,s+1-\tilde{\delta}_{\beta}]}
				C\tilde{\delta}_{\beta}^{\frac{1-\kappa}{2}}
				\left(
				\|\mathcal{Z}\|_{\overline{O_{r+1}^{r-1}}}^2+1
				\right)
				\triplenorm{\mathbf{R}^+_r\RR{\mathcal{W}}}_{\gamma_0,-1-\kappa,O_{\tilde{\delta}_{\beta}+r}^r}^{(r)}\leq
				\triplenorm{\mathbf{R}^+_s\RR{\mathcal{W}}}_{\gamma_0,-1-\kappa,O_{1+s}^s}^{(s)}.
			\end{split}
		\end{align}
		We choose a partition $\{r_i\}_{0\leq i\leq m}$ of $[s,s+1]$ such that
		\begin{align*}
			&r_0=s,\qquad r_m=s+1,\\
			&r_1-r_0\leq \tilde{\delta}_{\beta},\qquad
			r_m-r_{m-1}\leq \tilde{\delta}_{\beta},\\
			&r_{i+1}-r_i=\tilde{\delta}_{\beta},
			\qquad 1\leq i\leq m-2.
		\end{align*}
		Clearly, the number of subintervals satisfies
		\begin{align}\label{MAMSs}
			\frac{1}{\tilde{\delta}_{\beta}}
			\leq m
			\leq
			2+\frac{1}{\tilde{\delta}_{\beta}}.
		\end{align}
		Then from \eqref{D0} and \eqref{OMaslsa}, we obtain
		\begin{align}\label{D00}
			\begin{split}
				\triplenorm{\mathbf{R}^+_s\RR{\mathcal{U}}}_{\gamma_1,0,O_{r_1}^s}^{(s)}
				&\leq
				\beta
				\triplenorm{\mathbf{R}^+_s\RR{\mathcal{U}}}_{\gamma_1,0,O_{r_1}^s}^{(s)}+
				\triplenorm{\mathbf{R}^+_s\RR{\mathcal{W}}}_{\gamma_0,-1-\kappa,O_{1+s}^s}^{(s)}
				+
				C\Vert u\Vert_{L^{\infty}(\T^d)}.
			\end{split}
		\end{align}
		Thanks to \eqref{521}, we have $\mathcal{R}\RR{\mathcal{U}}(r_1,\cdot)
		=
		\mathrm{Pr}_{\gY}
		\big(
		\mathbf{R}^{+}_{s}\RR{\mathcal{U}}
		\big)(r_1,\cdot),$
		and therefore, by \eqref{D00},
		\begin{align}\label{IKmals}
			\begin{split}
				\Vert \mathcal{R}\RR{\mathcal{U}} (r_1,\cdot) \Vert_{L^{\infty}(\T^d)}
				&\leq
				\triplenorm{\mathbf{R}^+_s\RR{\mathcal{U}}}_{\gamma_1,0,O_{r_1}^s}^{(s)}
				\\
				&\leq
				\frac{
					\triplenorm{\mathbf{R}^+_s\RR{\mathcal{W}}}_{\gamma_0,-1-\kappa,O_{1+s}^s}^{(s)}
					+
					C\Vert u\Vert_{L^{\infty}(\T^d)}
				}{1-\beta}.
			\end{split}
		\end{align}
		It follows from \hyperref[Step10]{\textbf{Step 1}} that, for each
		$1\leq i\leq m-1$,
		\begin{align}\label{JMa}
			\mathbf{R}^+_{r_i}\RR{\mathcal{U}}
			&=
			\mathcal{G}^{r_i}_{\gamma_{0},1-\kappa}
			\Big(
			(\mathbf{R}^+_{r_i}\RR{\mathcal{V}}
			\star
			\RR{\mathcal{U}})
			\tilde{\star}\rb
			+
			\mathbf{R}^+_{r_i}\RR{\mathcal{W}}
			\Big)
			+\mu\mathcal{G}^{r_i}_{\gamma_{1},2}
			\left(
			\mathbf{R}^+_{r_i}\RR{\mathcal{U}}
			\right)
			+
			\RR{G^{r_i}}
			\big(
			\mathcal{R}\RR{\mathcal{U}}(r_i,\cdot)
			\big).
		\end{align}
		Again from \eqref{D0} and \eqref{OMaslsa}, we obtain
		\begin{align*}
			\triplenorm{\mathbf{R}^+_{r_i}\RR{\mathcal{U}}}_{\gamma_1,0,O_{r_{i+1}}^{r_i}}^{(r_i)}
			&\leq
			\beta
			\triplenorm{\mathbf{R}^+_{r_i}\RR{\mathcal{U}}}_{\gamma_1,0,O_{r_{i+1}}^{r_i}}^{(r_i)}
			+
			\triplenorm{\mathbf{R}^+_s\RR{\mathcal{W}}}_{\gamma_0,-1-\kappa,O_{1+s}^s}^{(s)}
			+
			C\Vert \mathcal{R}\RR{\mathcal{U}}(r_i,\cdot)\Vert_{L^{\infty}(\T^{d})}.
		\end{align*}
		Thus, since
		\[
		\mathcal{R}\RR{\mathcal{U}}(r_{i+1},\cdot)
		=
		\mathrm{Pr}_{\gY}
		\big(
		\mathbf{R}^+_{r_i}\RR{\mathcal{U}}
		\big)(r_{i+1},\cdot),
		\]
		it follows that
		\begin{align}\label{IK45}
			\begin{split}
				\Vert\mathcal{R}\RR{\mathcal{U}}(r_{i+1},\cdot)\Vert_{L^{\infty}(\T^d)}
				&\leq
				\triplenorm{\mathbf{R}^+_{r_i}\RR{\mathcal{U}}}_{\gamma_1,0,O_{r_{i+1}}^{r_i}}^{(r_i)}
				\\
				&\leq
				\frac{
					\triplenorm{\mathbf{R}^+_s\RR{\mathcal{W}}}_{\gamma_0,-1-\kappa,O_{1+s}^s}^{(s)}
					+
					C\Vert \mathcal{R}\RR{\mathcal{U}}(r_i,\cdot)\Vert_{L^{\infty}(\T^d)}
				}{1-\beta}.
			\end{split}
		\end{align}
		To summarize, from \eqref{IKmals} and \eqref{IK45}, for
		\[
		A_i
		:=
		\triplenorm{\mathbf{R}^+_{r_i}\RR{\mathcal{U}}}_{\gamma_1,0,O_{r_{i+1}}^{r_i}}^{(r_i)},
		\qquad
		L_0
		:=
		\frac{
			\triplenorm{\mathbf{R}^+_s\RR{\mathcal{W}}}_{\gamma_0,-1-\kappa,O_{1+s}^s}^{(s)}
		}{1-\beta},
		\qquad
		L_1
		:=
		\frac{C}{1-\beta},
		\]
		we have
		\begin{align*}
			A_i
			&\leq
			L_0+L_1A_{i-1},
			\qquad
			1\leq i\leq m-1,
			\\
			A_0
			&\leq
			L_0+L_1\Vert u\Vert_{L^\infty(\T^d)}.
		\end{align*}
		Thus, from the discrete Gr\"onwall lemma, we obtain
		\begin{align}\label{MAxAi}
			\max_{0\leq i\leq m-1}
			\left\{
			\triplenorm{\mathbf{R}^+_{r_i}\RR{\mathcal{U}}}_{\gamma_1,0,O_{r_{i+1}}^{r_i}}^{(r_i)}
			\right\}
			\leq
			\frac{L_1^m-1}{L_1-1}L_0
			+
			L_1^m
			\Vert u\Vert_{L^\infty(\T^d)}.
		\end{align}
		Therefore, for arbitrary $b,c\in[s,s+1]$ with $0<c-b\leq\tilde{\delta}_\beta$, applying
		\eqref{D0}, \eqref{OMaslsa}, \eqref{MAMSs}, and \eqref{MAxAi}, we obtain
		\begin{align}\label{Bnas}
			\triplenorm{\mathbf{R}^+_b\RR{\mathcal{U}}}_{\gamma_1,0,O_c^b}^{(b)}
			\lesssim
			L_1^{\frac{1}{\tilde{\delta}_\beta}+3}
			\left(
			\triplenorm{\mathbf{R}^+_s\RR{\mathcal{W}}}_{\gamma_0,-1-\kappa,O_{1+s}^s}^{(s)}
			+
			\Vert u\Vert_{L^\infty(\T^d)}
			\right).
		\end{align}
		Thanks to \eqref{Bnas}, we are now in a position to apply Lemma~\ref{global} with $\tilde{\delta}_{\beta}$ defined in \eqref{kappa}. In particular for $\beta=\frac{1}{2}$, we obtain
		\begin{align*}
			\triplenorm{\mathbf{R}^+_s \RR{\mathcal{U}}}_{\gamma_1,0,O_{1+s}^s}^{(s)}
			\lesssim
			\left(\frac{1}{\tilde{\delta}_{\frac{1}{2}}}\right)^{\frac{\gamma_1+2}{2}}
			\big(\|\Gamma\|_{\overline{O_{s+1}^s}}\big)^2
			L_{1}^{\frac{1}{\tilde{\delta}_{\frac{1}{2}}}+3}
			\left(
			\triplenorm{\mathbf{R}^+_s\RR{\mathcal{W}}}_{\gamma_0,-1-\kappa,O_{1+s}^s}^{(s)}
			+\Vert u \Vert_{L^{\infty}(\T^d)}
			\right).
		\end{align*}
		Then, using \eqref{kappa} and after some straightforward algebraic manipulations, we obtain \eqref{Atyasww}.
	\end{proof}
	\begin{remark}
		The identity \eqref{OL<45} is, in fact, equivalent to saying that 
		$\RR{\mathcal{U}}$ satisfies the following equation on $[s,1+s]$:
		\begin{align}\label{MAIN222}
			\begin{cases}
				(\partial_t-\Delta)\RR{\mathcal{U}}
				=
				\big(\RR{\mathcal{V}}\star\RR{\mathcal{U}}\big)
				\tilde{\star}\rb
				+\RR{\mathcal{W}}
				+\mu\RR{\mathcal{U}}, & t>s,\\
				\mathcal{R}\RR{\mathcal{U}}(s,\cdot)=u, & t=s.
			\end{cases}
		\end{align}
	\end{remark}
	
	\begin{remark}
		The previous proposition can potentially be modified using the greedy-point technique; see, e.g., \cite{CLL13,GVR25,BG26}. In the case where $\RR{\mathcal{V}}=\gY$ and $\RR{\mathcal{W}}=0$, for Gaussian noises satisfying the corresponding assumptions required for the greedy-point estimates, this technique can be used to prove that the solutions to the parabolic Anderson model have finite moments of all orders. This is typically achieved by applying Borell's inequality. Since this technique is well established, we do not pursue the details here. See also \cite{XDQT21} for related moment bounds obtained using Feynman--Kac functional representations.
	\end{remark}
	As a consequence of Proposition~\ref{onwall}, we establish the following a priori estimate for the linearized equation \eqref{MAIN2}. This corollary is one of the central ingredients of the paper, since it provides the fundamental estimate that enables us to apply the multiplicative ergodic theorem and, consequently, obtain the Lyapunov exponents.
	\begin{corollary}\label{UJU}
		Consider the setting of Proposition~\ref{kkkl} with $T=1$ and assume that
		$\Sigma\in C_b^3(\mathbb{R},\mathbb{R})$.
		Then there exist $\widetilde{C}_3,\widetilde{C}_4\geq 1$ such that, for any
		$u_1\in B_{L^{\infty}(\mathbb{T}^d)}(u_0,\sigma)$ and
		$\tilde{u}\in L^{\infty}(\mathbb{T}^d)$, we have
		\begin{align*}
			&\triplenorm{\mathbf{R}^+_s D_{u_1}\RR{\mathcal{J}_1}[\tilde{u}]}
			_{\gamma_1,\,0,\,O_{1+s}^s}^{(s)}
			\\
			&\qquad\leq
			\mathcal{E}_1\Big(
			\|\mathcal{Z}\|_{\overline{O_{s+2}^{s-1}}},
			[\ \gI\ ]_{\overline{O_{s+1}^s}},
			\|\Gamma\|_{\overline{O_{s+1}^s}},
			\sup_{w\in O_{s+1}^s}
			\|\mathcal{U}^{\gY}_{u_1}(w)\|,
			\|u_1\|_{L^{\infty}(\mathbb{T}^d)}
			\Big)
			\Vert \tilde{u} \Vert_{L^{\infty}(\T^d)},
		\end{align*}
		where
		\begin{align*}
			\mathcal{E}_1(b,c,d,f,g)
			&:=
			\widetilde{C}_3 P_2(b,c,d,f,g)
			\exp\Big(
			\widetilde{C}_4 P_3(b,c,d,f,g)
			\Big),
		\end{align*}
		with
		\begin{align*}
			P_{1}(b,c,d,f,g)
			:={}&
			d^2
			(1+b^2)^{\frac{\gamma_1+2}{1-\kappa}}
			\big(1+|\mu|+c+f^{\gamma_1-1}\big)^{\frac{\gamma_1+2}{1-\kappa}}
			\\
			&\times
			\big(
			f+g+(1+b^2)(1+c^2+f^{\gamma_1})
			\big),
		\end{align*}
		and
		\begin{align*}
			P_2(b,c,d,f,g)
			&:=
			d^2
			(1+b^2)^{\frac{\gamma_1+2}{1-\kappa}}
			\bigg(
			1+c^2+\big(P_1(b,c,d,f,g)\big)^2
			\bigg)^{\frac{\gamma_1+2}{1-\kappa}},
			\\
			P_3(b,c,d,f,g)
			&:=
			(1+b^2)^{\frac{2}{1-\kappa}}
			\bigg(
			1+c^2+\big(P_1(b,c,d,f,g)\big)^2
			\bigg)^{\frac{2}{1-\kappa}},
			\\
			&\qquad \text{for } b,c,d,f,g \geq 0.
		\end{align*}
	\end{corollary}
	\begin{comment}
		content...
		
		\begin{corollary}\label{UJU}
			Consider the setting of Proposition~\ref{kkkl} with $T=1$ and assume that
			$\Sigma\in C_b^3(\mathbb{R},\mathbb{R})$.
			Then there exist $\widetilde{C}_3,\widetilde{C}_4\geq 1$ such that, for any
			$u_1\in B_{L^{\infty}(\mathbb{T}^d)}(u_0,\sigma)$ and
			$\tilde{u}\in L^{\infty}(\mathbb{T}^d)$, we have
			\begin{align*}
				\triplenorm{D_{u_1}\RR{\mathcal{J}_1}[\tilde{u}] }_{\gamma_1,\,0,\,O_{1+s}}^{(s)}\leq \mathcal{E}\Big(\|\mathcal{Z}\|_{\overline{O_{s+1+T}^{s-1}}},[\ \gI\ ]_{O_{s+T}^s}, \|\Gamma\|_{\overline{O_{s+T}^s}},\sup_{w\in O_{s+T}^s}
				\|\mathcal{U}^{\gY}_{u_1}(w)\|,\|u_1\|_{L^{\infty}(\mathbb{R}^d)}\Big) \Vert \tilde{u} \Vert_{L^{\infty}(\T^d)}
			\end{align*}
			where
			\begin{align*}
				&\mathcal{E}(b,c,d,f,g)=\widetilde{C}_3 P_{2}(b,c,d,f,g)\exp\Big(\widetilde{C}_4 P_{3}(b,c,d,f,g)\Big)
			\end{align*}
			
			where for
			\begin{align*}
				P_{1}(b,c,d,f,g):=d^2(1+b)^{\frac{\gamma_{1}}{1-\kappa}}\big(1+c+f^{\gamma_1-1}\big)^{\frac{\gamma_{1}}{1-\kappa}}\big(f+g+(1+b)(1+c^2+f^{\gamma_{1}})\big)
			\end{align*}
			we have
			\begin{align*}
				\begin{split}
					P_{2}(b,c,d,f,g)=d^2(1+b)^{\frac{\gamma_{1}}{1-\kappa}}\bigg(1+c^2+\big(P_{1}(b,c,d,f,g)\big)^2\bigg)^{\frac{\gamma_{1}}{1-\kappa}}, \\
					P_{3}(b,c,d,f,g)=(1+b)^{\frac{2}{1-\kappa}}\bigg(1+c^2+\big(P_{1}(b,c,d,f,g)\big)^2\bigg)^{\frac{2}{1-\kappa}} .
				\end{split}
			\end{align*}
		\end{corollary}
	\end{comment}
	\begin{proof}
		Thanks to Proposition~\ref{kkkl}, we have
		\begin{align*}
			\mathbf{R}^+_{s}D_{u_1}\RR{\mathcal{J}_1}[\tilde{u}]
			&=
			\mathcal{G}^{s}_{\gamma_{0},1-\kappa}
			\Big(
			\mathbf{R}^+_{s}\Big(
			\Sigma'(\RR{\mathcal{J}_1}(u_1))
			\star
			D_{u_1}\RR{\mathcal{J}_1}[\tilde{u}]
			\Big)
			\tilde{\star}\rb
			\Big)
			+\mu\mathcal{G}^{s}_{\gamma_{1},2}
			\big(
			\mathbf{R}^+_{s}D_{u_1}\RR{\mathcal{J}_1}[\tilde{u}]
			\big)
			+\RR{G^{s}}(\tilde{u}).
		\end{align*}
		Then from Proposition \ref{onwall}, for the constants $\widetilde{C}_1,\widetilde{C}_2\geq 1$, we have
		\begin{align}
			\begin{split}
				&\triplenorm{\mathbf{R}^+_s D_{u_1}\RR{\mathcal{J}_1}[\tilde{u}]}_{\gamma_1,0,O_{1+s}^s}^{(s)}\\&\leq \widetilde{C}_1
				\big(\|\Gamma\|_{\overline{O_{s+1}^s}}\big)^2
				\big(1+\|\mathcal{Z}\|_{\overline{O_{s+2}^{s-1}}}^2\big)^{
					\frac{\gamma_1+2}{1-\kappa}}\big(\triplenorm{\mathbf{R}^+_s \Sigma'(\RR{\mathcal{J}_1}(u_1))}_{\gamma_1,0,O_{1+s}^s}^{(s)}+1\big)^{\frac{\gamma_1+2}{1-\kappa}}	 \Vert \tilde{u} \Vert_{L^{\infty}(\T^d)}\\
				&\qquad\times\exp\Big(\widetilde{C}_2 \big(1+\|\mathcal{Z}\|_{\overline{O_{s+2}^{s-1}}}^2\big)^{\frac{2}{1-\kappa}}\big(\triplenorm{\mathbf{R}^+_s \Sigma'(\RR{\mathcal{J}_1}(u_1))}_{\gamma_1,0,O_{1+s}^s}^{(s)}+1\big)^{\frac{2}{1-\kappa}}	\Big).
			\end{split}
		\end{align} 
		The proof now follows from Remark~\ref{BOUNDED} and Lemma~\ref{Alex}.
	\end{proof}
	Let us now continue with the technical lemma.
	\begin{lemma}\label{DCXSZ}
		Suppose that $F \in C^{2,\nu}_b(\mathbb{R},\mathbb{R})$ for some $\nu \in (0,1]$, $T\in (0,1]$, and that
		$\RR{\mathcal{U}}, \RR{\widetilde{\mathcal{U}}} \in \mathcal{D}^{\gamma_1,0}_{P_s,0}\big|_{O_{T+s}}$
		such that
		\begin{align*}
			\forall w \in O_{T+s}^{s}: \qquad
			\mathcal{U}^{\gI}(w) = \Sigma(\mathcal{U}^{\gY}(w))
			\quad \text{and} \quad
			\widetilde{\mathcal{U}}^{\gI}(w) = \Sigma(\widetilde{\mathcal{U}}^{\gY}(w)).
		\end{align*}
		Then
		\begin{align*}
			&\triplenorm{\mathbf{R}^+_{s}F\big(\RR{\mathcal{U}}\big) - \mathbf{R}^+_{s}F\big(\RR{\widetilde{\mathcal{U}}}\big)}_{\gamma_1,0,O_{T+s}^s}^{(s)}
			\\&\lesssim
			M_{\nu}\!\left(
			\triplenorm{\mathbf{R}^+_{s}\RR{\mathcal{U}} - \mathbf{R}^+_{s}\RR{\widetilde{\mathcal{U}}}}_{\gamma_1,0,O_{T+s}^s}^{(s)}
			\right)
			Q_1\!\left(
			[\ \gI \ ]_{\overline{{O}_{T+s}^s}},
			\triplenorm{\mathbf{R}^+_{s}\RR{\mathcal{U}}}_{\gamma_1,0,O_{T+s}^s}^{(s)},
			\triplenorm{\mathbf{R}^+_{s}\RR{\widetilde{\mathcal{U}}}}_{\gamma_1,0,O_{T+s}^s}^{(s)}
			\right),
		\end{align*}
		where
		\begin{align*}
			M_{\nu}(b) &:= \max\{b, b^{\nu}\},\\
			Q_1(b,c,d) &:= \max\Big\{
			(1+b)c+1, b(1+b)^\nu(1+b+c+d),c+c^2+d,(1+b+c)(1+d)
			\Big\},\\
			&\qquad \text{for } b,c,d \geq 0.
		\end{align*}
	\end{lemma}
	\begin{proof}
		See Appendix \ref{AAzza1}.
	\end{proof}
	\begin{definition}
		For multivariable polynomials, saying that they are increasing in each argument means that, if all arguments except one are fixed, increasing that argument does not decrease the value of the polynomial.
	\end{definition}
	The preceding lemma yields the following result.
	\begin{corollary}\label{SASewwewewewew}
		Consider the setting of Proposition~\ref{kkkl} with $T=1$ and assume that
		$\Sigma\in C_b^{3,\nu}(\mathbb{R},\mathbb{R})$ for some $\nu\in(0,1]$.
		Then, for any
		$u_1,u_2\in B_{L^\infty(\mathbb{T}^d)}(u_0,\sigma)$,
		there exists a polynomial $Q_2$, increasing in each of its arguments, such that
		\begin{align}\label{UASJewesdf}
			\begin{split}
				&\triplenorm{
					\mathbf{R}^+_s \Sigma'(\RR{\mathcal{J}_1}(u_2))
					-\mathbf{R}^+_s \Sigma'(\RR{\mathcal{J}_1}(u_1))
				}_{\gamma_1,0,O_{1+s}^s}^{(s)}
				\\
				&\lesssim
				\max\Big\{
				\Vert u_2-u_1\Vert_{L^{\infty}(\T^d)},
				\Vert u_2-u_1\Vert_{L^{\infty}(\T^d)}^{\nu}
				\Big\}
				\\
				&\quad\times
				\mathcal{E}_2\Big(
				\|\mathcal{Z}\|_{\overline{O_{s+2}^{s-1}}},
				[\ \gI\ ]_{\overline{O_{s+1}^s}},
				\|\Gamma\|_{\overline{O_{s+1}^s}},
				\sup_{\substack{r\in[0,1]\\w\in O_{s+1}^s}}
				\left\|
				\mathcal{U}^{\gY}_{u_1+r(u_2-u_1)}(w)
				\right\|,
				\|u_1\|_{L^{\infty}(\mathbb{T}^d)},
				\|u_2\|_{L^{\infty}(\mathbb{T}^d)}
				\Big).
			\end{split}
		\end{align}
		where
		\begin{align*}
			\mathcal{E}_2(b,c,d,f,g,h)
			&:=
			\exp\Big(
			Q_2(b,c,d,f,g,h)
			\Big),
		\end{align*}
		for $b,c,d,f,g,h\geq 0$.
	\end{corollary}
	\begin{proof}
		Thanks to Lemma \ref{DCXSZ}, we have
		\begin{align*}
			&\triplenorm{\mathbf{R}^+_s \Sigma'( \RR{\mathcal{J}_1}(u_2) )-\mathbf{R}^+_s \Sigma'(\RR{\mathcal{J}_1}(u_1))}_{\gamma_1,0,O_{1+s}^s}^{(s)}
			\\&\lesssim
			M_{\nu}\!\left(
			\triplenorm{\mathbf{R}^+_{s}\RR{\mathcal{J}_1}(u_2) - \mathbf{R}^+_{s}\RR{\mathcal{J}_1}(u_1)}_{\gamma_1,0,O_{1+s}^s}^{(s)}
			\right)
			Q_1\!\left(
			[\ \gI \ ]_{\overline{{O}_{1+s}^s}},
			\triplenorm{\mathbf{R}^+_{s}\RR{\mathcal{J}_1}(u_1)}_{\gamma_1,0,O_{1+s}^s}^{(s)},
			\triplenorm{\mathbf{R}^+_{s}\RR{\mathcal{J}_1}(u_2)}_{\gamma_1,0,O_{1+s}^s}^{(s)}
			\right),
		\end{align*}
		Thanks to Proposition \ref{kkkl}, we have
		\begin{align*}
			\RR{\mathcal{J}_1}(u_2) - \RR{\mathcal{J}_1}(u_1)=\int_{0}^{1}D_{u_1+r(u_2-u_1)}\RR{\mathcal{J}_1}[u_2-u_1]\mathrm{d}r
		\end{align*}
		In particular, we have
		\begin{align*}
			\triplenorm{\mathbf{R}^+_s \RR{\mathcal{J}_1}(u_2)-\mathbf{R}^+_s \RR{\mathcal{J}_1}(u_1)}_{\gamma_1,0,O_{1+s}^s}^{(s)}\leq \int_{0}^{1} \triplenorm{\mathbf{R}^+_s D_{u_1+r(u_2-u_1)}\RR{\mathcal{J}_1}[u_2-u_1]}_{\gamma_1,0,O_{1+s}^s}^{(s)}\mathrm{d}r
		\end{align*}
		The claim now follows from a tedious but straightforward application of Lemma~\ref{Alex}, Lemma~\ref{DCXSZ}, and Corollary~\ref{UJU}.
	\end{proof}
	The next proposition provides the final technical estimates needed to deduce the existence of invariant manifolds. It shows that, after linearization, the difference between the linearized solutions corresponding to two different initial conditions can be effectively controlled. Our goal is not only to establish this continuity property, but also to derive a priori bounds satisfying suitable integrability conditions for applications of our theory to stochastic models.
	\begin{proposition}\label{UJUJUJU}
		Consider the setting of Corollary \ref{SASewwewewewew}.
		Then there exists a polynomial $Q_3$, increasing in each of its arguments, such that
		\begin{align}\label{I789s}
			\begin{split}
				&\triplenorm{D_{u_2}\RR{\mathcal{J}_1}[\tilde{u}]
					-
					D_{u_1}\RR{\mathcal{J}_1}[\tilde{u}]
				}_{\gamma_1,0,O_{1+s}}^{(s)} \lesssim
				\max\big\{\Vert u_2-u_1\Vert_{L^{\infty}(\T^d)}, \Vert u_2-u_1\Vert_{L^{\infty}(\T^d)}^{\nu}\big\} \Vert \tilde{u}\Vert_{L^{\infty}(\T^d)}
				\\&\quad \times\mathcal{E}_3\Big(
				\|\mathcal{Z}\|_{\overline{O_{s+2}^{s-1}}},
				[\ \gI\ ]_{\overline{O_{s+1}^s}},
				\|\Gamma\|_{\overline{O_{s+1}^s}},
				\sup_{\substack{r\in[0,1]\\w\in O_{s+1}^s}}
				\left\|
				\mathcal{U}^{\gY}_{u_1+r(u_2-u_1)}(w)
				\right\|,
				\|u_1\|_{L^{\infty}(\mathbb{T}^d)},
				\|u_2\|_{L^{\infty}(\mathbb{T}^d)}
				\Big)
			\end{split}
		\end{align}
		where
		\begin{align*}
			\mathcal{E}_3(b,c,d,f,g,h)
			&:=
			\exp\Big(
			Q_3(b,c,d,f,g,h)
			\Big),
		\end{align*}
		for $b,c,d,f,g,h\geq 0$.
	\end{proposition}
	\begin{proof}
		Thanks to Proposition~\ref{kkkl}, we have
		\begin{align*}
			& \mathbf{R}^+_{s}D_{u_2}\RR{\mathcal{J}_1}[\tilde{u}]-\mathbf{R}^+_{s}D_{u_1}\RR{\mathcal{J}_1}[\tilde{u}]
			\\&	=
			\mathcal{G}^{s}_{\gamma_{0},1-\kappa}
			\Big(
			\big(\mathbf{R}^+_{s}\Sigma'(\RR{\mathcal{J}_1}(u_2)) \star D_{u_2}\RR{\mathcal{J}_1}[\tilde{u}]\big) \tilde{\star}\rb-\big(\mathbf{R}^+_{s}\Sigma'(\RR{\mathcal{J}_1}(u_1)) \star D_{u_1}\RR{\mathcal{J}_1}[\tilde{u}]\big)\tilde{\star}\rb
			\Big)\\&+\mu\mathcal{G}^{s}_{\gamma_{1},2}\bigg( \mathbf{R}^+_{s}D_{u_2}\RR{\mathcal{J}_1}[\tilde{u}]-\mathbf{R}^+_{s}D_{u_1}\RR{\mathcal{J}_1}[\tilde{u}]\bigg)\\&=\mathcal{G}^{s}_{\gamma_{0},1-\kappa}\Bigg(\Bigg(\Big(\mathbf{R}^+_{s}\Sigma'(\RR{\mathcal{J}_1}(u_1))\Big)\star\Big( D_{u_2}\RR{\mathcal{J}_1}[\tilde{u}]-D_{u_1}\RR{\mathcal{J}_1}[\tilde{u}]\Big)\Bigg)\tilde{\star}\rb\Bigg)\\&+ \mathcal{G}^{s}_{\gamma_{0},1-\kappa}\Bigg(\Bigg(  \Big(\mathbf{R}^+_{s}\Sigma'(\RR{\mathcal{J}_1}(u_2))-\mathbf{R}^+_{s}\Sigma'(\RR{\mathcal{J}_1}(u_1))\Big)\star  D_{u_2}\RR{\mathcal{J}_1}[\tilde{u}]\Bigg)\tilde{\star}\rb \Bigg)+\mu\mathcal{G}^{s}_{\gamma_{1},2}\bigg( D_{u_2}\RR{\mathcal{J}_1}[\tilde{u}]-D_{u_1}\RR{\mathcal{J}_1}[\tilde{u}]\bigg).
		\end{align*}
		Now we can apply Proposition~\ref{onwall} with
		\begin{align*}
			&\RR{\mathcal{V}}=\Sigma'(\RR{\mathcal{J}_1}(u_1)),
			\\&\RR{\mathcal{W}}=\bigg(\Big(\Sigma'(\RR{\mathcal{J}_1}(u_2))-\Sigma'(\RR{\mathcal{J}_1}(u_1))\Big)\star  D_{u_2}\RR{\mathcal{J}_1}[\tilde{u}]\bigg)\tilde{\star} \rb
			\\
			&u=0.
		\end{align*}
		First, note that from Remark~\ref{BOUNDED} and Lemma~\ref{Alex},
		\begin{align}\label{ASas7qwqwqwqc}
			\begin{split}
				&\triplenorm{\mathbf{R}^+_s \RR{\mathcal{V}}}_{\gamma_1,0,O_{1+s}^s}^{(s)}\\&\lesssim 	1+\big([ \ \gI \ ]_{\overline{O_{T+s}^{\,s}}}\big)^2
				+
				\bigg( P_{1}\Big(\|\mathcal{Z}\|_{\overline{O_{s+2}^{s-1}}},[\ \gI\ ]_{\overline{O_{s+1}^s}},\|\Gamma\|_{\overline{O_{s+1}^s}},\sup_{w\in O_{s+1}^s}
				\|\mathcal{U}^{\gY}_{u_1}(w)\|,\|u_1\|_{L^{\infty}(\mathbb{T}^d)}\Big)\bigg)^2.
			\end{split}
		\end{align}
		From \eqref{multi} and Lemma \ref{stat},
		\begin{align}\label{Ausjadfwe}
			\begin{split}
				&\triplenorm{\mathbf{R}^+_s \RR{\mathcal{W}}}_{\gamma_0,-1-\kappa,O_{1+s}^s}^{(s)}\lesssim \triplenorm{\mathbf{R}^+_s \Sigma'(\RR{\mathcal{J}_1}(u_2))-\mathbf{R}^+_s \Sigma'(\RR{\mathcal{J}_1}(u_1))}_{\gamma_1,0,O_{1+s}^s}^{(s)}\triplenorm{ \mathbf{R}^+_s D_{u_2}\RR{\mathcal{J}_1}[\tilde{u}]}_{\gamma_1,0,O_{1+s}^s}^{(s)} .
			\end{split}
		\end{align}
		Thanks to Corollary \ref{UJU} and Corollary \ref{SASewwewewewew}, one can obtain an analogous bound for the right-hand side of \eqref{Ausjadfwe} as claimed in \eqref{I789s}. Keeping this in mind, by using \eqref{ASas7qwqwqwqc} and Proposition~\ref{onwall}, together with some tedious but straightforward calculations, one can prove the claim.
	\end{proof}

	\begin{comment}
		content...
		
		\begin{remark}
			Let $\xi^\epsilon$ be a smooth aoprociatin of $\xi$ in $C^{-1-k}$ such that for caocntst $C_\epsilon$ wjich ius duverfing one have
			\begin{align*}
				sdsdsd
			\end{align*}
			\begin{align*}
				& \Pi_z(\rb) = \bullet, \quad
				\Pi_z(\gXXXX) = X^{i} - X^{i}(z), \quad
				\Pi_z(\gI) = \gII - \gII(z), \\
				& \Pi_z(\gIb) = (\gII - \gII(z))\,\bullet - C, \quad
				\Pi_z(\gY) = 1, \quad
				\Pi_z(\gXXXX\rb) = \Pi_z(\gXXXX)\,\bullet .
			\end{align*}
		\end{remark}
	\end{comment}
	\section{Probabilistic Preparations}\label{PRO}
	In this section, we collect the essential probabilistic tools needed to show that the solution of the SPDE \eqref{MAIN} naturally induces a random dynamical system. The results established in this section, together with the estimates obtained in the previous section, provide the foundational tools for establishing our main results, which will be presented in the next section.
	\subsection{Basic definitions}
	We begin by recalling the fundamental concepts from the theory of random dynamical systems for completeness. For further details, we refer the reader to the monograph~\cite{Arn98}.
	\begin{definition}
		Let $(\overline{\Omega},\overline{\mathcal{F}},\overline{\mathbb{P}})$ be a probability space and let $\mathbb{T}$ be either $\mathbb{Z}$ or $\mathbb{R}$. Assume that there exists a family of measurable maps $\{\overline{\theta}_t\}_{t\in\mathbb{T}}$ from $\overline{\Omega}$ to itself such that
		\begin{enumerate}
			\item $\overline{\theta}_{0}=\operatorname{id}$,
			\item for every $s,t\in\mathbb{T}$,
			\[
			\overline{\theta}_{t+s}
			=
			\overline{\theta}_{t}\circ\overline{\theta}_{s},
			\]
			\item if $\mathbb{T}=\mathbb{R}$, then the map
			\[
			(t,\omega)\mapsto\overline{\theta}_{t}\omega
			\]
			is $\mathcal{B}(\mathbb{R})\otimes\overline{\mathcal{F}}/\overline{\mathcal{F}}$-measurable,
			\item for every $t\in\mathbb{T}$,
			\[
			\overline{\mathbb{P}}\circ\overline{\theta}_t^{-1}
			=
			\overline{\mathbb{P}}.
			\]
		\end{enumerate}
		We then call $(\overline{\Omega},\overline{\mathcal{F}},\{\overline{\theta}_t\}_{t\in\mathbb{T}},\overline{\mathbb{P}})$ an
		\emph{invertible measurable metric dynamical system}. Moreover, it is called \emph{ergodic} if, for every
		$A\in\overline{\mathcal{F}}$ satisfying
		\[
		\overline{\theta}_t^{-1}A=A,
		\qquad \text{for every } t\in\mathbb{T},
		\]
		one has
		\[
		\overline{\mathbb{P}}(A)\in\{0,1\}.
		\]
	\end{definition}
	Another fundamental concept is the cocycle, which we define next.
	\begin{definition}\label{DEF_CO}
		Let $\mathcal{X}$ be a separable Banach space and let 
		$(\overline{\Omega},\overline{\mathcal{F}},\{\overline{\theta}_t\}_{t\in\mathbb{T}},\overline{\mathbb{P}})$ be an invertible 
		measure-preserving dynamical system. Denote by $\mathbb{T}^{+}$ the 
		non-negative part of $\mathbb{T}$. A jointly measurable map
		\begin{align*}
			\varphi \colon \mathbb{T}^{+}\times\overline{\Omega}\times\mathcal{X}
			\rightarrow \mathcal{X}
		\end{align*}
		satisfying
		\begin{align*}
			\varphi(0,\omega,x) &= x,\\
			\varphi(s+t,\omega,x)
			&=
			\varphi\bigl(s,\overline{\theta}_t\omega,\varphi(t,\omega,x)\bigr),
			\qquad s,t\in\mathbb{T}^{+},
		\end{align*}
		is called a \emph{measurable cocycle}. It is called a 
		\emph{$C^{k}$-cocycle} if, for every fixed 
		$(t,\omega)\in\mathbb{T}^{+}\times\overline{\Omega}$, the map
		\[
		\varphi(t,\omega,\cdot)\colon\mathcal{X}\rightarrow\mathcal{X}
		\]
		is of class $C^{k}$. If $\varphi(t,\omega,\cdot)$ is linear for every 
		$(t,\omega)$, then $\varphi$ is called a \emph{linear cocycle}.
		Moreover, if 
		$\varphi(t,\omega,\cdot)$ is compact for every 
		$t\in\mathbb{T}^{+}\setminus\{0\}$ and $\omega\in\overline{\Omega}$, then 
		$\varphi$ is called a \emph{compact cocycle}.
		
	\end{definition}
	In order to apply the theory of random dynamical systems, we need to provide a more precise description of the underlying probabilistic objects. Indeed, so far, all estimates and results have been deterministic and based on the assumption that a regularity structure is given. In applications, regularity structures are typically constructed using probabilistic tools. Therefore, before showing how one can associate a random dynamical system to the SPDE \eqref{MAIN}, we first describe more precisely the probabilistic construction of the regularity structure, which depends on the underlying noise. We begin by specifying the class of noises considered in this manuscript.
	Our choice is motivated by the desire to remain close to the classical
	setting of SPDEs. The framework, however, also allows for more general
	Gaussian noises, including noises with fractional temporal behaviour. We
	do not pursue this generality here, as its treatment would require
	additional technical developments and would lengthen the
	paper. Instead, at the end of this section, we briefly discuss how the
	arguments can be adapted to this more general setting.
	\begin{definition}\label{GGAHSs}
		Let
		\[
		Q:L^2(\mathbb{T}^d)\to L^2(\mathbb{T}^d)
		\]
		be bounded, self-adjoint, and nonnegative. Set
		\[
		\mathcal H_0:=L^2\bigl(\mathbb R;L^2(\mathbb T^d)\bigr)
		\]
		and define the positive semidefinite form
		\[
		\langle f,g\rangle_Q
		:=
		\int_{\mathbb R}
		\left\langle Q^{1/2}f(t,\cdot),
		Q^{1/2}g(t,\cdot)\right\rangle_{L^2(\mathbb T^d)}
		\,\mathrm dt
		=
		\int_{\mathbb R}
		\left\langle Qf(t,\cdot),g(t,\cdot)\right\rangle_{L^2(\mathbb T^d)}
		\,\mathrm dt .
		\]
		Let
		\[
		\mathcal N_Q
		:=
		\{f\in\mathcal H_0:\langle f,f\rangle_Q=0\}.
		\]
		We define \(\overline{\mathcal H}^{Q}\) as the completion of
		\(\mathcal H_0/\mathcal N_Q\) with respect to the inner product induced by
		\(\langle\cdot,\cdot\rangle_Q\). In particular,
		\[
		\|[f]\|_{\overline{\mathcal H}^{Q}}^2
		=
		\int_{\mathbb R}
		\|Q^{1/2}f(t,\cdot)\|_{L^2(\mathbb T^d)}^2\,\mathrm dt,
		\qquad
		[f]\in\mathcal H_0/\mathcal N_Q.
		\]
		Then \(\overline{\mathcal H}^{Q}\) is a separable Hilbert space.
	\end{definition}
	To define the noise considered in this manuscript, it is more convenient to work with an isonormal Gaussian process, which can be defined on any separable Hilbert space.
	
	\begin{definition}\label{UAsjwew}
		Let $(\Omega,\mathcal{F},\mathbb{P})$ be a probability space. We say that a
		centered Gaussian process
		\[
		W^{Q}=\{W^{Q}(f): f\in	\overline{\mathcal{H}}^{Q}\}
		\]
		is an \emph{isonormal Gaussian process over $	\overline{\mathcal{H}}^{Q}$} if
		\begin{align}\label{I4766sd}
			\mathbb{E}\bigl[W^{Q}(f)W^{Q}(g)\bigr]
			=
			\langle f,g\rangle_{\overline{\mathcal{H}}^{Q}},
			\qquad f,g\in\overline{\mathcal{H}}^{Q}.
		\end{align}
	\end{definition}
	\begin{comment}
		content...
		
		\begin{align}
			\int_{\mathbb{R}\times\mathbb{T}^d}
			\bigl(\mathcal{D}_{-}^{\alpha} f\bigr)\,
			\bigl(Q \mathcal{D}_{-}^{\alpha} g\bigr)\,dx\,dt
			=
			\int_{\mathbb{R}}
			\sum_{k\in\mathbb{Z}^d}
			|\tau|^{2\alpha}\, q_k\,
			\hat f(\tau,k)\,\overline{\hat g(\tau,k)}\,d\tau.
		\end{align}
		Since $\mathcal{W}^{H,Q}$ is separable.
	\end{comment}
	For the sake of completeness, we briefly describe the construction of an
	isonormal Gaussian process on $\overline{\mathcal{H}}^{Q}$.
	\begin{lemma}\label{AYsHSsqq}
		Let the space $\mathbb{R}^{\mathbb{N}}$ be endowed with the product topology.
		Since $\mathbb{R}^{\mathbb{N}}$ is a Polish space, we consider the probability
		space
		\[
		(\Omega,\mathcal{F},\mathbb{P})
		=
		\left(\mathbb{R}^{\mathbb{N}},
		\mathcal{B}(\mathbb{R}^{\mathbb{N}}),
		\gamma^{\otimes\mathbb{N}}\right),
		\]
		where $\gamma$ denotes the standard Gaussian measure on $\mathbb{R}$.
		Let $(f_n)_{n\geq1}$ be an orthonormal basis of
		$\overline{\mathcal{H}}^{Q}$. For $f\in\overline{\mathcal{H}}^{Q}$, define
		\[
		W^{Q}(f)
		:=
		L^2(\Omega)\text{-}\lim_{N\to\infty}
		\sum_{n=1}^{N}
		\omega_n
		\langle f,f_n\rangle_{\overline{\mathcal{H}}^{Q}}.
		\]
		Then
		\[
		W^{Q}
		=
		\{W^{Q}(f):f\in\overline{\mathcal{H}}^{Q}\}
		\]
		is an isonormal Gaussian process over $\overline{\mathcal{H}}^{Q}$.
	\end{lemma}
	
	\begin{proof}
		For $f\in\overline{\mathcal{H}}^{Q}$, define the partial sums
		\[
		W_N^{Q}(f)
		:=
		\sum_{n=1}^{N}
		\omega_n
		\langle f,f_n\rangle_{\overline{\mathcal{H}}^{Q}}.
		\]
		Since $(\omega_n)_{n\geq1}$ are independent standard Gaussian random
		variables, for $M<N$ we have
		\[
		\mathbb{E}\left[
		\left|W_N^{Q}(f)-W_M^{Q}(f)\right|^2
		\right]
		=
		\sum_{n=M+1}^{N}
		\left|\langle f,f_n\rangle_{\overline{\mathcal{H}}^{Q}}\right|^2.
		\]
		By Parseval's identity,
		\[
		\sum_{n=1}^{\infty}
		\left|\langle f,f_n\rangle_{\overline{\mathcal{H}}^{Q}}\right|^2
		=
		\|f\|_{\overline{\mathcal{H}}^{Q}}^2,
		\]
		and hence
		\[
		\mathbb{E}\left[
		\left|W_N^{Q}(f)-W_M^{Q}(f)\right|^2
		\right]
		\longrightarrow 0
		\]
		as $M,N\to\infty$. Thus, $(W_N^{Q}(f))_{N\geq1}$ converges in
		$L^2(\Omega)$, which defines $W^{Q}(f)$.
		For any $g_1,\ldots,g_m\in\overline{\mathcal{H}}^{Q}$, the vector
		\[
		\bigl(W_N^{Q}(g_1),\ldots,W_N^{Q}(g_m)\bigr)
		\]
		is centered Gaussian for every $N$. Passing to the limit in $L^2(\Omega)$,
		we conclude that
		\[
		\bigl(W^{Q}(g_1),\ldots,W^{Q}(g_m)\bigr)
		\]
		is also centered Gaussian. Moreover, for $f,g\in\overline{\mathcal{H}}^{Q}$,
		\[
		\begin{aligned}
			\mathbb{E}\left[W^{Q}(f)W^{Q}(g)\right]
			&=
			\sum_{n=1}^{\infty}
			\langle f,f_n\rangle_{\overline{\mathcal{H}}^{Q}}
			\langle g,f_n\rangle_{\overline{\mathcal{H}}^{Q}}=
			\langle f,g\rangle_{\overline{\mathcal{H}}^{Q}},
		\end{aligned}
		\]
		where the last equality follows from Parseval's identity. Therefore,
		$W^{Q}$ is an isonormal Gaussian process over $\overline{\mathcal{H}}^{Q}$.
	\end{proof}
	The role of the operator $Q$ is to regularize the space-time white noise sufficiently so that it fits into our framework. Throughout this section, we assume that $Q$ satisfies the following assumption.
	\begin{assumption}\label{asasww}
		Let
		\[
		\mathbb{T}^d=(\mathbb{R}/2\pi\mathbb{Z})^d,
		\]
		and let
		\[
		Q:L^2(\mathbb{T}^d)\to L^2(\mathbb{T}^d)
		\]
		be a bounded, self-adjoint, and nonnegative operator. We assume that,
		on the complexification of \(L^2(\mathbb{T}^d)\), \(Q\) is a Fourier
		multiplier, namely
		\[
		Qe_k=q_ke_k,
		\qquad
		e_k(x)=e^{ik\cdot x},
		\qquad
		k\in\mathbb{Z}^d.
		\]
		The Fourier multipliers satisfy
		\[
		q_k=q_{-k}\geq 0,
		\qquad
		k\in\mathbb{Z}^d,
		\]
		and there exist constants \(C>0\) and
		\[
		\beta\in\left(0,\frac d2\right)
		\]
		such that
		\[
		q_k\leq C(1+|k|^2)^{-\beta},
		\qquad
		k\in\mathbb{Z}^d.
		\]
		Here
		\[
		|k|=\sqrt{k_1^2+\cdots+k_d^2}.
		\]
	\end{assumption}

	\begin{definition}\label{Olasmudfa}
		Let $\xi^{Q}\in\mathcal{S}'(\mathbb{R}\times\mathbb{T}^d)$ be the random distribution defined by
		\[
		\langle \xi^{Q},\psi\rangle
		=
		W^{Q}(\psi),
		\qquad
		\psi\in C_c^\infty(\mathbb{R}\times\mathbb{T}^d).
		\]
	\end{definition}
	
	\begin{remark}\label{BNASSSSS}
		Using standard techniques from Fourier analysis, one can show that, almost
		surely,
		\[
		\xi^{Q}\in
		C^{\,\beta-\frac d2-1-\varepsilon}_{(2,1)},
		\qquad
		\forall\,\varepsilon>0.
		\]
	\end{remark}
	With these preparations, we now fix our assumptions.
	\begin{assumption}\label{ASAS78asaaa}
		We assume that:
		\begin{enumerate}
			\item Assumption~\ref{asasww} holds.
			
			\item We suppose that $W^{Q}$ is an isonormal Gaussian process over
			$\overline{\mathcal{H}}^{Q}$. Moreover, $\xi^Q$ denotes the random
			distribution defined in Definition~\ref{Olasmudfa}.
			
			\item For $\beta\in(0,\frac{d}{2})$ and $\kappa \in (0,\frac{1}{3})$, we assume that
			\begin{align}\label{asasaiasa}
				\beta-\frac{d}{2}>-\kappa .
			\end{align}
			
			\item We consider the equation
			\begin{align}\label{MAINAA}
				\begin{cases}
					(\partial_t-\Delta)u=\mu u+\Sigma(u)\,\xi^Q,
					& t>0,\\
					u(0,\cdot)=u_0,\qquad u_0\in L^\infty(\mathbb T^d).
				\end{cases}
			\end{align}
			where $\mu\leq0$ and $\Sigma\in C^2(\R,\R)$. We assume that \eqref{MAINAA} admits a unique
			global solution for every $u_0\in L^\infty(\mathbb T^d)$.
			
			\item We assume that
			\[
			(\Omega,\mathcal F,\mathbb P)
			=
			\left(
			\mathbb R^{\mathbb N},
			\mathcal B(\mathbb R^{\mathbb N}),
			\gamma^{\otimes\mathbb N}
			\right),
			\]
			as introduced in Lemma~\ref{AYsHSsqq}.
			
			\item The sequence $(f_n)_{n\geq1}$ is an orthonormal basis of
			$\overline{\mathcal{H}}^{Q}$.
		\end{enumerate}
	\end{assumption}
	\begin{remark}
		The global well-posedness assumed above is known under suitable
		additional assumptions on $\Sigma$. For sufficiently small $\kappa$,
		it holds for $\Sigma\in C_b^2(\mathbb R,\mathbb R)$ under the conditions
		considered in \cite{CFW26,SZZ26}. Moreover, an inspection of the proof
		in \cite{ES26} shows that, in the regime $\kappa<1/3$, the corresponding
		a priori bounds hold for $\Sigma\in C_b^3(\mathbb R,\mathbb R)$.
	\end{remark}
	To formulate the solution theory for the singular SPDE \eqref{MAINAA} and to
	construct a metric random dynamical system, we make use of the Wiener chaos
	decomposition associated with the underlying Gaussian probability space. We
	briefly recall the following definition.
	\begin{definition}\label{BASsdswsw}
		Let $H_n$ denote the $n$-th Hermite polynomial and define
		$H_{e_i}^{n}(\omega):=H_n(\omega_i)$. Let $\Lambda$ be the set of all
		multi-indices
		\[
		\alpha=(\alpha_1,\alpha_2,\dots),\qquad \alpha_i\in\mathbb N_0,
		\]
		such that only finitely many components are non-zero. For
		$\alpha\in\Lambda$, we define
		\[
		\alpha!:=\prod_{i=1}^{\infty}\alpha_i!,
		\qquad
		|\alpha|:=\sum_{i=1}^{\infty}\alpha_i .
		\]
		The generalized Hermite polynomial associated with $\alpha$ is defined by
		\[
		H_\alpha(\omega)
		:=
		\prod_{i=1}^{\infty}H_{\alpha_i}(\omega_i),
		\qquad \omega\in\mathbb R^{\mathbb N},
		\]
		where the product is finite since $H_0=1$ and only finitely many
		$\alpha_i$ are non-zero.
		Let $\mathcal C_0$ denote the space of constant functions. For $n\ge1$, set
		\[
		\Lambda_n:=\{\alpha\in\Lambda:|\alpha|=n\},
		\]
		and define the $n$-th Wiener chaos by
		\[
		\mathcal H_n
		:=
		\overline{\mathrm{span}}
		\{H_\alpha(\omega):\alpha\in\Lambda_n\}
		\subset L^2(\Omega,\mathcal F,\mathbb P).
		\]
		Then $\mathcal H_n$ is a closed Hilbert subspace of $L^2(\Omega)$, and
		$\{H_\alpha:\alpha\in\Lambda_n\}$ forms an orthogonal basis of
		$\mathcal H_n$. Moreover,
		\[
		L^2(\Omega,\mathcal F,\mathbb P)
		=
		\mathcal C_0
		\oplus
		\bigoplus_{n=1}^{\infty}\mathcal H_n,
		\]
		where the sum is orthogonal in $L^2(\Omega)$. We refer to
		\cite[Chapter~1]{Nua05} for further details.
	\end{definition}
	Let us now explain how the probabilistic objects introduced above allow us to
	give a rigorous meaning to the SPDE \eqref{MAINAA} and to associate a random
	dynamical system with it.
	Before going further, we briefly explain the main strategy as follows.
	\begin{enumerate}
		\item We exploit well-established results from regularity structures theory:
		by regularizing the noise and introducing the corresponding renormalization
		counter terms, one can give a rigorous meaning to the SPDE \eqref{MAINAA}.
		In particular, this allows us to construct a convergent sequence of
		approximations to the SPDE \eqref{MAINAA}.	
		\item We construct a suitable Polish subspace $\mathcal{M}_0$ of
		$(\mathcal{M},d_{\mathcal M})$ and push forward the probability measure onto
		this space.
		\item Using the topological properties of $\mathcal{M}_0$, we are able to
		define an ergodic metric dynamical system and associate it with the solution
		of the SPDE \eqref{MAINAA}. The flow property of the solution and the
		invariance of the law of the noise under time translations then allow us to
		establish the cocycle property. 
	\end{enumerate}
	\begin{definition}
		Let $\rho:\mathbb{R}\times\mathbb{R}^d \to [0,2]$ be a smooth, compactly supported function, integrating to one and even in both variables. For each $\epsilon>0$ and $z=(t,x)\in \R\times\T^d$, we define
		\begin{align}
			\xi^{\epsilon}(z)=W^Q\left(\rho^{\epsilon}_{z}\right)
			:=W^Q\left((\rho^{\epsilon}_{z})^{\mathrm{per}}\right),
		\end{align}
		where $\rho^{\epsilon}_{z}$ is defined in \eqref{AAASSsssww}, and $(\rho^{\epsilon}_{z})^{\mathrm{per}}$ denotes the periodicisation of $\rho^{\epsilon}_{z}$ defined by
		\begin{align*}
			(\rho_z^\varepsilon)^{\mathrm{per}}
			:\mathbb{R}\times\mathbb{T}^d &\longrightarrow \mathbb{R},\\
			(\rho_z^\varepsilon)^{\mathrm{per}}(t,x)
			&:=
			\sum_{k\in\mathbb{Z}^d}
			\rho_z^\varepsilon\bigl(t,x+2\pi k\bigr).
		\end{align*}
		We further set
		\begin{align*}
			\mathcal{I}[\xi^{\epsilon}](z):=(K* \xi^{\epsilon})(z).
		\end{align*}
	\end{definition}
	Clearly, $\xi^{\epsilon}$ and $\mathcal{I}[\xi^{\epsilon}]$ belong to $\mathcal{H}_1$, cf.\ Definition \ref{BASsdswsw}. From the basic properties of Gaussian analysis, the Wick product belongs to $\mathcal{H}_2$ and is given by
	\begin{align*}
		\xi^{\epsilon}\diamond \mathcal{I}[\xi^{\epsilon}](z)
		:=
		\xi^{\epsilon}(z)\, \mathcal{I}[\xi^{\epsilon}](z)
		-
		\mathbb{E}\big[\xi^{\epsilon}(z)\, \mathcal{I}[\xi^{\epsilon}](z)\big].
	\end{align*}
	\begin{remark}
		By a standard calculation, one can see that
		\begin{align}\label{AISmfdsd}
			\begin{split}
				&\mathbb{E}\big[\xi^{\epsilon}(z)\,\mathcal{I}[\xi^{\epsilon}](z)\big] = \mathbb{E}\left[ \int_{\mathbb{R}\times\mathbb{R}^d} K(z-z_1)\, W^Q(\rho_{z_1}^\epsilon)\, W^Q(\rho_z^\epsilon)\, \mathrm{d}z_1 \right] \\ &= \int_{\mathbb{R}\times\mathbb{R}^d} K(z-z_1)\, \mathbb{E}\big[ W^Q(\rho_{z_1}^\epsilon)\, W^Q(\rho_z^\epsilon) \big]\, \mathrm{d}z_1 = \int_{\mathbb{R}\times\mathbb{R}^d} K(z-z_1)\, \langle (\rho^{\epsilon}_{z})^{\mathrm{per}},(\rho^{\epsilon}_{z_1})^{\mathrm{per}}\rangle_{\overline{\mathcal{H}}^{Q}}\, \mathrm{d}z_1 \\ &= \int_{\mathbb{R}\times\mathbb{R}^d} K(z-z_1)\, \langle (\rho^{\epsilon}_{z-z_1})^{\mathrm{per}},(\rho^{\epsilon})^{\mathrm{per}}\rangle_{\overline{\mathcal{H}}^{Q}}\,\mathrm{d}z_1 =:C_\epsilon .
			\end{split}
		\end{align}
		The constant $C_\epsilon$ is independent of $z$. Depending on the regularity of $Q$, the family $(C_\epsilon)_{\epsilon>0}$ may remain bounded or diverge as $\epsilon\to0$.
	\end{remark}
	\begin{definition}
		For $t\in\mathbb{R}$, the translation operator
		\[
		U_t:\mathcal{D}'(\R \times \mathbb{T}^{d})
		\longrightarrow
		\mathcal{D}'(\R \times \mathbb{T}^{d})
		\]
		is defined by
		\[
		\langle U_t\Upsilon,\varphi\rangle
		:=
		\langle\Upsilon,\varphi(\,\cdot-(t,0)\,)\rangle,
		\qquad
		\varphi\in C_c^\infty(\mathbb{R}^{d+1}),
		\]
		for every $\Upsilon\in\mathcal{D}'(\mathbb{R}^{d+1})$. In particular, if
		$\Upsilon=f$ is a sufficiently regular function, then
		\[
		(U_tf)(s,x)=f(s+t,x),
		\qquad (s,x)\in\mathbb{R}\times\mathbb{R}^d.
		\]
	\end{definition}
	A fundamental property of the isonormal Gaussian process $W^Q$ on
	$\overline{\mathcal H}^{Q}$ is its invariance under time translations.
	More precisely, the law of the Gaussian family associated with $W^Q$ is
	invariant under the action of the translation operator $U_t$. The following
	lemma makes this property precise.
	\begin{lemma}\label{UYHNAss}
		Let $f,g\in \overline{\mathcal H}^{Q}$, and let
		$F:\mathbb{R}^2\to\mathbb{R}$ be a bounded Borel measurable function.
		Then, for every $t\in\mathbb{R}$,
		\[
		\mathbb{E}\Big[F\big(W^Q(U_t f),W^Q(U_t g)\big)\Big]
		=
		\mathbb{E}\Big[F\big(W^Q(f),W^Q(g)\big)\Big].
		\]
	\end{lemma}
	\begin{proof}
		Observe that both
		\[
		\bigl(W^Q(f),W^Q(g)\bigr)
		\qquad\text{and}\qquad
		\bigl(W^Q(U_t f),W^Q(U_t g)\bigr)
		\]
		are centered Gaussian vectors. Hence, it is enough to show that their
		covariance matrices coincide.
		Since $W^Q$ is an isonormal Gaussian process, we have
		\[
		\mathbb{E}\left[\left(W^Q(f)\right)^2\right]
		=
		\|f\|_{\overline{\mathcal H}^{Q}}^2,
		\qquad
		\mathbb{E}\left[\left(W^Q(g)\right)^2\right]
		=
		\|g\|_{\overline{\mathcal H}^{Q}}^2,
		\]
		and
		\[
		\mathbb{E}\bigl[W^Q(f)W^Q(g)\bigr]
		=
		\langle f,g\rangle_{\overline{\mathcal H}^{Q}}.
		\]
		Moreover, the translation operator $U_t$ preserves the inner product on
		$\overline{\mathcal H}^{Q}$, that is,
		\[
		\langle U_tf,U_tg\rangle_{\overline{\mathcal H}^{Q}}
		=
		\langle f,g\rangle_{\overline{\mathcal H}^{Q}}.
		\]
		Therefore,
		\[
		\mathbb{E}\left[\left(W^Q(U_tf)\right)^2\right]
		=
		\|U_tf\|_{\overline{\mathcal H}^{Q}}^2
		=
		\|f\|_{\overline{\mathcal H}^{Q}}^2,
		\]
		and similarly,
		\[
		\mathbb{E}\left[\left(W^Q(U_tg)\right)^2\right]
		=
		\|g\|_{\overline{\mathcal H}^{Q}}^2.
		\]
		Furthermore,
		\[
		\mathbb{E}\bigl[W^Q(U_tf)W^Q(U_tg)\bigr]
		=
		\langle U_tf,U_tg\rangle_{\overline{\mathcal H}^{Q}}
		=
		\langle f,g\rangle_{\overline{\mathcal H}^{Q}}.
		\]
		Hence, the Gaussian vectors
		\[
		\bigl(W^Q(f),W^Q(g)\bigr)
		\quad\text{and}\quad
		\bigl(W^Q(U_tf),W^Q(U_tg)\bigr)
		\]
		have the same mean and covariance matrix, and therefore the same
		distribution. The desired identity follows.
	\end{proof}
	\subsection{The metric dynamical system associated with the noise}
	While $\mathcal{M}$ is a complete metric space, it is not necessarily separable. Moreover, it is rather large. On the other hand, Polish spaces play a central role in the theory of random dynamical systems. This motivates us to restrict the model space to a separable and complete subspace.
	\begin{lemma}\label{polikks}
		We equip $C(\mathbb{R}\times\mathbb{T}^d)$ with the topology generated by the
		seminorms
		\[
		p_n(f)
		:=
		\sup_{(t,x)\in[-n,n]\times\mathbb{T}^d}|f(t,x)|,
		\qquad n\in\mathbb{N}.
		\]
		We define the map
		\[
		\mathbf{M}:
		C(\mathbb{R}\times\mathbb{T}^d)
		\times
		\mathbb{R}
		\longrightarrow
		\mathcal M
		\]
		by
		\[
		\mathbf{M}(f,c)
		=
		(\Pi^{(f,c)},\Gamma^{(f,c)})
		\]
		with $(\Pi^{(f,c)}, \Gamma^{(f,c)})$ defined as follows.
		The map
		\[
		\Pi^{(f,c)} : \mathbb{R}^{d+1} \to \mathcal{L}(\mathcal{T},\mathcal{D}'(\mathbb{R}^{d+1}))
		\]
		is defined by
		\begin{align}\label{Ik78965}
			\begin{split}
				\Pi^{(f,c)}_z(\rb) &= f, \quad
				\Pi^{(f,c)}_z(\gXXXX) = X^{i} - X^{i}(z), \quad
				\Pi^{(f,c)}_z(\gI) = K*f - (K*f)(z), \\
				\Pi^{(f,c)}_z(\gIb) &= \big( K*f - (K*f)(z)\big) f - c, \quad
				\Pi^{(f,c)}_z(\gY) = 1, \quad
				\Pi^{(f,c)}_z(\gXXXX\rb) = \Pi^{(f,c)}_z(\rb) \Pi^{(f,c)}_z(\gXXXX).
			\end{split}
		\end{align}
		Also, $\Gamma^{(f,c)}:\mathbb{R}^{d+1}\times\mathbb{R}^{d+1}\rightarrow \mathcal{G} $
		is defined by
		\[
		\Gamma^{(f,c)}(z_1,z_2)
		=
		\Gamma^{(f,c)}_{z_1,z_2},
		\]
		where $\Gamma^{(f,c)}_{z_1,z_2}$ is the linear map given by
		\begin{align}\label{Ik789655}
			\begin{split}
				\Gamma^{(f,c)}_{z_1,z_2}(\rb) &= \rb, \quad
				\Gamma^{(f,c)}_{z_1,z_2}(\gI) = \gI - (\Pi^{(f,c)}_{z_1} \gI)(z_2)\, \gY, \quad
				\Gamma^{(f,c)}_{z_1,z_2}(\gY) = \gY, \\
				\Gamma^{(f,c)}_{z_1,z_2}(\gXXXX) &= \gXXXX - (\Pi^{(f,c)}_{z_1}\gXXXX)(z_2)\, \gY, \quad
				\Gamma^{(f,c)}_{z_1,z_2}(\gIb) = \gIb - (\Pi^{(f,c)}_{z_1} \gI)(z_2)\, \rb, \\
				\Gamma^{(f,c)}_{z_1,z_2}(\gXXXX\rb) &= \gXXXX\rb - (\Pi^{(f,c)}_{z_1} \gXXXX)(z_2)\, \rb.
			\end{split}
		\end{align}
		We set
		\[
		\mathcal{M}_0 :=
		\overline{\mathbf{M}\left(C(\mathbb{R}\times\mathbb{T}^d)\times \mathbb{R}\right)}^{\,d_{\mathcal{M}}}.
		\]
		Then $\mathcal{M}_0$ is a Polish space.
	\end{lemma}
	
	\begin{proof}
		The space
		\[
		C(\mathbb{R}\times\mathbb{T}^d)\times
		\mathbb{R}
		\]
		equipped with the locally uniform topologies generated by the above
		seminorms is a Polish
		space.
		By construction, the map $\mathbf{M}$ is continuous.
		Therefore,
		\[
		\mathbf{M}\left(
		C(\mathbb{R}\times\mathbb{T}^d)
		\times
		\mathbb{R}
		\right)
		\]
		is a separable subset of $\mathcal{M}$.
		Since $\mathcal{M}_0$ is the closure of a separable set in the complete metric
		space $(\mathcal M,d_{\mathcal M})$, it is a complete and separable metric
		space. Hence, $\mathcal{M}_0$ is Polish.
	\end{proof}
	\begin{remark}
		By construction, every model of the form $\mathbf{M}(f,c)$ realizes
		the kernel $K$ for the abstract integration map $\mathcal{I}$ in the
		sense of \eqref{trtetre}. Since this property is preserved under
		convergence in the model topology, every model in $\mathcal{M}_0$
		also realizes $K$ for $\mathcal{I}$.
	\end{remark}
	\begin{lemma}\label{JMasss}
		Let $(z_j)_{j\ge1}\subset \mathbb{R}\times\mathbb{T}^d$ be a fixed dense subset of $\mathbb{R}\times\mathbb{T}^d$. For every $\epsilon>0$, there exists a full-measure Borel set $\Omega^{\epsilon}\subset\Omega$ such that for every $\omega\in\Omega^{\epsilon}$, the following assertions hold:
		\begin{enumerate}
			\item For every $j\ge1$, the random variable
			\begin{align}\label{I1479AS}
				\begin{split}
					\mathcal{P}^{\epsilon}(\omega,z_j)
					&:=W^Q\bigl((\rho^{\epsilon}_{z_j})^{\mathrm{per}}\bigr)(\omega)\\
					&=\sum_{i\ge1}\omega_i
					\big\langle (\rho^{\epsilon}_{z_j})^{\mathrm{per}},
					f_i\big\rangle_{\overline{\mathcal{H}}^{Q}}
				\end{split}
			\end{align}
			is well defined (equivalently, the series converges).
			
			\item There exists a continuous function
			\[
			\mathcal{P}^{\epsilon}(\omega,\cdot)
			\in C(\mathbb{R}\times\mathbb{T}^d)
			\]
			such that
			\[
			\mathcal{P}^{\epsilon}(\omega,z_j)
			=
			W^Q\bigl((\rho^{\epsilon}_{z_j})^{\mathrm{per}}\bigr)(\omega)
			\]
			for every $j\ge1$, where the right-hand side is given by
			\eqref{I1479AS}.
		\end{enumerate}
	\end{lemma}
	\begin{proof}
		For each $z\in \mathbb{R}\times\mathbb{T}^{d}$, we have
		$W^Q((\rho^{\epsilon}_{z})^{\mathrm{per}})\in\mathcal{H}_{1}$.
		Hence, by \cite[Lemma 10.5]{Hai14} and \eqref{I4766sd}, for every $z,w\in\mathbb{R}\times\mathbb{T}^{d}$,
		\[
		\mathbb{E}\Big[
		\big|
		W^Q((\rho^{\epsilon}_{z})^{\mathrm{per}})
		\big|^{2p}
		\Big]
		\lesssim
		\left(
		\sum_{i\ge1}
		\big\langle
		(\rho^{\epsilon}_{z})^{\mathrm{per}},
		f_i
		\big\rangle_{\overline{\mathcal{H}}^{Q}}^{2}
		\right)^p
		=
		\|
		(\rho^{\epsilon}_{z})^{\mathrm{per}}
		\|_{\overline{\mathcal{H}}^{Q}}^{2p}
		<\infty.
		\]
		
		Similarly,
		\begin{align*}
			&\mathbb{E}\Big[
			\big|
			W^Q((\rho^{\epsilon}_{w})^{\mathrm{per}})
			-
			W^Q((\rho^{\epsilon}_{z})^{\mathrm{per}})
			\big|^{2p}
			\Big]
			\lesssim
			\left(
			\mathbb{E}\Big[
			\big|
			W^Q((\rho^{\epsilon}_{w})^{\mathrm{per}})
			-
			W^Q((\rho^{\epsilon}_{z})^{\mathrm{per}})
			\big|^{2}
			\Big]
			\right)^{p} \\&=
			\left(
			\sum_{i\ge1}
			\left(
			\big\langle
			(\rho^{\epsilon}_{w})^{\mathrm{per}}
			-
			(\rho^{\epsilon}_{z})^{\mathrm{per}},
			f_i
			\big\rangle_{\overline{\mathcal{H}}^{Q}}
			\right)^2
			\right)^p =
			\|
			(\rho^{\epsilon}_{w})^{\mathrm{per}}
			-
			(\rho^{\epsilon}_{z})^{\mathrm{per}}
			\|_{\overline{\mathcal{H}}^{Q}}^{2p} \\
			&\le
			R_{\epsilon}\,
			\|w-z\|^{2p}.
		\end{align*}
		For every $z_j$ in the prescribed countable dense subset, the series in
		\eqref{I1479AS} converges almost surely. Choosing $p$ sufficiently large, the
		above increment estimate and the Kolmogorov--Chentsov theorem, applied on
		$[-n,n]\times\mathbb T^d$ for every $n\in\mathbb N$, imply that the process
		admits a continuous modification. Taking the intersection of the corresponding
		countably many full-measure sets, we obtain a full-measure Borel set
		$\Omega^\epsilon$ on which both assertions hold.
	\end{proof}
	\begin{remark}\label{Asaisada}
		For technical reasons in Lemma~\ref{JMasss}, we fixed a dense sequence $(z_j)_{j\geq 1}$.
		However, the construction works for any countable dense subset.
		Note that any two continuous versions are indistinguishable in the sense that if we consider two versions
		$\mathcal{P}^{\epsilon}(\omega,\cdot)$ and $\widetilde{\mathcal{P}}^{\epsilon}(\omega,\cdot)$ defined on sets of full measure and continuous, then
		\[
		\mathbb{P}\Big(
		\omega :
		\mathcal{P}^{\epsilon}(\omega,\cdot)
		=
		\widetilde{\mathcal{P}}^{\epsilon}(\omega,\cdot)
		\Big)
		=1.
		\]
	\end{remark}
	\begin{lemma}\label{ASHyasas}
		Let $\epsilon>0$, and consider the setting of Lemma~\ref{JMasss}.
		Assume also that $C_\epsilon$ is the constant defined in
		\eqref{AISmfdsd}. Then the map
		\[
		\mathcal{Z}^{\epsilon}:
		(\Omega,\mathcal{B}(\Omega))
		\longrightarrow
		(\mathcal{M}_0,\mathcal{B}(\mathcal{M}_0)),
		\]
		defined by
		\begin{align}\label{Baxczxx}
			\mathcal{Z}^{\epsilon}(\omega)
			:=
			\begin{cases}
				\mathbf{M}\bigl(
				\mathcal{P}^{\epsilon}(\omega),
				C_\epsilon
				\bigr),
				& \omega\in\Omega^\epsilon,
				\\[1mm]
				\mathbf{M}(0,0),
				& \omega\notin\Omega^\epsilon,
			\end{cases}
		\end{align}
		is Borel measurable, where $\mathcal{P}^{\epsilon}(\omega)(z):=\mathcal{P}^{\epsilon}(\omega, z)$.
	\end{lemma}
	\begin{proof}
		First recall that, by Lemma \ref{polikks}, the space $\mathcal{M}_0$ endowed with the subspace topology inherited from $\mathcal{M}$ is a Polish space.
		Moreover, the topology of $\mathcal{M}$ is generated by the metric defined in \eqref{NHMAsi78}. Hence, the induced topology on $\mathcal{M}_0$ is also metrizable and generated by the corresponding restricted seminorms. Assume that $(f_j,c_j)_{j\geq 1}$ is a dense sequence in
		\[
		C(\mathbb{R}\times\mathbb{T}^d)\times \mathbb{R}.
		\]
		In particular, $\bigl(\mathbf{M}(f_j,c_j)\bigr)_{j\geq 1}$
		is dense in $\mathcal{M}_0$. Since $\Omega^\epsilon$ is Borel, it suffices to verify the measurability on $\Omega^\epsilon$.
		With these preliminaries, it is clear that $\mathcal{Z}^{\epsilon}$ is measurable if for all $n,j \geq 1$ and $r>0$ we have
		\begin{align*}
			\Bigl\{ \omega \in \Omega^\epsilon :
			\bigl\|
			\mathbf{M}\bigl(\mathcal{P}^{\epsilon}(\omega),
			C_\epsilon\bigr)
			;
			\mathbf{M}(f_j,c_j)
			\bigr\|_{B(0,n)} \leq r
			\Bigr\}
			\in \mathcal{B}(\Omega).
		\end{align*}
		The latter claim follows from \eqref{NBAVsss}, \eqref{NBA1}, and \eqref{NBA2}. Indeed, by Lemma~\ref{JMasss}, after setting
		$\mathcal{P}^{\epsilon}(\omega,\cdot)=0$ on $(\Omega^\epsilon)^c$, the maps
		
		$$
		\omega\longmapsto \mathcal{P}^{\epsilon}(\omega,z_j),
		\qquad j\geq1,
		$$
		
		are Borel measurable. Since $(z_j)_{j\geq1}$ is dense and
		$\mathcal{P}^{\epsilon}(\omega,\cdot)$ is continuous, this implies that
		
		$$
		\omega\longmapsto \mathcal{P}^{\epsilon}(\omega,\cdot)
		$$
		is Borel measurable as a map into
		$C(\mathbb{R}\times\mathbb{T}^d)$. Hence $\mathcal{Z}^{\epsilon}$ is Borel measurable.
	\end{proof}
	The latter lemma is a crucial point that allows us to explicitly construct our metric dynamical system. This is the topic of the next result.
	\begin{lemma}\label{Aaasqwqwa}
		Consider the setting of Lemma \ref{ASHyasas} with the probability space
		\[
		\left(\mathcal{M}_0, \mathcal{B}(\mathcal{M}_0), (\mathcal{Z}^{\epsilon})_{\#}\mathbb{P}\right),
		\]
		where $(\mathcal{Z}^{\epsilon})_{\#}\mathbb{P}$ denotes the pushforward measure of $\mathbb{P}$ under $\mathcal{Z}^{\epsilon}$. Then the following statements hold:
		
		\begin{enumerate}
			\item One can define a family of continuous maps
			\[
			\Theta=(\Theta_t)_{t\in \mathbb{R}}
			\]
			such that for each $t\in\mathbb{R}$ (with a slight abuse of notation, we still write $U_t$), we have
			\begin{align*}
				&\Theta_t:\mathcal{M}_0\to\mathcal{M}_0,\\
				&\Theta_t(\Pi,\Gamma)=\bigl(U_t\Pi,U_t\Gamma\bigr),
			\end{align*}
			where, for $		\tau\in\{\rb, \gIb, \gXXXX\rb, \gY, \gI, \gXXXX:\, i=1,\ldots,d\}$
			and $z\in\mathbb{R}^{d+1}$, we define
			\[
			(U_t\Pi)(\tau)(z):=U_t\bigl(\Pi_{z+(t,0)}\tau\bigr).
			\]
			Moreover, $\Theta_t\Gamma$ is defined by \eqref{B1478} using $\Theta_t\Pi$.
			
			\item For each $t,s\in\mathbb{R}$,
			\begin{align}\label{UAJsss}
				\Theta_{t+s}=\Theta_t\circ\Theta_s.
			\end{align}
			Moreover,
			\begin{align}\label{UAJsss1}
				(\Theta_t)_{\#}(\mathcal{Z}^{\epsilon})_{\#}\mathbb{P}
				=
				(\mathcal{Z}^{\epsilon})_{\#}\mathbb{P},
			\end{align}
			and hence $\Theta_t$ is measure-preserving.	
		\end{enumerate}
	\end{lemma}
	\begin{proof}
		It is clear that for each $t\in\mathbb{R}$ one has
		$\Theta_t(\mathcal{M}_0)\subseteq\mathcal{M}$. We prove that, in fact,
		$\Theta_t(\mathcal{M}_0)\subseteq\mathcal{M}_0$. First note that, by definition,
		$\Theta_t:\mathcal{M}_0\to\mathcal{M}$ is continuous. Moreover,
		$$
		\Theta_t\mathbf{M}(f,c)
		=
		\mathbf{M}(U_tf,c).
		$$
		Since $U_t$ is a bijection on 
		$C(\mathbb{R}\times\mathbb{T}^d)$, it follows that
		\begin{align*}
			\Theta_t\bigl(\mathbf{M}(C(\mathbb{R}\times\mathbb{T}^d)\times\mathbb{R})\bigr)
			=
			\mathbf{M}(C(\mathbb{R}\times\mathbb{T}^d)
			\times \mathbb{R}).
		\end{align*}
	Thus, by continuity of $\Theta_t$ and the definition of $\mathcal{M}_0$,
	\[
	\Theta_t(\mathcal{M}_0)\subseteq\mathcal{M}_0.
	\]
Moreover, the map
\[
(t,M)\longmapsto \Theta_tM
\]
is continuous. Indeed, if $(t_n,M_n)\to(t,M)$, then for every compact set $\mathcal O\subset\mathbb R^{d+1}$, by the triangle inequality,
\[
\|\Theta_{t_n}M_n;\Theta_tM\|_{\mathcal O}
\leq
\|\Theta_{t_n}M_n;\Theta_{t_n}M\|_{\mathcal O}
+
\|\Theta_{t_n}M;\Theta_tM\|_{\mathcal O}.
\]
Since $(t_n)$ is bounded, the first term converges to zero by $M_n\to M$ and the definition of the local model seminorms, while the second converges to zero by continuity of translations in the model topology. This proves the first item. The proof of \eqref{UAJsss} is straightforward. We now turn to the proof of \eqref{UAJsss1}. To this end, it suffices to show that
		\begin{align*}
			\int_{\Omega} \overline{F}\bigl(\mathcal{Z}^{\epsilon}(\omega)\bigr)\,\mathbb{P}(\mathrm{d}\omega)
			=
			\int_{\Omega}  \overline{F}\bigl(\Theta_t \mathcal{Z}^{\epsilon}(\omega)\bigr)\,\mathbb{P}(\mathrm{d}\omega),
		\end{align*}
		for every bounded measurable function $ \overline{F}:\mathcal{M}_0\to\mathbb{R}$.
		Equivalently, by \eqref{Baxczxx}, it suffices to show that
		\begin{align*}
			\int_{\Omega}
			\overline{F}\left(
			\mathbf{M}\bigl(
			\mathcal{P}^{\epsilon}(\omega),
			C_\epsilon
			\bigr)
			\right)
			\mathbb{P}(\mathrm{d}\omega)
			=
			\int_{\Omega}
			\overline{F}\left(
			\Theta_t
			\mathbf{M}\bigl(
			\mathcal{P}^{\epsilon}(\omega),
			C_\epsilon
			\bigr)
			\right)
			\mathbb{P}(\mathrm{d}\omega).
		\end{align*}
		This follows from Lemma \ref{UYHNAss}, Lemma~\ref{JMasss} and Remark~\ref{Asaisada}.
	\end{proof}
	\begin{remark}\label{REFRass}
		For each $\epsilon>0$, let
		\[
		\xi^\epsilon(\omega,z):=\mathcal{P}^\epsilon(\omega,z).
		\]
		Since $\xi^\epsilon$ is smooth, one may consider the classical random
		parabolic PDE
		\[
		\begin{cases}
			(\partial_t-\Delta)u_\epsilon
			=
			\mu u_\epsilon
			+
			\Sigma(u_\epsilon)\xi^\epsilon
			-
			C_\epsilon\Sigma'(u_\epsilon)\Sigma(u_\epsilon),
			& t>0,\\
			u_\epsilon(0,\cdot)=u_0\in L^\infty(\mathbb T^d).
		\end{cases}
		\]
		On the other hand, let $\RR{\mathcal{U}}_{u_0}$ denote the solution of the equation
		\begin{align*}
			\begin{cases}
				(\partial_t - \Delta) \RR{\mathcal{U}}
				=\mu  \RR{\mathcal{U}}+ \Sigma(\RR{\mathcal{U}})\tilde{\star} \rb,
				& t > 0, \\
				\mathcal{R}(\RR{\mathcal{U}})(0,\cdot) = {u_0},
				& t=0,
			\end{cases}
		\end{align*}
		with respect to $\mathcal{Z}^{\epsilon}$, cf. Definition~\ref{dfn}.
		Since
		\begin{align}
			\begin{split}
				&\Pi_z^{\big(
					\mathcal{P}^{\epsilon}(\omega),
					C_\epsilon
					\big)}(\rb)
				=\mathcal{P}^{\epsilon}(\omega,\cdot)=	\xi^\epsilon(\omega,\cdot),\\
				&\Pi_z^{\bigl(
					\mathcal{P}^{\epsilon}(\omega),
					C_\epsilon
					\bigr)}(\gIb)
				=
				\bigl(K*\mathcal P^\epsilon
				-(K*\mathcal P^\epsilon)(z)\bigr)
				\mathcal P^\epsilon
				-C_\epsilon,
			\end{split}
		\end{align}
		and $\mathcal P^\epsilon$ is a smooth function, the reconstruction operator is simply
		evaluation. Consequently,
		\begin{align*} 
			&\mathcal{R}\left(\Sigma(\RR{\mathcal{U}}_{u_0})\tilde{\star} \rb\right)(z)=\Pi_z^{\big(
				\mathcal{P}^{\epsilon}(\omega),
				C_\epsilon
				\big)}\big(\Sigma(\RR{\mathcal{U}}_{u_0})\tilde{\star} \rb\big)(z)\\&=\Sigma\left({\mathcal{U}}^{\gY}(z)\right)\xi^\epsilon(\omega,z)-C_{\epsilon}\Sigma^{\prime}\left({\mathcal{U}}^{\gY}(z)\right)\Sigma\left({\mathcal{U}}^{\gY}(z)\right).
		\end{align*}
		Therefore, $\mathcal R(\RR{\mathcal{U}}_{u_0})$ satisfies the same classical PDE as
		$u_\epsilon$. By uniqueness of classical solutions, we conclude that
		\[
		\mathcal R(\RR{\mathcal{U}}_{u_0})=u_\epsilon.
		\]
	\end{remark}
	The final ingredient in the construction of the solution theory for
	\eqref{MAINAA} is to show that the family of models $\left(\mathcal{Z}^{\epsilon}\right)_{\epsilon>0}$
	converges to a limiting model, which we denote by $\mathcal{Z}$. This, in
	turn, provides the solution theory for \eqref{MAINAA} by passing to the limit
	in the corresponding sequence of renormalized equations. Moreover, we prove
	that the limiting model satisfies the same properties as those established in
	Lemma~\ref{Aaasqwqwa}. 
	\begin{proposition}\label{Ujasn785sw}
		Consider the situation of Lemma~\ref{Aaasqwqwa}. Then there exists a limiting random model
		${\mathcal Z}:\Omega\rightarrow\mathcal{M}_0$, independent of the choice of the mollifier $\rho$,
		such that, for each compact set $\mathcal{O}\subset\R^{d+1}$,
		\begin{align}\label{Oaslnr}
			\lim_{\epsilon\rightarrow 0}
			\E\big[\Vert\mathcal{Z}^{\epsilon};{\mathcal Z}\Vert_{\mathcal{O}}\big]
			=0.
		\end{align}
		Moreover, for every $p\geq1$,
		\begin{align}\label{uniform-model-moments}
			\sup_{\epsilon\in(0,1)}
			\E\big[
			\Vert\mathcal{Z}^{\epsilon}\Vert_{\mathcal{O}}^{p}
			\big]
			<\infty.
		\end{align}
		Finally for each $t\in\R$
		\begin{align}\label{sdKisd4ws}
			(\Theta_t)_{\#}(\mathcal{Z})_{\#}\mathbb{P}
			=
			(\mathcal{Z})_{\#}\mathbb{P}.
		\end{align}
	\end{proposition}
	\begin{proof}
		The first two claims follow by adapting the proof of \cite[Theorem~10.19]{Hai14} to the present setting. Finally, \eqref{sdKisd4ws} follows immediately from \eqref{UAJsss1} and \eqref{Oaslnr}.
	\end{proof}
	\begin{corollary}\label{HAYSda}
		Consider the setting of Proposition~\ref{Ujasn785sw}. Then
		\[
		\left(\mathcal{M}_0, \mathcal{B}(\mathcal{M}_0),\{\Theta_t\}_{t\in\mathbb{R}},(\mathcal{Z})_{\#}\mathbb{P}\right)
		\]
		is an {invertible measurable metric dynamical system}.
	\end{corollary}
	\begin{proof}
		The map $(t,M)\mapsto\Theta_tM$ is continuous and hence measurable. Therefore, the claim follows immediately from Lemma~\ref{Aaasqwqwa} and Proposition~\ref{Ujasn785sw}.
	\end{proof}
	\subsection{Ergodicity}
	The final ingredient is to establish the ergodicity of
	
	$$
	\Theta_t:\mathcal{M}_0\to\mathcal{M}_0,
	\qquad t\neq0.
	$$
	
	We now turn to this objective. The main idea is to first define a shift map on
	$\Omega$ and prove its ergodicity. The ergodicity of $\Theta_t$ then follows by
	transporting this property through the map $\mathcal Z$.
	\begin{lemma}\label{BGAssss}
		Let
		\begin{align}\label{Imasko74}
			C_{i,j}(t):=\langle U_t f_i,f_j\rangle_{\overline{\mathcal{H}}^{Q}},
			\qquad i,j\geq1 .
		\end{align}
		Then there exists a measurable subset $\Omega_t \subseteq \Omega$ of full measure such that, for each $\omega \in \Omega_t$, the series
		\[
		\sum_{j\ge1} C_{i j}(t)\,\omega_j
		\]
		converges for every $i\ge1$.
		In particular, the map
		\begin{align}\label{asaas26}
			\hat{\theta}_t:\Omega \to \Omega
		\end{align}
		defined by
		\begin{align}\label{HYjas}
			\hat{\theta}_t(\omega):=
			\begin{cases}
				\left( \sum_{j\ge1} C_{i j}(t)\,\omega_j \right)_{i\ge1},
				& \text{if } \omega \in \Omega_t,\\[2mm]
				0=(0)_{i\ge1},
				& \text{otherwise},
			\end{cases}
		\end{align}
		is Borel measurable and satisfies
		\[
		(\hat{\theta}_t)_\# \mathbb{P} = \mathbb{P}.
		\]
	\end{lemma}
	\begin{proof}
		Let $i\ge1$. Since $(\omega_j)_{j\ge1}$ are i.i.d. standard Gaussian random variables and
		
		$$
		\sum_{j\ge1}\left(C_{ij}(t)\right)^2
		=
		\|U_tf_i\|_{\overline{\mathcal H}^Q}^2
		=
		1,
		$$
		
		the partial sums
		$$
		\sum_{j=1}^N C_{ij}(t)\omega_j
		$$
		form an $L^2$-bounded martingale. Hence, by the martingale convergence theorem, the series
		\begin{align}\label{Oams478}
			\sum_{j\ge1} C_{ij}(t)\omega_j
		\end{align}
		converges almost surely and in $L^2(\Omega)$. Since $i\geq 1$ is countable, this implies the existence of a measurable set $\Omega_t \subset \Omega$ of full measure such that, for every $\omega \in \Omega_t$, all the series in \eqref{Oams478} are well-defined and converge simultaneously for all $i\ge1$.
		For $\omega \in \Omega_t$, define
		\begin{align}\label{Km4785as}
			(\hat{\theta}_t \omega)_i := \sum_{j\ge1} C_{ij}(t)\,\omega_j,
			\qquad i\ge1.
		\end{align}
		By its definition, $\hat{\theta}_t:\Omega\to\Omega$ is Borel measurable.
		Now let $i,k\ge1$. By independence of $(\omega_j)_{j\ge1}$ and standard Gaussian computations,
		\[
		\mathbb{E}\big[(\hat{\theta}_t\omega)_i(\hat{\theta}_t\omega)_k\big]
		=
		\sum_{j\ge1} C_{ij}(t)\,C_{kj}(t).
		\]
		Hence the covariance matrix of $\hat{\theta}_t\omega$ is
		\begin{align}\label{BBB47}
			\mathrm{Cov}(\hat{\theta}_t\omega)
			=
			C(t)C(t)^\top.
		\end{align}
		Here $C(t)^\top$ denotes the transpose (equivalently, the adjoint on $\ell^2$) of the operator $C(t)$. Since $\langle U_t f_i,U_tf_j\rangle_{\overline{\mathcal{H}}^{Q}}=\langle  f_i,f_j\rangle_{\overline{\mathcal{H}}^{Q}}$, we conclude that
		\begin{align}\label{UJMa}
			C(t)C(t)^\top = I.
		\end{align}
		Consequently, $(\hat{\theta}_t\omega)_{i\ge1}$ is a centered Gaussian family with covariance identity, hence has the same law as $(\omega_i)_{i\ge1}$. Therefore,
		\[
		(\hat{\theta}_t)_\# \P = \P.
		\]
	\end{proof}
	Now that we have established that, for each \(t \in \mathbb{R}\), the map \(\hat{\theta}_t\) is measure-preserving, a natural question is whether, for \(t \neq 0\), this map is ergodic. In the next result, we address this question.
	
	\begin{proposition}\label{Ik78assq}
		Consider the setting of Lemma \ref{BGAssss}. Then, for each \(t \neq 0\), the map \(\hat{\theta}_t : \Omega \to \Omega\) is ergodic.
	\end{proposition}
	\begin{proof}
		\textbf{Step 1.}\label{Step001} Recall that
		\[
		(\Omega,\mathcal{F},\mathbb{P}) = \left(\mathbb{R}^{\mathbb{N}}, \mathcal{B}(\mathbb{R}^{\mathbb{N}}), \gamma^{\otimes \mathbb{N}}\right).
		\]
		In this Wiener space, $\ell^{2}$ is the corresponding Cameron--Martin space.
		To prove ergodicity, it is sufficient to show that if \(\zeta \in L^{2}(\Omega)\) satisfies
		\begin{align}\label{Y11}
			\zeta(\hat{\theta}_t \omega) = \zeta(\omega), \quad \mathbb{P}\text{-a.s.},
		\end{align}
		then \(\zeta\) is almost surely constant.
		To this end, we use the Wiener chaos expansion (cf. Definition \ref{BASsdswsw}). From \cite[Proposition 1.1.1]{Nua05}, the projection of \(\zeta\) onto \(\mathcal{H}_n\) is given by
		\[
		I_0(\zeta) := \mathbb{E}[\zeta(\omega)],
		\qquad n = 0,
		\]
		and for \(n \ge 1\) ,
		\begin{align}\label{ASasasasas}
			I_n(\zeta)(\overline{\omega}) := \sum_{\alpha \in \Lambda_n} \ \mathbb{E}\!\left[\zeta(\omega) \, \sqrt{\alpha!}H_{\alpha}(\omega)\right]\sqrt{\alpha!} H_{\alpha}(\overline{\omega}).
		\end{align}
		We show that for each \(n \ge 1\), one has \(I_n(\zeta) = 0\). Then, it follows that
		\(\zeta(\omega) = \mathbb{E}[\zeta(\omega)]\) almost surely, and hence the claim follows. For $n\geq 1$, first note that the assumption on $\zeta$ and lemma \ref{BGAssss} imply
		\begin{align}\label{Y22}
			\mathbb{E}\!\left[\zeta(\omega) \, H_{\alpha}(\omega)\right]=\mathbb{E}\!\left[	\zeta(\hat{\theta}_t \omega) \, H_{\alpha}(\omega)\right]=\mathbb{E}\!\left[	\zeta(\omega) \, H_{\alpha}(\hat{\theta}_{-t} \omega)\right]
		\end{align}
		From \eqref{Imasko74} and \eqref{BBB47}, one can associate with \(\hat{\theta}_{-t}\) the orthogonal operator
		\[
		C(-t): \ell^2 \to \ell^2,
		\]
		defined by
		\[
		\bigl(C(-t)v\bigr)_i
		:=
		\sum_{j\geq 1} C_{ij}(-t)\,v_j,
		\qquad v=(v_j)_{j\geq1}\in \ell^2.
		\]
		From \cite[Equation~(1.1)]{Nua05}, for each \(v\in \ell^2\), one has almost surely
		\begin{align}\label{JMasssa}
			\exp\left( \langle v,\omega\rangle-\frac{1}{2}\|v\|^2\right)
			=
			\sum_{\alpha\in\Lambda} v^{\alpha} H_{\alpha}(\omega),
		\end{align}
		where for \(\alpha = (\alpha_1,\alpha_2,\dots)\in\Lambda\),
		\begin{align*}
			\|v\|^2 := \sum_{i\geq 1}(v_i)^2, 
			\qquad 
			v^{\alpha} = \prod_{i\geq 1} (v_i)^{\alpha_i},
			\qquad 
			\langle v,\omega\rangle = \sum_{i\geq 1} v_i \omega_i.
		\end{align*}
		Note that \(\langle v,\omega\rangle\) is well defined on a set of full measure. Recall that \(C(-t)^{\top}=C(t)\) and $\|C(t)v\|^2=\Vert v\Vert^2$, so we have
		\begin{align*}
			\exp\left( \langle C(t) v,\omega\rangle-\frac{1}{2}\|C(t)v\|^2\right)
			=
			\exp\left( \langle v, C(-t)\omega\rangle-\frac{1}{2}\|v\|^2\right).
		\end{align*}
		Consequently, together with \eqref{Km4785as} and \eqref{JMasssa}, we conclude that for each \(v\in \ell^2\) (on a set of full measure)
		\begin{align}\label{Assawqqw}
			\sum_{\alpha\in\Lambda}\left(C(t)v\right)^{\alpha}H_{\alpha}(\omega)
			=
			\sum_{\alpha\in\Lambda}v^{\alpha}H_{\alpha}(\hat{\theta}_{-t}\omega).
		\end{align}
		Taking the homogeneous component of degree $n$ in \eqref{Assawqqw}, we obtain
		\begin{align*}
			\sum_{\alpha\in\Lambda_n}\left(C(t)v\right)^{\alpha}H_{\alpha}(\omega)
			=
			\sum_{\alpha\in\Lambda_n}v^{\alpha}H_{\alpha}(\hat{\theta}_{-t}\omega).
		\end{align*}
		\textbf{Step 2.}\label{Step002}
		Let $\alpha \in \Lambda_n$, the aim is to obtain an identity like
		\[
		H_{\alpha}(\hat{\theta}_{-t}\omega)=\sum_{\overline{\alpha}\in\Lambda_n} K({\overline{\alpha}}) H_{\overline{\alpha}}(\omega)
		\]
		where the coefficients $K({\overline{\alpha}})$ depend on the $C(t)$. For this purpose, we first introduce some auxiliary notation. Let \((\tilde{e}_i)_{i\geq1}\) denote the canonical  orthonormal basis of \(\ell^2\).
		For \(n \ge 1\), we consider elementary tensors of the form
		\[
		\tilde{e}_{i_1} \otimes \cdots \otimes \tilde{e}_{i_n}, 
		\qquad (i_1,\dots,i_n) \in \mathbb{N}^n.
		\]
		We denote by \(S_n\), the set of all permutations of \(\{1,2,\ldots,n\}\).
		We define an equivalence relation on such tensors by declaring that
		\[
		\tilde{e}_{i_1} \otimes \cdots \otimes \tilde{e}_{i_n}
		\sim
		\tilde{e}_{j_1} \otimes \cdots \otimes \tilde{e}_{j_n}
		\]
		if there exists \(\sigma \in S_n\) such that
		\[
		j_k = i_{\sigma(k)},
		\qquad k=1,\ldots,n.
		\]
		
		For \(\alpha=(\alpha_i)_{i\geq 1} \in \Lambda_n\),
		\begin{itemize}
			\item We associate to each \(\alpha \in \Lambda_n\) the ordered representation \(I=(i_1,\dots,i_n)\) satisfying
			\begin{align}\label{Y33}
				1\leq	i_1 \le \cdots \le i_n,
			\end{align}
			and
			\[
			\alpha_i = \#\{k : i_k = i\},
			\qquad i\geq 1.
			\]
			
			Conversely, for each \(I=(i_1,\dots,i_n)\) satisfying \eqref{Y33}, we associate the multi-index
			\(\alpha^{I}=(\alpha^{I}_{i})_{i\geq 1}\in\Lambda\) defined by
			\begin{align}\label{X1}
				\alpha^{I}_{i}
				=
				\#\{k : i_k = i\},
				\qquad i\geq 1.
			\end{align}
			\item We then define
			\[
			\tilde{e}_\alpha := \bigotimes_{k=1}^n \tilde{e}_{i_k}.
			\]
			Also
			\begin{align}\label{X2}
				H_{(i_1,...,i_n)}(\omega):=H_{\alpha}(\omega).
			\end{align}
		\end{itemize}
		With these preliminaries let ${\alpha}\in\Lambda_{n}$ and $\tilde{e}_{{\alpha}} := \bigotimes_{k=1}^n \tilde{e}_{i_k}$, then
		\begin{align*}
			&\left(C(t)v\right)^{{\alpha}}=\prod_{k=1}^{n}\left(C(t)v\right)_{i_k}=\prod_{k=1}^{n}\left(\sum_{j_{k}=1}^{\infty}C_{i_k,j_k}(t)v_{j_k}\right)\\&=\sum_{j_1,...j_n}\left(\prod_{k=1}^{n}C_{i_k,j_k}(t)\right)\left(\prod_{k=1}^{n}v_{j_k}\right)=\sum_{j_1 \le \cdots \le j_n}\left(\prod_{k=1}^{n}v_{j_k}\right)\left(\frac{1}{\alpha^{J}!}\sum_{\sigma \in S_{n}}\prod_{k=1}^{n}C_{i_k,j_{\sigma(k)}}(t)\right).
		\end{align*}
		Consequently
		\begin{align*}
			&\sum_{{\alpha}\in\Lambda_n}\left(C(t)v\right)^{{\alpha}}H_{{\alpha}}(\omega)=\sum_{i_1 \le \cdots \le i_n}H_{(i_1,...,i_n)}(\omega)\sum_{j_1 \le \cdots \le j_n}\left(\left(\prod_{k=1}^{n}v_{j_k}\right)\left(\frac{1}{\alpha^{J}!}\sum_{\sigma \in S_{n}}\prod_{k=1}^{n}C_{i_k,j_{\sigma(k)}}(t)\right)\right)\\&=\sum_{j_1 \le \cdots \le j_n}\left(\prod_{k=1}^{n}v_{j_k}\right)\sum_{i_1 \le \cdots \le i_n}H_{(i_1,...,i_n)}(\omega)\left(\frac{1}{\alpha^{J}!}\sum_{\sigma \in S_{n}}\prod_{k=1}^{n}C_{i_k,j_{\sigma(k)}}(t)\right).
		\end{align*}
		On the other hand,
		\begin{align*}
			\sum_{\alpha\in\Lambda_n}v^{\alpha}H_{\alpha}(\hat{\theta}_{-t}\omega)
			=
			\sum_{j_1 \le \cdots \le j_n}
			\left(\prod_{k=1}^{n}v_{j_k}\right)
			H_{(j_1,\ldots,j_n)}(\hat{\theta}_{-t}\omega).
		\end{align*}
		Thus, by \eqref{Assawqqw}, for every \((j_1,\dots,j_n)\in\mathbb{N}^n\) satisfying
		\(j_1 \le \cdots \le j_n\), we obtain
		\begin{align}\label{NN12}
			H_{(j_1,\ldots,j_n)}(\hat{\theta}_{-t}\omega)
			=
			\sum_{i_1 \le \cdots \le i_n}
			H_{(i_1,\ldots,i_n)}(\omega)
			\left(\frac{1}{\alpha^{J}!}
			\sum_{\sigma\in S_n}
			\prod_{k=1}^{n}
			C_{i_k,j_{\sigma(k)}}(t)
			\right).
		\end{align}
		Thanks to \eqref{Y22} and \eqref{NN12}, for each $(j_1,\dots,j_n)\in\mathbb{N}^n$ with $j_1 \le \cdots \le j_n$
		\begin{align}\label{LPoa}
			\begin{split}
				& \mathbb{E}\!\left[\zeta(\omega) \, H_{(j_1,\ldots,j_n)}(\omega)\right]= \mathbb{E}\!\left[\zeta(\omega) \, H_{(j_1,\ldots,j_n)}(\hat{\theta}_{-t} \omega)\right] \\&= \sum_{i_1 \le \cdots \le i_n}
				\mathbb{E}\!\left[\zeta(\omega) \, H_{(i_1,\ldots,i_n)}( \omega)\right]
				\left(\frac{1}{\alpha^{J}!}
				\sum_{\sigma\in S_n}
				\prod_{k=1}^{n}
				C_{i_k,j_{\sigma(k)}}(t)
				\right).
			\end{split}
		\end{align}
		\textbf{Step 3.}\label{Step003} 
	All series in this step are understood in the corresponding $L^2$ spaces;
	the displayed almost-everywhere identities refer to their $L^2$ limits.	Recall that for every \(i,j\ge1\),
		\begin{align}\label{Yhanss}
			\begin{split}
				\langle f_i,f_j\rangle_{\overline{\mathcal{H}}^{Q}}
				&=
				\int_{\mathbb{R}}
				\left\langle
				Q^{\frac12}f_i(t,\cdot),
				Q^{\frac12}f_j(t,\cdot)
				\right\rangle_{L^2(\mathbb{T}^d)}
				\,\mathrm{d}t  \\
				&=
				\int_{\mathbb{R}}
				\int_{\mathbb{T}^d}
				\bigl(Q^{\frac12}f_i(t,\cdot)\bigr)(x)\,
				\bigl(Q^{\frac12}f_j(t,\cdot)\bigr)(x)
				\,\mathrm{d}x\,\mathrm{d}t
				=
				\delta_{ij}.
			\end{split}
		\end{align}
		We define 
		\begin{align}\label{Iias78as}
			F(t_1,\ldots,t_n)(x_1,\ldots,x_n)
			:=
			\sum_{i_1\le\cdots\le i_n}\mathbb{E}\!\left[\zeta\,H_{(i_1,\ldots,i_n)}\right]
			\sum_{\sigma\in S_n}
			\prod_{k=1}^{n}
			\bigl(Q^{\frac12}f_{i_{\sigma(k)}}(t_k,\cdot)\bigr)(x_k)
			.
		\end{align}
		Thanks to \eqref{Yhanss}, for every
		\(I=(i_1,\ldots,i_n)\) and
		\(J=(j_1,\ldots,j_n)\) satisfying
		\(i_1\le\cdots\le i_n\) and
		\(j_1\le\cdots\le j_n\), we have
		\begin{align}\label{orthogonality}
			\begin{split}
				&\int_{\mathbb{R}^n}
				\int_{\mathbb{T}^{nd}}
				\left(
				\sum_{\sigma\in S_n}
				\prod_{k=1}^{n}
				\bigl(Q^{\frac12}f_{i_{\sigma(k)}}(t_k,\cdot)\bigr)(x_k)
				\right)
				\left(
				\sum_{\tau\in S_n}
				\prod_{k=1}^{n}
				\bigl(Q^{\frac12}f_{j_{\tau(k)}}(t_k,\cdot)\bigr)(x_k)
				\right)
				\,\mathrm{d}x_1\cdots\mathrm{d}x_n\,
				\mathrm{d}t_1\cdots\mathrm{d}t_n
				\\
				&\qquad
				=
				n!\sqrt{\alpha^I!}\,\sqrt{\alpha^J!}\,\delta_{I,J},
			\end{split}
		\end{align}
		Thus, by \eqref{ASasasasas}, \cite[Proposition~1.1.1]{Nua05} and \eqref{orthogonality}, we have
		\begin{align}\label{Im784aw}
			\begin{split}
				\int_{\R^n} 	\Vert F(t_1,...,t_n)\Vert_{L^{2}(\T^{nd})}^2\mathrm{d}t_1\cdots\mathrm{d}t_n&=n!\Vert I_n(\zeta)\Vert_{L^{2}(\Omega)}^2\\&=n!\sum_{\substack{i_1 \le \cdots \le i_n}} \alpha^{I}!\ \left(\mathbb{E}\!\left[\zeta\,H_{(i_1,\ldots,i_n)}\right]\right)^2\leq n!\Vert\zeta\Vert_{L^{2}(\Omega)}^2.
			\end{split}
		\end{align}
		We show that
		\[
		F(t_1,\ldots,t_n)
		=
		F(t_1+t,\ldots,t_n+t)
		\quad
		\text{for Lebesgue-a.e. } (t_1,\ldots,t_n)\in\mathbb{R}^n.
		\]
		First, note that by \eqref{Yhanss},
		\[
		Q^{\frac12}f_i(s,\cdot)\in L^2(\mathbb T^d)
		\quad\text{for Lebesgue-a.e. }s\in\mathbb R,
		\qquad i\ge1.
		\]
		Also, by definition,
		\begin{align}\label{N1236as9}
			\begin{split}
				&F(t_1+t,\ldots,t_n+t)(x_1,\ldots,x_n) \\
				&\qquad=
				\sum_{i_1\le\cdots\le i_n}\mathbb{E}\!\left[\zeta\,H_{(i_1,\ldots,i_n)}\right]
				\sum_{\sigma\in S_n}
				\prod_{k=1}^{n}
				\bigl(Q^{\frac12}f_{i_{\sigma(k)}}(t_k+t,\cdot)\bigr)(x_k)
				.
			\end{split}
		\end{align}
		Furthermore, recall that
		\[
		C_{i,j}(t):=\langle U_t f_i,f_j\rangle_{\overline{\mathcal{H}}^{Q}}.
		\]
		Expanding $U_t f_{i_{\sigma(k)}}$ with respect to the orthonormal basis
		$(f_j)_{j\geq1}$, for each $1\leq k\leq n$ and $\sigma\in S_n$, we have
		\begin{align*}
			\bigl(Q^{\frac12}f_{i_{\sigma(k)}}(t_k+t,\cdot)\bigr)(x_k)
			&=
			\left(
			Q^{\frac12}
			\left(
			U_t f_{i_{\sigma(k)}}(t_k,\cdot)
			\right)
			\right)(x_k)\
			\\	&=
			\sum_{j_k\geq1}
			\langle U_t f_{i_{\sigma(k)}},f_{j_k}\rangle_{\overline{\mathcal H}^{Q}}
			\bigl(Q^{\frac12}f_{j_k}(t_k,\cdot)\bigr)(x_k)\
			&=
			\sum_{j_k\geq1}
			C_{i_{\sigma(k)},j_k}(t)
			\bigl(Q^{\frac12}f_{j_k}(t_k,\cdot)\bigr)(x_k),
		\end{align*}
		for Lebesgue-a.e. $(t_k,x_k)\in\mathbb R\times\mathbb T^d$.
		Thus, 
		\begin{align*}
			&		\sum_{\sigma\in S_n}\prod_{k=1}^{n}
			\bigl(Q^{\frac12}f_{i_{\sigma(k)}}(t_k+t,\cdot)\bigr)(x_k)=\sum_{\sigma\in S_n}\sum_{j_1,...,j_n\geq 1}\left(\prod_{k=1}^{n}C_{i_{\sigma(k)},j_k}(t)\right)\left(\prod_{k=1}^{n}\left(Q^{\frac12}f_{j_{k}}(t_{k},.)\right)(x_k)\right)\\&=\sum_{j_1,...,j_n\geq 1}\left(\sum_{\sigma\in S_n}\prod_{k=1}^{n}C_{i_{\sigma(k)},j_k}(t)\right)\left(\prod_{k=1}^{n}\left(Q^{\frac12}f_{j_{k}}(t_{k},.)\right)(x_k)\right),\\&=\sum_{j_1 \le \cdots \le j_n}\left(\sum_{\sigma\in S_n}\frac{1}{\alpha^{J}!}\prod_{k=1}^{n}C_{i_{\sigma(k)},j_k}(t)\right)\left(\sum_{\overline{\sigma}\in S_n}\prod_{k=1}^{n}\left(Q^{\frac12}f_{j_{\overline{\sigma}(k)}}(t_{k},.)\right)(x_k)\right)\\& \quad\text{for Lebesgue-a.e. } (t_1,\ldots,t_n,x_1,\ldots,x_n)
			\in \mathbb R^n\times\mathbb{T}^{nd}.
		\end{align*} 
		By substituting the above identity into the right-hand side of \eqref{N1236as9}, we obtain, for Lebesgue-a.e. \((t_1,\ldots,t_n,x_1,\ldots,x_n)
		\in \mathbb R^n\times\mathbb{T}^{nd}\),
		\begin{align}\label{M1236}
			\begin{split}
				&F(t_1+t,\ldots,t_n+t)(x_1,\ldots,x_n)\\&=\sum_{\substack{ i_1 \le \cdots \le i_n}}	\mathbb{E}\!\left[\zeta(\omega)\, H_{(i_1,\ldots,i_n)}(\omega)\right]
				\left[\sum_{j_1 \le \cdots \le j_n}\left(\sum_{\sigma\in S_n}\frac{1}{\alpha^{J}!}\prod_{k=1}^{n}C_{i_{\sigma(k)},j_k}(t)\right)\left(\sum_{\overline{\sigma}\in S_n}\prod_{k=1}^{n}\left(Q^{\frac12}f_{j_{\overline{\sigma}(k)}}(t_{k},.)\right)(x_k)\right)\right]
				\\&=\sum_{j_1 \le \cdots \le j_n}\left(\sum_{\overline{\sigma}\in S_n}\prod_{k=1}^{n}\left(Q^{\frac12}f_{j_{\overline{\sigma}(k)}}(t_{k},.)\right)(x_k)\right)\left[\sum_{\substack{i_1 \le \cdots \le i_n}}\mathbb{E}\!\left[\zeta(\omega)\, H_{(i_1,\ldots,i_n)}(\omega)\right]\left(\frac{1}{\alpha^{J}!}\sum_{\sigma\in S_n}\prod_{k=1}^{n}C_{i_{\sigma(k)},j_k}(t)\right)\right].
			\end{split}
		\end{align}
		Thanks to \eqref{LPoa}, we have
		\begin{align}\label{B12369}
			\begin{split}
				&\sum_{\substack{i_1 \le \cdots \le i_n}}\mathbb{E}\!\left[\zeta(\omega)\, H_{(i_1,\ldots,i_n)}(\omega)\right]\left(\frac{1}{\alpha^{J}!}\sum_{\sigma\in S_n}\prod_{k=1}^{n}C_{i_{\sigma(k)},j_k}(t)\right)\\&=\sum_{\substack{i_1 \le \cdots \le i_n}}\mathbb{E}\!\left[\zeta(\omega)\, H_{(i_1,\ldots,i_n)}(\omega)\right]\left(\frac{1}{\alpha^{J}!}\sum_{\sigma\in S_n}\prod_{k=1}^{n}C_{i_{k},j_{\sigma(k)}}(t)\right)=\mathbb{E}\!\left[\zeta(\omega) \, H_{(j_1,\ldots,j_n)}(\omega)\right].
			\end{split}
		\end{align}
		Thus, from \eqref{N1236as9}, \eqref{M1236}, and \eqref{B12369}, we obtain that
		\begin{align}\label{Ikassqq}
			\begin{split}
				&	F(t_1,\ldots,t_n)(x_1,\ldots,x_n)
				=
				F(t_1+t,\ldots,t_n+t)(x_1,\ldots,x_n)
				\\& \quad\text{for Lebesgue-a.e. } (t_1,\ldots,t_n,x_1,\ldots,x_n)
				\in \mathbb R^n\times\mathbb{T}^{nd}.
			\end{split}
		\end{align}
		Recall from \eqref{Im784aw} that
		$$
		\int_{\mathbb{R}^n}
		\|F(t_1,\ldots,t_n)\|_{L^{2}(\mathbb{T}^{nd})}^2
		\,\mathrm{d}t_1\cdots \mathrm{d}t_n
		<\infty.
		$$
		By \eqref{Ikassqq}, the integrand is invariant under translation by
		\((t,\ldots,t)\). Since \(t\neq0\), a nonzero invariant function of this form cannot belong to \(L^1(\mathbb R^n)\). Hence,
		$$
		F=0.
		$$
		In particular, from \eqref{Iias78as} and \eqref{orthogonality} we conclude that
		$$
		\mathbb{E}\!\left[\zeta\,H_{(i_1,\ldots,i_n)}\right]=0
		$$
		for all \(I=(i_1,\ldots,i_n)\) with \(i_1\le\cdots\le i_n\). Therefore, by \eqref{ASasasasas},
		$$
		I_n(\zeta)=0,
		\qquad n\ge1.
		$$
		Consequently,
		$$
		\zeta=I_0(\zeta)=\mathbb{E}[\zeta]
		\quad\text{a.s.},
		$$
		which proves that \(\hat{\theta}_t\) is ergodic for every \(t\neq0\).
	\end{proof}
	The previous result is the main ingredient in proving the ergodicity of
	\[
	\Theta_t:\mathcal{M}_0\to\mathcal{M}_0,
	\qquad t\neq0.
	\]
	Before stating the ergodicity result, we first establish the following
	technical lemma.
	\begin{lemma}\label{IAskca1}
		Let $(\overline{z}_j)_{j\ge1}\subset \mathbb{R}\times\mathbb{T}^d$ be a sequence in $\mathbb{R}\times\mathbb{T}^d$. Then, for each $t\in\mathbb{R}$, there exists a full-measure set  such that
		\begin{align}\label{IAskca}
			\sum_{i\ge1}
			\bigl\langle (\rho^{\epsilon}_{\overline{z}_k})^{\mathrm{per}}, f_i \bigr\rangle_{\overline{\mathcal{H}}^{Q}}
			\left(\sum_{j\ge1} C_{i,j}(t)\omega_j\right)
			=
			\sum_{j\ge1}
			\omega_j	\left(
			\sum_{i\ge1} C_{i,j}(t)
			\bigl\langle (\rho^{\epsilon}_{\overline{z}_k})^{\mathrm{per}}, f_i \bigr\rangle_{\overline{\mathcal{H}}^{Q}}
			\right),
			\quad \forall k\ge1.
		\end{align}
	\end{lemma}
	
	\begin{proof}
		Since we are considering a countable sequence $(\overline{z}_j)_{j\ge1}$, it is enough to show that the identity \eqref{IAskca} holds on a set of full measure for each $\overline{z}_k$.
		Recall that $(f_i)_{i\ge1}$ is an orthonormal basis of $\overline{\mathcal{H}}^{Q}$. By Parseval's identity, we obtain
		\begin{align}\label{Ujkasawe}
			\sum_{j\ge1}
			\left|
			\bigl\langle(\rho_{\overline{z}_k}^\epsilon)^{\mathrm{per}},f_j\bigr\rangle_{\overline{\mathcal{H}}^{Q}}
			\right|^2
			=
			\|(\rho_{\overline{z}_k}^\epsilon)^{\mathrm{per}}\|_{\overline{\mathcal{H}}^{Q}}^2
			<\infty.
		\end{align}	
		For $N\ge1$, define
		\begin{align}\label{Ikasgfa}
			Y^{(N)} := \sum_{i=1}^N \bigl\langle(\rho_{\overline{z}_k}^\epsilon)^{\mathrm{per}},f_i\bigr\rangle_{\overline{\mathcal{H}}^{Q}} \left(\sum_{j\ge1} C_{i,j}(t)\omega_j\right),
			\qquad
			\mathcal F_n := \sigma(\omega_1,\dots,\omega_n).
		\end{align}
	From \eqref{UJMa} and \eqref{Ujkasawe}
		\begin{align}\label{aksasq}
			\E\left[	(Y^{(N)})^2 \right]=\sum_{i=1}^{N}\left( \bigl\langle(\rho_{\overline{z}_k}^\epsilon)^{\mathrm{per}},f_i\bigr\rangle_{\overline{\mathcal{H}}^{Q}}\right)^{2}<\infty
		\end{align}
		Set
		\[
		Y^{(N)}_n := \mathbb E\!\left[Y^{(N)} \mid \mathcal F_n\right].
		\]
		Then, from \eqref{aksasq}, one can easily see that $(Y^{(N)}_n)_{n\ge1}$ is an $L^2$-martingale. Thanks to Doob's theorem,
		\[
		Y^{(N)}_n \to Y^{(N)} \quad \text{in } L^2(\Omega).
		\]
		Since $\mathbb E[\omega_j|\mathcal F_n]=\omega_j$ for $j\le n$ and $0$ otherwise, we obtain
		\[
		Y^{(N)}_n
		=
		\sum_{i=1}^N 	\bigl\langle (\rho^{\epsilon}_{{\overline{z}_k}})^{\mathrm{per}}, f_i \bigr\rangle_{\overline{\mathcal{H}}^{Q}} \sum_{j=1}^n C_{i,j}(t)\omega_j
		=
		\sum_{j=1}^n \omega_j\left(\sum_{i=1}^N C_{i,j}(t)	\bigl\langle (\rho^{\epsilon}_{\overline{z}_k})^{\mathrm{per}}, f_i \bigr\rangle_{\overline{\mathcal{H}}^{Q}}\right).
		\]
		Hence,
		\begin{align}\label{Haks85a}
			Y^{(N)}
			=
			\sum_{j\ge1}\omega_j\left(\sum_{i=1}^N C_{i,j}(t)	\bigl\langle (\rho^{\epsilon}_{\overline{z}_k})^{\mathrm{per}}, f_i \bigr\rangle_{\overline{\mathcal{H}}^{Q}}\right)
			\quad \text{in } L^2(\Omega).
		\end{align}
		Now define $B=(B_j)_{j\geq 1}$ and $B^N=\left(B_j^{(N)}\right)_{j\geq 1}\in\ell^2$ by
		\[
		B_j^{(N)} := \sum_{i=1}^N C_{i,j}(t)	\bigl\langle (\rho^{\epsilon}_{\overline{z}_k})^{\mathrm{per}}, f_i \bigr\rangle_{\overline{\mathcal{H}}^{Q}},
		\qquad
		B_j := \sum_{i\ge1} C_{i,j}(t)	\bigl\langle (\rho^{\epsilon}_{\overline{z}_k})^{\mathrm{per}}, f_i \bigr\rangle_{\overline{\mathcal{H}}^{Q}}.
		\]	
		Since $C(t)^\top:\ell^2\to\ell^2$ is continuous, it follows from \eqref{Ujkasawe} that $B^{(N)}\to B$ in $\ell^2$.
		By the Gaussian isometry,
		\[
		\mathbb E\Big|\sum_{j\ge1} (B_j^{(N)}-B_j)\omega_j\Big|^2
		=
		\sum_{j\ge1} |B_j^{(N)}-B_j|^2
		\to 0,
		\]
		so by \eqref{Haks85a}
		\begin{align}\label{T11}
			Y^{(N)}=	\sum_{j\ge1} B_j^{(N)}\omega_j
			\to
			\sum_{j\ge1} B_j\omega_j
			\quad \text{in } L^2(\Omega).
		\end{align}
		On the other hand, by \eqref{Ikasgfa},
		\begin{align}\label{T111}
			Y^{(N)} \to \sum_{i\ge1} 	\bigl\langle (\rho^{\epsilon}_{\overline{z}_k})^{\mathrm{per}}, f_i \bigr\rangle_{\overline{\mathcal{H}}^{Q}} \left(\sum_{j\ge1} C_{i,j}(t)\omega_j\right)
			\quad \text{in } L^2(\Omega).
		\end{align}
		By uniqueness of limits in $L^2(\Omega)$, we conclude from \eqref{T11} and \eqref{T111} that, on a set of full measure
		\begin{align}\label{JAsuas}
			\sum_{i\ge1} 	\bigl\langle (\rho^{\epsilon}_{\overline{z}_k})^{\mathrm{per}}, f_i \bigr\rangle_{\overline{\mathcal{H}}^{Q}} \left(\sum_{j\ge1} C_{i,j}(t)\omega_j\right)
			=
			\sum_{j\ge1}\omega_j\left(\sum_{i\ge1} C_{i,j}(t)	\bigl\langle (\rho^{\epsilon}_{\overline{z}_k})^{\mathrm{per}}, f_i \bigr\rangle_{\overline{\mathcal{H}}^{Q}}\right).
		\end{align}
	\end{proof}
	Now we are ready to state the main result of this section.
	\begin{theorem}\label{Ujasm}
		Under the assumptions of Corollary~\ref{HAYSda}, the quadruple
		\[
		(\mathcal{M}_0,\mathcal{B}(\mathcal{M}_0),(\mathcal{Z})_{\#}\mathbb{P},
		\{\Theta_t\}_{t\in\mathbb R})
		\]
		is an ergodic, invertible measurable metric dynamical system.
	\end{theorem}
	\begin{proof}
		From Corollary~\ref{HAYSda}, it remains only to prove the ergodicity of
		$\Theta_t$ for each $t\neq0$. We divide the proof into two steps.
		
		\medskip
		\noindent\textbf{Step 1.}
		In this step, we prove that, for each $\epsilon\in(0,1)$,
		\[
		\mathcal{Z}^{\epsilon}(\hat{\theta}_{t}\omega)
		=
		\Theta_{-t}\mathcal{Z}^{\epsilon}(\omega)
		\]
		on a set of full measure, which may depend on $\epsilon$ and $t$, where
		\[
		\mathcal{Z}^{\epsilon}:
		(\Omega,\mathcal{B}(\Omega))
		\longrightarrow
		(\mathcal{M}_0,\mathcal{B}(\mathcal{M}_0))
		\]
		is the Borel measurable map introduced in Lemma~\ref{ASHyasas}.
		By Lemma~\ref{JMasss}, there exists a
		full-measure set on which $\mathcal{P}^{\epsilon}(\omega,\cdot)$ is a
		continuous function satisfying
		\begin{align}\label{Olaspsq}
			\mathcal{P}^{\epsilon}(\omega,z_k)
			=
			\sum_{i\ge1}\omega_i
			\big\langle
			(\rho^{\epsilon}_{z_k})^{\mathrm{per}},
			f_i
			\big\rangle_{\overline{\mathcal{H}}^{Q}},
			\qquad \forall\, k\ge1,
		\end{align}
		where $(z_j)_{j\geq1}$ is the dense sequence introduced in Lemma~\ref{JMasss}.
		Note that $(z_j-(t,0))_{j\ge1}$ is also a dense subset of
		$\mathbb{R}\times\mathbb{T}^{d}$. Therefore, by Remark~\ref{Asaisada}, we may choose the continuous version $\mathcal{P}^{\epsilon}(\omega,\cdot)$ so that, on a full-measure set, in addition to \eqref{Olaspsq}, it also satisfies
		\begin{align}\label{Olaspsq_shift}
			\mathcal{P}^{\epsilon}\bigl(\omega,z_k-(t,0)\bigr)
			=
			\sum_{i\ge1}\omega_i
			\big\langle
			(\rho^{\epsilon}_{z_k-(t,0)})^{\mathrm{per}},
			f_i
			\big\rangle_{\overline{\mathcal{H}}^{Q}},
			\qquad \forall\, k\ge1.
		\end{align}
		From \eqref{HYjas} and \eqref{Olaspsq}, on a set of full measure, we have
		\begin{align}\label{IAkslas8}
			\begin{split}
				\mathcal{P}^{\epsilon}(\hat{\theta}_{t}(\omega),z_k)
				&=
				\sum_{i\ge1} (\hat{\theta}_{t}(\omega))_{i}
				\big\langle (\rho^{\epsilon}_{z_k})^{\mathrm{per}}, f_i \big\rangle_{\overline{\mathcal{H}}^{Q}} \\
				&=
				\sum_{i\ge1}\big\langle (\rho^{\epsilon}_{z_k})^{\mathrm{per}}, f_i \big\rangle_{\overline{\mathcal{H}}^{Q}}
				\bigg(\sum_{j\ge1} C_{ij}(t)\,\omega_j\bigg)
				,
				\qquad \forall k\ge1.
			\end{split}
		\end{align}
		Also, from \eqref{Olaspsq_shift} on a set of full measure,
		\begin{align*}
			\bigl(U_{-t} \mathcal{P}^{\epsilon}(\omega)\bigr)(z_k)=\mathcal{P}^{\epsilon}\left(\omega,z_k-(t,0)\right)=	\sum_{i\ge1}\omega_i
			\big\langle
			(\rho^{\epsilon}_{z_k-(t,0)})^{\mathrm{per}},
			f_i
			\big\rangle_{\overline{\mathcal{H}}^{Q}},
			\qquad \forall\, k\ge1
		\end{align*}
		Recall that $(f_i)_{i\ge1}$ is an orthonormal basis of $\overline{\mathcal{H}}^{Q}$. Also, it is readily verified that
		\[
		(\rho^{\epsilon}_{z_k-(t,0)})^{\mathrm{per}}
		=
		U_{t}(\rho^{\epsilon}_{z_k})^{\mathrm{per}}.
		\]
		Therefore,
		\begin{align*}
			&	\bigl(U_{-t} \mathcal{P}^{\epsilon}(\omega)\bigr)(z_k)
			=
			\sum_{i\ge1}
			\langle U_{t}(\rho_{z_k}^\epsilon)^{\mathrm{per}},f_i\rangle_{\overline{\mathcal{H}}^{Q}}\,\omega_i=
			\sum_{i\ge1}\omega_i
			\langle (\rho_{z_k}^\epsilon)^{\mathrm{per}},U_{-t}f_i\rangle_{\overline{\mathcal{H}}^{Q}}\\\
			&=
			\sum_{i\ge1}\omega_i
			\left\langle
			(\rho_{z_k}^\epsilon)^{\mathrm{per}},
			\sum_{j\ge1}\langle U_{-t}f_i,f_j\rangle_{\overline{\mathcal{H}}^{Q}}f_j
			\right\rangle_{\overline{\mathcal{H}}^{Q}}
			=
			\sum_{i\ge1}\omega_i\left(\sum_{j\ge1}
			C_{i,j}(-t)\,
			\bigl\langle(\rho_{z_k}^\epsilon)^{\mathrm{per}},f_j\bigr\rangle_{\overline{\mathcal{H}}^{Q}}\right)
			\\
			&=
			\sum_{j\ge1}\omega_j\left(\sum_{i\ge1}
			C_{j,i}(-t)\,
			\bigl\langle(\rho_{z_k}^\epsilon)^{\mathrm{per}},f_i\bigr\rangle_{\overline{\mathcal{H}}^{Q}}\right)=	\sum_{j\ge1}\omega_j\left(\sum_{i\ge1}
			C_{i,j}(t)\,
			\bigl\langle(\rho_{z_k}^\epsilon)^{\mathrm{per}},f_i\bigr\rangle_{\overline{\mathcal{H}}^{Q}}\right).
		\end{align*}
		Consequently, by Lemma~\ref{IAskca1} and \eqref{IAkslas8}, there exists a
		full-measure set on which
		\begin{align}\label{Ol45as}
			\mathcal{P}^{\epsilon}(\hat{\theta}_{t}(\omega),z_k)
			=
			\bigl(U_{-t}\mathcal{P}^{\epsilon}(\omega)\bigr)(z_k),
			\qquad \forall k\geq1.
		\end{align}
		Keeping in mind that $(z_j)_{j\geq1}$ is dense in
		$\mathbb{R}\times\mathbb{T}^d$, and using the continuity of
		$z\mapsto\mathcal{P}^{\epsilon}(\omega,z)$ and
		$z\mapsto\bigl(U_{-t}\mathcal{P}^{\epsilon}(\omega)\bigr)(z)$ on a set of full
		measure, together with \eqref{Ol45as}, we conclude that, on a set of full
		measure,
		\begin{align}\label{Ikamsw25}
			\mathcal{P}^{\epsilon}(\hat{\theta}_{t}(\omega),z)
			=
			\bigl(U_{-t}\mathcal{P}^{\epsilon}(\omega)\bigr)(z),
			\qquad \forall z\in\mathbb{R}\times\mathbb{T}^{d}.
		\end{align}
		Consequently, together with \eqref{Baxczxx}, we conclude that on a set of full
		measure,
		\begin{align*}
			\mathcal{Z}^{\epsilon}(\hat{\theta}_{t}(\omega))
			=
			\Theta_{-t}\mathcal{Z}^{\epsilon}(\omega).
		\end{align*}
		\textbf{Step 2.} Thanks to Proposition~\ref{Ujasn785sw},
		
		$$
		\lim_{\epsilon\rightarrow 0}
		\mathbb{E}\Vert\mathcal{Z}^{\epsilon};\mathcal{Z}\Vert_{\mathcal{O}}=0.
		$$
		
		Therefore, by passing to a suitable sequence $(\epsilon_j)_{j\geq1}$ converging to zero, and using the continuity of
		$\Theta_{-t}:\mathcal{M}_0\to\mathcal{M}_0$ together with the identity established in Step~1, we conclude that, on a set of full measure, denoted by
		$\tilde{\Omega}_t$, we have
		\begin{align}\label{Olas[sd]}
			\mathcal{Z}(\hat{\theta}_{t}(\omega))
			=
			\Theta_{-t}\mathcal{Z}(\omega).
		\end{align}
		To prove the ergodicity of $\Theta_t$, let
		$A\in\mathcal{B}(\mathcal{M}_0)$ satisfy
		\[
		(\mathcal{Z})_{\#}\mathbb{P}\bigl(A\Delta\Theta_tA\bigr)=0.
		\]
		It remains to show that
		\[
		(\mathcal{Z})_{\#}\mathbb{P}(A)\in\{0,1\}.
		\]
		To this end, note that $\Theta_t\circ\Theta_{-t}=I$ and, on a set of full
		measure,
		\[
		\hat{\theta}_t\circ\hat{\theta}_{-t}(\omega)=\omega.
		\]
		Therefore, together with \eqref{Olas[sd]}, we obtain
		\begin{align}\label{IK478sd}
			\begin{split}
				&(\mathcal{Z})_{\#}\mathbb{P}\bigl(A\Delta\Theta_tA\bigr)
				=
				\mathbb{P}\!\left(
				\mathcal{Z}^{-1}(A)
				\Delta
				\mathcal{Z}^{-1}(\Theta_tA)
				\right) \\
				&=
				\mathbb{P}\!\left(
				\mathcal{Z}^{-1}(A)
				\Delta
				(\hat{\theta}_t)^{-1}\!\left(\mathcal{Z}^{-1}(A)\right)
				\right)
				=
				\mathbb{P}\!\left(
				\mathcal{Z}^{-1}(A)
				\Delta
				(\hat{\theta}_{-t})\!\left(\mathcal{Z}^{-1}(A)\right)
				\right)
				=0.
			\end{split}
		\end{align}
		From Proposition~\ref{Ik78assq}, we know that
		$\hat{\theta}_{-t}:\Omega\to\Omega$ is ergodic. Hence, from
		\eqref{IK478sd}, we have
		\[
		(\mathcal{Z})_{\#}\mathbb{P}(A)
		=
		\mathbb{P}\bigl(\mathcal{Z}^{-1}(A)\bigr)
		\in\{0,1\}.
		\]
		This finishes the proof.
	\end{proof}
	We can now prove that the solution to \eqref{MAINAA} generates a natural
	cocycle. This follows directly from the previous results.
	
	\begin{theorem}\label{ASA8eweaas}
		Consider the setting of Theorem~\ref{Ujasm}. Then, for each $s\in\R$, there exists a global solution
		$\RR{\mathcal{U}}(s,u_0,\mathcal{Z})$ to the following equation:
		\begin{align}
			\begin{cases}
				(\partial_t-\Delta)\RR{\mathcal{U}}
				=
				\mu\RR{\mathcal{U}}
				+
				\Sigma(\RR{\mathcal{U}})\tilde{\star}\rb,
				& t>s,\\
				\mathcal{R}(\RR{\mathcal{U}})(s,\cdot)=u_0,
				& t=s.
			\end{cases}
		\end{align}
		with respect to the model $\mathcal{Z}$. For $t\geq s$, 
		let
		\[
		\tilde{\phi}(s,t,\mathcal{Z},u_0)\in L^{\infty}(\mathbb{T}^d)
		\]
		be given by
		\begin{align}\label{OLmasw}
			\tilde{\phi}(s,t,\mathcal{Z},u_0)(x)
			= \mathcal{R}\Big(\mathbf{R}^+_{s}\RR{\mathcal{U}}(s,u_0,\mathcal{Z})\Big)(t,x)
		\end{align}
		and define
		\[
		\overline{\phi}(t,\mathcal{Z},u_0)
		:=
		\tilde{\phi}(0,t,\mathcal{Z},u_0).
		\]
		Then the following statements hold:
		\begin{enumerate}
			\item For $s_1<s_2<s_3$,
			\begin{align}\label{ASk85}
				\tilde{\phi}\left(s_1,s_3,\mathcal{Z},u_0\right)
				=
				\tilde{\phi}\left(s_2,s_3,\mathcal{Z},
				\tilde{\phi}\left(s_1,s_2,\mathcal{Z},u_0\right)\right).
			\end{align}
			
			\item For $s_1,s_2>0$, we have
			\begin{align}\label{Asaiqwq}
				\overline{\phi}(s_1+s_2,\mathcal{Z},u_0)
				=
				\overline{\phi}\left(s_1,\Theta_{s_2}\mathcal{Z},
				\overline{\phi}(s_2,\mathcal{Z},u_0)\right).
			\end{align}
			In particular, $\overline{\phi}$ satisfies the cocycle property over
			\begin{align}\label{AIknas}
				(\mathcal{M}_0,\mathcal{B}(\mathcal{M}_0),
				(\mathcal{Z})_{\#}\mathbb{P},
				\{\Theta_t\}_{t\in\mathbb{R}}).
			\end{align}
		\end{enumerate}
	\end{theorem}
	\begin{proof}
		Recall that in Assumption~\ref{ASAS78asaaa}, we assume the existence of a
		global solution. The identity \eqref{ASk85} follows by an argument similar
		to the one used in \hyperref[Step10]{\textbf{Step 1}} of
		Proposition~\ref{onwall}.
		Moreover, by \eqref{OLmasw} and Definition~\ref{dfn}, one can easily
		verify that
		\begin{align}\label{Lm78a}
			\tilde{\phi}\left(s_2,s_1+s_2,\mathcal{Z},u_1\right)
			=
			\overline{\phi}\left(s_1,\Theta_{s_2}\mathcal{Z},u_1\right),
		\end{align}
		for every $u_1\in L^{\infty}(\mathbb{T}^d)$.
		Combining \eqref{ASk85} and \eqref{Lm78a}, we obtain \eqref{Asaiqwq},
		and hence the cocycle property follows.
	\end{proof}
	In the previous result, we showed that one can associate a cocycle to the solution of the SPDE~\eqref{MAIN}. However, a technical problem remains. The natural Banach space in which the cocycle is defined is $L^\infty(\mathbb T^d)$, which is not separable. For our main application, it is crucial that the cocycle be defined on a separable Banach space. This is mainly due to the measurability issues that arise in the absence of separability and that prevent us from applying the multiplicative ergodic theorem.
	Fortunately, in the present setting, this issue can easily be resolved by restricting the cocycle to  $C(\mathbb T^d)$.
	
	\begin{theorem}\label{ASAsewe}
		Consider the setting of Theorem~\ref{ASA8eweaas}. Then
		\begin{align*}
			\phi:[0,\infty)\times \mathcal{M}_0\times C(\mathbb T^d)
			\to C(\mathbb T^d),
			\qquad
			\phi(t,\mathcal{Z},u_0)
			:=\overline{\phi}(t,\mathcal{Z},u_0),
		\end{align*}
		is well-defined and is a continuous cocycle. Additionally, once
		$\Sigma\in C^{n}(\mathbb R,\mathbb R)$ for $n\geq3$, this cocycle is of
		class $C^{n-2}$.
	\end{theorem}
	
	\begin{proof}
		The proof is based on Theorem~\ref{ASA8eweaas} and Lemma~\ref{YHATSs}. The only part is that we must prove that for each $t>0$ and $\mathcal{Z}\in \mathcal{M}_0$ we have
		\begin{align*}
			\overline{\phi}(t,\mathcal{Z},C(\mathbb T^d))
			\subseteq C(\mathbb T^d).
		\end{align*}
		We show that, in fact, we have continuously
		\begin{align*}
			\overline{\phi}(t,\mathcal{Z},C(\mathbb T^d))
			\subseteq C^{1-\kappa}(\mathbb T^d).
		\end{align*}
		To this end, recall that
		\begin{align}
			L^{\infty}(\mathbb T^d)&\to \mathcal{D}^{\gamma_1,0}_{P_0,0},\ \quad u_{0}\mapsto \mathbf{R}^+_{0}\RR{\mathcal{U}}(0,u_0,\mathcal{Z})
		\end{align}
		is continuous. 
		Then, by the definition of $\triplenorm{\mathbf{R}^+_{0}\RR{\mathcal{U}}(0,u_0,\mathcal{Z})}_{\gamma_1,0,O_{0}^t}^{(0)}$,
		the claim follows.
	\end{proof}
	\begin{remark}\label{ASyachadw}
		As already mentioned, our choice of noise is mainly motivated by simplicity
		and by the desire to remain close to the classical setting. The same
		framework can, however, be adapted to other classes of Gaussian noises. In
		particular, from the perspective of random dynamical systems, one can
		consider the Gaussian process
		\[
		W^{H,Q}=\{W^{H,Q}(f):f\in\mathcal{W}^{H,Q}\},
		\]
		where $\mathcal{W}^{H,Q}$ is the completion of
		$C_c^\infty(\mathbb{R}\times\mathbb{T}^d)$ with respect to the inner product
		\[
		\langle f,g\rangle_{\mathcal{W}^{H,Q}}
		:=c_H
		\int_{\mathbb{R}}
		\sum_{k\in\mathbb{Z}^d}
		|\tau|^{1-2H}q_k\,
		\widehat f(\tau,k)
		\overline{\widehat g(\tau,k)}
		\,d\tau,
		\]
		where $H\in (0,1)$.
		The corresponding noise satisfies
		\[
		\xi^{H,Q}\in
		C^{2H-2+\beta-\frac d2-\varepsilon}_{(2,1)}
		\qquad\text{for every }\varepsilon>0.
		\]
		The arguments developed in this section can be adapted to this setting,
		with the corresponding modifications of the regularity estimates. This class of noises is particularly interesting in the present context due to
		its temporal dependence structure and, in particular, its temporal stationarity. To satisfy the regularity assumption in \eqref{asasaiasa}, we
		require
		\begin{align}
			2H-2+\beta-\frac{d}{2}>-1-\kappa.
		\end{align}
		In particular, for fixed $\beta$, $d$, and $\kappa$, this condition
		restricts how far $H$ can be below $\frac12$. It should also be possible to consider noises with fractional spatial
		regularity.
	\end{remark}

	\section{Main results}\label{TNT}
	We now have all the ingredients needed to establish our main results. Following the standard approach in dynamical systems, our goal is to describe the local dynamics of the cocycle $\varphi$ in a neighborhood of an equilibrium. In particular, we prove the existence of three families of invariant manifolds that locally characterize the dynamical behavior of the solution.
	We first collect the assumptions under which these results hold.
	\begin{assumption}\label{zasdasdaz}
		Throughout this section, we assume the following.
		\begin{enumerate}
			\item We assume that Assumption~\ref{ASAS78asaaa} holds.
			
			\item For the SPDE
			\begin{align}\label{MAINAAAA}
				\begin{cases}
					(\partial_t-\Delta)u=\mu u+\Sigma(u)\,\xi^Q,
					& t>0,\\
					u(0,.)=u_0\in L^\infty(\mathbb T^d),
					& t=0,
				\end{cases}
			\end{align}
			we assume that $\mu\leq0$ and
			$\Sigma(\cdot)\in C^{3+\nu}_{b}(\mathbb{R},\mathbb{R})$ for some
			$\nu\in(0,1)$. The solution theory for this equation is based on the
			random model ${\mathcal Z}$ specified in
			Theorem~\ref{ASA8eweaas}.
			
			\item For simplicity, and by an abuse of notation, we set
			\[
			\left({\Omega},{\mathcal{F}},{\P},\{{\theta}_t\}_{t\in\mathbb{R}}\right)
			=
			\left(
			\mathcal{M}_0,
			\mathcal{B}(\mathcal{M}_0),
			(\mathcal{Z})_{\#}\mathbb{P},
			\{\Theta_t\}_{t\in\mathbb{R}}
			\right).
			\]
			We also set
			\begin{align*}
				\varphi \colon [0,\infty)\times{\Omega}\times C(\T^d)
				\rightarrow C(\T^d),
			\end{align*}
			where
			\begin{align*}
				\varphi^{t}_{\omega}(u):=\phi(t,\omega,u),
			\end{align*}
			as defined in Theorem~\ref{ASAsewe}. This map is a continuous
			cocycle and is of class $C^{1}$.
		\end{enumerate}
	\end{assumption}
	As already mentioned, our goal is to describe the dynamics around the
	equilibrium. In the theory of random dynamical systems, an equilibrium is
	formulated in terms of a random point that is invariant under the flow. We
	refer to such a random point as a stationary point.
	
	\begin{definition}
		A random point $Y:\Omega\to C(\T^d)$ is called a \emph{stationary point} for
		the cocycle $\varphi$ if
		\begin{enumerate}
			\item $Y$ is measurable and
			\begin{align}\label{AS89asdxa}
				\mathbb{E}\big[\|Y_{\omega}\|_{C(\T^d)}^p\big]<\infty
				\qquad\text{for all }p\geq1;
			\end{align}
			\item for every $t\geq0$ and $\omega\in\Omega$,
			\[
			\varphi^t_\omega(Y_\omega)=Y_{\theta_t\omega}.
			\]
		\end{enumerate}
	\end{definition}
	
	\begin{remark}
		If $\Sigma(0)=0$, then the zero function is a stationary point. More generally,
		the existence of a stationary point is expected when $\mu<0$ is sufficiently
		large in absolute value. We also note that, when the system admits an invariant
		measure, one can construct a new random dynamical system in which the
		corresponding stationary random point is viewed as the zero equilibrium; see,
		for example, \cite[Lemma~7.2.1]{Arn98}.
	\end{remark}
	\begin{remark}\label{sds78asdaad}
		For each $t_{0}>0$,	recall that $\theta_{t_0}:\Omega\to\Omega$ is ergodic. Consider the random variable $X:\Omega\to[0,\infty]$ satisfying $	\mathbb{E}[X]<\infty.$
		Set
		$$
		Y:=X\circ\theta_{t_0}-X.
		$$
		Since $\theta_{t_0}$ is measure-preserving and $X\in L^1$, we have $Y\in L^1$ and
		$$
		\mathbb{E}[Y]
		=
		\mathbb{E}[X\circ\theta_{t_0}]-\mathbb{E}[X]
		=0.
		$$
		Hence, by the Birkhoff ergodic theorem,
		$$
		\frac{1}{m}\sum_{k=0}^{m-1}Y(\theta_{kt_0}\omega)
		\longrightarrow 0
		\qquad\text{almost surely}.
		$$
		On the other hand,
		$$
		\frac{1}{m}\sum_{k=0}^{m-1}Y(\theta_{kt_0}\omega)
		=
		\frac{X(\theta_{mt_0}\omega)-X(\omega)}{m}.
		$$
		Since $X(\omega)<\infty$ almost surely, we conclude that
		$$
		\frac{X(\theta_{mt_0}\omega)}{m}
		\longrightarrow0
		\qquad\text{almost surely}.
		$$
	\end{remark}
	
	\begin{remark}
		As we will see throughout this section, we apply the estimates established in
		Section~\ref{GPAM}, together with the construction introduced in the previous
		section, to derive the main results. Since the estimates in
		Section~\ref{GPAM} are deterministic, we emphasize that our results can be
		applied to different types of noise, such as the one introduced in
		Remark~\ref{ASyachadw}. This also allows us to study stationary points rather
		than being restricted to deterministic equilibria.
	\end{remark}
	We now make the following additional assumption.
	
	\begin{assumption}
		In addition to Assumption~\ref{zasdasdaz}, we assume that
		$Y:\Omega\to C(\T^d)$ is a stationary point for the cocycle $\varphi$.
		We define
		\begin{align*}
			\psi \colon [0,\infty)\times\Omega\times C(\T^d)
			\rightarrow C(\T^d),
			\qquad
			\psi^{t}_{\omega}(u)
			:=
			D_{Y_\omega}\varphi^{t}_{\omega}(u).
		\end{align*}
		That is, $\psi$ is the linearization of $\varphi$ around the stationary
		point $Y$. Moreover, $\psi$ is a linear and continuous cocycle.
	\end{assumption}
	The following lemma shows that $\psi$ is a compact operator. This is particularly
	important, as it allows us to apply the multiplicative ergodic theorem later.
	\begin{lemma}\label{Asqwqw}
		For each $t> 0$ and $\omega\in\Omega$, the following map is compact:
		\begin{align*}
			C(\mathbb{T}^d)\to C(\mathbb{T}^d), \qquad
			u\mapsto \psi^{t}_{\omega}(u).
		\end{align*}
	\end{lemma}
	
	\begin{proof}
		By inspecting the proof of Theorem~\ref{ASAsewe}, we can easily see that the following map is continuous:
		\begin{align*}
			C(\mathbb{T}^d)\to C^{1-\kappa}(\mathbb{T}^d), \qquad
			u\mapsto \psi^{t}_{\omega}(u).
		\end{align*}
		The proof follows by noting that the Hölder space $C^{1-\kappa}(\mathbb T^d)$ embeds compactly into $C(\mathbb \T^d)$.
		
	\end{proof}
	\begin{remark}
		Note that in Propositions \ref{kkkl}, \ref{onwall}, and Proposition \ref{UJUJUJU}, we can clearly obtain the same results using the space $C(\T^d)$ instead of $L^{\infty}(\T^d)$. In the sequel, we use this fact without explicitly mentioning it.
	\end{remark}
	With these preliminaries in place, we now present the main strategy for establishing the existence of invariant manifolds. A direct approach, such as the Lyapunov--Perron method, is often lengthy and technically demanding. However, thanks to the abstract results in \cite{GVR23A} and \cite{GVR25C}, it is sufficient to verify a number of hypotheses, which follow readily from the results established in the previous sections. We first summarize the main strategy.
	
	\begin{enumerate}
		\item Rather than studying the nonlinear cocycle $\varphi$ directly, we first analyze its linearization $\psi$. Since $\psi$ is linear, its asymptotic behavior is considerably easier to describe. The main tool is the Multiplicative Ergodic Theorem. A key technical requirement for the validity of this theorem is
		\begin{align}\label{asdeerwdc}
			\sup_{t\in[0,1]}
			\log^{+}\big(
			\Vert\psi^{t}_{\omega}\Vert_{\mathcal{L}(C(\T^d),C(\T^d))}
			\big)
			\in L^{1}(\Omega),
		\end{align}
		where by $\mathcal{L}(C(\T^d),C(\T^d))$ we mean the usual operator space.
		This assumption is verified using the estimate established in Corollary~\ref{UJU}, together with the a priori bound for the solution. The required compactness follows from Lemma~\ref{Asqwqw}.
		\item The Multiplicative Ergodic Theorem yields a finite or countable family of deterministic numbers, called the \emph{Lyapunov exponents}; see item~(1) of Theorem~\ref{METT}.
		
		\item Corresponding to each Lyapunov exponent, there exists an invariant linear subspace, called the \emph{Oseledets subspace}, for the cocycle $\psi$; see item~(2) of Theorem~\ref{METT}. These subspaces are naturally classified into three families: the unstable subspaces, corresponding to positive Lyapunov exponents; the center subspace, corresponding to the zero Lyapunov exponent; and the stable subspaces, corresponding to negative Lyapunov exponents; see Definition~\ref{ASAqvfdg}.
		
		\item The dynamics of $\psi$ on these subspaces are fundamentally different. On the stable subspaces, trajectories converge exponentially fast to zero. On the unstable subspaces, trajectories grow exponentially fast away from zero. In contrast, the dynamics on the center subspace are neither exponentially contracting nor exponentially expanding. Consequently, the center subspace is typically where bifurcations may occur.
		\item Finally, we use these invariant subspaces of the linear cocycle $\psi$ as tangent spaces to construct invariant manifolds for the nonlinear cocycle $\varphi$. More precisely, we show that each family of Oseledets subspaces gives rise to a corresponding family of random invariant manifolds for $\varphi$. Furthermore, the dynamics on these invariant manifolds inherit the qualitative behavior of the corresponding invariant subspaces of the linearized cocycle. The main technical requirement is the \emph{first-order approximation condition}; see Definition~\ref{first-order-condition}. Once this condition is verified, the existence of the three types of invariant manifolds follows directly, provided that the Lyapunov exponents satisfy the corresponding sign conditions. This provides a unified approach and avoids treating each type of invariant manifold separately. The main ingredient in this step is the estimate established in Proposition~\ref{UJUJUJU}.
	\end{enumerate}
	\begin{definition}\label{first-order-condition}
		We say that the first-order approximation condition holds if there exists a
		random variable $\rho:\Omega\to\mathbb{R}_+$ such that
		\[
		\liminf_{t\to\infty}
		\frac{1}{t}\log\rho(\theta_t\omega)\geq 0
		\qquad\text{almost surely} \]
		and such that, for almost every $\omega\in\Omega$ and all
		$u_1,u_2\in C(\mathbb{T}^d)$ satisfying
		\[
		\|u_1\|_{C(\mathbb{T}^d)},\,
		\|u_2\|_{C(\mathbb{T}^d)}
		<\rho(\omega),
		\]
		the following conditions hold:
		\begin{itemize}
			\item
			\begin{align}\label{A78aqwz}
				\begin{split}
					&\sup_{s\in[0,1]}
					\left\|
					\varphi^s_\omega(u_2+Y_\omega)
					-\varphi^s_\omega(u_1+Y_\omega)
					-\psi^s_\omega[u_2-u_1]
					\right\|_{C(\mathbb{T}^d)}
					\\
					&\qquad\leq
					\|u_2-u_1\|_{C(\mathbb{T}^d)}
					\Big(
					\|u_1\|_{C(\mathbb{T}^d)}^r
					+
					\|u_2\|_{C(\mathbb{T}^d)}^r
					\Big)
					h\Big(
					\|u_1\|_{C(\mathbb{T}^d)}
					+
					\|u_2\|_{C(\mathbb{T}^d)}
					\Big)
					R(\omega);
				\end{split}
			\end{align}
			
			\item $r\in(0,1]$;
			
			\item $h:\mathbb{R}_+\to\mathbb{R}_+$ is an increasing $C^1$ function;
			
			\item $R:\Omega\to\mathbb{R}_+$ is measurable and
			\begin{align}\label{RTRTR}
				\lim_{t\to\infty}
				\frac{1}{t}\log^+R(\theta_t\omega)=0
				\qquad\text{a.s.}
			\end{align}
		\end{itemize}
	\end{definition}
	\begin{remark}\label{IOpasasasas}
		Thanks to Remark \ref{sds78asdaad}, a practical way to check \eqref{RTRTR} is to show that
		\begin{align*}
			\log^+\big(\sup_{s\in [0,1]}R(\theta_s\omega)\big)\in L^{1}(\Omega).
		\end{align*}
	\end{remark}
	\begin{remark}
		The results of Sections~\ref{GPAM} and~\ref{PRO} play complementary roles in this section. The results of Section~\ref{GPAM} provide the technical estimates, while those of Section~\ref{PRO} provide the framework needed to apply these estimates in order to study the dynamics of the SPDE~\eqref{MAIN}.
	\end{remark}
	The following result  provides an integrable bound for the
	solution of the SPDE~\eqref{MAINAAAA}.
	
	\begin{lemma}\label{ASasd78axx}
		There exist a constant $C_\mu>0$ and a polynomial ${P_0}$ such that
		\begin{align}\label{A75qwewc}
			\sup_{s\in[0,1]}
			\|\varphi^s_\omega(u)\|_{C(\T^d)}
			\leq
			C_\mu
			\max\left\{
			\|u\|_{C(\T^d)},
			P_0\left(
			\|\Pi\|_{\overline{O_{2}^{-1}}},
			[\ \gI\ ]_{\overline{O_{2}^{-1}}}
			\right)
			\right\}.
		\end{align}
		In particular, for every random point $u_\omega$ such that
		$\|u_\omega\|_{C(\T^d)}$ has finite moments of all orders, we have
		\begin{align}\label{Asddadvse}
			\mathbb{E}\big[
			(
			\sup_{t\in[0,1]}
			\|\varphi^t_\omega(u_\omega)\|_{C(\T^d)}
			)^p
			\big]
			<\infty,
			\qquad \forall p\geq1.
		\end{align}
	\end{lemma}
	\begin{proof}
		For \eqref{A75qwewc}, see \cite[Theorems~1 and~3]{CFW26}, valid for
		$\kappa<\bar\kappa\approx0.1321$. An analogous
		a priori estimate in the full subcritical regime is established in
		\cite[Theorem~1.1]{ES26}.
		For the second part, Proposition~\ref{Ujasn785sw} yields, for every
		$p\geq1$,
		\[
		\sup_{\epsilon\in(0,1)}
		\E\Big[
		\Vert\mathcal{Z}^{\epsilon}\Vert_{\overline{O_{2}^{-1}}}^{p}
		\Big]
		<\infty,
		\]
		and
		\begin{align}\label{sadsdsdfsfs}
			\lim_{\epsilon\rightarrow0}
			\E\Vert\mathcal Z^\epsilon;\mathcal Z
			\Vert_{\overline{O_{2}^{-1}}}
			=0.
		\end{align}
		Hence, by passing to a subsequence and using Fatou's lemma,
		$\mathcal Z$ has finite moments of every order. Together with
		\eqref{A75qwewc}, this yields \eqref{Asddadvse}.
	\end{proof}
	\subsection{Multiplicative Ergodic Theorem}
	We are now ready to address the first step of our strategy, namely, the application of the Multiplicative Ergodic Theorem (MET) to the linear cocycle $\psi$.
	\begin{theorem}[Multiplicative Ergodic Theorem]\label{METT}
		For $\lambda\in\mathbb{R}\cup\{-\infty\}$, define
		\[
		F_\lambda(\omega)
		:=
		\left\{
		\xi \in C(\mathbb{T}^d):
		\limsup_{t\to\infty}
		\frac1t
		\log\|\psi_\omega^t (\xi)\|_{C(\mathbb{T}^d)}
		\le\lambda
		\right\}.
		\]
		
		Then there exists a $\theta_t$-invariant subset of $\Omega$ of full measure for all $t\in\mathbb{R}$, again denoted by $\Omega$, on which the following assertions hold:
		\begin{enumerate}
			
			\item There exists a decreasing sequence of Lyapunov exponents
			\[
			\lambda_1>\lambda_2>\cdots,
			\]
			or a finite sequence followed by $-\infty$, satisfying
			\[
			\lim_{i\to\infty}\lambda_i=-\infty.
			\]
			
			\item Corresponding to every finite Lyapunov exponent $\lambda_i$, there exists a finite-dimensional invariant subspace
			\[
			H_\omega^i\subset C(\mathbb{T}^d),
			\]
			called the \emph{Oseledets subspace}, such that
			\[
			\psi_\omega^t(H_\omega^i)
			=
			H_{\theta_t\omega}^i,
			\qquad
			t\ge0.
			\]
			
			\item The spaces $H_\omega^i$ determine the filtration
			\[
			F_{\lambda_1}(\omega)=C(\mathbb{T}^d),
			\qquad
			F_{\lambda_i}(\omega)
			=
			H_\omega^i
			\oplus
			F_{\lambda_{i+1}}(\omega),
			\]
			and hence
			\[
			C(\mathbb{T}^d)
			=
			\bigoplus_{j=1}^iH_\omega^j
			\oplus
			F_{\lambda_{i+1}}(\omega)
			\]
			for every $i\ge1$.
			
			\item If $h_\omega\in H_\omega^i\setminus \lbrace 0\rbrace$, then
			\[
			\lim_{t\to\infty}
			\frac1t
			\log
			\|\psi_\omega^t (h_\omega)\|_{C(\mathbb{T}^d)}
			=
			\lambda_i,
			\]
			and
			\[
			\lim_{t\to\infty}
			\frac1t
			\log
			\big\|
			(\psi_{\theta_{-t}\omega}^t|_{H^{i}_{\theta_{-t}\omega}})^{-1}
			(h_\omega)
			\big\|_{C(\mathbb{T}^d)}
			=
			-\lambda_i.
			\]
		\end{enumerate}
	\end{theorem}
	\begin{proof}
		The goal is to apply the multiplicative ergodic theorem for Banach spaces
		stated in \cite[Theorem~1.21]{GVR23A}. We first fix an arbitrary time step
		$t_{0}\in (0,1]$ and obtain the result for the discrete cocycle $(\psi^{nt_0}_{\omega})_{n\geq 0}$. By Lemma~\ref{Asqwqw}, the operator
		$\psi^{t_0}_{\omega}$ is compact for every $\omega\in\Omega$. It remains to
		verify the integrability condition
		\begin{align}\label{ASaickasda}
			\log^{+}\Big(
			\sup_{s\in[0,t_0]}	\|\psi^{s}_{\omega}\|_{\mathcal{L}(C(\T^d),C(\T^d))}
			\Big)
			\in L^{1}(\Omega).
		\end{align}
		To this end, recall that $\psi^{t}_{\omega}=D_{Y_{\omega}}\varphi^{t}_{\omega}$. Thanks to Corollary \ref{UJU}, for each $\tilde{u}\in C(\T^d)$, we have
		\begin{align}\label{TRT1}
			\begin{split}
				&	\sup_{s\in[0,t_0]}\Vert \psi^{s}_{\omega}[\tilde{u}]\Vert\leq \triplenorm{\mathbf{R}^+_0 D_{u_1}\RR{\mathcal{J}_1}[\tilde{u}] }_{\gamma_1,\,0,\,O_{1}}^{(0)}
				\\
				&\qquad\leq
				\mathcal{E}_1\Big(
				\|\mathcal{Z}\|_{\overline{O_{2}^{-1}}},
				[\ \gI\ ]_{\overline{O_{1}^0}},
				\|\Gamma\|_{\overline{O_{1}^0}},
				\sup_{z\in O_{1}^0}
				\|\mathcal{U}^{\gY}_{Y_\omega}(z)\|,
				\|Y_\omega\|_{C(\mathbb{T}^d)}
				\Big)
				\Vert \tilde{u} \Vert_{C(\T^d)},
			\end{split}
		\end{align}
		Note that from \eqref{A75qwewc},
		\begin{align}\label{TRT2}
			\begin{split}
				\sup_{z\in O_{1}^0}
				\|\mathcal{U}^{\gY}_{Y_\omega}(z)\|&\leq \sup_{s\in [0,1]}\Vert\varphi^{s}_{\omega}(Y_\omega)\Vert_{C(\T^d)}=\sup_{s\in [0,1]}\Vert Y_{\theta_s\omega}\Vert_{C(\T^d)}\\&\leq C_\mu
				\max\left\{
				\|Y_\omega\|_{C(\T^d)},
				P_0\left(
				\|\Pi\|_{\overline{O_{1}^0}},
				[\ \gI\ ]_{\overline{O_{1}^0}}
				\right)
				\right\}\in\bigcap_{p\geq 1}L^{p}(\Omega).
			\end{split}
		\end{align}
		With an argument analogous to that in the proof of Lemma \ref{ASasd78axx}, we have
		\begin{align}\label{TRT3}
			\|\mathcal{Z}\|_{\overline{O_{2}^{-1}}},
			[\ \gI\ ]_{\overline{O_{1}^0}},
			\|\Gamma\|_{\overline{O_{1}^0}}\in\bigcap_{p\geq 1}L^{p}(\Omega)
		\end{align}
		Now thanks to the definition of the function $\mathcal{E}_1$ in Corollary \ref{UJU} and together with \eqref{TRT1}-\eqref{TRT3}, one can indeed derive \eqref{ASaickasda}. Thus, we verify all the criteria needed for the requirement of \cite[Theorem~1.21]{GVR23A}. And so the multiplicative ergodic theorem holds for the discrete cocycle $(\psi^{nt_0}_{\omega})_{n\geq 0}$. To obtain the continuous version, it is sufficient to check that by doing an analogous argument,
		\begin{align*}
			\log^{+}\Big(
			\sup_{s\in[0,t_0]}	\|\psi^{t_0-s}_{\theta_s\omega}\|_{\mathcal{L}(C(\T^d),C(\T^d))}
			\Big)
			\in L^{1}(\Omega).
		\end{align*}
		Then, by a standard adaptation of the arguments in \cite[Theorem 3.3 and Lemma 3.4]{LL10}, one can obtain the continuous-time version of the multiplicative ergodic theorem from the discrete case.
	\end{proof}
	\begin{definition}\label{ASAqvfdg}
		Consider the setting of Theorem~\ref{METT}. We set, whenever the corresponding spaces are well defined,
		\begin{align}\label{linn}
			S_{\omega}
			:= F_{\lambda^-}(\omega),\quad
			U_{\omega}
			:= \bigoplus_{i:\lambda_i>0}H_{\omega}^{i},\quad
			C_{\omega}
			&:= H^{i_c}_{\omega},
		\end{align}
		where 	$$
		\lambda^- := \sup\{\lambda_i:\lambda_i<0\},
		\qquad
		\lambda_{i_c}=0.
		$$
	\end{definition}
	
\begin{remark}
	Let
	\[
	C^{1-\kappa-\epsilon}_0(\mathbb T^d)
	:=
	\overline{C^\infty(\mathbb T^d)}^{\,C^{1-\kappa-\epsilon}}
	\]
	denote the little Hölder space, which is separable. By
	Lemma~\ref{Asqwqw}, the cocycle \(\psi\) also acts on
	\(C^{1-\kappa-\epsilon}_0(\mathbb T^d)\), and hence the Multiplicative
	Ergodic Theorem can be applied on this space. Moreover, the Lyapunov
	exponents and, under the natural embedding into \(C(\mathbb T^d)\), the
	corresponding Oseledets subspaces coincide with those obtained on
	\(C(\mathbb T^d)\); see~\cite[Theorems~5.17 and~5.18]{BGS2025}.
\end{remark}
	\begin{remark}
		There is a rich literature on the Multiplicative Ergodic Theorem; see, e.g., \cite{LL10,GTQ15,Blu16}. The formulation developed in \cite{GVRS22,GVR23A}, which we use in this manuscript, is particularly flexible and well suited to applications, especially when verifying the assumptions of the theorem. 
	\end{remark}
	\subsection{Invariant manifolds}
	We are now in a position to establish the existence of invariant manifolds (stable, unstable, and center) around the stationary point. Having the multiplicative ergodic theorem at hand, the only remaining technical condition that we need is the first-order approximation condition, which is the topic of the next proposition.
	\begin{proposition}\label{FIRTA}
		There exists an increasing $C^1$ function $h:\R_{+}\to \R_{+}$ and a measurable function $R:\Omega\to\mathbb{R}_+$ such that, for each $u_1,u_2\in C(\mathbb{T}^d)$,
		\begin{align*}
			\Vert u_i\Vert_{C(\T^d)}\leq 1 , \ \  i=1,2 ,
		\end{align*}
		we have
		\begin{align}
			\begin{split}
				&\sup_{s\in[0,1]}
				\left\|
				\varphi^s_\omega(u_2+Y_\omega)
				-\varphi^s_\omega(u_1+Y_\omega)
				-\psi^s_\omega[u_2-u_1]
				\right\|_{C(\mathbb{T}^d)}
				\\
				&\qquad\leq
				\|u_2-u_1\|_{C(\mathbb{T}^d)}
				\Big(
				\|u_1\|_{C(\mathbb{T}^d)}^\nu
				+
				\|u_2\|_{C(\mathbb{T}^d)}^\nu
				\Big)
				h\Big(
				\|u_1\|_{C(\mathbb{T}^d)}
				+
				\|u_2\|_{C(\mathbb{T}^d)}
				\Big)
				R(\omega);
			\end{split}
		\end{align}
		Also,
		\begin{align*}
			\lim_{t\to\infty}
			\frac{1}{t}\log^+R(\theta_t\omega)=0
			\qquad\text{a.s.}
		\end{align*}
		In particular, the first-order approximation condition holds.
	\end{proposition}
	\begin{proof}
		For $\tau \in [0,1]$, let
		\begin{align*}
			L^{\tau}_{Y_\omega}(u_1,u_2):=u_1+Y_\omega+\tau (u_2-u_1),
		\end{align*}
		then we have
		\begin{align*}
			\varphi^s_\omega(u_2+Y_\omega)
			-\varphi^s_\omega(u_1+Y_\omega)
			-\psi^s_\omega(u_2-u_1)=\int_{0}^{1}\big[D_{L^{\tau}_{Y_\omega}(u_1,u_2)}\varphi^{s}_{\omega}-D_{Y_{\omega}}\varphi_{\omega}^s\big][u_2-u_1]\mathrm{d}\tau.
		\end{align*}
		Thanks to Proposition \ref{UJUJUJU} 
		\begin{align}\label{I7895454s}
			\begin{split}
				&\sup_{\tau\in [0,1]}\sup_{s\in [0,1]} \big\Vert\big[D_{u_1+Y_\omega+\tau (u_2-u_1)}\varphi^{s}_{\omega}-D_{Y_{\omega}}\varphi^s_\omega\big][u_2-u_1]\big\Vert_{C(\T^d)}\leq \\&\sup_{\tau\in [0,1]}\triplenorm{D_{L^{\tau}_{Y_\omega}(u_1,u_2)}\RR{\mathcal{J}_1}\big[u_2-u_1\big]
					-
					D_{Y_\omega}\RR{\mathcal{J}_1}\big[u_2-u_1\big]
				}_{\gamma_1,0,O_{1}}^{(0)} \\&\lesssim
				\sup_{\tau \in [0,1]}\max\big\{\Vert u_1+\tau (u_2-u_1)\Vert_{C(\T^d)}, \Vert  u_1+\tau (u_2-u_1)\Vert_{C(\T^d)}^{\nu}\big\} \Vert u_2-u_1\Vert_{C(\T^d)}
				\\& \times\sup_{\tau \in [0,1]}\mathcal{E}_3\Big(
				\|\mathcal{Z}\|_{\overline{O_{2}^{-1}}},
				[\ \gI\ ]_{\overline{O_{1}^0}},
				\|\Gamma\|_{\overline{O_{1}^0}},\sup_{\substack{r\in[0,1]\\w\in O_{1}^0}}
				\left\|
				\mathcal{U}^{\gY}_{Y_{\omega}+ru_1+r\tau(u_2-u_1)}(w)
				\right\|,
				\|Y_{\omega}\|_{C(\mathbb{T}^d)},
				\|L^{\tau}_{Y_\omega}(u_1,u_2)\|_{C(\mathbb{T}^d)}
				\Big)\\&\lesssim \big(\Vert u_1\Vert^{\nu}_{C(\T^d)}+\Vert u_2\Vert^{\nu}_{C(\T^d)}\big)\Vert u_2-u_1\Vert_{C(\T^d)}\big(1+\Vert u_1\Vert_{C(\T^d)}+\Vert u_2\Vert_{C(\T^d)}\big)\\& \times\sup_{\tau \in [0,1]}\mathcal{E}_3\Big(
				\|\mathcal{Z}\|_{\overline{O_{2}^{-1}}},
				[\ \gI\ ]_{\overline{O_{1}^0}},
				\|\Gamma\|_{\overline{O_{1}^0}},\sup_{\substack{r\in[0,1]\\w\in O_{1}^0}}
				\left\|
				\mathcal{U}^{\gY}_{Y_{\omega}+ru_1+r\tau(u_2-u_1)}(w)
				\right\|,
				\|Y_{\omega}\|_{C(\mathbb{T}^d)},
				\|L^{\tau}_{Y_\omega}(u_1,u_2)\|_{C(\mathbb{T}^d)}
				\Big)
			\end{split}
		\end{align}
		Note that from \eqref{A75qwewc}, for a polynomial $P_0$
		\begin{align}\label{TRT2aaa}
			\begin{split}
				&	\sup_{\substack{r,\tau\in[0,1]\\ z\in O_1^0}}
				\left\|
				\mathcal{U}^{\gY}_{Y_\omega+r u_1+r\tau(u_2-u_1)}(z)
				\right\|\leq 	\sup_{\substack{r,\tau\in[0,1]\\ s\in[0,1]}}
				\left\|
				\varphi_\omega^s
				\big(
				Y_\omega+r u_1+r\tau(u_2-u_1)
				\big)
				\right\|_{C(\T^d)}\\&\lesssim 
				\max\bigg\{
				\|Y_\omega\|_{C(\T^d)}+\|u_1\|_{C(\T^d)}+\|u_2\|_{C(\T^d)},
				P_0\left(
				\|\Pi\|_{\overline{O_{1}^0}},
				[\ \gI\ ]_{\overline{O_{1}^0}}
				\right)
				\bigg\}
			\end{split}
		\end{align}
		Recall that for a polynomial $Q_3$ which is increasing in its argument, we have
		\begin{align*}
			\mathcal{E}_3(b,c,d,f,g,h)
			&:=
			\exp\Big(
			Q_3(b,c,d,f,g,h)
			\Big),
		\end{align*}
		Consequently, using \eqref{TRT2}, \eqref{TRT2aaa}, and \eqref{I7895454s}, together with some straightforward algebraic and tedious calculations, we can conclude that one can obtain an increasing $C^1$ function $h:\mathbb{R}_+\to\mathbb{R}_+$ and a polynomial $Q_4$ such that 
		\begin{align}
			\begin{split}
				&\sup_{s\in [0,1]} \big\Vert \varphi^s_\omega(u_2+Y_\omega)
				-\varphi^s_\omega(u_1+Y_\omega)
				-\psi^s_\omega(u_2-u_1)\big\Vert\\&\leq
				\|u_2-u_1\|_{C(\mathbb{T}^d)}\big(\Vert u_1\Vert^{\nu}_{C(\T^d)}+\Vert u_2\Vert^{\nu}_{C(\T^d)}\big) h\big(\Vert u_1\Vert_{C(\T^d)}+\Vert u_2\Vert_{C(\T^d)}\big)\\
				&\times \underbrace{\exp\bigg(Q_{4}\big(	\|\mathcal{Z}\|_{\overline{O_{2}^{-1}}},
					[\ \gI\ ]_{\overline{O_{1}^0}},
					\|\Gamma\|_{\overline{O_{1}^0}},\|Y_\omega\|_{C(\T^d)}\big)\bigg)}_{=R(\omega)}
			\end{split}
		\end{align}
		The final step is to verify \eqref{RTRTR}. By \eqref{AS89asdxa}, \eqref{TRT2}, and the same argument as in \eqref{TRT3}, we have
		$$
		\log^+\Big(\sup_{s\in[0,1]}R(\theta_s\omega)\Big)\in L^1(\Omega).
		$$
		Hence, Remark~\ref{IOpasasasas} yields \eqref{RTRTR}, and therefore the first-order approximation condition holds.
	\end{proof}
	We begin now with the existence of the stable manifold.
	\begin{theorem}[Local stable manifolds]\label{stable_manifold}
		Consider the situation of Theorem \ref{METT}. Let $\lambda(s)\coloneqq \sup\{\lambda_j:\lambda_j<0\}$ and fix an arbitrary time step $t_0\in (0,1]$.
		Then, for every
		\[
		0<\upsilon<-\lambda(s),
		\]
	there exists a family of immersed submanifolds
		\[
		S^\upsilon_{\mathrm{loc}}(\omega)\subset C(\mathbb{T}^d)
		\]
		and a set $\widetilde{\Omega}\subset\Omega$ of full measure which is $\theta_{t_0}$-invariant and such that the following properties hold:
		\begin{enumerate}
			
			\item There exist random variables 
			\[
			\rho_1^\upsilon(\omega),\rho_2^\upsilon(\omega),
			\]
			which are positive and finite on $\widetilde{\Omega}$, satisfying
			\begin{align}\label{eqn:rho_temp}
				\liminf_{p\to\infty}\frac1p
				\log\rho_i^\upsilon(\theta_{pt_0}\omega)\geq0,
				\qquad i=1,2,
			\end{align}
			such that
			\begin{align}\label{invari}
				\begin{split}
					&\left\{
					u\in C(\mathbb{T}^d):
					\sup_{n\geq0}e^{nt_0\upsilon}
					\left\|\varphi^{nt_0}_{\omega}(u)-Y_{\theta_{nt_0}\omega}\right\|_{C(\mathbb{T}^d)}
					<\rho_1^\upsilon(\omega)
					\right\}\subseteq
					S^\upsilon_{\mathrm{loc}}(\omega)
					\\
					&\qquad\subseteq
					\left\{
					u\in C(\mathbb{T}^d):
					\sup_{n\geq0}e^{nt_0\upsilon}
					\left\|\varphi^{nt_0}_{\omega}(u)-Y_{\theta_{nt_0}\omega}\right\|_{C(\mathbb{T}^d)}
					<\rho_2^\upsilon(\omega)
					\right\}.
				\end{split}
			\end{align}
			
			\item The tangent space of $S^\upsilon_{\mathrm{loc}}(\omega)$ at the stationary point $Y_\omega$ is given by
			\[
			T_{Y_\omega}S^\upsilon_{\mathrm{loc}}(\omega)
			=
			F_{\lambda(s)}(\omega).
			\]
			
			\item There exists a random variable $N(\omega)$ such that, for all 
			$n\geq N(\omega)$,
			\[
			\varphi^{nt_0}_{\omega}
			\left(S^\upsilon_{\mathrm{loc}}(\omega)\right)
			\subseteq
			S^\upsilon_{\mathrm{loc}}(\theta_{nt_0}\omega).
			\]
			
			\item \label{aaaa} For
			\[
			0<\upsilon_1\leq\upsilon_2<-\lambda(s),
			\]
			we have
			\[
			S^{\upsilon_2}_{\mathrm{loc}}(\omega)
			\subseteq
			S^{\upsilon_1}_{\mathrm{loc}}(\omega),
			\]
			and, for $n\geq N(\omega)$,
			\[
			\varphi^{nt_0}_{\omega}
			\left(S^{\upsilon_1}_{\mathrm{loc}}(\omega)\right)
			\subseteq
			S^{\upsilon_2}_{\mathrm{loc}}(\theta_{nt_0}\omega).
			\]
			Consequently, for every 
			$u\in S^\upsilon_{\mathrm{loc}}(\omega)$,
			\begin{align}\label{}
				\limsup_{n\to\infty}
				\frac1n
				\log
				\left\|
				\varphi^{nt_0}_{\omega}(u)-Y_{\theta_{nt_0}\omega}
				\right\|_{C(\mathbb{T}^d)}
				\leq
				t_0\lambda(s).
			\end{align}
			
			\item We have
			\begin{align}\label{78421215}
\begin{split}
	\limsup_{n\to\infty}\frac1n
\log\Bigg[
\sup\Bigg\{
\frac{
	\left\|
	\varphi^{nt_0}_{\omega}(u)
	-\varphi^{nt_0}_{\omega}(\tilde{u})
	\right\|_{C(\mathbb{T}^d)}
}{
	\|u-\tilde{u}\|_{C(\T^d)}
}:
\ u\neq\tilde{u},\
u,\tilde{u}\in S^\upsilon_{\mathrm{loc}}(\omega)
\Bigg\}
\Bigg]
\leq
t_0\lambda(s).
\end{split}
			\end{align}
		\end{enumerate}
		
	\end{theorem}
	\begin{proof}
		We aim to apply \cite[Theorem 2.10]{GVR23A}. The requirements of this result are provided by Theorem \ref{METT} and Proposition \ref{FIRTA}.
	\end{proof}
	\begin{remark}
		The stable manifold theorem has several important roles. One of those that we discuss in this paper concerns local stability. This result is also important for proving synchronization. Note that the stable manifold always exists, since the Lyapunov exponents converge to $-\infty$.
	\end{remark}
	Next, we state the unstable manifold theorem.
	\begin{theorem}[Local unstable manifolds]\label{unstable}
		Consider the situation of Theorem \ref{METT} and suppose that 
		$\lambda_1>0$. Define
		\[
		\lambda(u)\coloneqq \inf\{\lambda_i:\lambda_i>0\}.
		\]
		Fix an arbitrary time step $t_{0}\in (0,1]$. Then, for every
		\[
		0<\upsilon<\lambda(u),
		\]
		there exists a family of immersed submanifolds
		\[
		U^\upsilon_{\mathrm{loc}}(\omega)\subset C(\mathbb{T}^d)
		\]
		and a set $\widetilde{\Omega}\subset\Omega$ of full measure which is $\theta_{t_0}$-invariant such that the following properties hold.	
		\begin{enumerate}
			
			\item There exist random variables 
			\[
			\widetilde{\rho}_1^\upsilon(\omega),
			\widetilde{\rho}_2^\upsilon(\omega),
			\]
			which are positive and finite on $\widetilde{\Omega}$ and satisfy
			\[
			\liminf_{p\to\infty}
			\frac1p
			\log
			\widetilde{\rho}_i^\upsilon(\theta_{-pt_0}\omega)
			\geq0,
			\qquad i=1,2.
			\]
			Moreover,
			\begin{align*}
				&\Bigg\{
				u_\omega\in C(\mathbb{T}^d):
				\exists\{u_{\theta_{-nt_0}\omega}\}_{n\geq1}
				\text{ such that }
				\varphi^{mt_0}_{\theta_{-nt_0}\omega}
				(u_{\theta_{-nt_0}\omega})
				=
				u_{\theta_{(m-n)t_0}\omega},
				\quad 0\leq m\leq n,
				\\
				&\qquad
				\sup_{n\geq0}
				e^{nt_0\upsilon}
				\left\|
				u_{\theta_{-nt_0}\omega}
				-
				Y_{\theta_{-nt_0}\omega}
				\right\|_{C(\T^d)}
				<
				\widetilde{\rho}_1^\upsilon(\omega)
				\Bigg\}
				\\
				&\qquad\subseteq
				U^\upsilon_{\mathrm{loc}}(\omega)
				\subseteq
				\Bigg\{
				u_\omega\in C(\mathbb{T}^d):
				\exists\{u_{\theta_{-nt_0}\omega}\}_{n\geq1}
				\text{ such that }
				\varphi^{mt_0}_{\theta_{-nt_0}\omega}
				(u_{\theta_{-nt_0}\omega})
				=
				u_{\theta_{(m-n)t_0}\omega},
				\quad 0\leq m\leq n,
				\\
				&\qquad
				\sup_{n\geq0}
				e^{nt_0\upsilon}
				\left\|
				u_{\theta_{-nt_0}\omega}
				-
				Y_{\theta_{-nt_0}\omega}
				\right\|_{C(\T^d)}
				<
				\widetilde{\rho}_2^\upsilon(\omega)
				\Bigg\}.
			\end{align*}
			
			\item The tangent space of the unstable manifold is given by
			\[
			T_{Y_\omega}U^\upsilon_{\mathrm{loc}}(\omega)
			=
			\bigoplus_{\lambda_i>0}H^i_\omega .
			\]
			
			\item There exists a random variable $N(\omega)$ such that, for every 
			$n\geq N(\omega)$,
			\[
			U^\upsilon_{\mathrm{loc}}(\omega)
			\subseteq
			\varphi^{nt_0}_{\theta_{-nt_0}\omega}
			\left(
			U^\upsilon_{\mathrm{loc}}(\theta_{-nt_0}\omega)
			\right).
			\]
			
			\item For
			\[
			0<\upsilon_1\leq\upsilon_2<\lambda(u),
			\]
			we have
			\[
			U^{\upsilon_2}_{\mathrm{loc}}(\omega)
			\subseteq
			U^{\upsilon_1}_{\mathrm{loc}}(\omega).
			\]
			Moreover, for $n\geq N(\omega)$,
			\[
			U^{\upsilon_1}_{\mathrm{loc}}(\omega)
			\subseteq
			\varphi^{nt_0}_{\theta_{-nt_0}\omega}
			\left(
			U^{\upsilon_2}_{\mathrm{loc}}(\theta_{-nt_0}\omega)
			\right).
			\]
			Consequently, for every
			$u_\omega\in U^\upsilon_{\mathrm{loc}}(\omega)$,
			\begin{align}\label{eqn:contr_char_1}
				\limsup_{n\to\infty}
				\frac1n
				\log
				\left\|
				u_{\theta_{-nt_0}\omega}
				-
				Y_{\theta_{-nt_0}\omega}
				\right\|_{C(\T^d)}
				\leq
				-t_0\lambda(u).
			\end{align}
			
			\item We have
			\begin{align*}
				\limsup_{n\to\infty}
				\frac1n
				\log
				\Bigg[
				\sup
				\Bigg\{
				\frac{
					\left\|
					u_{\theta_{-nt_0}\omega}
					-
					\tilde{u}_{\theta_{-nt_0}\omega}
					\right\|_{C(\T^d)}
				}{
					\left\|
					u_\omega-\tilde{u}_\omega
					\right\|_{C(\T^d)}
				}
				:
				u_\omega\neq\tilde{u}_\omega,\ 
				u_\omega,\tilde{u}_\omega
				\in U^\upsilon_{\mathrm{loc}}(\omega)
				\Bigg\}
				\Bigg]
				\leq
				-t_0\lambda(u).
			\end{align*}
			
		\end{enumerate}
		
	\end{theorem}
	\begin{proof}
		The result follows directly from \cite[Theorem~2.17]{GVR23A}, whose requirements are provided by Theorem \ref{METT} and Proposition \ref{FIRTA}.
	\end{proof}
	\begin{remark}
		The unstable manifold describes the directions along which trajectories diverge from the stationary point in forward time. In particular, it characterizes the local behavior of trajectories that originate near the stationary point and move away from it. This provides a natural connection to the global dynamics and, in particular, to the possible attraction toward other invariant sets or attractors.
	\end{remark}
	\begin{remark}\label{CISa}
		In the statements of Theorems~\ref{stable_manifold} and \ref{unstable}, we
		first discretize the cocycle and then state the corresponding results for the
		resulting discrete-time cocycle. This shows that the invariant manifold is eventually
		invariant under the discrete-time cocycle. A natural question is whether this
		invariance property also holds for the continuous-time cocycle. In fact, we
		obtain a slightly weaker result: for every
		$0<\upsilon_1<-\lambda(s)$, on a set of full measure, there exists a time
		$t(\omega)>0$ such that, for all $t>t(\omega)$,
		\[
		\varphi^{t}_{\omega}
		\big(S^{\upsilon_1}_{\mathrm{loc}}(\omega)\big)
		\subset
		S^{\upsilon_1}_{\mathrm{loc}}(\theta_t\omega).
		\]	
		We sketch the main idea. Indeed, let $t\geq0$ and write
		\[
		t=m_t t_0+r,
		\qquad
		m_t:=\left\lfloor\frac{t}{t_0}\right\rfloor\in\mathbb{N}_0,
		\qquad
		0\leq r<t_0.
		\]
		Let
		\[
		0<\upsilon_1<\upsilon_2<-\lambda(s)
		\]
		and let $u\in S^{\upsilon_1}_{\mathrm{loc}}(\omega)$. Then
		\begin{align}\label{Ujmas89qwz}
			\begin{split}
				&\varphi^{t}_{\omega}(u)-Y_{\theta_t\omega}
				\\
				&\quad=
				\varphi^{r}_{\theta_{m_t t_0}\omega}
				\big(
				\varphi^{m_t t_0}_{\omega}(u)
				\big)
				-
				\varphi^{r}_{\theta_{m_t t_0}\omega}
				\big(
				Y_{\theta_{m_t t_0}\omega}
				\big)
				\\
				&\quad=
				\int_0^1
				D_{\alpha\varphi^{m_t t_0}_{\omega}(u)
					+(1-\alpha)Y_{\theta_{m_t t_0}\omega}}
				\varphi^{r}_{\theta_{m_t t_0}\omega}
				\Big[
				\varphi^{m_t t_0}_{\omega}(u)
				-
				Y_{\theta_{m_t t_0}\omega}
				\Big]
				\,\mathrm{d}\alpha .
			\end{split}
		\end{align}
		By Theorem~\ref{stable_manifold}, Item~\ref{aaaa}, for all sufficiently
		large $t$,
		\begin{align}\label{sdf9sfsad}
			\varphi^{m_t t_0}_{\omega}
			\left(S^{\upsilon_1}_{\mathrm{loc}}(\omega)\right)
			\subseteq
			S^{\upsilon_2}_{\mathrm{loc}}
			(\theta_{m_t t_0}\omega).
		\end{align}
	Moreover, by \eqref{invari} together with \eqref{78421215}, the convergence
	\[
	\varphi^{mt_0}_{\omega}(u)
	\longrightarrow
	Y_{\theta_{mt_0}\omega}
	\]
	is exponentially fast uniformly in
	$u\in S^{\upsilon_1}_{\mathrm{loc}}(\omega)$. Therefore, by Corollary~\ref{UJU},
		Remark~\ref{sds78asdaad}, and the estimates established above, we obtain
		\begin{align}\label{Y78azswds}
			\lim_{m\to\infty}
			\frac{1}{m}
	\sup_{\substack{
			u\in S_{\mathrm{loc}}^{\upsilon_1}(\omega)\\
			0\leq\alpha\leq1\\
			0\leq r\leq t_0}}
			\log^+
			\left\|
			D_{\alpha\varphi^{mt_0}_{\omega}(u)
				+(1-\alpha)Y_{\theta_{mt_0}\omega}}
			\varphi^{r}_{\theta_{mt_0}\omega}
			\right\|_{\mathcal{L}(C(\T^d),C(\T^d))}
			=0.
		\end{align}
		Applying \eqref{Ujmas89qwz} with $m_t$ replaced by $m_t+n$, together with \eqref{Y78azswds}, yields the required estimate uniformly in $n\geq0$. Hence, by \eqref{invari},
		$$
		\varphi_\omega^t(u)
		\in S_{\mathrm{loc}}^{\upsilon_1}(\theta_t\omega)
		$$	
		for all sufficiently large $t$. An analogous argument applies to the unstable manifold in Theorem~\ref{unstable}.
	\end{remark}
	Let us now state the final result of this subsection, which addresses the existence of the center manifold.
	\begin{theorem}[Local center manifolds]\label{center}
		In Theorem~\ref{METT}, suppose that, for  $ i_c\geq 1$,
		we have $\lambda_{i_c}=0$. Fix an arbitrary time step $t_0>0$ and assume
		\[
		0 < \nu < \min\!\bigl\lbrace \lambda_{i_c-1},\, -\lambda_{i_c+1}\bigr\rbrace,
		\]
		with the convention that $\lambda_0=\infty$ if $i_c=1$.
		Let
		\[
		\mathbb{N}_0t_0
		=
		\bigl\{nt_0:n\in\mathbb{N}_0=\mathbb{N}\cup\lbrace0\rbrace\bigr\}.
		\]
		Then there exist a $\theta_{t_0}$-invariant subset
		$\tilde{\Omega}\subseteq\Omega$ of full measure, a continuous cocycle
		\[
		\bar{\varphi} \colon \mathbb{N}_0 t_0 \times \tilde{\Omega} \times C(\T^d) \to C(\T^d) ,
		\]
		and a positive random variable
		$\rho^c \colon \tilde{\Omega} \to (0,\infty)$ such that
		\[
		\liminf_{n\to\pm\infty}
		\frac{1}{|n|}\log\rho^c\!\bigl(\theta_{nt_0}\omega\bigr)\geq0.
		\]
		Moreover, the functions $\rho^c$ and $\bar{\varphi}$ satisfy
		\[
		\|u\|_{C(\T^d)}\leq\rho^c(\omega)
		\implies
		\bar{\varphi}^{t_0}_\omega(Y_\omega+u)
		=
		\varphi^{t_0}_\omega(Y_\omega+u).
		\]
		In addition, for every $\omega\in\tilde{\Omega}$, there exists a function
		\[
		h^c_\omega\colon C_\omega\to\mathcal{M}^{c,\nu}_\omega\subset C(\T^d),
		\]
		such that:
		\begin{itemize}
			\item[(i)] $h^c_\omega$ is a Lipschitz homeomorphism with
			$$
			h^c_\omega(0)=Y_\omega,
			$$
			and is differentiable at zero.
			\item[(ii)] $\mathcal{M}^{c,\nu}_\omega$ is a topological Banach manifold\footnote{That is, $\mathcal{M}^{c,\nu}_\omega$ is locally homeomorphic to an open subset of the Banach space $C_\omega$.} modeled on $C_\omega$.

			\item[(iii)] $\mathcal{M}^{c,\nu}_\omega$ is
			$\bar{\varphi}$-invariant; that is, for every $n\in\mathbb{N}_0$,
			\[
			\bar{\varphi}^{nt_0}_\omega
			\bigl(\mathcal{M}^{c,\nu}_\omega\bigr)
			\subseteq
			\mathcal{M}^{c,\nu}_{\theta_{nt_0}\omega}.
			\]
			
			\item[(iv)] For every
			$u_\omega\in\mathcal{M}^{c,\nu}_\omega$, there exists a sequence
			$\{u_{\theta_{-nt_0}\omega}\}_{n\geq1}$ such that, if we define
			\[
			j\leq0:\qquad
			\bar{\varphi}^{jt_0}_\omega(u_\omega)
			:=u_{\theta_{jt_0}\omega},
			\]
			then
			\[
			\forall(k,j)\in\mathbb{N}_0\times\mathbb{Z}:\qquad
			\bar{\varphi}^{kt_0}_{\theta_{jt_0}\omega}
			\bigl(\bar{\varphi}^{jt_0}_\omega(u_\omega)\bigr)
			=
			\bar{\varphi}^{(k+j)t_0}_\omega(u_\omega).
			\]
			Moreover,
			\[
			\sup_{j\in\mathbb{Z}}
			\exp(-\nu t_0|j|)
			\left\|
			\bar{\varphi}^{jt_0}_\omega(u_\omega)
			-
			Y_{\theta_{jt_0}\omega}
			\right\|_{C(\T^d)}
			<\infty.
			\]
		\end{itemize}
	\end{theorem}
	\begin{proof}
		The result follows directly from \cite[Theorem~3.14]{GVR25C}, with the same requirements as for the stable and unstable manifold results, i.e., Theorem \ref{METT} and Proposition \ref{FIRTA}.
	\end{proof}
	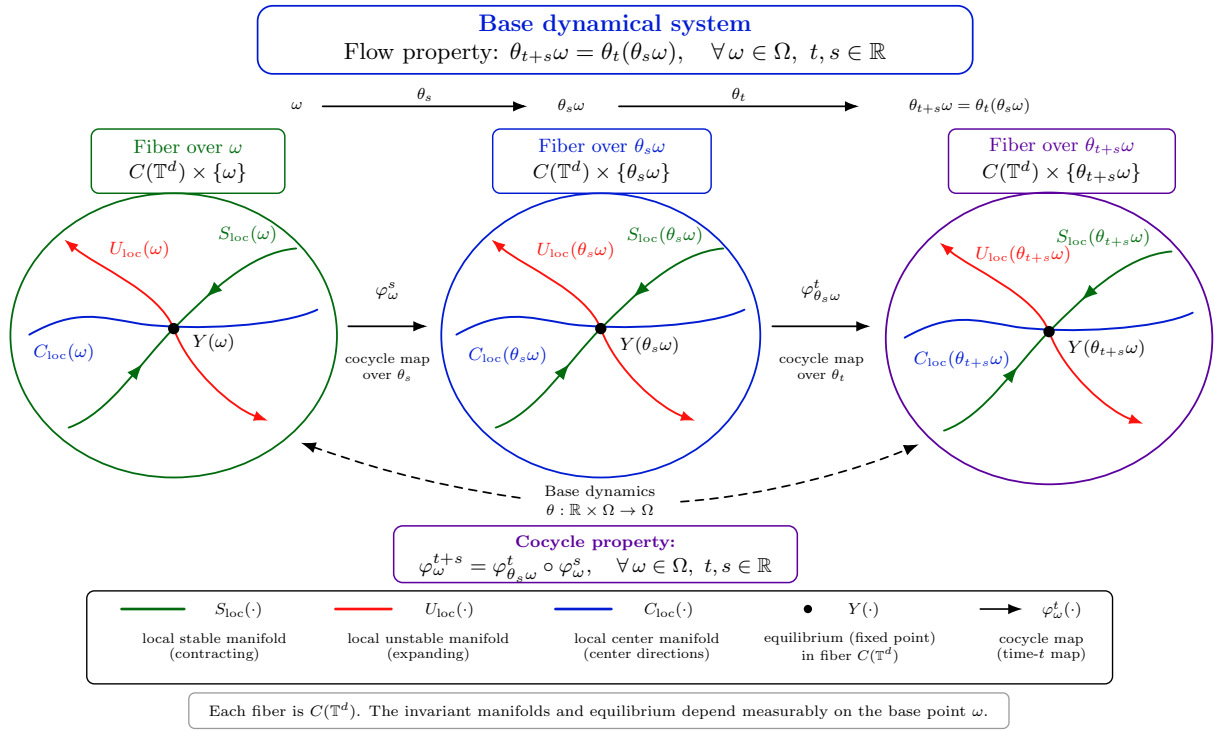
\begin{figure}[ht]
		\centering
		\begin{tikzpicture}[x=1cm,y=1cm,
			every node/.style={inner sep=0pt,outer sep=0pt,scale=.75},
			line cap=round,line join=round,
			>={Latex[length=1.9mm,width=1.4mm]}]
			\path[use as bounding box] (0,0) rectangle (16.11,9.76);
			
			% Base dynamical system
			\draw[baseblue,line width=.65pt,rounded corners=2mm]
			(3.41,8.78) rectangle (12.87,9.68);
			\node[baseblue,font=\fontsize{11.5}{12}\selectfont\bfseries]
			at (8.14,9.43) {Base dynamical system};
			\node[scale=1.30,font=\fontsize{9}{11}\selectfont] at (8.14,9.06)
			{Flow property: $\theta_{t+s}\omega=\theta_t(\theta_s\omega),\quad
				\forall\,\omega\in\Omega,\ t,s\in\mathbb R$};
			
			\node[font=\fontsize{8.5}{10}\selectfont] at (3.93,8.36) {$\omega$};
			\node[font=\fontsize{8.5}{10}\selectfont] at (7.56,8.36)
			{$\theta_s\omega$};
			\node[font=\fontsize{8.5}{10}\selectfont] at (12.83,8.36)
			{$\theta_{t+s}\omega=\theta_t(\theta_s\omega)$};
			\draw[->,line width=.65pt] (4.27,8.35)--(6.98,8.35)
			node[midway,above=2pt,font=\fontsize{8}{9}\selectfont]
			{$\theta_s$};
			\draw[->,line width=.65pt] (8.19,8.35)--(11.37,8.35)
			node[midway,above=2pt,font=\fontsize{8}{9}\selectfont]
			{$\theta_t$};
			
			% Fiber labels
			\draw[stablegreen,line width=.55pt,rounded corners=1mm]
			(1.21,7.20) rectangle (3.77,8.05);
			\node[stablegreen,scale=1.08,font=\fontsize{8.6}{9.5}\selectfont]
			at (2.49,7.82) {Fiber over $\omega$};
			\node[scale=1.25,font=\fontsize{8.2}{9}\selectfont]
			at (2.49,7.48)
			{$C(\mathbb T^d)\times\{\omega\}$};
			
			\draw[baseblue,line width=.55pt,rounded corners=1mm]
			(6.52,7.20) rectangle (9.41,8.05);
			\node[baseblue,scale=1.08,font=\fontsize{8.6}{9.5}\selectfont]
			at (7.965,7.82) {Fiber over $\theta_s\omega$};
			\node[scale=1.25,font=\fontsize{8.2}{9}\selectfont]
			at (7.965,7.48)
			{$C(\mathbb T^d)
				\times\{\theta_s\omega\}$};
			
			\draw[fiberpurple,line width=.55pt,rounded corners=1mm]
			(12.55,7.20) rectangle (15.55,8.05);
			\node[fiberpurple,scale=1.08,font=\fontsize{8.6}{9.5}\selectfont]
			at (14.05,7.82) {Fiber over $\theta_{t+s}\omega$};
			\node[scale=1.25,font=\fontsize{8.2}{9}\selectfont]
			at (14.05,7.48)
			{$C(\mathbb T^d)
				\times\{\theta_{t+s}\omega\}$};
			
			% Draw the ellipse and manifolds in one fiber
			\newcommand{\drawfiber}[5]{%
				\begin{scope}[shift={(#1,#2)}]
					\draw[#3,line width=.68pt]
					(0,-.09) ellipse [x radius=#4,y radius=#5];
					
					\draw[baseblue,line width=.74pt]
					(-1.91,-.08)
					.. controls (-1.30,.26) and (-1.00,.16) .. (-.43,.06)
					.. controls (.26,-.03) and (1.44,.03) .. (1.90,.25);
					
					\draw[stablegreen,line width=.72pt,
					postaction={decorate},
					decoration={markings,
						mark=at position .68 with {\arrow{Latex}}}]
					(-1.39,-1.31)
					.. controls (-.85,-1.12) and (-.44,-.44) .. (0,0);
					
					\draw[stablegreen,line width=.72pt,
					postaction={decorate},
					decoration={markings,
						mark=at position .70 with {\arrow{Latex}}}]
					(1.62,1.06)
					.. controls (1.00,1.01) and (.44,.44) .. (0,0);
					
					\draw[unstablered,->,line width=.72pt]
					(0,0) .. controls (-.20,.50) and (-1.00,.81) .. (-1.45,1.18);
					
					\draw[unstablered,->,line width=.72pt]
					(0,0) .. controls (.20,-.55) and (.78,-1.05) .. (1.25,-1.22);
					
					\fill (0,0) circle[radius=.075cm];
				\end{scope}%
			}
			
			\drawfiber{2.30}{5.41}{stablegreen}{2.16cm}{1.88cm}
			\drawfiber{7.95}{5.41}{baseblue}{2.11cm}{1.88cm}
			\drawfiber{13.88}{5.37}{fiberpurple}{2.16cm}{1.93cm}
			
			% Labels in the left fiber
			\node[stablegreen,scale=1.12,font=\fontsize{7.8}{9}\selectfont]
			at (3.28,6.66) {$S_{\mathrm{loc}}(\omega)$};
			\node[unstablered,scale=1.12,font=\fontsize{7.8}{9}\selectfont]
			at (1.85,6.45) {$U_{\mathrm{loc}}(\omega)$};
			\node[baseblue,scale=1.12,font=\fontsize{7.8}{9}\selectfont]
			at (.85,5.09) {$C_{\mathrm{loc}}(\omega)$};
			\node[anchor=west,scale=1.12,font=\fontsize{7.8}{9}\selectfont]
			at (2.55,5.22) {$Y(\omega)$};
			
			% Labels in the middle fiber
			\node[stablegreen,scale=1.12,font=\fontsize{7.8}{9}\selectfont]
			at (8.8,6.64) {$S_{\mathrm{loc}}(\theta_s\omega)$};
			\node[unstablered,scale=1.12,font=\fontsize{7.8}{9}\selectfont]
			at (7.62,6.44) {$U_{\mathrm{loc}}(\theta_s\omega)$};
			\node[baseblue,scale=1.12,font=\fontsize{7.8}{9}\selectfont]
			at (6.72,5.06) {$C_{\mathrm{loc}}(\theta_s\omega)$};
			\node[anchor=west,scale=1.12,font=\fontsize{7.8}{9}\selectfont]
			at (8.20,5.19) {$Y(\theta_s\omega)$};
			
			% Labels in the right fiber
			\node[stablegreen,scale=1.12,font=\fontsize{7.8}{9}\selectfont]
			at (14.6,6.61) {$S_{\mathrm{loc}}(\theta_{t+s}\omega)$};
			\node[unstablered,scale=1.12,font=\fontsize{7.8}{9}\selectfont]
			at (13.54,6.39) {$U_{\mathrm{loc}}(\theta_{t+s}\omega)$};
			\node[baseblue,scale=1.12,font=\fontsize{7.8}{9}\selectfont]
			at (12.73,5.00) {$C_{\mathrm{loc}}(\theta_{t+s}\omega)$};
			\node[anchor=west,scale=1.12,font=\fontsize{7.8}{9}\selectfont]
			at (14.13,5.15) {$Y(\theta_{t+s}\omega)$};
			
			% Maps between fibers
			\draw[->,line width=.67pt] (4.61,5.44)--(5.66,5.44);
			\node[scale=1.12,font=\fontsize{8.2}{9}\selectfont]
			at (5.13,5.91) {$\varphi_\omega^s$};
			\node[align=center,font=\fontsize{7.3}{8.3}\selectfont]
			at (5.13,4.90) {cocycle map\\over $\theta_s$};
			
			\draw[->,line width=.67pt] (10.22,5.44)--(11.52,5.44);
			\node[scale=1.12,font=\fontsize{8.2}{9}\selectfont]
			at (10.87,5.91) {$\varphi_{\theta_s\omega}^{t}$};
			\node[align=center,font=\fontsize{7.3}{8.3}\selectfont]
			at (10.87,4.90) {cocycle map\\over $\theta_t$};
			
			% Base dynamics
			\draw[dashed,-{Latex[length=2mm]},line width=.65pt,
			dash pattern=on 3pt off 2.5pt]
			(6.88,3.17)
			.. controls (5.51,3.37) and (4.71,3.65) .. (3.98,3.90);
			\draw[dashed,-{Latex[length=2mm]},line width=.65pt,
			dash pattern=on 3pt off 2.5pt]
			(9.01,3.12)
			.. controls (10.19,3.20) and (11.40,3.56) .. (12.16,3.88);
			\node[align=center,scale=1.08,
			font=\fontsize{7.5}{8.6}\selectfont]
			at (7.94,3.13)
			{Base dynamics\\$\theta:\mathbb R\times\Omega\to\Omega$};
			
			% Cocycle property
			\draw[fiberpurple,line width=.57pt,rounded corners=1.2mm]
			(5.21,2.05) rectangle (10.55,2.80);
			\node[fiberpurple,scale=1.05,
			font=\fontsize{8.5}{10}\selectfont\bfseries]
			at (7.88,2.57) {Cocycle property:};
			\node[scale=1.26,font=\fontsize{8}{9}\selectfont]
			at (7.88,2.25)
			{$\varphi_\omega^{t+s}
				=\varphi_{\theta_s\omega}^{t}\circ\varphi_\omega^{s},
				\quad\forall\,\omega\in\Omega,\ t,s \geq 0$};
			
			% Legend
			\draw[line width=.52pt,rounded corners=1mm]
			(1.15,.71) rectangle (14.72,1.94);
			
			\draw[stablegreen,line width=1pt] (1.61,1.70)--(2.44,1.70);
			\node[font=\fontsize{8}{9}\selectfont]
			at (3.16,1.69) {$S_{\mathrm{loc}}(\cdot)$};
			\node[align=center,font=\fontsize{6.6}{7.3}\selectfont]
			at (2.84,1.19) {local stable manifold\\(contracting)};
			
			\draw[unstablered,line width=1pt] (4.44,1.70)--(5.18,1.70);
			\node[font=\fontsize{8}{9}\selectfont]
			at (5.95,1.69) {$U_{\mathrm{loc}}(\cdot)$};
			\node[align=center,font=\fontsize{6.6}{7.3}\selectfont]
			at (5.67,1.19) {local unstable manifold\\(expanding)};
			
			\draw[baseblue,line width=1pt] (7.35,1.70)--(8.09,1.70);
			\node[font=\fontsize{8}{9}\selectfont]
			at (8.85,1.69) {$C_{\mathrm{loc}}(\cdot)$};
			\node[align=center,font=\fontsize{6.6}{7.3}\selectfont]
			at (8.56,1.19)
			{local center manifold\\(center directions)};
			
			\fill (10.69,1.70) circle[radius=.055cm];
			\node[font=\fontsize{8}{9}\selectfont]
			at (11.43,1.69) {$Y(\cdot)$};
			\node[align=center,font=\fontsize{6.6}{7.3}\selectfont]
			at (11.25,1.19)
			{equilibrium (fixed point)\\
				in fiber $C(\mathbb T^d)$};
			
			\draw[->,line width=.62pt] (12.97,1.70)--(13.53,1.70);
			\node[font=\fontsize{8}{9}\selectfont]
			at (14.05,1.69) {$\varphi_\omega^t(\cdot)$};
			\node[align=center,font=\fontsize{6.6}{7.3}\selectfont]
			at (13.78,1.19) {cocycle map\\(time-$t$ map)};
			
			% Bottom note
			\draw[gray!80,line width=.48pt,rounded corners=.8mm]
			(2.50,.12) rectangle (13.35,.59);
			\node[scale=1.09,font=\fontsize{7.1}{8}\selectfont]
			at (7.925,.35)
			{Each fiber is $C(\mathbb T^d)$.
				The invariant manifolds and equilibrium depend measurably
				on the base point $\omega$.};
		\end{tikzpicture}
		\caption{Visualization of the invariant manifolds.}
		\label{fig:invariant-manifolds}
	\end{figure}
\subsection{Stability}

A useful consequence of the stable manifold theorem is that, if the top Lyapunov exponent is negative, the stable manifold contains a neighborhood of the stationary point. This allows us to deduce local exponential stability from the linearized dynamics.
We first recall the following lemma, which is an immediate consequence of the stable manifold theorem.
\begin{lemma}\label{ASasxassas}
	Consider the setting of Theorem~\ref{stable_manifold}. Assume that
	\[
	\lambda(s)=\lambda_1<0,
	\]
	and fix an arbitrary time step $t_0\in(0,1]$. Then, for every
	\[
	0<\upsilon<-\lambda_1,
	\]
	there exists a set of full measure $\widetilde{\Omega}$ and a positive random variable $\rho_3^\upsilon$ satisfying
	\begin{align}\label{eqn:rho_tempa}
		\liminf_{p\to\infty}\frac1p
		\log\rho_3^\upsilon(\theta_{pt_0}\omega)\geq0,
	\end{align}
	such that, for every $\omega\in\widetilde{\Omega}$,
	\begin{enumerate}
		\item
		\begin{align}\label{asaswqew}
			B_{C(\T^d)}\bigl(Y_\omega,\rho_3^\upsilon(\omega)\bigr)
			\subseteq
			S^\upsilon_{\mathrm{loc}}(\omega).
		\end{align}
		
		\item We have
		\begin{align}\label{eqn:local_contraction}
			\limsup_{t\to\infty}
			\frac1t
			\log
			\left[
			\sup_{\substack{
					u\neq\tilde u\\
					u,\tilde u\in
					B_{C(\T^d)}(Y_\omega,\rho_3^\upsilon(\omega))}}
			\frac{
				\left\|
				\varphi^t_\omega(u)
				-\varphi^t_\omega(\tilde u)
				\right\|_{C(\T^d)}
			}{
				\|u-\tilde u\|_{C(\T^d)}
			}
			\right]
			\leq
			\lambda_1,
		\end{align}
Consequently
		\begin{align}\label{eqn:contr_char6}
			\limsup_{t\to\infty}
			\frac1t
			\log
			\left[
			\sup_{u\in B_{C(\T^d)}(Y_\omega,\rho_3^\upsilon(\omega))}
			\left\|
			\varphi^{t}_{\omega}(u)-Y_{\theta_t\omega}
			\right\|_{C(\mathbb{T}^d)}
			\right]
			\leq
			\lambda_1.
		\end{align}
	\end{enumerate}
\end{lemma}

\begin{proof}
	The proof essentially follows from Theorem~\ref{stable_manifold} and Remark~\ref{CISa}. Indeed, since $\lambda_1<0$ is the top Lyapunov exponent, we have
	\[
	F_{\lambda_1}(\omega)=C(\T^d).
	\]
	Then, by following the proof of \cite[Theorem~2.10]{GVR23A}, one obtains a positive random variable $\rho_3^\upsilon$ satisfying \eqref{eqn:rho_tempa} such that \eqref{asaswqew} holds. The second assertion follows from Theorem~\ref{stable_manifold} together with Remark~\ref{CISa}.
\end{proof}
The next result shows that, in some cases, one can prove that such an exponential estimate holds.

\begin{theorem}\label{stabi}
	Consider the SPDE
	\begin{align}\label{MAINAAssAA}
		\begin{cases}
			(\partial_t-\Delta)u^{\delta}
			=
			\mu u^{\delta}
			+\delta\Sigma(u^\delta)\,\xi^Q,
			& t>0,\\
			u^\delta(0,\cdot)=u_0\in L^\infty(\mathbb T^d),
		\end{cases}
	\end{align}
	where $\Sigma(0)=0$, $\mu<0$, and $\delta>0$. Then, for $\delta>0$ sufficiently small, the top Lyapunov exponent $\lambda_1^\delta$ is negative. In particular, the conclusions of Lemma~\ref{ASasxassas} hold with $Y_\omega=0$.
\end{theorem}
\begin{proof}
Thanks to Lemma~\ref{ASasxassas}, it remains to prove that the top Lyapunov exponent is negative. Let
	\[
	\psi^{t,\delta}_\omega:=D_0\varphi^{t,\delta}_\omega .
	\]
	By Theorem~\ref{METT}, for almost every $\omega$ there exists
	\[
	u_\omega\in H^1_\omega\setminus\{0\}
	\]
	such that
	\begin{align}\label{Ikaswwe25}
		\lambda_1^\delta
		=
		\lim_{m\to\infty}
		\frac{1}{mt_0}
		\log
		\left\|
		\psi^{mt_0,\delta}_\omega(u_\omega)
		\right\|_{C(\T^d)} .
	\end{align}
	Moreover, by the cocycle property,
	\begin{align*}
		\frac{1}{knt_0}
		\log
		\left\|
		\psi^{knt_0,\delta}_\omega
		\right\|_{\mathcal L(C(\T^d),C(\T^d))}
		\leq
		\frac1k
		\sum_{0\leq j<k}
		\frac{
			\log
			\left\|
			\psi^{nt_0,\delta}_{\theta_{jnt_0}\omega}
			\right\|_{\mathcal L(C(\T^d),C(\T^d))}
		}{nt_0}.
	\end{align*}
	Hence, by the Birkhoff ergodic theorem and the ergodicity of
	$\theta_{nt_0}$,
	\begin{align*}
		\limsup_{k\to\infty}
		\frac{1}{knt_0}
		\log
		\left\|
		\psi^{knt_0,\delta}_\omega
		\right\|_{\mathcal L(C(\T^d),C(\T^d))}
		\leq
		\frac{1}{nt_0}
		\E\left[
		\log
		\left\|
		\psi^{nt_0,\delta}_\omega
		\right\|_{\mathcal L(C(\T^d),C(\T^d))}
		\right],
		\qquad \mathbb P\text{-a.s.}
	\end{align*}
	Together with \eqref{Ikaswwe25}, this yields
	\begin{align}\label{Kaser}
		\lambda_1^\delta
		\leq
		\inf_{n\geq1}
		\frac{1}{nt_0}
		\E\left[
		\log
		\left\|
		\psi^{nt_0,\delta}_\omega
		\right\|_{\mathcal L(C(\T^d),C(\T^d))}
		\right].
	\end{align}
	We now use \eqref{Kaser} to show that $\lambda_1^\delta<0$ for sufficiently small $\delta$. By Proposition~\ref{kkkl}, the linearized equation at the stationary solution $0$ is
	\begin{align}\label{MAINAAssAAsasd}
		\begin{cases}
			(\partial_t-\Delta)v^\delta
			=
			\mu v^\delta
			+\delta\Sigma'(0)v^\delta\,\xi^Q,
			& t>0,\\
			v^\delta(0,\cdot)=v_0\in C(\T^d),
		\end{cases}
	\end{align}
	where the solution is understood in the sense of Definition~\ref{dfn}.
	Let
	\[
	v^\delta(t,x)=e^{\mu t}\overline v^\delta(t,x).
	\]
	Then, by the properties of the reconstruction operator, $\overline v^\delta$ solves
	\begin{align}\label{MAINAAssAAsassdd}
		\begin{cases}
			(\partial_t-\Delta)\overline v^\delta
			=
			\delta\Sigma'(0)\overline v^\delta\,\xi^Q,
			& t>0,\\
			\overline v^\delta(0,\cdot)=v_0\in C(\T^d).
		\end{cases}
	\end{align}
	Consequently,
	\[
	\psi^{t,\delta}_\omega
	=
	e^{\mu t}\overline\psi^{t,\delta}_\omega,
	\]
	where $\overline\psi^{t,\delta}$ denotes the cocycle associated with
	\eqref{MAINAAssAAsassdd}. Thus, by \eqref{Kaser},
	\begin{align}\label{small-noise-Lyap}
		\lambda_1^\delta
		\leq
		\mu
		+
		\frac{1}{t_0}
		\E\left[
		\log
		\left\|
		\overline\psi^{t_0,\delta}_\omega
		\right\|_{\mathcal L(C(\T^d),C(\T^d))}
		\right].
	\end{align}
Regarding \eqref{MAINAAssAAsassdd} as an equation driven by the corresponding model scaled by $\delta$, the stability of the solution with respect to the model, together with Proposition~\ref{onwall} and the dominated convergence theorem, yields
\[
\lim_{\delta\to0}
\frac{1}{t_0}
\E\left[
\log
\left\|
\overline\psi^{t_0,\delta}_\omega
\right\|_{\mathcal L(C(\T^d),C(\T^d))}
\right]
=
\frac{1}{t_0}
\log
\left\|
e^{t_0\Delta}
\right\|_{\mathcal L(C(\T^d),C(\T^d))}
=0.
\]
	Since $\mu<0$, \eqref{small-noise-Lyap} implies that, for $\delta>0$
	sufficiently small,
	\[
	\lambda_1^\delta<0.
	\]
\end{proof}
	
	\appendix
	\section{ Proofs of Lemma \ref{Alex} and Lemma \ref{DCXSZ}}
	\begin{proof}[\textbf{Proof of Lemma \ref{Alex}}]\label{Alexi}
		Some parts of the proof are similar to \cite[Lemmas~2 and~3 and Theorem~4]{CFW26}.
		The main difference is that we formulate the argument within the framework of
		\cite{Hai14}. In particular, we use the Schauder-type estimates established in
		\cite{Hai14}. Moreover, we employ a patching argument to extend the local
		estimates. First, recall that for
		$w\in O^s_{s+T}$, 
		\begin{align*}
			\left(\mathbf{R}^+_s\RR{\mathcal{U}}_{u_0}\right)(w)
			:= {}&
			\mathcal{G}^{s}_{\gamma_{0},1-\kappa}
			\big(
			\Sigma(\RR{\mathcal{U}}_{u_0}) \tilde{\star} \rb
			\big)(w)
			+\mu\mathcal{G}^{s}_{\gamma_{1},2}
			\left(\RR{\mathcal{U}}_{u_0}\right)(w)
			+\RR{G^{s}}(u_0)(w).
		\end{align*}	
		Using an analogous argument to that in the proof of Proposition~\ref{onwall}
		(cf.~\hyperref[Step10]{\textbf{Step 1}}), we deduce that, for every
		$r\in[s,s+T)$ and $w\in O^{r}_{s+T}$,
		\begin{align}\label{CZ1}
			\begin{split}
				\left(\mathbf{R}^+_r\RR{\mathcal{U}}_{u_0}\right)(w)
				={}&
				\mathcal{G}^{r}_{\gamma_{0},1-\kappa}
				\big(
				\Sigma(\RR{\mathcal{U}}_{u_0}) \tilde{\star} \rb
				\big)(w)
				+\mu\mathcal{G}^{r}_{\gamma_{1},2}
				\left(\RR{\mathcal{U}}_{u_0}\right)(w)
				+\RR{G^{r}}
				\big(\mathcal{U}^{\gY}_{u_0}(r,\cdot)\big)(w),
				\\
				&\mathcal{R}\RR{\mathcal{U}}_{u_0}(r,x)
				=\mathcal{U}^{\gY}_{u_0}(r,x),
				\quad
				\text{with the convention that, if } r=s,\;
				\mathcal{U}^{\gY}_{u_0}(s,\cdot)=u_0(\cdot).
			\end{split}
		\end{align}
		Thus, \begin{align}\label{CZ22} 
			\left(\mathbf{R}^+_r\RR{\mathcal{U}}_{u_0}\right)(w) = \mathcal{Q}_{s+T-r}^{r} \left( \mathcal{R}\RR{\mathcal{U}}_{u_0}(r,\cdot), \RR{\mathcal{U}}_{u_0} \right)(w). 
		\end{align}
		Consider
		\begin{align}\label{BBBBB}
			r\in [s,s+T)
			\quad \text{and} \quad
			0<\delta\leq s+T-r.
		\end{align}
		With these preparations, we divide the rest of the proof into several steps.
		
		\textbf{Step 1.}\label{Step1111}
		In this step, for $a\in \mathcal{A}^-=\{\rb, \gIb, \gXXXX\rb, i=1,\cdots, d\}$, we focus on estimating
		\begin{align*}
			\sup_{\substack{
					z,w\in O_{\delta+r}^{r},\ z\neq w\\
					d(z,w)\leq\Vert z,w\Vert_{P_r}
			}}
			\frac{
				\Vert z,w\Vert_{P_r}^{\gamma_1}
				\,
				\Big\Vert
				\mathrm{Pr}_{a}\Big[
				\big(\Sigma(\RR{\mathcal{U}}_{u_0})\tilde{\star}\rb\big)(z)
				-\Gamma_{z,w}
				\big(
				\Sigma(\RR{\mathcal{U}}_{u_0})\tilde{\star}\rb
				\big)(w)
				\Big]
				\Big\Vert
			}{
				d(z,w)^{\gamma_0-|a|_h}
			}.
		\end{align*}
		\textbf{Case $a=\rb$}: Let
		\begin{align}\label{Bnassssw}
			\begin{split}
				&\eta_{1}(z,w):=-\Sigma(\mathcal{U}^{\gY}_{u_0}(w))\Pi_z(\gI)(w),\\
				&\eta_{2}(z,w):=-\sum_{i=1}^{d}\mathcal{U}^{\gXXXX}_{u_0}(w)	\Pi_z(\gXXXX)(w).
			\end{split}
		\end{align}
		Thus
		\begin{align}\label{NMasz}
			\mathrm{Pr}_{\gY}\big[
			\RR{\mathcal{U}}_{u_0}(z)-\Gamma_{z,w}\left(
			\RR{\mathcal{U}}_{u_0}(w)
			\right)\big]
			=
			\mathcal{U}^{\gY}_{u_0}(z)-\mathcal{U}^{\gY}_{u_0}(w)-\eta_1(z,w)-\eta_2(z,w).
		\end{align}
		By \eqref{B1478} and \eqref{Hnasssss}, we have
		\begin{align}
			\begin{split}
				&\mathrm{Pr}_{\rb}\big[
				\big(\Sigma(\RR{\mathcal{U}}_{u_0}) \tilde{\star} \rb
				\big)(z)-\Gamma_{z,w}\big(
				\left(
				\Sigma(\RR{\mathcal{U}}_{u_0}) \tilde{\star} \rb
				\right)(w)
				\big)\big]\\&=\Sigma(\mathcal{U}^{\gY}_{u_0}(z))-\Sigma(\mathcal{U}^{\gY}_{u_0}(w))+\Sigma^{\prime}(\mathcal{U}^{\gY}_{u_0}(w))\Sigma(\mathcal{U}^{\gY}_{u_0}(w))\Pi_z(\gI)(w) +\sum_{i=1}^{d}\Sigma^{\prime}(\mathcal{U}^{\gY}_{u_0}(w))\mathcal{U}^{\gXXXX}_{u_0}(w)	\Pi_z(\gXXXX)(w)\\&=\underbrace{\Sigma(\mathcal{U}^{\gY}_{u_0}(z))-\Sigma\big(\mathcal{U}^{\gY}_{u_0}(w)+\eta_1(z,w)+\eta_2(z,w)\big)}_{\text{I}(\RR{\mathcal{U}}_{u_0},z,w)}\\
				&+\underbrace{\Sigma\big(\mathcal{U}^{\gY}_{u_0}(w)+\eta_1(z,w)+\eta_2(z,w)\big)-\Sigma\big(\mathcal{U}^{\gY}_{u_0}(w)+\eta_2(z,w)\big)-\Sigma^{\prime}\big(\mathcal{U}^{\gY}_{u_0}(w)\big)\eta_1(z,w)}_{\text{II}(\RR{\mathcal{U}}_{u_0},z,w)}\\&+\underbrace{\Sigma\big(\mathcal{U}^{\gY}_{u_0}(w)+\eta_2(z,w)\big)-\Sigma(\mathcal{U}^{\gY}_{u_0}(w))-\Sigma^{\prime}(\mathcal{U}^{\gY}_{u_0}(w))\eta_2(z,w)}_{\text{III}(\RR{\mathcal{U}}_{u_0},z,w)}.
			\end{split}
		\end{align}
		From the Taylor expansion, it is clear that
		\begin{enumerate}
			\item
			\begin{align*}
				&\text{I}(\RR{\mathcal{U}}_{u_0},z,w) \\
				&= \int_{0}^{1}
				\Sigma'\bigg(
				\alpha \mathcal{U}^{\gY}_{u_0}(z)
				+ (1-\alpha)\big(\mathcal{U}^{\gY}_{u_0}(w)+\eta_1(z,w)+\eta_2(z,w)\big)
				\bigg) 
				\mathrm{Pr}_{\gY}\big[
				\RR{\mathcal{U}}_{u_0}(z)-\Gamma_{z,w}\big(\RR{\mathcal{U}}_{u_0}(w)\big)
				\big]\,\mathrm{d}\alpha.
			\end{align*}
			\item 
			\begin{align*}
				& \text{II}(\RR{\mathcal{U}}_{u_0},z,w) \\
				&= \int_0^1\int_0^1
				\Sigma''\big(\mathcal{U}^{\gY}_{u_0}(w)+\beta\eta_2(z,w)+\alpha\beta\eta_1(z,w)\big)
				\big[\eta_2(z,w)+\alpha\eta_1(z,w)\big]\eta_1(z,w)
				\,\mathrm{d}\beta\,\mathrm{d}\alpha.
			\end{align*}
			
			\item 
			\begin{align*}
				\text{III}(\RR{\mathcal{U}}_{u_0},z,w)
				&=
				\int_{0}^{1}
				\big[
				\Sigma^{\prime}\big(\mathcal{U}^{\gY}_{u_0}(w)+\alpha\eta_{2}(z,w)\big)
				-
				\Sigma^{\prime}\big(\mathcal{U}^{\gY}_{u_0}(w)\big)
				\big]
				\eta_{2}(z,w)\,\mathrm{d}\alpha
				\\
				&=
				\int_0^1\int_0^1
				\alpha\Sigma^{\prime\prime}\big(
				\mathcal{U}^{\gY}_{u_0}(w)+\alpha\beta\eta_2(z,w)
				\big)
				\,\,(\eta_2(z,w))^2
				\,\mathrm{d}\beta\,\mathrm{d}\alpha.
			\end{align*}
			Thanks to the assumptions on $\Sigma$ and a simple interpolation argument, we obtain that
			\begin{align*}
				\Vert \text{III}(\RR{\mathcal{U}}_{u_0},z,w)\Vert
				\lesssim
				\Vert\eta_{2}(z,w)\Vert^{\gamma_1}.
			\end{align*}
		\end{enumerate}
		Using these latter expressions together with the assumptions on $\Sigma$, we conclude that
		\begin{align}\label{PLPLP}
			\begin{split}
				&\big\Vert \mathrm{Pr}_{\rb}\big[
				\big(\Sigma(\RR{\mathcal{U}}_{u_0}) \tilde{\star} \rb
				\big)(z)-\Gamma_{z,w}\big(
				\left(
				\Sigma(\RR{\mathcal{U}}_{u_0}) \tilde{\star} \rb
				\right)(w)
				\big)\big]\big\Vert \lesssim \big\Vert \mathrm{Pr}_{\gY}\big[
				\RR{\mathcal{U}}_{u_0}(z)-\Gamma_{z,w}\big(
				\RR{\mathcal{U}}_{u_0}(w)
				\big)\big]\big\Vert\\&+\Vert \Pi_z(\gI)(w)\Vert^{2}+d(z,w) \Vert \Pi_z(\gI)(w)\Vert\sum_{i=1}^{d}\Vert\mathcal{U}^{\gXXXX}_{u_0}(w)\Vert+ d(z,w)^{\gamma_1}\sum_{i=1}^{d}\Vert\mathcal{U}^{\gXXXX}_{u_0}(w)\Vert^{\gamma_1}.
			\end{split}
		\end{align}
		Since $\gamma_1<2-2\kappa$ and, for
		$z,w\in O_{\delta+r}^{\,r}$, we have $\|z,w\|_{P_r}\leq \delta^{\frac12},$ it follows from \eqref{PLPLP} that
		\begin{align}\label{A11}
			\begin{split}
				&\sup_{\substack{ z,w \in O_{\delta+r}^{r},\ z\neq w\\ d(z,w)\leq\Vert z,w\Vert_{P_r}}}
				\frac{\Vert z,w \Vert_{P_{r}}^{\gamma_1} \, \big\Vert \mathrm{Pr}_{\rb}\big[
					\big(\Sigma(\RR{\mathcal{U}}_{u_0}) \tilde{\star} \rb
					\big)(z)-\Gamma_{z,w}\big(
					\left(
					\Sigma(\RR{\mathcal{U}}_{u_0}) \tilde{\star} \rb
					\right)(w)
					\big)\big]\big\Vert}{d(z,w)^{\gamma_1}}\\&\lesssim
				\sup_{\substack{ z,w \in O_{\delta+r}^{r},\ z\neq w\\ d(z,w)\leq\Vert z,w\Vert_{P_r}}}
				\frac{\Vert z,w \Vert_{P_{r}}^{\gamma_1}\big\Vert \mathrm{Pr}_{\gY}\big[
					\RR{\mathcal{U}}_{u_0}(z)-\Gamma_{z,w}\big(
					\RR{\mathcal{U}}_{u_0}(w)
					\big)\big]\big\Vert}{d(z,w)^{\gamma_1}}+\delta^{1-\kappa}\big(
				[\ \gI\ ]_{\overline{{O}_{\delta+r}^r}}\big)^2\\&+\delta^{\frac{1-\kappa}{2}}[\ \gI\ ]_{\overline{{O}_{\delta+r}^r}}\sup_{w\in O^{r}_{\delta+r}}\sum_{i=1}^{d}\Vert w\Vert_{P_{r}}\Vert\mathcal{U}^{\gXXXX}_{u_0}(w)\Vert+\big(\sum_{i=1}^{d}\sup_{w\in O^{r}_{\delta+r}}\Vert w\Vert_{P_{r}}\Vert\mathcal{U}^{\gXXXX}_{u_0}(w)\big)^{\gamma_1}
			\end{split}
		\end{align}
		\textbf{Case} $a=\gIb$:
		By definition,
		\begin{align}\label{ZZXZX}
			\begin{split}
				&\mathrm{Pr}_{\gIb}\big[
				\big(\Sigma(\RR{\mathcal{U}}_{u_0}) \tilde{\star} \rb
				\big)(z)-\Gamma_{z,w}\big(
				\left(
				\Sigma(\RR{\mathcal{U}}_{u_0}) \tilde{\star} \rb
				\right)(w)
				\big)\big]\\&= \Sigma^{\prime}(\mathcal{U}^{\gY}_{u_0}(z))\Big[\Sigma(\mathcal{U}^{\gY}_{u_0}(z))-\Sigma(\mathcal{U}^{\gY}_{u_0}(w))\Big]+\Sigma(\mathcal{U}^{\gY}_{u_0}(w))\Big[\Sigma^{\prime}(\mathcal{U}^{\gY}_{u_0}(z))-\Sigma^{\prime}(\mathcal{U}^{\gY}_{u_0}(w))\Big].
			\end{split}
		\end{align}
		Let $\widetilde{\Sigma}$ denote either $\Sigma$ or $\Sigma'$. In either case, we have
		\begin{align}\label{LLMOLPa}
			\begin{split}
				&\widetilde{\Sigma}(\mathcal{U}^{\gY}_{u_0}(z))
				-\widetilde{\Sigma}(\mathcal{U}^{\gY}_{u_0}(w))
				\\
				&=
				\int_{0}^{1}
				\widetilde{\Sigma}^{\prime}\big(
				\alpha\mathcal{U}^{\gY}_{u_0}(z)
				+(1-\alpha)\mathcal{U}^{\gY}_{u_0}(w)
				\big)
				\mathrm{Pr}_{\gY}\big[
				\RR{\mathcal{U}}_{u_0}(z)-\Gamma_{z,w}\big(
				\RR{\mathcal{U}}_{u_0}(w)
				\big)
				\big]
				\,\mathrm{d}\alpha
				\\
				&\quad+
				\int_{0}^{1}
				\widetilde{\Sigma}^{\prime}\big(
				\alpha\mathcal{U}^{\gY}_{u_0}(z)
				+(1-\alpha)\mathcal{U}^{\gY}_{u_0}(w)
				\big)
				\big[
				\eta_1(z,w)+\eta_2(z,w)
				\big]
				\,\mathrm{d}\alpha .
			\end{split}
		\end{align}
		Therefore, from the assumptions on $\Sigma$ and since $\gamma_1<2-2\kappa$, we conclude from \eqref{ZZXZX} that
		\begin{align}\label{A33}
			\begin{split}
				&\sup_{\substack{
						z,w \in O_{\delta+r}^{r},\ z\neq w\\
						d(z,w)\leq\Vert z,w\Vert_{P_r}}}
				\frac{
					\Vert z,w \Vert_{P_{r}}^{\gamma_1}
					\big\Vert\mathrm{Pr}_{\gIb}\big[
					\big(\Sigma(\RR{\mathcal{U}}_{u_0}) \tilde{\star} \rb\big)(z)
					-\Gamma_{z,w}\big(
					\big(\Sigma(\RR{\mathcal{U}}_{u_0}) \tilde{\star} \rb\big)(w)
					\big)
					\big]\big\Vert
				}{
					d(z,w)^{\gamma_1-1+\kappa}
				}
				\\
				&\lesssim
				\delta^{\frac{1-\kappa}{2}}	\sup_{\substack{
						z,w \in O_{\delta+r}^{r},\ z\neq w\\
						d(z,w)\leq\Vert z,w\Vert_{P_r}}}
				\frac{
					\Vert z,w \Vert_{P_{r}}^{\gamma_1}
					\big\Vert \mathrm{Pr}_{\gY}\big[
					\RR{\mathcal{U}}_{u_0}(z)-\Gamma_{z,w}\big(
					\RR{\mathcal{U}}_{u_0}(w)
					\big)
					\big]\big\Vert
				}{
					d(z,w)^{\gamma_1}
				}
				\\&+\delta^{1-\kappa}[\ \gI \ ]_{\overline{{O}_{\delta+r}^r}}
				+\delta^{\frac{1-\kappa}{2}}
				\sum_{i=1}^{d}\sup_{w\in O_{\delta+r}^r}
				\Vert w\Vert_{P_s}\Vert \mathcal{U}^{\gXXXX}_{u_0}(w)\Vert .
			\end{split}
		\end{align} 
		\textbf{Case $a=\gXXXX\rb$}: By definition,
		\begin{align*}
			&\mathrm{Pr}_{\gXXXX\rb}\big[
			\big(\Sigma(\RR{\mathcal{U}}_{u_0}) \tilde{\star} \rb\big)(z)
			-\Gamma_{z,w}\big(
			\big(\Sigma(\RR{\mathcal{U}}_{u_0}) \tilde{\star} \rb\big)(w)
			\big)
			\big]\\&=	\mathcal{U}^{\gXXXX}_{u_0}(z)
			\big[
			\Sigma^{\prime}(\mathcal{U}^{\gY}_{u_0}(z))
			-
			\Sigma^{\prime}(\mathcal{U}^{\gY}_{u_0}(w))
			\big]+
			\Sigma^{\prime}(\mathcal{U}^{\gY}_{u_0}(w))
			\big[
			\mathcal{U}^{\gXXXX}_{u_0}(z)
			-
			\mathcal{U}^{\gXXXX}_{u_0}(w)
			\big]
		\end{align*}	
		
		Also,
		\begin{align}\label{BBNMNa}
			\begin{split}
				&\big\Vert\mathrm{Pr}_{\gXXXX\rb}\big[
				\big(\Sigma(\RR{\mathcal{U}}_{u_0}) \tilde{\star} \rb\big)(z)
				-\Gamma_{z,w}\big(
				\big(\Sigma(\RR{\mathcal{U}}_{u_0}) \tilde{\star} \rb\big)(w)
				\big)
				\big]\big\Vert
				\\
				&\leq 
				\big\Vert
				\mathcal{U}^{\gXXXX}_{u_0}(z)
				\big[
				\Sigma^{\prime}(\mathcal{U}^{\gY}_{u_0}(z))
				-
				\Sigma^{\prime}(\mathcal{U}^{\gY}_{u_0}(w))
				\big]\big\Vert+
				\big\Vert
				\Sigma^{\prime}(\mathcal{U}^{\gY}_{u_0}(w))
				\big[
				\mathcal{U}^{\gXXXX}_{u_0}(z)
				-
				\mathcal{U}^{\gXXXX}_{u_0}(w)
				\big]
				\big\Vert
				\\
				&\leq
				\big\Vert
				\mathcal{U}^{\gXXXX}_{u_0}(z)
				\big[
				\Sigma^{\prime}(\mathcal{U}^{\gY}_{u_0}(z))
				-
				\Sigma^{\prime}\big(
				\mathcal{U}^{\gY}_{u_0}(w)+\eta_1(z,w)+\eta_2(z,w)
				\big)
				\big]
				\big\Vert
				\\
				&\quad+
				\big\Vert
				\mathcal{U}^{\gXXXX}_{u_0}(z)
				\big[
				\Sigma^{\prime}\big(
				\mathcal{U}^{\gY}_{u_0}(w)+\eta_1(z,w)+\eta_2(z,w)
				\big)
				-
				\Sigma^{\prime}\big(
				\mathcal{U}^{\gY}_{u_0}(w)+\eta_2(z,w)
				\big)
				\big]
				\big\Vert
				\\
				&\quad+
				\big\Vert
				\mathcal{U}^{\gXXXX}_{u_0}(z)
				\big[
				\Sigma^{\prime}\big(
				\mathcal{U}^{\gY}_{u_0}(w)+\eta_2(z,w)
				\big)
				-
				\Sigma^{\prime}(\mathcal{U}^{\gY}_{u_0}(w))
				\big]
				\big\Vert
				\\
				&\quad+
				\big\Vert
				\Sigma^{\prime}(\mathcal{U}^{\gY}_{u_0}(w))
				\big[
				\mathcal{U}^{\gXXXX}_{u_0}(z)
				-
				\mathcal{U}^{\gXXXX}_{u_0}(w)
				\big]
				\big\Vert .
			\end{split}
		\end{align}
		We now estimate the terms on the right-hand side of \eqref{BBNMNa} as follows:
		\begin{enumerate}
			\item
			For the first term, using the identity
			$AB=(AB^{\frac{1}{\gamma_1}})B^{\frac{\gamma_1-1}{\gamma_1}}$
			and the assumptions on $\Sigma$, we obtain
			\begin{align*}
				&\big\Vert
				\mathcal{U}^{\gXXXX}_{u_0}(z)
				\big[
				\Sigma^{\prime}(\mathcal{U}^{\gY}_{u_0}(z))
				-
				\Sigma^{\prime}\big(
				\mathcal{U}^{\gY}_{u_0}(w)+\eta_1(z,w)+\eta_2(z,w)
				\big)
				\big]
				\big\Vert
				\\
				&\lesssim
				\Vert \mathcal{U}^{\gXXXX}_{u_0}(z)\Vert
				\big\Vert
				\mathrm{Pr}_{\gY}\big[
				\RR{\mathcal{U}}_{u_0}(z)
				-\Gamma_{z,w}\big(
				\RR{\mathcal{U}}_{u_0}(w)
				\big)
				\big]
				\big\Vert^{\frac{\gamma_1-1}{\gamma_1}}.
			\end{align*}
			
			\item
			For the second term, using the boundedness of $\Sigma$, $\Sigma^{\prime}$, and $\Sigma^{\prime\prime}$, we have
			\begin{align*}
				&\big\Vert
				\mathcal{U}^{\gXXXX}_{u_0}(z)
				\big[
				\Sigma^{\prime}\big(
				\mathcal{U}^{\gY}_{u_0}(w)+\eta_1(z,w)+\eta_2(z,w)
				\big)
				-
				\Sigma^{\prime}\big(
				\mathcal{U}^{\gY}_{u_0}(w)+\eta_2(z,w)
				\big)
				\big]
				\big\Vert
				\\
				&\lesssim
				\Vert \mathcal{U}^{\gXXXX}_{u_0}(z)\Vert
				\Vert\Pi_z(\gI)(w)\Vert .
			\end{align*}
			
			\item
			For the third term, using the boundedness of $\Sigma$, $\Sigma^{\prime}$, and $\Sigma^{\prime\prime}$ together with an interpolation argument yields
			\begin{align*}
				&\big\Vert
				\mathcal{U}^{\gXXXX}_{u_0}(z)
				\big[
				\Sigma^{\prime}\big(
				\mathcal{U}^{\gY}_{u_0}(w)+\eta_2(z,w)
				\big)
				-
				\Sigma^{\prime}(\mathcal{U}^{\gY}_{u_0}(w))
				\big]
				\big\Vert
				\\
				&\lesssim
				\Vert \mathcal{U}^{\gXXXX}_{u_0}(z)\Vert
				\Vert\eta_{2}(z,w)\Vert^{\gamma_1-1}.
			\end{align*}
			
			\item
			Thanks to the boundedness of $\Sigma^{\prime}$, we have
			\begin{align*}
				&\big\Vert
				\Sigma^{\prime}(\mathcal{U}^{\gY}_{u_0}(w))
				\big[
				\mathcal{U}^{\gXXXX}_{u_0}(z)
				-
				\mathcal{U}^{\gXXXX}_{u_0}(w)
				\big]
				\big\Vert
				\\
				&\lesssim
				\Vert
				\mathcal{U}^{\gXXXX}_{u_0}(z)
				-
				\mathcal{U}^{\gXXXX}_{u_0}(w)
				\Vert .
			\end{align*}
		\end{enumerate}
		Putting these estimates together, and keeping in mind that $\gamma_1<2-2\kappa$, we conclude that
		\begin{align}\label{NNNNN}
			\begin{split}
				&\sum_{i=1}^{d}\sup_{\substack{
						z,w \in O_{\delta+r}^{r},\ z\neq w\\
						d(z,w)\leq\Vert z,w\Vert_{P_r}}}
				\frac{\Vert z,w \Vert_{P_{r}}^{\gamma_1}\big\Vert \mathrm{Pr}_{\gXXXX\rb}\big[
					\big(\Sigma(\RR{\mathcal{U}}_{u_0}) \tilde{\star} \rb
					\big)(z)-\Gamma_{z,w}\big(
					\left(
					\Sigma(\RR{\mathcal{U}}_{u_0}) \tilde{\star} \rb
					\right)(w)
					\big)\big]\big\Vert}{d(z,w)^{\gamma_1-1}}\\&\lesssim \Big[\sum_{i=1}^{d}\sup_{w\in O_{\delta+r}^r } \Vert w\Vert_{P_r}\Vert \mathcal{U}^{\gXXXX}(w)\Vert\Big]\bigg[\sup_{\substack{
						z,w \in O_{\delta+r}^{r},\ z\neq w\\
						d(z,w)\leq\Vert z,w\Vert_{P_r}}}
				\frac{\Vert z,w \Vert_{P_{r}}^{\gamma_1}\big\Vert \mathrm{Pr}_{\gY}\big[
					\RR{\mathcal{U}}_{u_0}(z)-\Gamma_{z,w}\big(
					\RR{\mathcal{U}}_{u_0}(w)
					\big)\big]\big\Vert}{d(z,w)^{\gamma_1}}\bigg]^{\frac{\gamma_1-1}{\gamma_1}}\\&+ \delta^{\frac{1-\kappa}{2}}[\ \gI \ ]_{{O}_{\delta+r}^r}\Big[\sum_{i=1}^{d}\sup_{w\in O_{\delta+r}^{r} } \Vert w\Vert_{P_r}\Vert \mathcal{U}^{\gXXXX}(w)\Vert\Big]+\Big[\sum_{i=1}^{d}\sup_{w\in {O}_{\delta+r}^r} \Vert w\Vert_{P_r}\Vert \mathcal{U}^{\gXXXX}(w)\Vert\Big]^{\gamma_1}\\&+\sum_{i=1}^d\sup_{\substack{
						z,w \in O_{\delta+r}^{r},\ z\neq w\\
						d(z,w)\leq\Vert z,w\Vert_{P_r}}}
				\frac{
					\Vert z,w \Vert_{P_{r}}^{\gamma_1}
					\big\Vert \mathrm{Pr}_{\gXXXX}\big[
					\RR{\mathcal{U}}_{u_0}(z)-\Gamma_{z,w}\big(
					\RR{\mathcal{U}}_{u_0}(w)
					\big)
					\big]\big\Vert
				}{
					d(z,w)^{\gamma_1-1}
				}.
			\end{split}
		\end{align}
		By using the inequality
		\begin{align*}
			AB^{\frac{\gamma_1-1}{\gamma_1}}\lesssim A^{\gamma_1}+B,
		\end{align*}
		it follows from \eqref{NNNNN} that
		\begin{align}\label{A44}
			\begin{split}
				&\sum_{i=1}^{d}\sup_{\substack{
						z,w \in O_{\delta+r}^{r},\ z\neq w\\
						d(z,w)\leq\Vert z,w\Vert_{P_r}}}
				\frac{\Vert z,w \Vert_{P_{r}}^{\gamma_1}\big\Vert \mathrm{Pr}_{\gXXXX\rb}\big[
					\big(\Sigma(\RR{\mathcal{U}}_{u_0}) \tilde{\star} \rb
					\big)(z)-\Gamma_{z,w}\big(
					\left(
					\Sigma(\RR{\mathcal{U}}_{u_0}) \tilde{\star} \rb
					\right)(w)
					\big)\big]\big\Vert}{d(z,w)^{\gamma_1-1}}\\&\lesssim \Big[\sum_{i=1}^{d}\sup_{w\in O_{\delta+r}^{r} } \Vert w\Vert_{P_r}\Vert \mathcal{U}^{\gXXXX}(w)\Vert\Big]^{\gamma_1}+ \sup_{\substack{
						z,w \in O_{\delta+r}^{r},\ z\neq w\\
						d(z,w)\leq\Vert z,w\Vert_{P_r}}}
				\frac{\Vert z,w \Vert_{P_{r}}^{\gamma_1}\big\Vert \mathrm{Pr}_{\gY}\big[
					\RR{\mathcal{U}}_{u_0}(z)-\Gamma_{z,w}\big(
					\RR{\mathcal{U}}_{u_0}(w)
					\big)\big]\big\Vert}{d(z,w)^{\gamma_1}}\\&+\delta^{\frac{1-\kappa}{2}}[\ \gI \ ]_{{O}_{\delta+r}^r}\Big[\sum_{i=1}^{d}\sup_{w\in O_{\delta+r}^{r} } \Vert w\Vert_{P_r}\Vert \mathcal{U}^{\gXXXX}(w)\Vert\Big]+\sum_{i=1}^d\sup_{\substack{
						z,w \in O_{\delta+r}^{r},\ z\neq w\\
						d(z,w)\leq\Vert z,w\Vert_{P_r}}}
				\frac{
					\Vert z,w \Vert_{P_{r}}^{\gamma_1}
					\big\Vert \mathrm{Pr}_{\gXXXX}\big[
					\RR{\mathcal{U}}_{u_0}(z)-\Gamma_{z,w}\big(
					\RR{\mathcal{U}}_{u_0}(w)
					\big)
					\big]\big\Vert
				}{
					d(z,w)^{\gamma_1-1}
				}.
			\end{split}
		\end{align}
		\textbf{Step 2.}\label{Step2222}  Now we establish a clean estimate for
		\[
		\triplenorm{\mathbf{R}^+_r\left(\Sigma(\RR{\mathcal{U}}_{u_0}) \tilde{\star} \rb\right)}_{\gamma_0,-1-\kappa,O_{\delta+r}^r}^{(r)}.
		\]
		First,	we now proceed to estimate
		\begin{align*}
			\left(
			\sum_{i=1}^{d}
			\sup_{w\in O_{\delta+r}^r}
			\Vert w\Vert_{P_r}\Vert \mathcal{U}^{\gXXXX}(w)\Vert
			\right)^{\gamma_1}.
		\end{align*}
		For $w=(t,x)\in O_{\delta+r}^{r}$, we set
		\begin{align}\label{TTyyta}
			\begin{split}
				\mathcal{P}_{1}(w,\delta+r)
				&=
				\sup_{\substack{
						z \in O_{\delta+r}^r\\
						d(w,z)\leq\frac{\Vert w\Vert_{P_r}}{2}
				}}
				\big\Vert 
				\mathcal{U}^{\gY}_{u_0}(z)
				-\mathcal{U}^{\gY}_{u_0}(w)
				-\eta_1(z,w)
				\big\Vert ,
				\\
				\mathcal{P}_{2}(w,\delta+r)
				&=
				\sup_{\substack{
						z\in O_{\delta+r}^r,\ z\neq w\\
						d(w,z)\leq\frac{\Vert w\Vert_{P_r}}{2}
				}}
				\frac{
					\big\Vert 
					\mathrm{Pr}_{\gY}\big[
					\RR{\mathcal{U}}_{u_0}(z)-\Gamma_{z,w}
					\big(
					\RR{\mathcal{U}}_{u_0}(w)
					\big)
					\big]\big\Vert
				}{
					d(z,w)^{\gamma_1}
				},
				\\
				R(w,\delta)
				&=
				\min\bigg\{
				\frac{\Vert w\Vert_{P_r}}{2},
				\left[
				\frac{\mathcal{P}_{1}(w,\delta+r)}
				{\mathcal{P}_{2}(w,\delta+r)}
				\right]^{\frac{1}{\gamma_1}}
				\bigg\}\footnotemark.
			\end{split}
		\end{align}
		\footnotetext{We assume that
			$\mathcal{P}_{1}(w,\delta+r)$ and $\mathcal{P}_{2}(w,\delta+r)$ are non-zero.
			The degenerate cases lead to the same final estimate.}
		For each $1\leq j\leq d$, we choose
		$\tilde{w}_{j}=(t,\tilde{x})\in O_{\delta+r}^{\,r}$ such that
		\begin{align*}
			X^{k}(\tilde{w}_{j})=
			\begin{cases}
				X^{j}(w)-R(w,\delta)\dfrac{\mathcal{U}^{\gXXXXX}_{u_0}(w)}
				{\Vert \mathcal{U}^{\gXXXXX}_{u_0}(w)\Vert}\footnotemark,
				& k=j,
				\\[6pt]
				X^{k}(w),
				& \text{otherwise}.
			\end{cases}
		\end{align*}
		\footnotetext{We assume that
			$\Vert \mathcal{U}^{\gXXXXX}_{u_0}(w)\Vert \neq 0$.
			Otherwise, the final estimate is trivial.}	
		Thus,
		\[
		d(w,\tilde{w}_j)=R(w,\delta)\leq \frac{\Vert w\Vert_{P_r}}{2}.
		\]
		
		From \eqref{Bnassssw}, we obtain
		\[
		-\eta_2(\tilde{w}_{j},w)=\sum_{i=1}^{d}\mathcal{U}^{\gXXXX}_{u_0}(w)\Pi_{\tilde{w}_j}(\gXXXX)(w)
		=
		R(w,\delta)\Vert \mathcal{U}^{\gXXXXX}_{u_0}(w)\Vert.
		\]
		Therefore, together with \eqref{NMasz}, and \eqref{TTyyta}, we conclude that
		\begin{align*}
			\Vert \mathcal{U}^{\gXXXXX}_{u_0}(w)\Vert
			\leq
			\mathcal{P}_{2}(w,\delta+r)\,R(w,\delta)^{\gamma_1-1}
			+
			\frac{\mathcal{P}_{1}(w,\delta+r)}{R(w,\delta)}.
		\end{align*}
		and hence, by substituting the definition of $R(w,\delta)$, we obtain:
		\begin{align*}
			&\Vert \mathcal{U}^{\gXXXXX}_{u_0}(w)\Vert
			\\&\leq
			\big[\mathcal{P}_{2}(w,\delta+r)\big]^{\frac{1}{\gamma_1}}
			\big[\mathcal{P}_{1}(w,\delta+r)\big]^{\frac{\gamma_1-1}{\gamma_1}}
			+
			\max\bigg\{
			\big[\mathcal{P}_{2}(w,\delta+r)\big]^{\frac{1}{\gamma_1}}
			\big[\mathcal{P}_{1}(w,\delta+r)\big]^{\frac{\gamma_1-1}{\gamma_1}},
			\frac{2\mathcal{P}_{1}(w,\delta+r)}{\Vert w\Vert_{P_r}}
			\bigg\}.
		\end{align*}
		Thus
		\begin{align*}
			&	\Vert w\Vert_{P_r}\Vert \mathcal{U}^{\gXXXXX}_{u_0}(w)\Vert
			\leq
			\big[\Vert w\Vert_{P_r}^{\gamma_1}\mathcal{P}_{2}(w,\delta+r)\big]^{\frac{1}{\gamma_1}}
			\big[\mathcal{P}_{1}(w,\delta+r)\big]^{\frac{\gamma_1-1}{\gamma_1}} \\
			&\quad +
			\max\bigg\{
			\big[\Vert w\Vert_{P_r}^{\gamma_1}\mathcal{P}_{2}(w,\delta+r)\big]^{\frac{1}{\gamma_1}}
			\big[\mathcal{P}_{1}(w,\delta+r)\big]^{\frac{\gamma_1-1}{\gamma_1}},
			2\mathcal{P}_{1}(w,\delta+r)
			\bigg\}.
		\end{align*}
		Consequently
		\begin{align}\label{TTT0}
			\begin{split}
				&\big(\sum_{i=1}^{d}\sup_{w\in O_{\delta+r}^r } \Vert w\Vert_{P_r}\Vert \mathcal{U}^{\gXXXX}_{u_0}(w)\Vert\big)^{\gamma_1}\\&
				\lesssim 
				\max\bigg\{
				\big[\sup_{w\in O_{\delta+r}^r}\Vert w\Vert_{P_r}^{\gamma_1}\mathcal{P}_{2}(w,\delta+r)\big]
				\big[\sup_{w\in O_{\delta+r}^r}\mathcal{P}_{1}(w,\delta+r)\big]^{\gamma_1-1} ,
				\big[\sup_{w\in O_{\delta+r}^r}\mathcal{P}_{1}(w,\delta+r)\big]^{\gamma_1} 
				\bigg\}.
			\end{split}
		\end{align}
		By the definition, we have
		\begin{align}\label{TTT1}
			\sup_{w\in O_{\delta+r}^r}\mathcal{P}_{1}(w,\delta+r)
			\lesssim
			\sup_{w\in O_{\delta+r}^r}\Vert \mathcal{U}^{\gY}_{u_0}(w)\Vert
			+
			\delta^{\frac{1-\kappa}{2}}[\ \gI \ ]_{\overline{{O}_{\delta+r}^r}},
		\end{align}
		and
		\begin{align}\label{TTT2}
			\sup_{w\in O_{\delta+r}^r}\Vert w\Vert_{P_r}^{\gamma_1}\mathcal{P}_{2}(w,\delta+r)
			\lesssim
			\sup_{\substack{
					z ,w\in O_{\delta+r}^r,\ z\neq w\\
					d(w,z)\leq \Vert z,w\Vert_{P_r}
			}}
			\frac{
				\Vert z,w \Vert_{P_{r}}^{\gamma_1}
				\big\Vert \mathrm{Pr}_{\gY}\big[
				\RR{\mathcal{U}}_{u_0}(z)-\Gamma_{z,w}\big(
				\RR{\mathcal{U}}_{u_0}(w)
				\big)\big]\big\Vert
			}{
				d(z,w)^{\gamma_1}
			}.
		\end{align}
		Thus, from \eqref{TTT0}--\eqref{TTT2}, we conclude that
		\begin{align}\label{A22}
			\begin{split}
				&\big[\sum_{i=1}^{d}\sup_{w\in O_{\delta+r}^r } 
				\| w\|_{P_r}\|\mathcal{U}^{\gXXXX}_{u_0}(w)\|\big]^{\gamma_1}\\
				&\lesssim \max\Bigg\{
				\Big[
				\sup_{\substack{
						z ,w\in O_{\delta+r}^r,\ z\neq w\\
						d(w,z)\leq \Vert z,w\Vert_{P_r}
				}}
				\frac{\| z,w \|_{P_{r}}^{\gamma_1}
					\big\| \mathrm{Pr}_{\gY}\big[
					\RR{\mathcal{U}}_{u_0}(z)
					-\Gamma_{z,w}\big(
					\RR{\mathcal{U}}_{u_0}(w)
					\big)\big]\big\|}
				{d(z,w)^{\gamma_1}}
				\Big]\\
				&\quad \times
				\Big[
				\Big(\sup_{w\in O_{\delta+r}^r}
				\|\mathcal{U}^{\gY}_{u_0}(w)\|\Big)^{\gamma_1-1}
				+ \delta^{\frac{(1-\kappa)(\gamma_1-1)}{2}}
				\big([\ \gI\ ]_{\overline{{O}_{\delta+r}^r}}\big)^{\gamma_1-1}
				\Big]\\
				&\quad ,
				\Big(\sup_{w\in O_{\delta+r}^r}
				\|\mathcal{U}^{\gY}_{u_0}(w)\|\Big)^{\gamma_1}
				+ \delta^{\frac{(1-\kappa)\gamma_1}{2}}
				\big([\ \gI\ ]_{\overline{{O}_{\delta+r}^r}}\big)^{\gamma_1}
				\Bigg\}.
			\end{split}
		\end{align}
		Note that for $a\in \mathcal{A}^-=\{\rb, \gIb, \gXXXX\rb,\, i=1,\ldots,d\}$, we have
		\begin{align*}
			\sup_{z\in O_{\delta+r}^r\setminus P_r}
			\Vert z\Vert_{P_r}^{|a|_h+1+\kappa}
			\big\Vert \mathrm{Pr}_{a}\big[
			\mathbf{R}^{+}_{r}\left(\Sigma(\RR{\mathcal{U}}_{u_0}) \tilde{\star} \rb\right)(z)
			\big]\big\Vert
			\lesssim 1+\triplenorm{\mathbf{R}^+_r\RR{\mathcal{U}}_{u_0}}_{\gamma_1,0,O_{\delta+r}^r}^{(r)}.
		\end{align*}
		Recall that $0<\delta \leq 1$ and $\gamma_1=\gamma_0+1+\kappa$;  therefore, together with \eqref{A11}, \eqref{A33} and \eqref{A44}, we deduce that
		\begin{align}
			\begin{split}
				&\triplenorm{\mathbf{R}^+_r\left(\Sigma(\RR{\mathcal{U}}_{u_0}) \tilde{\star} \rb
					\right)}_{\gamma_0,-1-\kappa,O_{\delta+r}^r}^{(r)}\lesssim \Big[\sum_{i=1}^{d}\sup_{w\in O_{\delta+r}^r } \Vert w\Vert_{P_r}\Vert \mathcal{U}^{\gXXXX}(w)\Vert\Big]^{\gamma_1}\\&+\left[1+\delta^{\frac{1-\kappa}{2}}[\ \gI \ ]_{\overline{{O}_{\delta+r}^r}}\right]\triplenorm{\mathbf{R}^+_r\RR{\mathcal{U}}_{u_0}}_{\gamma_1,0,O_{\delta+r}^r}^{(r)}+\delta^{1-\kappa}\left[[\ \gI \ ]_{\overline{{O}_{\delta+r}^r}}+\big([\ \gI \ ]_{\overline{{O}_{\delta+r}^r}}\big)^2\right]+1.
			\end{split}
		\end{align}
		Consequently, by \eqref{A22}, we obtain
		\begin{align}\label{A6}
			\begin{split}
				&\triplenorm{\mathbf{R}^+_r\left(\Sigma(\RR{\mathcal{U}}_{u_0}) \tilde{\star} \rb
					\right)}_{\gamma_0,-1-\kappa,O_{\delta+r}^r}^{(r)}\lesssim \\&+\triplenorm{\mathbf{R}^+_r\RR{\mathcal{U}}_{u_0}}_{\gamma_1,0,O_{\delta+r}^r}^{(r)}\Big(\sup_{w\in O_{\delta+r}^r}
				\|\mathcal{U}^{\gY}_{u_0}(w)\|\Big)^{\gamma_1-1}+\Big(\sup_{w\in O_{\delta+r}^r}
				\|\mathcal{U}^{\gY}_{u_0}(w)\|\Big)^{\gamma_1}\\&+\left[1+\delta^{\frac{1-\kappa}{2}}[\ \gI \ ]_{\overline{{O}^{r}_{\delta+r}}}\right]\triplenorm{\mathbf{R}^+_r\RR{\mathcal{U}}_{u_0}}_{\gamma_1,0,O_{\delta+r}^r}^{(r)}\\&+\delta^{1-\kappa}\left[[\ \gI \ ]_{\overline{{O}_{\delta+r}^r}}+\big([\ \gI \ ]_{\overline{{O}^{r}_{\delta+r}}}\big)^2\right]+ 
				\delta^{\frac{(1-\kappa)\gamma_1}{2}}\big([\ \gI\ ]_{\overline{{O}_{\delta+r}^r}}\big)^{\gamma_1}+1,
			\end{split}
		\end{align}
		where we used the fact that
		\begin{align*}
			b^{\gamma_1-1}\lesssim 1+b,
			\qquad b\geq 0.
		\end{align*}
		\textbf{Step 3.}\label{Step3333} We now finalize the proof. Recall that \eqref{CZ1} and \eqref{CZ22} hold. By \eqref{A-1}, \eqref{fixed-point}, and \eqref{A6}, one can find a constant $C\geq 1$ such that
		\begin{align}\label{JMMAs}
			\begin{split}
				&\triplenorm{\mathbf{R}^+_r\RR{\mathcal{U}}_{u_0}}_{\gamma_1,0,O_{\delta+r}^r}^{(r)}\leq 	C\delta^{\frac{1-\kappa}{2}}\left[ \|\mathcal{Z}\|_{\overline{O_{r+1}^{r-1}}}^2
				+1\right]
				\triplenorm{\mathbf{R}^+_r \left(\Sigma(\RR{\mathcal{U}}_{u_0}) \tilde{\star} \rb\right)}
				_{\gamma_0,\,-1-\kappa,\,O_{\delta+r}^r}^{(r)}
				\\&+	C|\mu| \delta^{\frac{1-\kappa}{2}}
				\left[
				\|\mathcal{Z}\|_{\overline{O_{r+1}^{r-1}}}^2+1
				\right]
				\triplenorm{\mathbf{R}^+_r\RR{\mathcal{U}}_{u_0}}_
				{\gamma_1,0,O_{\delta+r}^r}^{(r)}	+C\|\mathcal{U}^{\gY}_{u_0}(r,\cdot)\|_{L^{\infty}(\mathbb{T}^d)}\\&\leq C\|\mathcal{U}^{\gY}_{u_0}(r,\cdot)\|_{L^{\infty}(\mathbb{T}^d)}+M_{1}(\mathcal{Z},r,\delta+r)\\&+\left[ M_{2}(\mathcal{Z},r,\delta+r)+M_{3}(\mathcal{Z},r,\delta+r)\Big(\sup_{w\in O_{\delta+r}^r}
				\|\mathcal{U}^{\gY}_{u_0}(w)\|\Big)^{\gamma_1-1}\right]\triplenorm{\mathbf{R}^+_r\RR{\mathcal{U}}_{u_0}}_{\gamma_1,0,O_{\delta+r}^r}^{(r)}\\&+M_{3}(\mathcal{Z},r,\delta+r)\Big(\sup_{w\in O_{\delta+r}^r}
				\|\mathcal{U}^{\gY}_{u_0}(w)\|\Big)^{\gamma_1},
			\end{split}
		\end{align}
		where for some $\widetilde{C}\geq 1$
		\begin{align*}
			&M_{1}(\mathcal{Z},r,\delta+r)\\&=\widetilde{C}\delta^{\frac{1-\kappa}{2}}\left[  \|\mathcal{Z}\|_{\overline{O_{r+1}^{r-1}}}^2
			+1\right]\left( \delta^{1-\kappa}\left[ [\ \gI \ ]_{\overline{{O}_{\delta+r}^r}}+\big([\ \gI\ ]_{\overline{{O}_{\delta+r}^r}}\big)^2\right]+ 
			\delta^{\frac{(1-\kappa)\gamma_1}{2}}\big([\ \gI\ ]_{\overline{{O}_{\delta+r}^r}}\big)^{\gamma_1}+1\right),
			\\&M_{2}(\mathcal{Z},r,\delta+r)=\widetilde{C}\delta^{\frac{1-\kappa}{2}}\left[  \|\mathcal{Z}\|_{\overline{O_{r+1}^{r-1}}}^2
			+1\right]\left[1+|\mu|+\delta^{\frac{1-\kappa}{2}}[\ \gI\ ]_{\overline{{O}_{\delta+r}^r}}\right],\\
			&M_{3}(\mathcal{Z},r,\delta+r)=\widetilde{C}\delta^{\frac{1-\kappa}{2}}\left[  \|\mathcal{Z}\|_{\overline{O_{r+1}^{r-1}}}^2
			+1\right].
		\end{align*}
		Clearly, for each $r,\delta$ satisfying \eqref{BBBBB}, the following inequalities hold:
		\begin{align*}
			M_{1}(\mathcal{Z},r,\delta+r)
			&\leq
			\widetilde{C}\delta^{\frac{1-\kappa}{2}}
			\left[\|\mathcal{Z}\|_{\overline{O_{s+1+T}^{s-1}}}^2+1\right]
			\left[
			1
			+\big([\ \gI\ ]_{\overline{O_{s+T}^s}}\big)^2
			\right] \\
			&=: \delta^{\frac{1-\kappa}{2}}
			\widetilde{M}_{1}(\mathcal{Z},s,T),\\[0.3em]
			M_{2}(\mathcal{Z},r,\delta+r)
			&\leq
			\widetilde{C}\delta^{\frac{1-\kappa}{2}}
			\left[\|\mathcal{Z}\|_{\overline{O_{s+1+T}^{s-1}}}^2+1\right]
			\left[1+|\mu|+[\ \gI\ ]_{\overline{O_{s+T}^s}}\right] \\
			&=: \delta^{\frac{1-\kappa}{2}}
			\widetilde{M}_{2}(\mathcal{Z},s,T),\\[0.3em]
			M_{3}(\mathcal{Z},r,\delta+r)
			&\leq
			\widetilde{C}\delta^{\frac{1-\kappa}{2}}
			\left[\|\mathcal{Z}\|_{\overline{O_{s+1+T}^{s-1}}}^2+1\right] \\
			&=: \delta^{\frac{1-\kappa}{2}}
			\widetilde{M}_{3}(\mathcal{Z},s,T).
		\end{align*}
		Consequently, it follows from \eqref{JMMAs} that
		\begin{align}\label{IAkwewe}
			\begin{split}
				&\triplenorm{\mathbf{R}^+_r\RR{\mathcal{U}}_{u_0}}_{\gamma_1,0,O_{\delta+r}^r}^{(r)}\\&\leq C\|\mathcal{U}^{\gY}_{u_0}(r,\cdot)\|_{L^{\infty}(\mathbb{T}^d)}+\delta^{\frac{1-\kappa}{2}}
				\widetilde{M}_{1}(\mathcal{Z},s,T)\\&+\left[ \delta^{\frac{1-\kappa}{2}}
				\widetilde{M}_{2}(\mathcal{Z},s,T)+\delta^{\frac{1-\kappa}{2}}
				\widetilde{M}_{3}(\mathcal{Z},s,T)\big(\sup_{w\in O_{s+T}^s}
				\|\mathcal{U}^{\gY}_{u_0}(w)\|\big)^{\gamma_1-1}\right]\triplenorm{\mathbf{R}^+_r\RR{\mathcal{U}}_{u_0}}_{\gamma_1,0,O_{\delta+r}^r}^{(r)}\\&+\delta^{\frac{1-\kappa}{2}}
				\widetilde{M}_{3}(\mathcal{Z},s,T)\Big(\sup_{w\in O_{\delta+r}^r}
				\|\mathcal{U}^{\gY}_{u_0}(w)\|\Big)^{\gamma_1},
			\end{split}
		\end{align}
		For an arbitrary $\beta\in(0,1)$, we choose $\delta_\beta\in(0,T]$ such that
		\begin{align*}
			&\delta_\beta^{\frac{1-\kappa}{2}}
			\widetilde{M}_{2}(\mathcal{Z},s,T)+\delta_\beta^{\frac{1-\kappa}{2}}
			\widetilde{M}_{3}(\mathcal{Z},s,T)\big(\sup_{w\in O_{s+T}^s}
			\|\mathcal{U}^{\gY}_{u_0}(w)\|\big)^{\gamma_1-1}\\&=\widetilde{C}
			\delta_\beta^{\frac{1-\kappa}{2}}\left[ \|\mathcal{Z}\|_{\overline{O_{s+1+T}^{s-1}}}^2
			+1\right]\left[1+|\mu|+[\ \gI \ ]_{\overline{{O}^{s}_{s+T}}}+\big(\sup_{w\in O_{s+T}^s}
			\|\mathcal{U}^{\gY}_{u_0}(w)\|\big)^{\gamma_1-1}\right]\leq \beta.
		\end{align*}
		In particular, we make the following explicit choice of $\delta_\beta$,
		\begin{align}\label{NMaz}
			\delta_\beta
			:=
			\min\Bigg\lbrace	\frac{
				\beta^{\frac{2}{1-\kappa}}
			}{
				\widetilde{C}^{\frac{2}{1-\kappa}}\left[\|\mathcal{Z}\|_{\overline{O_{s+1+T}^{s-1}}}^2+1\right]^{\frac{2}{1-\kappa}}
				\left[
				1+|\mu|+[\ \gI\ ]_{\overline{O_{s+T}^s}}
				+\big(\sup_{w\in O_{s+T}^s}
				\|\mathcal{U}^{\gY}_{u_0}(w)\|\big)^{\gamma_1-1}
				\right]^{\frac{2}{1-\kappa}}
			},T\Bigg\rbrace.
		\end{align}
		We aim to apply Lemma~\ref{global}. Let $b < c$ be points in $[s, s+T]$ such that $c-b \leq \delta_\beta$. By definition
		\begin{align}\label{NMasas}
			\begin{split}
				\delta_\beta^{\frac{1-\kappa}{2}}
				\widetilde{M}_{2}(\mathcal{Z},s,T)+\delta_\beta^{\frac{1-\kappa}{2}}
				\widetilde{M}_{3}(\mathcal{Z},s,T)\big(\sup_{w\in O_{s+T}^s}
				\|\mathcal{U}^{\gY}_{u_0}(w)\|\big)^{\gamma_1-1}\leq \beta.
			\end{split}
		\end{align}
		Now, thanks to \eqref{IAkwewe} and \eqref{NMasas}, we obtain 
		\begin{align*}
			\begin{split}
				&\triplenorm{\mathbf{R}^+_{b}\RR{\mathcal{U}}_{u_0}}_{\gamma_1,0,O_{c}^b}^{(b)}\\&\leq 	C\|\mathcal{U}^{\gY}_{u_0}(b,\cdot)\|_{L^{\infty}(\mathbb{T}^d)}+\widetilde{M}_{1}(\mathcal{Z},s,T)+\beta\triplenorm{\mathbf{R}^+_{b}\RR{\mathcal{U}}_{u_0}}_{\gamma_1,0,O_{c}^b}^{(b)}+\widetilde{M}_{3}(\mathcal{Z},s,T)\Big(\sup_{w\in O_{s+T}^s}
				\|\mathcal{U}^{\gY}_{u_0}(w)\|\Big)^{\gamma_1}
			\end{split}
		\end{align*}
		and so
		\begin{align}\label{ASSferg}
			\begin{split}
				&\triplenorm{\mathbf{R}^+_{b}\RR{\mathcal{U}}_{u_0}}_{\gamma_1,0,O_{c}^b}^{(b)}
				\\
				&\leq \frac{1}{1-\beta}\Big[
				C\|\mathcal{U}^{\gY}_{u_0}(b,\cdot)\|_{L^{\infty}(\mathbb{T}^d)}
				+\widetilde{M}_{1}(\mathcal{Z},s,T)
				+\widetilde{M}_{3}(\mathcal{Z},s,T)
				\Big(\sup_{w\in O_{s+T}^s}
				\|\mathcal{U}^{\gY}_{u_0}(w)\|\Big)^{\gamma_1}
				\Big]
				\\
				&\leq \frac{1}{1-\beta}\Big[
				C\sup_{w\in O_{s+T}^s}
				\|\mathcal{U}^{\gY}_{u_0}(w)\|
				+C\|u_0\|_{L^{\infty}(\mathbb{T}^d)}
				+\widetilde{M}_{1}(\mathcal{Z},s,T)
				+\widetilde{M}_{3}(\mathcal{Z},s,T)
				\Big(\sup_{w\in O_{s+T}^s}
				\|\mathcal{U}^{\gY}_{u_0}(w)\|\Big)^{\gamma_1}
				\Big]
				\\
				&=:\Lambda(\beta,u_0,\mathcal{Z}).
			\end{split}
		\end{align}
		Let $\beta=\frac{1}{2}$. Then
		\begin{enumerate}
			\item If $\delta_{\frac{1}{2}}<T$, then from \eqref{NMaz}, \eqref{ASSferg} and Lemma~\ref{global}, we obtain
			\begin{align}\label{SSAS78sasas}
				\begin{split}
					&\triplenorm{\mathbf{R}^+_{s}\RR{\mathcal{U}}_{u_0}}_{\gamma_1,0,O_{s+T}}^{(s)}
					\lesssim
					\left(\frac{1}{\delta_{\frac{1}{2}}}\right)^{\frac{\gamma_1+2}{2}}
					\big( \|\Gamma\|_{\overline{O_{s+T}^s}}\big)^2
					\Lambda(\frac{1}{2},u_0,\mathcal{Z})\\
					&\lesssim P_{1}\Big(\|\mathcal{Z}\|_{\overline{O_{s+1+T}^{s-1}}},[\ \gI\ ]_{\overline{O_{s+T}^s}}, \|\Gamma\|_{\overline{O_{s+T}^s}},\sup_{w\in O_{s+T}^s}
					\|\mathcal{U}^{\gY}_{u_0}(w)\|,\|u_0\|_{L^{\infty}(\mathbb{T}^d)}\Big)
				\end{split}
			\end{align}
			where
			\begin{align*}
				P_{1}(b,c,d,f,g):=&
				d^2
				(1+b^2)^{\frac{\gamma_1+2}{1-\kappa}}
				\big(1+|\mu|+c+f^{\gamma_1-1}\big)^{\frac{\gamma_1+2}{1-\kappa}}
				\\
				&\times
				\big(
				f+g+(1+b^2)(1+c^2+f^{\gamma_1})
				\big).
			\end{align*}
			\item  If $\delta_{\frac{1}{2}}=T$, then from \eqref{ASSferg},
			\begin{align*}
				&\triplenorm{\mathbf{R}^+_{s}\RR{\mathcal{U}}_{u_0}}_{\gamma_1,0,O_{s+T}}^{(s)}\lesssim \tilde{P}_{1}\Big(\|\mathcal{Z}\|_{\overline{O_{s+1+T}^{s-1}}},[\ \gI\ ]_{\overline{O_{s+T}^s}}, \|\Gamma\|_{\overline{O_{s+T}^s}},\sup_{w\in O_{s+T}^s}
				\|\mathcal{U}^{\gY}_{u_0}(w)\|,\|u_0\|_{L^{\infty}(\mathbb{T}^d)}\Big)
			\end{align*}
			where
			\begin{align*}
				\tilde{P}_{1}(b,c,d,f,g)=\big(f+g+(1+b^2)(1+c^2+f^{\gamma_{1}})\big).
			\end{align*}
		\end{enumerate}
		Since $ \|\Gamma\|_{\overline{O_{s+T}^s}}\geq 1$, the bound obtained in \eqref{SSAS78sasas} is also valid when $\delta_{\frac{1}{2}}=T$. This finishes the proof.
	\end{proof}
	\begin{comment}
		content...
		
		For $z=(s_1,x)\in O^{s}_{s+1}$, if $\Vert z\Vert_{P_s}=\sqrt{s_1-s}\geq \sqrt{T_0}$ then clearly we have $z\in O^{s_1-T_0}_{s_1}$. And so for $ a\in \{\gY, \gI, \gXXXX: i=1,..,d\}$
		\begin{align*}
			(T_{0})^{|a|_h}\Vert \RR{\mathcal{U}}_{u_0}(z)\Vert_{a}\leq \triplenorm{\mathbf{R}^+_{s_1-T_0}\RR{\mathcal{U}}_{u_0}}_{\gamma_1,0,O_{s_1}}^{(s_1-T_0)}.
		\end{align*}
		Thus
		\begin{align*}
			\Vert z\Vert_{P_s}\Vert^{|a|_h} \Vert\RR{\mathcal{U}}_{u_0}(z)\Vert_{a}\leq \frac{\triplenorm{\mathbf{R}^+_{s_1-T_0}\RR{\mathcal{U}}_{u_0}}_{\gamma_1,0,O_{s_1}}^{(s_1-T_0)}}{(T_{0})^{|a|_h}}.
		\end{align*}
		
	\end{comment}
	\begin{proof}[\textbf{Proof of Lemma \ref{DCXSZ}}]\label{AAzza1}
		We estimate each of the terms appearing in
		\[
		\triplenorm{
			\mathbf{R}^{+}_{s}F\big(\RR{\mathcal{U}}\big)
			-
			\mathbf{R}^{+}_{s}F\big(\RR{\widetilde{\mathcal{U}}}\big)
		}_{\gamma_1,0,O_{T+s}^{s}}^{(s)}
		\]
		in \eqref{Q00}. Since the proof is somewhat lengthy, we divide it into several
		steps. We first introduce the following notation.
		\begin{align*}
			&\eta_1(z,w):=-\Sigma(\mathcal{U}^{\gY}(w))\Pi_z(\gI)(w), \qquad
			\eta_2(z,w):=-\sum_{i=1}^{d}\mathcal{U}^{\gXXXX}(w)\Pi_z(\gXXXX)(w), \\
			&\widetilde{\eta}_{1}(z,w):=-\Sigma(\widetilde{\mathcal{U}}^{\gY}(w))\Pi_z(\gI)(w), \qquad
			\widetilde{\eta}_{2}(z,w):=-\sum_{i=1}^{d}\widetilde{\mathcal{U}}^{\gXXXX}(w)\Pi_z(\gXXXX)(w),
		\end{align*}
		for all $z,w\in O_{T+s}^{s}$ such that $d(z,w)\leq \Vert z,w\Vert_{P_s}$. Also, throughout the proof, we repeatedly use the following elementary identity without further mention:
		\begin{align*}
			\prod_{i=1}^{m} A_i - \prod_{i=1}^{m} B_i
			= \sum_{j=1}^{m}
			\big(\prod_{i=1}^{j-1} B_i\big)
			(A_j - B_j)
			\big(\prod_{k=j+1}^{m} A_k\big),
			\qquad A_i,B_i\in\mathbb{R}, \ \ \prod_{i=1}^{0}(\cdot)=\prod_{i=m+1}^{m}(\cdot)=1.
		\end{align*}
		Also, $M_{\nu}(b) := \max\{b, b^{\nu}\}$.
		
		\textbf{Step 1.}\label{Step11}
		In this step, we estimate
		\begin{align*}
			\sup_{\substack{z,w \in O_{T+s}^{s}, z\neq w\\ d(z,w)\leq \Vert z,w \Vert_{P_{s}}}}
			\frac{
				\Vert z,w \Vert_{P_{s}}^{\gamma_1}
				\, \big\Vert \mathrm{Pr}_{\gY}\big[
				\big(F(\RR{\mathcal{U}})-F(\RR{\widetilde{\mathcal{U}}})\big)(z)
				-\Gamma_{z,w}\big(\big(F(\RR{\mathcal{U}})-F(\RR{\widetilde{\mathcal{U}}})\big)(w)\big)
				\big]\big\Vert
			}{
				d(z,w)^{\gamma_1}
			}.
		\end{align*}
		This is the most involved term in
		$\triplenorm{
			\mathbf{R}^{+}_{s}F\big(\RR{\mathcal{U}}\big)
			-
			\mathbf{R}^{+}_{s}F\big(\RR{\widetilde{\mathcal{U}}}\big)
		}_{\gamma_1,0,O_{T+s}^{s}}^{(s)}$.
		Note that
		\begin{align}\label{NASBs}
			\begin{split}
				&z,w \in O_{T+s}^{s}, \quad d(z,w)\leq \Vert z,w \Vert_{P_{s}}:\\
				&\Vert\eta_1(z,w)-\widetilde{\eta}_{1}(z,w)\Vert
				\leq \Vert \Pi_z(\gI)(w)\Vert\,
				\Vert \Sigma(\mathcal{U}^{\gY}(w)) - \Sigma(\widetilde{\mathcal{U}}^{\gY}(w))\Vert \\
				&\qquad \lesssim d(z,w)^{1-\kappa}[\ \gI \ ]_{\overline{{O}_{T+s}^s}}\,
				\triplenorm{ \mathbf{R}^{+}_{s}\RR{\mathcal{U}} - \mathbf{R}^{+}_{s}\RR{\widetilde{\mathcal{U}}} }_{\gamma_1,0,O_{T+s}^s}^{(s)},\\[0.5em]
				&\Vert\eta_2(z,w)-\widetilde{\eta}_{2}(z,w)\Vert
				\leq \sum_{i=1}^{d}\Vert \Pi_z(\gXXXX)(w)\Vert\,
				\Vert \mathcal{U}^{\gXXXX}(w)-\widetilde{\mathcal{U}}^{\gXXXX}(w)\Vert \\
				&\qquad \leq \sum_{i=1}^{d}d(z,w)\,
				\Vert \mathcal{U}^{\gXXXX}(w)-\widetilde{\mathcal{U}}^{\gXXXX}(w)\Vert \lesssim\triplenorm{\mathbf{R}^{+}_{s} \RR{\mathcal{U}} - \mathbf{R}^{+}_{s}\RR{\widetilde{\mathcal{U}}} }_{\gamma_1,0,O_{T+s}^s}^{(s)}.
			\end{split}
		\end{align}
		From \eqref{B1478}, \eqref{Hnasssss}, we have :
		\begin{align*}
			&\mathrm{Pr}_{\gY}\big[
			F(\RR{\mathcal{U}})(z)-\Gamma_{z,w}\big(
			F(\RR{\mathcal{U}})(w)
			\big)\big] \\
			&= \underbrace{F(\mathcal{U}^{\gY}(z))
				- F\big(\mathcal{U}^{\gY}(w)+\eta_1(z,w)+\eta_2(z,w)\big)}_{\text{I}(\RR{\mathcal{U}},z,w)} \\
			&\quad + \underbrace{
				F\big(\mathcal{U}^{\gY}(w)+\eta_1(z,w)+\eta_2(z,w)\big)
				- F\big(\mathcal{U}^{\gY}(w)+\eta_2(z,w)\big)
				- F^{\prime}\big(\mathcal{U}^{\gY}(w)\big)\eta_1(z,w)
			}_{\text{II}(\RR{\mathcal{U}},z,w)} \\
			&\quad + \underbrace{
				F\big(\mathcal{U}^{\gY}(w)+\eta_2(z,w)\big)
				- F(\mathcal{U}^{\gY}(w))
				- F^{\prime}(\mathcal{U}^{\gY}(w))\eta_2(z,w)
			}_{\text{III}(\RR{\mathcal{U}},z,w)}
		\end{align*}
		Using this expression, we conclude the following items:
		\begin{itemize}
			\item[1)] By a simple Taylor expansion, we have:
			\begin{align*}
				&\text{I}(\RR{\mathcal{U}},z,w) \\
				&= \int_{0}^{1}
				F'\bigg(
				\alpha \mathcal{U}^{\gY}(z)
				+ (1-\alpha)\big(\mathcal{U}^{\gY}(w)+\eta_1(z,w)+\eta_2(z,w)\big)
				\bigg) 
				\mathrm{Pr}_{\gY}\big[
				\RR{\mathcal{U}}(z)-\Gamma_{z,w}\big(\RR{\mathcal{U}}(w)\big)
				\big]\,\mathrm{d}\alpha.
			\end{align*}
			Together with the assumptions on $F$ and $\Sigma$, we obtain that
			\begin{align*}
				& \Vert\text{I}(\RR{\mathcal{U}},z,w)-\text{I}(\RR{\widetilde{\mathcal{U}}},z,w)\Vert	
				\lesssim \big\Vert \mathrm{Pr}_{\gY}\big[
				(\RR{\mathcal{U}}-\RR{\widetilde{\mathcal{U}}})(z)-\Gamma_{z,w}\big(
				(\RR{\mathcal{U}}-\RR{\widetilde{\mathcal{U}}})(w)
				\big)\big]\big\Vert+\\&\big(\Vert \mathcal{U}^{\gY}(z)-\widetilde{\mathcal{U}}^{\gY}(z)\Vert+\Vert \mathcal{U}^{\gY}(w)-\widetilde{\mathcal{U}}^{\gY}(w)\Vert+\Vert\eta_1(z,w)-\widetilde{\eta}_{1}(z,w)\Vert+\Vert\eta_2(z,w)-\widetilde{\eta}_{2}(z,w)\Vert \big)\\
				&\times \big\Vert \mathrm{Pr}_{\gY}\big[
				\RR{\mathcal{U}}(z)-\Gamma_{z,w}\big(
				\RR{\mathcal{U}}(w)
				\big)\big]\big\Vert.
			\end{align*}
			Therefore, from \eqref{NASBs}, we conclude that
			\begin{align*}
				&\sup_{\substack{z,w \in O_{T+s}^{s},z\neq w\\d(z,w)\leq \Vert z,w \Vert_{P_{s}}}} \frac{\Vert z,w \Vert_{P_{s}}^{\gamma_1} \Vert\text{I}(\RR{\mathcal{U}},z,w)-\text{I}(\RR{\widetilde{\mathcal{U}}},z,w)\Vert }{d(z,w)^{\gamma_1}}	
				\\&\lesssim   \triplenorm{  \mathbf{R}^{+}_{s}\RR{\mathcal{U}} -\mathbf{R}^{+}_{s} \RR{\widetilde{\mathcal{U}}} }_{\gamma_1,0,O_{T+s}^s}^{(s)}\big[\big(1+[\ \gI \ ]_{\overline{{O}_{T+s}^s}}\big)\triplenorm{ \RR{\mathcal{U}}}_{\gamma_1,0,O_{T+s}^s}^{(s)}+1 \big]
			\end{align*}
			\item[2)]  Again, a direct consequence of the Taylor expansion yields
			\begin{align*}
				& \text{II}(\RR{\mathcal{U}},z,w) \\
				&= \int_0^1\int_0^1
				F''\big(\mathcal{U}^{\gY}(w)+\beta\eta_2(z,w)+\alpha\beta\eta_1(z,w)\big)
				\big[\eta_2(z,w)+\alpha\eta_1(z,w)\big]\eta_1(z,w)
				\,\mathrm{d}\beta\,\mathrm{d}\alpha.
			\end{align*}
			Thus, under the assumptions on $F$ and $\Sigma$, we obtain
			\begin{align*}
				&\Vert \text{II}(\RR{\mathcal{U}},z,w) -\text{II}(\RR{\widetilde{\mathcal{U}}},z,w) \Vert\\
				&\lesssim \big(\Vert \mathcal{U}^{\gY}(w)-\widetilde{\mathcal{U}}^{\gY}(w)\Vert+\Vert\eta_1(z,w)-\widetilde{\eta}_{1}(z,w)\Vert+\Vert\eta_2(z,w)-\widetilde{\eta}_{2}(z,w)\Vert \big)^{\nu}\\&\times\big(\Vert \eta_1(z,w)\Vert^2+\Vert \eta_1(z,w)\Vert \Vert \eta_2(z,w)\Vert\big)\\&+\big(\Vert\eta_1(z,w)-\widetilde{\eta}_{1}(z,w)\Vert+\Vert\eta_2(z,w)-\widetilde{\eta}_{2}(z,w)\Vert\big)\Vert\eta_1(z,w)\Vert\\&+\big(\Vert \widetilde{\eta}_{1}(z,w)\Vert+\Vert \widetilde{\eta}_{2}(z,w)\Vert\big)\Vert\eta_1(z,w)-\widetilde{\eta}_{1}(z,w)\Vert.
			\end{align*}
			Consequently, from \eqref{NASBs}, it follows that
			\begin{align*}
				&\sup_{\substack{z,w \in O_{T+s}^{s},z\neq w\\d(z,w)\leq \Vert z,w \Vert_{P_{s}}}} \frac{\Vert z,w \Vert_{P_{s}}^{\gamma_1} \Vert\text{II}(\RR{\mathcal{U}},z,w)-\text{II}(\RR{\widetilde{\mathcal{U}}},z,w)\Vert }{d(z,w)^{\gamma_1}}	
				\\&\lesssim   \big(\triplenorm{\mathbf{R}^{+}_{s} \RR{\mathcal{U}} - \mathbf{R}^{+}_{s}\RR{\widetilde{\mathcal{U}}} }_{\gamma_1,0,O_{T+s}^s}^{(s)}\big)^{\nu}\big(1+[\ \gI \ ]_{\overline{{O}_{T+s}^s}}\big)^\nu[\ \gI \ ]_{\overline{{O}_{T+s}^s}}\bigg(\triplenorm{\mathbf{R}^{+}_{s} \RR{\mathcal{U}} }_{\gamma_1,0,O_{T+s}^s}^{(s)}+[\ \gI \ ]_{\overline{{O}_{T+s}^s}}\bigg)\\&+\triplenorm{\mathbf{R}^{+}_{s} \RR{\mathcal{U}} - \mathbf{R}^{+}_{s}\RR{\widetilde{\mathcal{U}}} }_{\gamma_1,0,O_{T+s}^s}^{(s)}[\ \gI \ ]_{\overline{{O}_{T+s}^s}}\bigg(1+[\ \gI \ ]_{\overline{{O}_{T+s}^s}}+\triplenorm{ \mathbf{R}^{+}_{s}\RR{\widetilde{\mathcal{U}}} }_{\gamma_1,0,O_{T+s}^s}^{(s)}\bigg)\\&\lesssim M_{\nu}(\triplenorm{ \mathbf{R}^{+}_{s}\RR{\mathcal{U}} - \mathbf{R}^{+}_{s}\RR{\widetilde{\mathcal{U}}} }_{\gamma_1,0,O_{T+s}^s}^{(s)})\\&\times[\ \gI \ ]_{\overline{{O}_{T+s}^s}}\big(1+[\ \gI \ ]_{\overline{{O}_{T+s}^s}}\big)^\nu\bigg(1+[\ \gI \ ]_{\overline{{O}_{T+s}^s}}+\triplenorm{\mathbf{R}^{+}_{s} \RR{\mathcal{U}} }_{\gamma_1,0,O_{T+s}^s}^{(s)}+\triplenorm{\mathbf{R}^{+}_{s} \RR{\widetilde{\mathcal{U}}} }_{\gamma_1,0,O_{T+s}^s}^{(s)}\bigg).
			\end{align*}
			\item[3)] Analogously to the previous items, from the Taylor expansion we have
			\begin{align*}
				&\text{III}(\RR{\mathcal{U}},z,w)= \int_0^1\int_0^1
				F^{\prime\prime}\big(\mathcal{U}^{\gY}(w)+\alpha\beta\eta_2(z,w)\big)
				\,\alpha\,(\eta_2(z,w))^2
				\,\mathrm{d}\beta\,\mathrm{d}\alpha.
			\end{align*}
			Then
			\begin{align*}
				&\Vert \text{III}(\RR{\mathcal{U}},z,w) -\text{III}(\RR{\widetilde{\mathcal{U}}},z,w) \Vert\\
				&\lesssim \big(\Vert \mathcal{U}^{\gY}(w)-\widetilde{\mathcal{U}}^{\gY}(w)\Vert+\Vert\eta_2(z,w)-\widetilde{\eta}_{2}(z,w)\Vert \big)^{\nu}\Vert \eta_2(z,w)\Vert^2\\&+\big(\Vert {\eta}_{2}(z,w)\Vert+\Vert \widetilde{\eta}_{2}(z,w)\Vert\big)\Vert\eta_2(z,w)-\widetilde{\eta}_{2}(z,w)\Vert.
			\end{align*}
			Consequently, from \eqref{NASBs},
			\begin{align*}
				&\sup_{\substack{z,w \in O_{T+s}^{s}, z\neq w\\d(z,w)\leq \Vert z,w \Vert_{P_{s}}}} \frac{\Vert z,w \Vert_{P_{s}}^{\gamma_1} \Vert\text{III}(\RR{\mathcal{U}},z,w)-\text{III}(\RR{\widetilde{\mathcal{U}}},z,w)\Vert }{d(z,w)^{\gamma_1}}\\&\lesssim M_{\nu}(\triplenorm{ \mathbf{R}^{+}_{s}\RR{\mathcal{U}} -\mathbf{R}^{+}_{s} \RR{\widetilde{\mathcal{U}}} }_{\gamma_1,0,O_{T+s}^s}^{(s)})\bigg[\triplenorm{\mathbf{R}^{+}_{s} \RR{\mathcal{U}} }_{\gamma_1,0,O_{T+s}^s}^{(s)}+\big(\triplenorm{ \mathbf{R}^{+}_{s}\RR{\mathcal{U}} }_{\gamma_1,0,O_{T+s}^s}^{(s)}\big)^2+\triplenorm{ \mathbf{R}^{+}_{s}\RR{\widetilde{\mathcal{U}}} }_{\gamma_1,0,O_{T+s}^s}^{(s)}\bigg].	
			\end{align*}
		\end{itemize}
		Finally, combining these three items, we conclude that
		\begin{align}\label{Z11}
			\begin{split}
				&
				\frac{
					\Vert z,w \Vert_{P_{s}}^{\gamma_1}
					\, \big\Vert \mathrm{Pr}_{\gY}\big[
					\big(F(\RR{\mathcal{U}})-F(\RR{\widetilde{\mathcal{U}}})\big)(z)
					-\Gamma_{z,w}\big(\big(F(\RR{\mathcal{U}})-F(\RR{\widetilde{\mathcal{U}}})\big)(w)\big)
					\big]\big\Vert
				}{
					d(z,w)^{\gamma_1}
				}
				\\
				&\lesssim
				M_{\nu}\!\left(
				\triplenorm{  \mathbf{R}^{+}_{s}\RR{\mathcal{U}} -\mathbf{R}^{+}_{s} \RR{\widetilde{\mathcal{U}}} }_{\gamma_1,0,O_{T+s}^s}^{(s)}
				\right)
				P\big(
				[\ \gI \ ]_{\overline{{O}_{T+s}^s}},
				\triplenorm{\mathbf{R}^{+}_{s} \RR{\mathcal{U}} }_{\gamma_1,0,O_{T+s}^s}^{(s)},
				\triplenorm{ \mathbf{R}^{+}_{s}\RR{\widetilde{\mathcal{U}}} }_{\gamma_1,0,O_{T+s}^s}^{(s)}
				\big),
			\end{split}
		\end{align}
		where
		\begin{align*}
			P(b,c,d)=\max\left\lbrace (1+b)c+1, b(1+b)^\nu(1+b+c+d),c+c^2+d\right\rbrace.
		\end{align*}
		\textbf{Step 2.}\label{Step22}
		Let us now estimate
		\begin{align*}
			\sup_{\substack{z,w \in O_{T+s}^{s}, z\neq w\\ d(z,w)\leq \Vert z,w \Vert_{P_{s}}}}
			\frac{
				\Vert z,w \Vert_{P_{s}}^{\gamma_1}
				\, \big\Vert \mathrm{Pr}_{\gI}\big[
				\big(F(\RR{\mathcal{U}})-F(\RR{\widetilde{\mathcal{U}}})\big)(z)
				-\Gamma_{z,w}\big(\big(F(\RR{\mathcal{U}})-F(\RR{\widetilde{\mathcal{U}}})\big)(w)\big)
				\big]\big\Vert
			}{
				d(z,w)^{\gamma_1-1+\kappa}
			}.
		\end{align*}
		The method of proof is similar to that of the previous step. From \eqref{B1478}, \eqref{Hnasssss}, by a direct use of the Taylor expansion, we have
		\begin{align}\label{Aza41}
			\begin{split}
				&\mathrm{Pr}_{\gI}\big[
				F(\RR{\mathcal{U}})(z)-\Gamma_{z,w}\big(F(\RR{\mathcal{U}})(w)\big)
				\big]
				=F^\prime(\mathcal{U}^{\gY}(w))\big[\Sigma(\mathcal{U}^{\gY}(z))-\Sigma(\mathcal{U}^{\gY}(w))\big]+\Sigma(\mathcal{U}^{\gY}(z))\big[F^\prime(\mathcal{U}^{\gY}(z))-F^\prime(\mathcal{U}^{\gY}(w))\big]\\
				&=
				F^{\prime}(\mathcal{U}^{\gY}(w))
				\int_{0}^{1}
				\Sigma'\big(
				\alpha \mathcal{U}^{\gY}(z)
				+(1-\alpha)\mathcal{U}^{\gY}(w)
				\big)
				\mathrm{Pr}_{\gY}\big[
				\RR{\mathcal{U}}(z)-\Gamma_{z,w}\big(\RR{\mathcal{U}}(w)\big)
				\big]
				\,\mathrm{d}\alpha
				\\
				&\quad+
				F^{\prime}(\mathcal{U}^{\gY}(w))
				\int_{0}^{1}
				\Sigma'\big(
				\alpha \mathcal{U}^{\gY}(z)
				+(1-\alpha)\mathcal{U}^{\gY}(w)
				\big)
				\big[\eta_1(z,w)+\eta_2(z,w)\big]
				\,\mathrm{d}\alpha
				\\
				&\quad+
				\Sigma(\mathcal{U}^{\gY}(z))
				\int_{0}^{1}
				F^{\prime\prime}\big(
				\alpha \mathcal{U}^{\gY}(z)
				+(1-\alpha)\mathcal{U}^{\gY}(w)
				\big)
				\mathrm{Pr}_{\gY}\big[
				\RR{\mathcal{U}}(z)-\Gamma_{z,w}\big(\RR{\mathcal{U}}(w)\big)
				\big]
				\,\mathrm{d}\alpha
				\\
				&\quad+
				\Sigma(\mathcal{U}^{\gY}(z))
				\int_{0}^{1}
				F^{\prime\prime}\big(
				\alpha \mathcal{U}^{\gY}(z)
				+(1-\alpha)\mathcal{U}^{\gY}(w)
				\big)
				\big[\eta_1(z,w)+\eta_2(z,w)\big]
				\,\mathrm{d}\alpha .
			\end{split}
		\end{align}
		
		Recalling the assumptions on $F$ and $\Sigma$, we then derive
		\begin{align}\label{LLm}
			\begin{split}
				&\big\Vert \mathrm{Pr}_{\gI}\big[
				\big(F(\RR{\mathcal{U}})-F(\RR{\widetilde{\mathcal{U}}})\big)(z)
				-\Gamma_{z,w}\big(\big(F(\RR{\mathcal{U}})-F(\RR{\widetilde{\mathcal{U}}})\big)(w)\big)
				\big]\big\Vert
				\\
				&\lesssim
				M_{\nu}\big(
				\Vert \mathcal{U}^{\gY}(z)-\widetilde{\mathcal{U}}^{\gY}(z)\Vert
				+\Vert \mathcal{U}^{\gY}(w)-\widetilde{\mathcal{U}}^{\gY}(w)\Vert
				\big)
				\\
				&\qquad\times
				\bigg[
				\big\Vert\mathrm{Pr}_{\gY}\big[
				\RR{\mathcal{U}}(z)-\Gamma_{z,w}\big(\RR{\mathcal{U}}(w)\big)
				\big]\big\Vert
				+\Vert\eta_1(z,w)\Vert
				+\Vert\eta_2(z,w)\Vert
				\bigg]
				\\
				&\quad+
				\bigg[
				\big\Vert \mathrm{Pr}_{\gY}\big[
				(\RR{\mathcal{U}}-\RR{\widetilde{\mathcal{U}}})(z)
				-\Gamma_{z,w}\big((\RR{\mathcal{U}}-\RR{\widetilde{\mathcal{U}}})(w)\big)
				\big]\big\Vert
				+	\\
				&\qquad\qquad
				\Vert\eta_1(z,w)-\widetilde{\eta}_{1}(z,w)\Vert
				+\Vert\eta_2(z,w)-\widetilde{\eta}_{2}(z,w)\Vert
				\bigg].
			\end{split}
		\end{align}
		Keeping in mind that $\gamma_1 < 2-2\kappa$ and using \eqref{NASBs}, we deduce that
		\begin{align}\label{LLMM}
			\begin{split}
				&z,w \in O_{T+s}^{s}, z\neq w , \quad d(z,w)\leq \Vert z,w \Vert_{P_{s}}:
				\\
				&
				\frac{
					\Vert z,w \Vert_{P_{s}}^{\gamma_1}
					\Vert\eta_1(z,w)-\widetilde{\eta}_{1}(z,w)\Vert
				}{
					d(z,w)^{\gamma_1-1+\kappa}
				}
				\lesssim
				\Vert z,w \Vert_{P_{s}}^{\gamma_1}
				d(z,w)^{2-2\kappa-\gamma_1}
				[\ \gI \ ]_{\overline{{O}_{T+s}^s}}
				\triplenorm{
					\mathbf{R}^{+}_{s}\RR{\mathcal{U}}
					-\mathbf{R}^{+}_{s}\RR{\widetilde{\mathcal{U}}}
				}_{\gamma_1,0,O_{T+s}^s}^{(s)}
				\\
				&\qquad\lesssim
				[\ \gI \ ]_{\overline{{O}_{T+s}^s}}\,
				\triplenorm{
					\mathbf{R}^{+}_{s}\RR{\mathcal{U}}
					-\mathbf{R}^{+}_{s}\RR{\widetilde{\mathcal{U}}}
				}_{\gamma_1,0,O_{T+s}^s}^{(s)},
				\\[0.5em]
				&
				\frac{
					\Vert z,w \Vert_{P_{s}}^{\gamma_1}
					\Vert\eta_2(z,w)-\widetilde{\eta}_{2}(z,w)\Vert
				}{
					d(z,w)^{\gamma_1-1+\kappa}
				}
				\lesssim
				\sum_{i=1}^{d}
				\Vert z,w \Vert_{P_{s}}^{\gamma_1}
				d(z,w)^{2-\gamma_1-\kappa}
				\Vert \mathcal{U}^{\gXXXX}(w)-\widetilde{\mathcal{U}}^{\gXXXX}(w)\Vert
				\\
				&\qquad\lesssim
				\triplenorm{
					\mathbf{R}^{+}_{s}\RR{\mathcal{U}}
					-\mathbf{R}^{+}_{s}\RR{\widetilde{\mathcal{U}}}
				}_{\gamma_1,0,O_{T+s}^s}^{(s)}.
			\end{split}
		\end{align}
		Thus, from \eqref{LLm} and \eqref{LLMM}, we deduce that
		\begin{align}\label{Z22}
			\begin{split}
				&\sup_{\substack{z,w \in O_{T+s}^{s},z\neq w\\d(z,w)\leq \Vert z,w \Vert_{P_{s}}}} \frac{	\Vert z,w \Vert_{P_{s}}^{\gamma_1}\big\Vert 	\mathrm{Pr}_{\gI}\big[
					\big(F(\RR{\mathcal{U}})-F(\RR{\widetilde{\mathcal{U}}})\big)(z)-\Gamma_{z,w}\big(\big(F(\RR{\mathcal{U}})-F(\RR{\widetilde{\mathcal{U}}})\big)(w)\big)
					\big]\big\Vert }{d(z,w)^{\gamma_1-1+\kappa}}\\&\lesssim M_{\nu}\big(\triplenorm{  \mathbf{R}^{+}_{s}\RR{\mathcal{U}} -\mathbf{R}^{+}_{s} \RR{\widetilde{\mathcal{U}}} }_{\gamma_1,0,O_{T+s}^s}^{(s)}\big)\big[1+[\ \gI \ ]_{\overline{{O}_{T+s}^s}}+\triplenorm{ \mathbf{R}^{+}_{s}\RR{\mathcal{U}}}_{\gamma_1,0,O_{T+s}^s}^{(s)}\big].
			\end{split}
		\end{align}
		\textbf{Step 3.}\label{Step33}
		Let us now estimate
		\begin{align*}
			\sup_{\substack{z,w \in O_{T+s}^{s},z\neq w\\ d(z,w)\leq \Vert z,w \Vert_{P_{s}}}}\sum_{i=1}^{d}
			\frac{
				\Vert z,w \Vert_{P_{s}}^{\gamma_1}
				\, \big\Vert \mathrm{Pr}_{\gXXXX}\big[
				\big(F(\RR{\mathcal{U}})-F(\RR{\widetilde{\mathcal{U}}})\big)(z)
				-\Gamma_{z,w}\big(\big(F(\RR{\mathcal{U}})-F(\RR{\widetilde{\mathcal{U}}})\big)(w)\big)
				\big]\big\Vert
			}{
				d(z,w)^{\gamma_1-1}
			}.
		\end{align*}
		Analogous to \eqref{Aza41}, we have
		\begin{align*}
			&	\mathrm{Pr}_{\gXXXX}\big[
			F(\RR{\mathcal{U}})(z)-\Gamma_{z,w}(
			F(\RR{\mathcal{U}})(w)
			)\big]\\&= \mathcal{U}^{\gXXXX}(z)\big[F^\prime(\mathcal{U}^{\gY}(z))-F^\prime(\mathcal{U}^{\gY}(w))\big]+F^\prime(\mathcal{U}^{\gY}(w))\big[\mathcal{U}^{\gXXXX}(z)-\mathcal{U}^{\gXXXX}(w)\big]\\&
			=\mathcal{U}^{\gXXXX}(z)	\int_{0}^{1} F^{\prime\prime}\big(
			\alpha \mathcal{U}^{\gY}(z)
			+ (1-\alpha)\mathcal{U}^{\gY}(w)\big)\mathrm{Pr}_{\gY}\big[
			\RR{\mathcal{U}}(z)-\Gamma_{z,w}\big(
			\RR{\mathcal{U}}(w)
			\big)\big]\mathrm{d}\alpha \\&+\mathcal{U}^{\gXXXX}(z)	\int_{0}^{1} F^{\prime\prime}\big(
			\alpha \mathcal{U}^{\gY}(z)
			+ (1-\alpha)\mathcal{U}^{\gY}(w)\big)[ \eta_1(z,w)
			+\eta_2(z,w)\big]\mathrm{d}\alpha+F^\prime(\mathcal{U}^{\gY}(w))\underbrace{\big[\mathcal{U}^{\gXXXX}(z)-\mathcal{U}^{\gXXXX}(w)\big]}_{\mathrm{Pr}_{\gXXXX}\big[
				\RR{\mathcal{U}}(z)-\Gamma_{z,w}(
				\RR{\mathcal{U}}(w)
				)\big]} .
		\end{align*}
		Therefore, for each $1 \leq i \leq d$, it holds that
		
		\begin{align}\label{NNMass}
			\begin{split}
				&\big\Vert 	\mathrm{Pr}_{\gXXXX}\big[
				\big(F(\RR{\mathcal{U}})-F(\RR{\widetilde{\mathcal{U}}})\big)(z)-\Gamma_{z,w}\big(\big(F(\RR{\mathcal{U}})-F(\RR{\widetilde{\mathcal{U}}})\big)(w)\big)
				\big]\big\Vert\\ \lesssim
				&\Vert  \mathcal{U}^{\gXXXX}(z)-\widetilde{\mathcal{U}}^{\gXXXX}(z)\Vert\bigg[\big\Vert\mathrm{Pr}_{\gY}\big[
				\RR{\mathcal{U}}(z)-\Gamma_{z,w}\left(
				\RR{\mathcal{U}}(w)
				\right)\big]\big\Vert+\Vert\eta_1(z,w)\Vert+\Vert\eta_2(z,w)\Vert\bigg]\\&+\Vert\widetilde{\mathcal{U}}^{\gXXXX}(z)\Vert
				M_{\nu}\Big(
				\Vert \mathcal{U}^{\gY}(z)-\widetilde{\mathcal{U}}^{\gY}(z)\Vert
				+
				\Vert \mathcal{U}^{\gY}(w)-\widetilde{\mathcal{U}}^{\gY}(w)\Vert
				\Big)
				\\
				&\times
				\bigg[
				\big\Vert\mathrm{Pr}_{\gY}\big[
				\RR{\mathcal{U}}(z)-\Gamma_{z,w}\big(
				\RR{\mathcal{U}}(w)
				\big)\big]\big\Vert
				+\Vert\eta_1(z,w)\Vert
				+\Vert\eta_2(z,w)\Vert
				\bigg]\\&+\Vert\widetilde{\mathcal{U}}^{\gXXXX}(z)\Vert \bigg[\big\Vert\mathrm{Pr}_{\gY}\big[
				(\RR{\mathcal{U}}-\RR{\widetilde{\mathcal{U}}})(z)-\Gamma_{z,w}\big(
				(\RR{\mathcal{U}}-\RR{\widetilde{\mathcal{U}}})(w)
				\big)\big]\big\Vert+\Vert\eta_1(z,w)-\widetilde{\eta}_{1}(z,w)\Vert +\Vert\eta_2(z,w)-\widetilde{\eta}_{2}(z,w)\Vert\bigg]\\&+\Vert \mathcal{U}^{\gY}(w)-\widetilde{\mathcal{U}}^{\gY}(w)\Vert\Vert \mathcal{U}^{\gXXXX}(z)-\mathcal{U}^{\gXXXX}(w)\Vert+\big\Vert \mathrm{Pr}_{\gXXXX}\big[
				(\RR{\mathcal{U}}-\RR{\widetilde{\mathcal{U}}})(z)-\Gamma_{z,w}\big(
				(\RR{\mathcal{U}}-\RR{\widetilde{\mathcal{U}}})(w)
				\big)\big]\big\Vert.
			\end{split}
		\end{align}
		For the second line of \eqref{NNMass}, one can see that
		\begin{align*}
			&z,w \in O_{T+s}^{s}, \quad
			d(z,w)\leq \Vert z,w \Vert_{P_{s}}
			=\min\{\Vert z\Vert_{P_s},\Vert w\Vert_{P_s}\}:
			\\
			&\frac{\Vert z,w \Vert_{P_{s}}^{\gamma_1}}{d(z,w)^{\gamma_1-1}}
			\Vert \mathcal{U}^{\gXXXX}(z)-\widetilde{\mathcal{U}}^{\gXXXX}(z)\Vert
			\bigg[
			\big\Vert\mathrm{Pr}_{\gY}\big[
			\RR{\mathcal{U}}(z)-\Gamma_{z,w}\big(\RR{\mathcal{U}}(w)\big)
			\big]\big\Vert
			+\Vert\eta_1(z,w)\Vert
			+\Vert\eta_2(z,w)\Vert
			\bigg]
			\\
			&\lesssim
			d(z,w)\Vert \mathcal{U}^{\gXXXX}(z)-\widetilde{\mathcal{U}}^{\gXXXX}(z)\Vert
			\bigg[
			\frac{\Vert z,w \Vert_{P_{s}}^{\gamma_1}}{d(z,w)^{\gamma_1}}
			\big\Vert\mathrm{Pr}_{\gY}\big[
			\RR{\mathcal{U}}(z)-\Gamma_{z,w}\big(\RR{\mathcal{U}}(w)\big)
			\big]\big\Vert
			\bigg]
			\\
			&+
			[\ \gI\ ]_{\overline{{O}_{T+s}^s}}\,
			\Vert z,w \Vert_{P_{s}}^{\gamma_1}
			d(z,w)^{2-\kappa-\gamma_1}
			\Vert \mathcal{U}^{\gXXXX}(z)-\widetilde{\mathcal{U}}^{\gXXXX}(z)\Vert\\&+
			\Vert z,w \Vert_{P_{s}}^{\gamma_1}
			d(z,w)^{2-\gamma_1}
			\Vert \mathcal{U}^{\gXXXX}(z)-\widetilde{\mathcal{U}}^{\gXXXX}(z)\Vert
			\sum_{j=1}^{d}\Vert \mathcal{U}^{\gXXXXX}(w)\Vert
			\\
			&\lesssim
			[\ \gI\ ]_{\overline{{O}_{T+s}^s}}
			\triplenorm{ \mathbf{R}^{+}_{s}\RR{\mathcal{U}} -\mathbf{R}^{+}_{s} \RR{\widetilde{\mathcal{U}}}}_{\gamma_1,0,O_{T+s}^s}^{(s)}
			+
			\triplenorm{ \mathbf{R}^{+}_{s}\RR{\mathcal{U}} -\mathbf{R}^{+}_{s} \RR{\widetilde{\mathcal{U}}}}_{\gamma_1,0,O_{T+s}^s}^{(s)}
			\triplenorm{\mathbf{R}^{+}_{s}\RR{\mathcal{U}}}_{\gamma_1,0,O_{T+s}^s}^{(s)}.
		\end{align*}	
		The same argument applies to the other terms on the right-hand side of \eqref{NNMass}.
		Therefore, we conclude that
		\begin{align}\label{Z33}
			\begin{split}
				&\sum_{i=1}^d\sup_{\substack{z,w \in O_{T+s}^{s}, z\neq w\\d(z,w)\leq \Vert z,w \Vert_{P_{s}}}}\frac{\Vert z,w \Vert_{P_{s}}^{\gamma_1}\big\Vert 	\mathrm{Pr}_{\gXXXX}\big[
					\big(F(\RR{\mathcal{U}})-F(\RR{\widetilde{\mathcal{U}}})\big)(z)-\Gamma_{z,w}\big(\big(F(\RR{\mathcal{U}})-F(\RR{\widetilde{\mathcal{U}}})\big)(w)\big)
					\big]\big\Vert}{d(z,w)^{\gamma_1-1}}\\&
				\lesssim   M_{\nu}\big(\triplenorm{  \mathbf{R}^{+}_{s}\RR{\mathcal{U}} -\mathbf{R}^{+}_{s} \RR{\widetilde{\mathcal{U}}} }_{\gamma_1,0,O_{T+s}^s}^{(s)}\big)\big(1+ [\ \gI \ ]_{\overline{{O}_{T+s}^s}}+\triplenorm{ \mathbf{R}^{+}_{s}\RR{{\mathcal{U}}}}_{\gamma_1,0,O_{T+s}^s}^{(s)}\big)\big(1+\triplenorm{ \mathbf{R}^{+}_{s}\RR{\widetilde{\mathcal{U}}}}_{\gamma_1,0,O_{T+s}^s}^{(s)}\big).
			\end{split}
		\end{align}
		\textbf{Step 4.}\label{Step44} Now we can finish the proof. Note that, from \eqref{Hnasssss} and the assumptions on $F$ and $\Sigma$, for each
		$a\in \{\gY,\gI,\gXXXX: i=1,\dots,d\}$, we have
		\begin{align}\label{Z44}
			\sup_{z\in O_{T+s}^{s}}
			\Vert z\Vert_{P_s}^{|a|_{h}}
			\big\Vert\mathrm{Pr}_{a}\big[
			\big(F(\RR{\mathcal{U}})-F(\RR{\widetilde{\mathcal{U}}})\big)(z)
			\big]\big\Vert
			\lesssim
			\Big(
			1+\triplenorm{\mathbf{R}^{+}_{s}\RR{{\mathcal{U}}}}_{\gamma_1,0,O_{T+s}^s}^{(s)}
			\Big)
			\triplenorm{\mathbf{R}^{+}_{s}\RR{\mathcal{U}}-\mathbf{R}^{+}_{s}\RR{{\mathcal{U}}}}_{\gamma_1,0,O_{T+s}^s}^{(s)}.
		\end{align}
		By \eqref{Z11}, \eqref{Z22}, \eqref{Z33}, and \eqref{Z44}, we deduce that
		\begin{align*}
			&\triplenorm{\mathbf{R}^{+}_{s}F\big(\RR{\mathcal{U}}\big) - \mathbf{R}^{+}_{s}F\big(\RR{\widetilde{\mathcal{U}}}\big)}_{\gamma_1,0,O_{T+s}^s}^{(s)}\\&\lesssim  M_{\nu}\big(\triplenorm{  \mathbf{R}^{+}_{s}\RR{\mathcal{U}} -\mathbf{R}^{+}_{s} \RR{\widetilde{\mathcal{U}}} }_{\gamma_1,0,O_{T+s}^s}^{(s)}\big) Q\big(
			[\ \gI \ ]_{\overline{{O}_{T+s}^s}},
			\triplenorm{ \mathbf{R}^{+}_{s}\RR{\mathcal{U}} }_{\gamma_1,0,O_{T+s}^s}^{(s)},
			\triplenorm{ \mathbf{R}^{+}_{s}\RR{\widetilde{\mathcal{U}}} }_{\gamma_1,0,O_{T+s}^s}^{(s)}
			\big),
		\end{align*}
		where 
		\begin{align*}
			Q(b,c,d)&=\max\left\lbrace P(b,c,d),(1+b+c)(1+d)\right\rbrace\\&=\max\left\lbrace (1+b)c+1, b(1+b)^\nu(1+b+c+d),c+c^2+d,(1+b+c)(1+d)\right\rbrace.
		\end{align*}
		This concludes the proof.
	\end{proof}
	\subsection*{Acknowledgements} \label{sec:acknowledgements} The author gratefully acknowledges funding from the Deutsche Forschungsgemeinschaft (DFG, German Research Foundation) through CRC/TRR 388 ``Rough Analysis, Stochastic Dynamics and Related Fields'' (Project-ID 516748464). The author also thanks Alexandra Blessing Neam\c{t}u for helpful suggestions that improved the clarity of the paper. 
	\\
	\\
	\textbf{Declaration on the use of AI tools.}
	The author acknowledges the use of OpenAI tools for language and stylistic editing, as well as for checking notation, correcting typographical errors, and refining the presentation of standard definitions. All mathematical results, proofs, and underlying ideas were developed solely by the author.
	
	\bibliographystyle{alpha}
	\bibliography{refs}
	
\end{document}